\documentclass[12pt,a4paper]{book}
\usepackage{amsthm,amsfonts,amsmath,amssymb,latexsym}
\usepackage{epsfig,graphics,color}
\usepackage[utf8]{inputenc}
\usepackage[matrix,arrow,curve]{xy}
\usepackage[breaklinks=true]{hyperref}
\usepackage{tikz-cd}
\usepackage{verbatim}
\usepackage{bm}
\usepackage{enumitem}

\newcommand{\ri}{\rightarrow}

\newcommand{\calm}{{\mathcal{M}}}
\newcommand{\calf}{{\mathcal{F}}}
\newcommand{\calc}{{\mathcal{C}_*}}
\newcommand{\call}{{\mathcal{L}}}
\newcommand{\calp}{{\mathcal{P}}}
\newcommand{\cals}{{\mathcal{S}}}
\newcommand{\calb}{{\mathcal{B}}}
\newcommand{\cale}{{\mathcal{E}}}

\newcommand{\gol}{\, \, \, \, \, \, \, \, }

\newcommand{\bfs}{{\bf S}}
\newcommand{\piu}{\pi_{1}}

\newcommand{\wu}{\overline{W}^u}
\newcommand{\crit}{\mathrm{ Crit}}

\newtheorem{PARA}{}[section]
\newtheorem{theorem}[PARA]{Theorem}
\newtheorem{corollary}[PARA]{Corollary}
\newtheorem{lemma}[PARA]{Lemma}
\newtheorem{proposition}[PARA]{Proposition}
\newtheorem{definition}[PARA]{Definition}

\theoremstyle{definition}
\newtheorem{remark}[PARA]{Remark}
\theoremstyle{theorem}
\newtheorem{example}[PARA]{Example}
\theoremstyle{theorem}
\newtheorem{question}[PARA]{Question}

\newcommand{\para}{\begin{PARA}\rm}
\newcommand{\arap}{\end{PARA}\rm}
\newcommand{\dfn}{\begin{definition}\rm}
\newcommand{\nfd}{\end{definition}\rm}
\newcommand{\rmk}{\begin{remark}\rm}
\newcommand{\kmr}{\end{remark}\rm}
\newcommand{\xmpl}{\begin{example}\rm}
\newcommand{\lpmx}{\end{example}\rm}

\newcommand{\cB}{\mathcal{B}}

\newcommand{\cC}{\mathcal{C}}
\newcommand{\cD}{\mathcal{D}}
\newcommand{\cF}{\mathcal{F}}
\newcommand{\cE}{\mathcal{E}}
\newcommand{\cG}{\mathcal{G}}
\newcommand{\cH}{\mathcal{H}}

\newcommand{\cK}{\mathcal{K}}
\newcommand{\cL}{\mathcal{L}}
\newcommand{\cM}{\mathcal{M}}

\newcommand{\cP}{\mathcal{P}}

\newcommand{\cY}{\mathcal{Y}}

\usepackage[scr=stixfancy]{mathalfa}
\newcommand{\scro}{\mathscr{o}}

\newcommand{\ori}{\mbox{Or}}
\newcommand{\coori}{\mbox{Coor}}
\newcommand{\linecall}{\overline{\call}}

\newcommand{\h}{\boldsymbol{\mathrm{h}}}

\newcommand{\K}{{\mathbb{K}}}

\newcommand{\R}{{\mathbb{R}}}

\newcommand{\Z}{{\mathbb{Z}}}
\newcommand{\caly}{{\mathcal{Y}}}
\newcommand{\im}{\mathrm{im}\,}        
\newcommand{\id}{\mathrm{ Id}}         
\newcommand{\Id}{\mathrm{ Id}}
\newcommand{\ind}{\mathrm{ind}\,}

\newcommand{\pr}{{\mathrm{pr}}}
\newcommand{\ev}{\mathrm{ev}}
\newcommand{\Crit}{\mathrm{ Crit}}
\newcommand{\Vect}{\mathrm{ Vect}}        
\newcommand{\End}{\mathrm{ End}}          

\newcommand{\Per}{\mathrm{Per}}

\newcommand{\Hom}{\mathrm{Hom}}

\newcommand{\eps}{{\varepsilon}}

\def\NABLA#1{{\mathop{\nabla\kern-.5ex\lower1ex\hbox{$#1$}}}}
\def\Nabla#1{\nabla\kern-.5ex{}_{#1}}
\def\Tabla#1{\Tilde\nabla\kern-.5ex{}_{#1}}
\renewcommand{\Tilde}{\widetilde}

\newcommand{\p}{{\partial}}

\newcommand{\ol}{\overline}
\newcommand{\ul}{\underline}

\newcommand{\fibreprod}[2]{\sideset{_{#1}}{_{#2}}{\mathop{\times}}}

\definecolor{vincent}{rgb}{0,0,1.0}

\definecolor{dartmouthgreen}{rgb}{0.05, 0.5, 0.06}

\newcommand{\footremember}[2]{
\footnote{#2}
\newcounter{#1}
\setcounter{#1}{\value{footnote}}
}
\newcommand{\footrecall}[1]{
\footnotemark[\value{#1}]
} 
\title{Morse Homology with DG Coefficients}
\author{%
  Jean-Fran\c{c}ois Barraud\footremember{Toulouse}{IMT, Universit\'e de Toulouse}
  \and Mihai Damian\footremember{Strasbourg}{IRMA, Universit\'e de Strasbourg}
  \and Vincent Humili\`ere\footremember{Jussieu}{IMJ-PRG, Sorbonne Universit\'e}
  \and Alexandru Oancea\footrecall{Strasbourg}
}
\date{\today}

\begin{document}


%
%
%
%


\maketitle


\vspace{-.5cm} 

\paragraph{Abstract.}
We develop a theory of Morse homology and cohomology with coefficients in a derived local system, for manifolds and also more generally for colimits of spaces that have the homotopy type of manifolds, with a view towards Floer theory. The model that we adopt for derived, or differential graded (DG) local systems is that of DG modules over chains on the based loop space of a manifold. These encompass both classical (non DG) local systems and chains on fibers of Hurewicz fibrations. We prove that the Morse homology and cohomology groups that we construct are isomorphic to DG Tor and Ext functors. The key ingredient in the definition is a notion of twisting cocycle obtained by evaluating into based loops a coherent system of representatives for the fundamental classes of the moduli spaces of Morse trajectories of arbitrary dimensions. From this perspective, our construction sits midway between classical Morse homology with twisted coefficients and more refined invariants of Floer homotopical flavor. 

The construction of the twisting cocycle is originally due to Barraud and Cornea with $\mathbb{Z}/2$-coefficients. We show that the twisting cocycle with integer coefficients is equivalent to Brown's universal twisting cocycle. We prove that Morse homology with coefficients in chains on the fiber of a Hurewicz fibration recovers the homology of the total space of the fibration. We study several structural properties of the theory: invariance, functoriality, and Poincaré duality, also in the nonorientable case.



\setcounter{tocdepth}{1}
\tableofcontents

\vfill \pagebreak

\chapter{Introduction and main results}

Given a based and path-connected topological space $X$, we denote by $C_*(\Omega X)$ the DGA of cubical chains on the based Moore loop space $\Omega X$. A \emph{DG local system} on $X$ is a chain complex $\cF$ which is a DG right $C_*(\Omega X)$-module. The purpose of this book is to lay down the foundations of Morse homology and cohomology theory with coefficients in $\cF$ for smooth manifolds.
The main source of DG local
systems are fibrations over $X$ (see below), but our setup also includes classical (non-DG) local
systems, in which case we recover homology with local
coefficients (see~\S5.1.2). 
As explained below, this is equivalent to giving Morse models for the DG Tor and Ext functors 
$$
\mathrm{Tor}^{C_*(\Omega X)} (\cF,\Z),\qquad \mathrm{Ext}_{C_*(\Omega X)}(\Z,\cF).
$$

{\it The Morse complex with DG coefficients}. Assume that $X$ is a smooth closed manifold and choose auxiliary data consisting of a Morse function $f:X\to \R$, a Morse-Smale negative pseudo-gradient vector field $\xi$, and an embedded tree $\cY\subset X$ rooted at a basepoint and containing the set $\Crit(f)$ of critical points of $f$.
Denote $|x|$ the Morse index of a critical point $x$. Following 
Barraud-Cornea~\cite{BC07}, we associate to the triple $(f,\xi,\cY)$ a collection of chains on the based loop space $(m_{x, y})$, $x,y\in \Crit(f)$, $m_{x, y} \in C_{|x|-|y|-1}(\Omega X)$ such that 
$$
\p m_{x, y} = \sum_z (-1)^{|x|-|z|} m_{x, z} m_{z, y},
$$
where the multiplication of chains is understood with respect to the Pontryagin product in $C_*(\Omega X)$. We call $(m_{x, y})$ \emph{the Barraud-Cornea twisting cocycle}. Roughly speaking $m_{x, y}$ is obtained by evaluating into $\Omega (X/\cY)\simeq \Omega X$ the fundamental class rel boundary of the compactified moduli space $\ol\cL(x,y)$ of pseudo-gradient trajectories  from $x$ to $y$. Given a DG local system $\cF$, we further define 
a chain complex 
$$
C_*(X;\cF) = \cF\otimes \langle \Crit (f)\rangle,
$$ 
with differential 
$$
\p(\alpha\otimes x) = \p \alpha \otimes x + (-1)^{|\alpha|} \sum_y \alpha m_{x, y}\otimes y. 
$$
We refer to $C_*(X,\cF)$ as being \emph{the Morse complex with DG-coefficients}, or \emph{the enriched Morse Complex}. Its homology $H_*(X;\cF)$ is called  \emph{Morse homology with coefficients in $\cF$}, or \emph{enriched Morse homology}.

This construction is due to Barraud and Cornea for $\cF=C_*(\Omega X;\Z/2)$, with module structure given by right multiplication. In that case $H_*(X;\cF)$ is the homology of a point and the original point of view of~\cite{BC07} was to emphasize the associated spectral sequence as being the object of interest. In the current work we insist on the target of the spectral sequence as being an object of interest in its own right, and view the spectral sequence as a tool to approach it. Regardless of the module of coefficients used, we will refer in the sequel to the previous construction as \emph{the Barraud-Cornea construction}.  

{\it Fibrations}. The main source of DG-local systems are Hurewicz fibrations $F\hookrightarrow E\to X$. The holonomy of such a fibration is encoded in a \emph{holonomy map} $F\times \Omega X\to F$ determined by a (unique up to homotopy) choice of \emph{lifting function}. This induces a right $C_*(\Omega X)$-module structure on $\cF=C_*(F)$. 

{\bf Theorem~A (Fibration theorem)}. 
{\it Let $E\to X$ be a Hurewicz fibration and let $\cF$ be the associated DG-local system on $X$. 
There is a chain map 
	$$
\Psi:	C_*(X;\cF)\longrightarrow C_*(E)
	$$
 which induces an isomorphism between the canonical spectral sequence of the Morse complex determined by the index filtration, and the Leray-Serre spectral sequence of the fibration $E$. In particular $\Psi$ induces an isomorphism 
 $$
 \Psi_* :H_*(X;\cF)\stackrel\simeq\longrightarrow H_*(E). 
 $$
}

We discuss lifting functions in~\S\ref{sec:fibrations} and prove Theorem~A as Theorem~\ref{thm:fibration}. This statement was proved previously with $\Z/2$-coefficients, for the path-loop fibration by Barraud and Cornea in~\cite[Theorem~1.1.d]{BC07} and for Hurewicz fibrations by Charette in~\cite{Charette2017}. 

{\it A hybrid construction.} We find it important to emphasize the hybrid character of this construction: \emph{Morse data} on $X$ is combined with the DGA of \emph{cubical chains} on the space of based loops $\Omega X$. For the purpose of the discussion in this paragraph, we emphasize this by denoting $\cF=\cF^{\mathrm{cubical}}$, $C_*(\Omega X)=C_*^{\mathrm{cubical}}(\Omega X)$, $C_*(X;\cF)=C_*^{\mathrm{Morse}}(X;\cF^{\mathrm{cubical}})$. This construction sits midway between two other ones: 
\begin{center}
{\small
\begin{tabular}{c|c|c}
$C_*^{\mathrm{cubical}}(X;\cF^{\mathrm{cubical}})$ & $\boldsymbol{C_*^{\mathrm{Morse}}(X;\cF^{\mathrm{cubical}})}$ & $C_*^{\mathrm{Morse}}(X;\cF^{\mathrm{Morse}})$ \\
& & \\
Cubical chains on $X$ & {\bf Morse theory on $\boldsymbol{X}$} & Morse theory on $X$ \\
& & \\
$C_*^{\mathrm{cubical}}(\Omega X)$-module & {\bf $\boldsymbol{C_*^{\mathrm{cubical}}(\Omega X)}$-module} & $C_*^{\mathrm{Morse}}(\Omega X)$-module
\end{tabular}
}
\end{center}
As shown in the table, one could consider a theory which uses cubical chains on both $X$ and $\Omega X$, and also a theory which uses Morse data on both $X$ and $\Omega X$. These theories would be equivalent models for the use of DG local coefficients, but their implementation would require different techniques of varying levels of complexity. The ``all cubical" theory on the first column is classical and is essentially due to Brown~\cite{Brown1959}. The ``all Morse" theory on the third column has not yet been fully constructed, though parts of it can be recovered from the existing literature, e.g. the $A_\infty$-algebra structure on $C_*^{\mathrm{Morse}}(\Omega X)$, which can be emulated from the $A_\infty$-algebra structure of wrapped Floer homology~\cite{Abouzaid2012a}. 

One important motivation for considering a theory involving Morse data on $X$ is that it admits a direct symplectic generalization in Floer theory. Among the two options represented by the second and third column in the previous table, the one that we study in this book is by far technically simpler. We will apply it to Floer theory in a sequel paper~\cite{BDHO-cotangent}. 

{\it Range of applicability of the Barraud-Cornea construction}. This construction applies to any Floer theory (Lagrangian, instanton etc.) and can be used to enrich Floer maps in a variety of settings, for example in the context of Lagrangian correspondences. It constitutes an efficient way of extracting algebraic information from the higher dimensional moduli spaces of Morse/Floer/instanton trajectories. Barraud and Cornea used it in a Lagrangian setting in~\cite{BC07}, and in a sequel paper we will give applications in a Hamiltonian setting~\cite{BDHO-cotangent}. One of the roles of this book is to serve as a blueprint for the expected structural properties of enriched Floer complexes.

{\it Relation to Brown's universal twisting cocycle}. The Brown cocycle of $X$, traditionally viewed as a degree $-1$ map $\varphi:C_*(X)\to C_{*-1}(\Omega X)$, can be interpreted~\cite{Brown1959,Gugenheim60} as an element 
$$
\mathfrak{m}_\varphi \in \End_{-1}(C_*(\Omega X)\otimes C_*(X))
$$
that satisfies the relation 
$$
(\p + \mathfrak{m}_\varphi)^2=0,
$$ 
where $\p$ is the tensor differential on $C_*(\Omega X)\otimes C_*(X)$. We view $\mathfrak{m}_\varphi$ as defining a perturbation of the tensor differential $\p$ on $C_*(\Omega X)\otimes C_*(X)$. An equivalent point of view stems from the  differential graded Lie algebra (dgLa) structure on $\End_*(C_*(\Omega X)\otimes C_*(X))$. Denoting the differential $D$ and the bracket $[\cdot,\cdot]$, the previous relation for $\mathfrak{m}_\varphi$ is equivalent to 
$$
D\mathfrak{m}_\varphi + \frac 1 2 [\mathfrak{m}_\varphi,\mathfrak{m}_\varphi]=0.
$$ 
We say that $\mathfrak{m}_\varphi$ is a \emph{Maurer-Cartan element}. 

Similarly, subtracting from the Barraud-Cornea cocycle $\mathfrak{m}=(m_{x, y})$ the contribution of the differential of the classical Morse complex $C_*(f)$ we obtain a Maurer-Cartan element
$$
\mathfrak{m}' \in \End_{-1}(C_*(\Omega X)\otimes C_*(f))
$$
that satisfies the relation 
$$
(\p + \mathfrak{m}')^2=0,
$$ 
where $\p$ is the tensor differential on $C_*(\Omega X)\otimes C_*(X)$. We view therefore $\mathfrak{m}'$ as defining a perturbation of the tensor differential $\p$ on $C_*(\Omega X)\otimes C_*(f)$. 

The modified Barraud-Cornea cocycle $\mathfrak{m}'$ and the Brown cocycle $\mathfrak{m}_\varphi$ can be interpreted from a unifying perspective in the context of differential homological algebra: they both provide semi-free $C_*(\Omega X)$-resolutions of the trivial $C_*(\Omega X)$-module $\mathbb{Z}$ (see Definition~\ref{defi:semi-free-resolution}). This 
follows from Theorem~A, respectively~\cite[Theorem~4.2]{Brown1959}, which provide quasi-isomorphisms 
$$
(C_*(\Omega X)\otimes C_*(f),\p + \mathfrak{m}')\stackrel\simeq\longrightarrow C_*(\cP_{\star\to X}X)
$$
and
$$
(C_*(\Omega X)\otimes C_*(X),\p + \mathfrak{m}_\varphi)\stackrel\simeq\longrightarrow C_*(\cP_{\star\to X}X),
$$
where $\cP_{\star\to X}X$ is the space of paths in $X$ starting at the basepoint. This space is contractible, so that $C_*(\cP_{\star\to X}X)$ is chain homotopy equivalent to the trivial $C_*(\Omega X)$-module $\mathbb{Z}$.

Let $C_*$ be a complex of free $\mathbb{Z}$-modules. A Maurer-Cartan element $\mathfrak{m}\in\End_{-1}(C_*(\Omega X)\otimes C_*)$ is said to be a \emph{universal cocycle} if $(C_*(\Omega X)\otimes C_*,\p + \mathfrak{m})$ is a semi-free resolution of the trivial $C_*(\Omega X)$-module $\mathbb{Z}$. 
With this definition at hand, our second main result reads as follows.

{\bf Theorem~B (Equivalence of Barraud-Cornea and Brown cocycles)}. {\it 
The Barraud-Cornea and Brown cocycles are universal, and therefore equivalent in the sense that they give rise to chain homotopy equivalent semi-free resolutions of the trivial $C_*(\Omega X)$-module $\Z$. 
}

That universality implies equivalence is a consequence of the fact that any two semi-free resolutions of the same $C_*(\Omega X)$-module are chain homotopy equivalent. In~\S\ref{sec:BC-Brown} we prove Theorem~B as Theorem~\ref{thm:equivalence-BC-Brown}. 

Brown's twisting cocycle is a foundational discovery which paved the way for the development of homotopical algebra. It plays a key role in Koszul duality, which is in turn the conceptual path to defining algebraic structures up to homotopy~\cite{LodayVallette-AlgOp}. See~\S\ref{sec:Brown-twisting-MC} for further references. Our Theorem~B places its Morse theoretic analogue, the Barraud-Cornea cocycle, on an equal footing.

{\it Tor, Ext, and Poincar\'e duality}. As a consequence of this discussion, we see that the Morse chain complex with DG-coefficients is a model for the derived tensor product with the trivial $C_*(\Omega X)$-module $\Z$,
$$
C_*(X;\cF)\simeq \cF\stackrel{L}\otimes_{C_*(\Omega X)} \Z,
$$
denoted also $\mathrm{Tor}_*^{C_*(\Omega X)}(-,\Z)$. In~\S\ref{sec:cohomology-PD} we define the \emph{Morse cochain complex} 
$$
C^*(X;\cG)
$$
with DG coefficients in a (cohomological) right $C_{-*}(\Omega X)$-module $\cG$, and we show in Proposition~\ref{prop:Morse-coh-as-derived-hom} that it can be interpreted as a derived $\Hom$ functor. More precisely, given a  
right $C_{-*}(\Omega X)$-module $\cG$, we denote $\ol\cG$ the  
right $C_*(\Omega X)$-module obtained by regrading in opposite degree, and $\ol\cG^{\mathrm{left}}$ the  
left $C_*(\Omega X)$-module obtained from $\ol\cG$ using the canonical involution on $\Omega X$ given by parametrizing the loops backwards. Then 
$$
C^*(X;\cG)\simeq R\Hom_{C_*(\Omega X)}(\Z,\ol\cG^{\mathrm{left}}). 
$$
Our main result in this context is the following Poincar\'e duality theorem. A direct proof in the orientable case starting from the algebraic definitions of homology as derived $\otimes$ and cohomology  as derived $\Hom$ was previously given by Malm in his thesis~\cite[Theorem~3.1.2]{Malm-thesis}.

{\bf Theorem~C (Poincar\'e duality)}. {\it Let $X$ be a closed manifold of dimension $n$ and let $\ul\scro^X$ be its orientation local system supported in degree $0$. Given $\cF$ a homological DG local system, let $\ol\cF$ be the cohomological local system obtained from $\cF$ by reversing the sign of the grading. There is an isomorphism 
$$
PD: H_*(X;\cF) \stackrel\simeq\longrightarrow H^{n-*}(X;\ol\cF\otimes \ul\scro^X).
$$
}

We prove Theorem~C in the orientable case as Theorem~\ref{thm:PD-orientable}, and in the general case as Theorem~\ref{thm:PD-non-orientable}.

\emph{Homology vs. homotopy.} Homology with DG local coefficients stands midway between homology theory and homotopy theory. At a conceptual level, this comes from the role played by universal cocycles in homotopical algebra. 

The following example is significant. We consider $X=S^2$, for which usual homology is unable to see homotopical information beyond degree $2$. It is well-known that $\pi_3(S^2)\simeq \Z$, with a generator represented by the Hopf fibration $S^3\to S^2$. 
The fibration gives rise to a DG local system with fiber $C_*(S^1)$ and $H_*(S^2;C_*(S^1))\simeq H_*(S^3)$ by Theorem~A. The top dimensional generator of this homology group is an instantiation of the generator of $\pi_3(S^2)$.

{\it Connection to Floer homotopy theory and previous related work.} The construction of the Morse complex with DG coefficients relies heavily on the original ideas of Barraud and Cornea~\cite{BC07} and connects to Floer homotopy theory. This is a rapidly unfolding field envisioned by Cohen-Jones-Segal~\cite{cohen-jones-segal} and going back to Floer~\cite{Floer-Witten}. For recent results inspired by this perspective, see  Abouzaid-Blumberg~\cite{Abouzaid-Blumberg}, Large~\cite{Large-thesis}, Porcelli~\cite{Porcelli}, Hirschi-Porcelli~\cite{Hirschi-Porcelli}, Abouzaid-McLean-Smith \cite{AMS}.

Floer observed in~\cite[\S2, pp.~212-213]{Floer-Witten} that ``the trajectory spaces are framed rather than only oriented" and that, as a consequence, ``it should be possible to obtain [generalized Floer cohomology theories]" in which ``[the boundary operator], in contrast to the singular case, would depend on trajectory spaces of arbitrary dimension". Floer homotopy theory aims to  extract generalized cohomology theories, or spectra, from higher dimensional trajectory spaces. Using the terminology of Cohen-Jones-Segal~\cite{cohen-jones-segal}, the ensemble of all trajectory spaces has the structure of a \emph{flow category}, see also Pardon~\cite{Pardon-algebraic} and Abouzaid~\cite{Abouzaid-flows}. 

One can place our book in context by discussing how much information one retains from the Morse or Floer trajectory spaces: classical Morse and Floer theory use $0$-dimensional moduli spaces; Floer homotopy theory uses the full moduli spaces together with their framing; our Morse homology with DG-coefficients, and the Floer homology with DG-coefficients from~\cite{BDHO-cotangent}, uses (representatives of) the fundamental classes rel boundary of all the compactified moduli spaces. This gives further meaning to the previous statement that our construction stands midway between homology and homotopy.

The original construction of Barraud-Cornea has been also revisited in recent years by Zhou~\cite{Zhou-MB,Zhou-ring}, Rezchikov~\cite{rezchikov}, Charette~\cite{Charette2017}. There is a substantial overlap between our~\S\ref{sec:fibrations} and the work of Charette, who proves essentially the same theorem as our Theorem~A, but with $\Z/2$-coefficients. 

Morse homology with classical (non DG) local coefficients has been widely used in the literature. We refer to the recent article by Banyaga-Hurtubise-Spaeth~\cite{Banyaga-Hurtubise-Spaeth} for a recent treatment and a comprehensive list of references. 

{\it Functoriality}. The bulk of the technical work in this book goes into defining over $\Z$ the Barraud-Cornea twisting cocycle, and also continuation cocycles associated to homotopies of Morse functions and pseudo-gradient vector fields, and homotopy cocycles associated to homotopies of such homotopies, and in proving their functoriality properties. The general structure obtained in this way is laid out in~\S\ref{sec:Morse-DGMorse}, and the structural properties of the resulting homology theory are summarized in the following statement.

{\bf Theorem~D (Functoriality)}. 
{\it A continuous map between smooth closed manifolds $\varphi:X\to Y$ induces in homology a canonical map 
$$
\varphi_*:H_*(X;\varphi^*\cF)\to H_*(Y;\cF),
$$
where $\varphi^*\cF$ denotes the $C_*(\Omega X)$-module structure induced on $\cF$ by the DG map $(\Omega\varphi)_*:C_*(\Omega X)\to C_*(\Omega Y)$. The map $\varphi_*$ has the following properties: 
\begin{enumerate} 
\item {\sc (Identity)} We have $\mathrm{Id}_*=\mathrm{Id}$. 
\item {\sc (Composition)} Given maps $X\stackrel{\varphi}\longrightarrow Y \stackrel{\psi}\longrightarrow Z$ and a DG local system $\cF$ on $Z$, 
we have 
$$
(\psi\varphi)^*\cF=\varphi^*\psi^*\cF
$$
and
$$
(\psi\varphi)_* = \psi_*\varphi_* : H_*(X;\varphi^*\psi^*\cF)\to H_*(Z;\cF).
$$
\item {\sc (Homotopy)} A homotopy of maps $\varphi_{[0,1]} =(\varphi_t)$, $t\in [0,1]$ induces a 
canonical isomorphism 
$$
\varphi_{[0,1]*}:H_*(X;\varphi_0^*\cF)\stackrel\simeq\longrightarrow H_*(X;\varphi_1^*\cF)
$$
such that $\varphi_{1*}\circ \varphi_{[0,1]*} = \varphi_{0*}$. 
\item {\sc (spectral sequence)} The morphism $\varphi_*$ is the limit of a morphism between spectral sequences that are canonically associated to the corresponding enriched complexes, given at the second page by 
$$ \varphi_{p,*}: H_{p}( X ; \varphi^*H_q(\calf) )\ri H_p( Y; H_q(\calf))$$
i.e., the usual direct map induced by $\varphi$ in (Morse)  homology with coefficients in $H_q(\calf)$. 
\end{enumerate}
}

We prove Theorem~D as Theorem~\ref{thm:f*!}. Similar properties hold for shriek maps $\varphi_!$ in homology (Theorem~\ref{thm:f*!}) and for maps $\varphi^*$ induced in cohomology (Proposition~\ref{prop:f_upper_*}). These are related to each other via Poincar\'e duality, and this whole discussion is carried over in~\S\S\ref{sec:functoriality}-\ref{sec:second-definition}.

{\it Further developments and questions}. There are many directions of study which open up from the current book. We mention the following.
\begin{itemize}
\item {\it Floer theory}. This book on Morse theory serves as a blueprint for Floer theory. The authors address Floer homology with DG-coefficients in their work~\cite{BDHO-cotangent}.  
\item {\it Quantitative aspects}. Morse and Floer complexes with DG-coefficients are naturally filtered. It would be interesting to develop applications, including in interaction with barcodes. 
\item {\it Relation to homotopy theory}. (i) How much of the homotopy of $X$ can be seen by fibrations over $X$ and their chain complexes? (ii) Study homology with DG coefficients in relation with generalized homology theories like stable homotopy theory and K-theory. This connects to the previous discussion on Floer homotopy theory. 
\item {\it String topology}. The homology $H_*(X;C_*(\Omega X)^{\mathrm{ad}})$ with DG-coefficients in $C_*(\Omega X)$ seen as a module over itself via the adjoint action, denoted $C_*(\Omega X)^{\mathrm{ad}}$, 
is isomorphic to the homology of the free loop space $\cL X$.  Robin Riegel is developing in his Ph.D. thesis~\cite{Riegel} a model for string topology using Morse theory with DG-coefficients. 
Relevant references for this perspective are~\cite{Gruher-Salvatore,Malm-thesis}.
\end{itemize}

We end this introduction with a discussion of our setup involving $C_*(\Omega X)$-modules. 

{\it Based loop vs. paths, DG-algebra vs. DG-category}. Morse trajectories evaluate naturally into Moore paths and they can be used to define morphisms in the DG-category $Path$ described as follows: the objects are given by the critical points of the Morse function, and the morphisms are given by (cubical) chains on the spaces of paths connecting these points. DG-local systems can be described as derived representations of the category $Path$, and it is possible to recast the results presented in this book in that language. 

In contrast, we use a DGA instead of a DG-category (chains on based loops instead of the category $Path$), and we use DG-modules instead of derived representations of $Path$. This results in a drastic simplification at the level of homological algebra. The price to pay is that we need to choose as additional data an embedded tree that connects the basepoint to the critical points of the Morse function. This choice results in evaluation maps that have a slightly more cumbersome form, involving a homotopy inverse of the map which collapses the tree to a point. Also, we need to additionally prove invariance of our constructions with respect to that choice of tree.

{\it $C_*(\Omega X)$-modules \emph{vs.} representations of the infinity groupoid}. It is a well-known fact that classical local systems on a manifold $X$ can equivalently be viewed either as $\Z[\pi_1(X)]$-modules, or as bundles of groups with trivial holonomy. In a similar fashion, DG-local systems, which we defined as differential graded $C_*(\Omega X)$-modules, should be equivalently viewed as representations of the full homotopy type of $X$, i.e. representations of the infinity groupoid $\pi_\infty(X)$. This was proved in full generality by 
Holstein~\cite[Theorem~26]{Holstein}. 
The topic is discussed from a broader perspective by Arias Abad and Sch\"{a}tz~\cite{Arias-Abad-Schatz}. See also Porta and Teyssier~\cite[Corollary~1.5]{Porta-Teyssier}.

{\it Structure of the book}. We start in~\S\ref{sec:Morse-DGMorse} with a comparison between the classical Morse homology and the DG Morse homology toolsets. This chapter can serve as a reading guide for chapters~\S\S\ref{section:Morse}-\ref{sec:invariance-homology}, in which we  construct the Morse chain complex with DG coefficients and prove its invariance. In~\S\ref{sec:BC-Brown} we prove Theorem~B, showing the equivalence between the Barraud-Cornea twisting cocycle and the Brown cocycle. This gives topological meaning and context for the construction of the Morse complex with DG coefficients. Chapter~\S\ref{sec:algebraicsetup} is of preparatory nature, and we discuss therein various algebraic properties of complexes of DG Morse type. In~\S\ref{sec:fibrations} we prove Theorem~A, showing that DG Morse homology recovers the homology of total spaces of fibrations. This is a fundamental result, which is also used in the proof of Theorem~B. In~\S\ref{sec:functoriality} we spell out various functoriality properties of Morse homology with DG-coefficients, involving both direct and shriek maps. These maps are constructed and their properties are established from two equivalent perspectives in~\S\S\ref{sec:functoriality-first-definition}-\ref{sec:second-definition}, which cover in particular Theorem~D. In~\S\ref{sec:degree} we discuss the functoriality properties of maps between closed manifolds of equal dimensions. In~\S\ref{sec:cohomology-PD} we define cohomology and prove the Poincar\'e duality Theorem~C in the orientable case. In~\S\ref{sec:non-orientable} we construct shriek maps and prove the Poincar\'e duality Theorem~C for non-orientable manifolds.  In~\S\ref{sec:beyond-manifolds} we extend our constructions (in particular Theorems A and D) to topological spaces that are homotopy equivalent to finite dimensional manifolds and to colimits of such spaces; we then apply this to free loop spaces.

{\it Disclaimer and acknowledgements.} The debt owed to~\cite{BC07} by this book is obvious. In order to acknowledge this, and also to stress the novelty of the main construction from~\cite{BC07}, the second, third, and fourth authors of this book (Damian, Humili\`ere, Oancea) insisted with the first author that we refer to it as the ``Barraud-Cornea construction/cocycle".

The authors were able to meet in excellent conditions during the Covid pandemic at CIRM Marseille in February 2021 and at IHP Paris in May-June 2021. They acknowledge partial funding from the ANR, grants no. 18-CE40-0009 (ENUMGEOM ) and 21-CE40-0002 (COSY). The third author is also partially supported by Institut Universitaire de France. The fourth author acknowledges support from a Fellowship of the University of Strasbourg Institute for Advanced Study (USIAS), within the French national programme ``Investment for the future" (IdEx-Unistra).

\chapter{Morse vs DG Morse homology toolset} \label{sec:Morse-DGMorse}

In this chapter we present the structure of the enriched theory at homology level, which we call ``the DG Morse homology toolset" (we do not address higher chain level structures in this book). It consists of data defining the differential, continuation maps, and homotopy maps. We start the discussion with a few definitions involving Maurer-Cartan elements. These allow us to place on the same formal footing the classical Morse toolset and the DG one, in the hope that this will render the new objects more familiar to the reader. 

\section{$\End$ and $\Hom$ complexes} 

Let $(C_*,\p)$ be a chain complex graded by the integers $\Z$, with differential $\p$ of degree $-1$. The graded $\text{End}$ complex
$\End(C_*)$ is a differential graded Lie algebra (dgLa). The degree $k$ component $\End_k$ consists of degree $k$ linear maps $C_*\to C_{*+k}$, the bracket is $[f,g]=f\circ g - (-1)^{|f|\cdot |g|}g\circ f$, and the differential is $D=[\p,\cdot]$. 

Let $(C_*^\pm,\p^\pm)$ be chain complexes. The $\Hom$ complex $\Hom(C_*^+,C_*^-)$ is a differential graded module, whose degree $k$ component is the space of degree $k$ linear maps $C_*^+\to C_{*+k}^-$, and with differential $D$ given on a homogeneous element by $Df=\p^-f - (-1)^{|f|}f\p^+$. We denote symbolically $Df=[\p,f]$. The $\Hom$ complex is a left Lie module over $\End(C^-)$ by post-composition and a right Lie module over $\End(C^+)$ by pre-composition.

{\bf Twisted differential}. A \emph{Maurer-Cartan element} for $(C_*,\p_0)$ is an element $\mathfrak{m}\in \End_{-1}(C_*)$ satisfying the \emph{Maurer-Cartan equation}
$$
D\mathfrak{m}+\frac 1 2 [\mathfrak{m},\mathfrak{m}]=0.
$$
This relation is equivalent to $(\p_0+\mathfrak{m})^2=0$, so that Maurer-Cartan elements can (and should) be viewed as perturbations of the differential $\p_0$ on $C_*$. We refer to $\p=\p_0+\mathfrak{m}$ as the \emph{twisted differential}. 

This is the most trivial instance of a very general philosophy initiated by Deligne~\cite{Deligne}, 
according to which any deformation problem is ``controlled" by a dgLA through the intermediate of its Maurer-Cartan elements. 

{\bf Chain maps}. Consider now complexes $(C_*^i,\p^i_0)$ with Maurer-Cartan elements $\mathfrak{m}^i$ for $i=0,1$,  denote $\p^i=\p^i_0+\mathfrak{m}^i$ the corresponding perturbed differentials and denote $D$ the differential on the complex $\Hom((C_*^0,\p^0),(C_*^1,\p^1))$. Since we use on $C^i$, $i=0,1$ the twisted differentials, we refer to this as the \emph{twisted $\Hom$ complex}. A degree $0$ chain map $\Psi:(C_*^0,\p^0)\to (C_*^1,\p^1)$ satisfies by definition $\p^1\Psi=\Psi\p^0$, and this is equivalent to 
$$
D\Psi=0,
$$
i.e., $\Psi\in\Hom_0((C_*^0,\p^0),(C_*^1,\p^1))$ is a cycle in the twisted $\Hom$ complex. 

{\bf Homotopies}. 
Given chain maps $\Psi_0, \Psi_1\in \Hom_0((C_*^0,\p^0),(C_*^1,\p^1))$ so that $D\Psi_0=D\Psi_1=0$, a chain homotopy between $\Psi_0$ and $\Psi_1$ is a degree $1$ map $\h :(C_*^0,\p^0)\to (C_*^1,\p^1)$ such that $\Psi_1 - \Psi_0 = \p^1 \h + \h \p^0$. Equivalently
$$
\Psi_1 - \Psi_0 = D \h,
$$
i.e., $\Psi_1 - \Psi_0$ is a boundary in the twisted $\Hom$ complex with primitive $\h\in\Hom_1((C_*^0,\p^0),(C_*^1,\p^1))$.

\section{The Morse toolset} 

Let $X$ be a closed manifold. The standard toolset in classical Morse homology consists of the differential, continuation maps and homotopy maps, see for example~\cite{Audin-Damian_English}. With a view towards the DG setup, we explain how to interpret these in the language of Maurer-Cartan elements and twisted $\Hom$ complexes of the previous section.

{\bf Differential}. Let $(f,\xi)$ be a regular pair consisting of a Morse function $f$ and a Morse-Smale pseudo-gradient $\xi$, and denote 
$$
C_\bullet=\langle \Crit(f)\rangle,
$$
the free $\Z$-module generated by the critical points of $f$.  

The Morse differential $\p$ can be seen as a perturbation of the zero differential $\p_0=0$ by a Maurer-Cartan element $\mathfrak{m}\in \End(C_\bullet)$. This Maurer-Cartan element is the matrix $\mathfrak{m}=(n_{x, y})$, $x,y\in\Crit(f)$ whose entries are given by the signed count of gradient trajectories between points $x$ and $y$ such that $|x|-|y|=1$, and are zero otherwise. The Maurer-Cartan relation amounts to the fact that this matrix squares to zero, and we have $\p=\mathfrak{m}$.  

The resulting twisted complex 
$$
MC_*(f,\xi)=(\langle \Crit(f)\rangle,\p)
$$ 
is \emph{the Morse complex}.

{\bf Continuation maps}. Let $(f_t,\xi_t)$, $t\in [0,1]$ be a regular homotopy interpolating between two regular pairs $(f_i,\xi_i)$, $i=0,1$. To this data we associate a degree $0$ \emph{continuation map} 
$$
\Psi:MC_*(f_0,\xi_0)\to MC_*(f_1,\xi_1).
$$ 
This is a chain map, i.e., a cycle in $\Hom(MC_*(f_0,\xi_0),MC_*(f_1,\xi_1))$. 

The map $\Psi$ is a matrix whose rows and columns are indexed by $\Crit(f_1)$, resp. $\Crit(f_0)$. The entries of this matrix can be defined as signed counts of pseudo-gradient trajectories in $[0,1]\times X$ for the function $(t,x)\mapsto g(t) + f_t(x)$, where $g:[0,1]\to \R$ is smooth with a critical maximum at $0$ and a critical minimum at $1$.

{\bf Homotopies}. Let $\{(f_{\tau,t},\xi_{\tau,t})\}$, $\tau\in[0,1]$ be a regular homotopy of homotopies connecting two regular homotopies $(f_{0,t}, \xi_{0,t})$ and $(f_{1,t}, \xi_{1,t})$ which interpolate between $(f_0,\xi_0)$ and $(f_1,\xi_1)$. To this data we associate a degree $1$ map 
$$
\h:MC_*(f_0,\xi_0)\to MC_{*+1}(f_1,\xi_1)
$$ 
which is a chain homotopy between the continuation chain maps $\Psi_0$ and $\Psi_1$ determined by $(f_{0,t}, \xi_{0,t})$, resp.  $(f_{1,t}, \xi_{1,t})$. Thus $\Psi_1 - \Psi_0$ is a boundary in the twisted $\Hom$ complex $\Hom(MC_*(f_0,\xi_0),MC_*(f_1,\xi_1))$, and $\h$ is a primitive of $\Psi_1 - \Psi_0$.

The map $\h$ is a matrix whose rows and columns are indexed by $\Crit(f_1)$, resp. $\Crit(f_0)$. The entries of this matrix can be defined as signed counts of pseudo-gradient trajectories in $[0,1]\times [0,1] \times X$ for the function $(\tau,t,x)\mapsto g(\tau) + g(t) + f_{\tau,t}(x)$, where $g:[0,1]\to \R$ is smooth with a critical maximum at $0$ and a critical minimum at $1$.

\section{The DG Morse toolset}\label{sec:DG-Morse}

The DG Morse homology toolset consists of the same kind of data as the Morse homology one, i.e., the differential, continuation maps and homotopies, except that the coefficients are enriched in cubical chains $R_*=C_*(\Omega X)$, and more generally in a DG right $R_*$-module $\cF_*$.

{\bf Differential}. To a Morse-Smale pair $(f,\xi)$ consisting of a Morse function and a generic pseudo-gradient, we associate a Maurer-Cartan element $\mathfrak{m}\in\End_{-1}(R_*\otimes C_\bullet)$, where $C_\bullet=\langle \Crit(f)\rangle$ is endowed with the zero differential as in the previous section. 

We view $\mathfrak{m}$ as a square matrix $(m_{x,y})$ acting from the right with coefficients in $R_*$ and indexed by the set $\Crit(f)$ which constitutes the distinguished basis of $C_\bullet$; in other words we write   
\[\mathfrak{m}(a\otimes x)=\sum_{y\in\Crit(f)}(-1)^{|a|}am_{x,y}\otimes y .\]
Intuitively, the chain $m_{x,y}\in C_{|x|-|y|-1}(\Omega X)$ is obtained by evaluating into based loops a chain level representative of the compactified moduli space $\overline\cL(x,y)$ of connecting gradient trajectories \emph{of arbitrary index} which run from $x$ to $y$. We call $\mathfrak{m}$ \emph{the Barraud-Cornea cocycle}. See~\S\ref{sec:construction-enriched-morse} for the details of the construction. 

\begin{proposition} \label{prop:MC-formula-equation}
The Maurer-Cartan equation 
$$
D\mathfrak{m}+ \frac 1 2 [\mathfrak m,\mathfrak m] =0
$$
is equivalent to the relation 
\begin{equation} \label{eq:MC-first-section}
\p m_{x,y}=\sum_z (-1)^{|x|-|z|}m_{x,z}m_{z,y}.
\end{equation}
\end{proposition}

\begin{proof}
Let $a\in R_*$ be homogeneous of degree $|a|$ and $x\in\Crit(f)$. The twisted differential acts by $\p(a\otimes x)=\p a \otimes x + (-1)^{|a|}\sum_y a m_{x,y}\otimes y$. We have 
\begin{align*}
\p\mathfrak m(a\otimes x) & = (\p \mathfrak m + \mathfrak m \p)(a\otimes x) \\
& = \p\left( (-1)^{|a|} a \otimes \mathfrak m(x)\right) + \mathfrak m(\p a\otimes x) \\
& = (-1)^{|a|}\p( \sum_y a m_{x,y}\otimes y) + (-1)^{|a|-1}\sum_y (\p a) m_{x,y}\otimes y \\
& = \sum_ya \p m_{x,y}\otimes y. 
\end{align*}
On the other hand $\frac 1 2 [\mathfrak m,\mathfrak m]= \mathfrak m\circ \mathfrak m$ computes to 
\begin{align*}
(\mathfrak m\circ \mathfrak m)(a\otimes x) & = \mathfrak m ( (-1)^{|a|}\sum_z a m_{x,z}\otimes z) \\
& = \sum_z (-1)^{|x|-|z|-1}\sum_y a m_{x,z}m_{z,y}\otimes y \\
& = \sum_y  a (\sum_z (-1)^{|x|-|z|-1}m_{x,z}m_{z,y}) \otimes y.
\end{align*}
It is clear now that the relation $\p m_{x,y}=\sum_z (-1)^{|x|-|z|}m_{x,z}m_{z,y}$ implies the Maurer-Cartan equation. The converse is seen to be true by taking $a$ to be the unit (basepoint).  
\end{proof}

Given a right $R_*$-module $\cF_*$, we induce tautologically a Maurer-Cartan element 
$$
\mathfrak{m}^\cF\in \End_{-1}(\cF_*\otimes C_\bullet).
$$ 
The resulting twisted complex  
$$
C_*(f,\xi;\cF)=(\cF_*\otimes C_\bullet,\p),
$$ 
$$
\p(\alpha\otimes x)=\p \alpha \otimes x + (-1)^{|\alpha|}\sum_y \alpha m_{x,y}\otimes y
$$
is \emph{the Morse complex with coefficients in the DG local system $\cF_*$}. 

There are choices involved in the construction of the Barraud-Cornea cocycle, and we discuss the ambiguity in~\S\S\ref{sec:construction-enriched-morse}--\ref{sec:invariance-homology}. 

\begin{remark} An inspection of the proof of Proposition~\ref{prop:MC-formula-equation} shows that, in this context, the Maurer-Cartan relation~\eqref{eq:MC-first-section} can be replaced with the slightly weaker one 
\begin{equation*}\label{MC'} 
\gol  \forall \alpha\in \calf, \gol \alpha\cdot (\partial m_{x,y}\, -\, \sum_z(-1)^{|x|-|z|}m_{x,z}m_{z,y}) \, =\, 0.
\end{equation*}
\end{remark}  

{\bf  Continuation maps}. Let $(f_t,\xi_t)$, $t\in[0,1]$ be a regular homotopy interpolating between two regular pairs $(f_0,\xi_0)$ and $(f_1,\xi_1)$. Let $\mathfrak{m}^i$, $i=0,1$ be a choice of Barraud-Cornea cocycles. To this data we associate a \emph{continuation map}, also called \emph{continuation cocycle}, 
$$
\Psi:C_*(f_0,\xi_0;R_*)\to C_*(f_1,\xi_1;R_*).
$$
This is a degree $0$ cycle in $\Hom(C_*(f_0,\xi_0;R_*), C_*(f_1,\xi_1;R_*))$, i.e., a chain map. 

We view $\Psi$ as a matrix $(\nu_{x_0,y_1})$ acting from the right with columns indexed by $x_0\in\Crit(f_0)$ and rows indexed by $y_1\in\Crit(f_1)$. Intuitively, the chain $\nu_{x_0,y_1}\in C_{|x_0|-|y_1|}(\Omega X)$ is obtained by evaluating into based loops a chain level representative of the fundamental class of the compactified moduli space of pseudo-gradient trajectories \emph{of arbitrary index} in $[0,1]\times X$, running from $(0,x_0)$ to $(1,y_1)$, for the function $(t,x)\mapsto g(t) + f_t(x)$, where $g:[0,1]\to \R$ is smooth with a critical maximum at $0$ and a critical minimum at $1$. See~\S\ref{sec:invariance-enriched} for the details of the construction.

\begin{proposition} \label{prop:continuation-DG}
The relation 
$$
D\Psi = 0
$$
is equivalent to 
\begin{equation} \label{eq:DGcont}
\p \nu_{x_0,y_1}=\sum_{z_0} m^0_{x_0,z_0}\nu_{z_0,y_1}+ \sum_{z_1} (-1)^{|x_0|-|z_1|-1} \nu_{x_0,z_1}m^1_{z_1,y_1}.
\end{equation}
\end{proposition}

\begin{proof}
Note that $D\Psi= \p^1 \Psi - \Psi \p^0$ since $\Psi$ has degree $0$. The map $\Psi$ acts by 
$\Psi(a\otimes x_0)=\sum_{y_1}a \nu_{x_0,y_1}\otimes y_1$, and we have 
\begin{align*}
D& \Psi(a \otimes x_0)\\
&=((\p+\mathfrak{m}^1)\Psi - \Psi (\p + \mathfrak{m}^0))(a\otimes x_0) \\
& = (\p+\mathfrak{m}^1)(\sum _{y_1}a \nu_{x_0,y_1} \otimes y_1)  \\
& \qquad \qquad - \Psi (\p a\otimes x_0) - \Psi((-1)^{|a|}\sum_{z_0}a m^0_{x_0,z_0}\otimes z_0) \\
& = \sum _{y_1}(\p a) \nu_{x_0,y_1} \otimes y_1 + \sum _{y_1} (-1)^{|a|}a (\p \nu_{x_0,y_1}) \otimes y_1 \\
& \qquad + \sum_{y_1} \sum _{z_1} (-1)^{|a|}(-1)^{|x_0|-|z_1|} a \nu_{x_0,z_1} m^1_{z_1,y_1}\otimes y_1 \\
& \qquad - \sum _{y_1}(\p a) \nu_{x_0,y_1} \otimes y_1 - \sum_{y_1}\sum_{z_0} (-1)^{|a|} a m^0_{x_0,z_0} \nu_{z_0,y_1}\otimes y_1 \\
& = (-1)^{|a|}\sum _{y_1} a \big(\p \nu_{x_0,y_1}  - \sum_{z_0}m^0_{x_0,z_0} \nu_{z_0,y_1} \\
& \qquad \qquad \qquad \qquad + \sum _{z_1} (-1)^{|x_0|-|z_1|} \nu_{x_0,z_1} m^1_{z_1,y_1}\big) \otimes y_1.
\end{align*}
The conclusion follows.
\end{proof}

Given a right $R_*$-module $\cF_*$, we induce a continuation cycle $\Psi^\cF$ in the twisted $\Hom$ complex $\Hom(C_*(f_0,\xi_0;\cF_*),C_*(f_1,\xi_1;\cF_*))$. This in turn induces a continuation chain map 
$$
\Psi^\cF:C_*(f_0,\xi_0;\cF_*)\to C_*(f_1,\xi_1;\cF_*),
$$ 
which acts by
$$
\Psi^\cF(\alpha\otimes x_0)=\sum_{y_1}\alpha \nu_{x_0,y_1}\otimes y_1.
$$

There are choices involved in the construction of the continuation cocycle, and we discuss the ambiguity in~\S\ref{sec:invariance-continuation}.

{\bf Homotopies}. Let $\{(f_{\tau,t},\xi_{\tau,t})\}$, $\tau\in [0,1]$ be a regular homotopy of homotopies connecting two regular homotopies $(f_{0,t}, \xi_{0,t})$ and $(f_{1,t}, \xi_{1,t})$ which interpolate between $(f_0,\xi_0)$ and $(f_1,\xi_1)$. Let $\mathfrak{m}^i=(m^i_{x,y})$ be a choice of Barraud-Cornea cocycles for $(f_i,\xi_i)$, $i=0,1$, and let $\Psi_0=(\nu^0_{x_0,y_1})$, $\Psi_1=(\nu^1_{x_0,y_1})$ be a choice of continuation cocycles for $(f_{0,t}, \xi_{0,t})$ and $(f_{1,t}, \xi_{1,t})$ respectively. To this data we associate a degree $1$ element 
$$
\h\in \Hom_1(C_*(f_0,\xi_0;R_*),C_*(f_1,\xi_1;R_*))
$$ 
such that $\Psi_1 - \Psi_0 = D\h$. We refer to it as \emph{the homotopy cocycle}.

We view $\h$ as a matrix $(h_{x_0,y_1})$ acting from the right with columns indexed by $x_0\in\Crit(f_0)$ and rows indexed by $y_1\in\Crit(f_1)$. Intuitively, the chain $h_{x_0,y_1}\in C_{|x_0|-|y_1|+1}(\Omega X)$ is obtained by evaluating into based loops a chain level representative of the fundamental class of the compactified moduli space of pseudo-gradient trajectories \emph{of arbitrary index} in $[0,1]\times [0,1] \times X$, running from $(0,0,x_0)$ to $(1,1,y_1)$, for the function $(\tau,t,x)\mapsto g(\tau) + g(t) + f_{\tau,t}(x)$, where $g:[0,1]\to \R$ is smooth with a critical maximum at $0$ and a critical minimum at $1$. See~\S\ref{sec:invariance-continuation} for the details of the construction.

\begin{proposition}\label{prop:homotopy-cocycle} 
The relation 
$$
D\h = \Psi_1 - \Psi_0
$$
is equivalent to 
\begin{align} \label{eq:alg-homotopy-eqn}
\partial h_{x_0,y_1}\, =\,  \nu^1_{x_0,y_1} & - \nu^0_{x_0, y_1} + \sum_{z_0} (-1)^{|x_0|-|z_0|} m^0_{x_0,z_0} h_{z_0,y_1}\\
& + \sum_{z_1}(-1)^{|x_0|-|z_1|} h_{x_0,z_1} m^1_{z_1,y_1}. \nonumber
\end{align}
\end{proposition}

\begin{proof}
The map $\h$ acts by $\h(a\otimes x_0)=(-1)^{|a|}\sum_{y_1}a h_{x_0,y_1}\otimes y_1$. 
Assume first that~\eqref{eq:alg-homotopy-eqn} is satisfied. We then obtain
\begin{align*} 
(\h & \p^0 + \p^1 \h)(a\otimes x) \\
& = \h(\partial a \otimes x + (-1)^{|a|} a \sum_z m^0_{x,z}\otimes z)
+ (-1)^{|a|}\p^1 ( a \sum_y h_{x,y}\otimes y) \\
& =  (-1)^{|a|-1} \partial a \sum_y h_{x,y}\otimes y + (-1)^{|a|}\sum_{z, y}(-1)^{|a| +|x|-|z|-1}a  m^0_{x,z} h_{z,y}\otimes y \\
& \qquad + \, (-1)^{|a|}\, \partial a \sum_y h_{x,y}\otimes y\, +\,  (-1)^{|a|}\, (-1)^{|a|}\, a \sum_y \partial h_{x,y}\otimes y\, \\
& \qquad +\, (-1)^{|a|}a \sum_{y, w}(-1)^{|a|+|x|-|y|-1}h_{x,y}\, m^1_{y,w}\otimes w. 
\end{align*}
After simplifying and interchanging $y$ and $w$ in the last sum we get using~\eqref{eq:alg-homotopy-eqn} 
\begin{align*} 
(\h&  \p^0 +\p^1 \h)(a\otimes x) \\
& =\, a \sum_{y} (\partial h_{x,y} -\!\!\sum_z(-1)^{|x|-|z|}m^0_{x,z}\, h_{z,y}- \!\!\sum_w(-1)^{|x|-|w|}h_{x,w}\, m^1_{w,y})\otimes y \\
& =\, a\sum_y(\nu^1_{x,y}-\nu^0_{x,y})\otimes y \\
& =\, (\Psi_1-\Psi_0)(a\otimes x). 
\end{align*}
This proves the relation $D\h = \Psi_1 - \Psi_0$. 
The converse is proved using the same computation by taking $a$ to be the unit. 
\end{proof}

Given a right $R_*$-module $\cF_*$, we induce an element 
$$
\h^\cF\in 
\Hom_1(C_*(f_0,\xi_0;\cF_*),C_*(f_1,\xi_1;\cF_*))
$$ 
which gives rise to a chain homotopy between the continuation maps 
$\Psi^\cF_0$ and $\Psi^\cF_1$ 
determined respectively by $(f_{0,t}, \xi_{0,t})$ and $(f_{1,t}, \xi_{1,t})$, i.e., 
$\Psi^\cF_1 - \Psi^\cF_0 = D\h^\cF$. 
The map $\h^\cF$ acts by 
$$
\h^\cF(\alpha\otimes x_0)=(-1)^{|\alpha|}\sum_{y_1}\alpha h_{x_0,y_1}\otimes y_1.
$$

There are choices involved in the construction of the homotopy map $\h$. However, unlike for the Barraud-Cornea cocycles and for the continuation cocycles, we will not discuss the ambiguity because it would amount to developing a theory of higher homotopies. See the work of Mazuir~\cite{Mazuir2} for the non-enriched setting.

\section{Vista on the DG Floer toolset}

Although we do not address Floer theory in this book, it is instructive to spell out the analogue of the previous discussion in that context, with a view towards future applications. 

In its classical formulation, the Floer toolset at homology level consists of the following data.
\begin{itemize}

\item to a regular pair $(H,J)$ consisting of an admissible Hamiltonian and an admissible almost complex structure, associate the Floer differential $\p$ (degree $-1$), which is a square zero degree $-1$ map from $C_\bullet=\langle \Per_1(H)\rangle$ to itself. Here $\Per_1(H)$ is the set of $1$-periodic orbits of $H$.

The differential $\p$ can be seen as a perturbation of the zero differential $\p_0=0$ by a Maurer-Cartan element $\mathfrak{m}\in \End_{-1}(C_\bullet)$. The resulting twisted complex  $FC_*(H,J)=(\langle \Per_1(H)\rangle,\p)$ is \emph{the Floer complex}. 

\item to a regular homotopy $(H_s,J_s)$ interpolating between regular pairs $(H_\pm,J_\pm)$ at $\pm\infty$, associate a \emph{continuation map} $$\Psi:FC_*(H_+,J_+)\to FC_*(H_-,J_-)$$ of degree $0$. This is a chain map, i.e., $\p_-\Psi-\Psi\p_+=0$. 

This is equivalent to saying that  $\Psi$ is a cycle in the twisted $\Hom$ complex $\Hom(FC_*(H_+,J_+),FC_*(H_-;J_-))$. 

\item to a regular homotopy of homotopies $\{(H^\lambda,J^\lambda)\}$ connecting two regular homotopies $(H^i, J^i)$, $i=0,1$ which interpolate between $(H_\pm,J_\pm)$, associate a degree $1$ map $$h:FC_*(H_+,J_+)\to FC_{*+1}(H_-,J_-)$$ which is a chain homotopy between the continuation chain maps $\Psi_i$ determined by $(H^i, J^i)$ for $i=0,1$, i.e., $\Psi_1-\Psi_0=\p^-h + h \p^+$. 

This is equivalent to saying that $\Psi_1 - \Psi_0$ is a boundary in the twisted $\Hom$ complex $\Hom(FC_*(H_+,J_+),FC_*(H_-;J_-))$, and $h$ is a primitive of $\Psi_1 - \Psi_0$. 
\end{itemize}

The DG Floer toolset consists of the same kind of data, but enriched with coefficients in cubical chains $R_*=C_*(\Omega\cL X)$, where $\cL X$ stands for the free loop space of $X$, and more generally enriched with coefficients in any DG local system over $\cL X$, i.e., DG complex $\cF_*$ which is a right $R_*$-module. 

\begin{itemize}

\item to a regular pair $(H,J)$ consisting of an admissible Hamiltonian and an admissible almost complex structure, associate a Maurer-Cartan element $\mathfrak{m}$ in the $\End$ complex $\End(R_*\otimes C_\bullet)$ endowed with the differential inherited from $R_*$. 

Given a right $R_*$-module $\cF_*$, we induce a Maurer-Cartan element $\mathfrak{m}^\cF$ in the $\End$ complex $\End(\cF_*\otimes C_\bullet)$ endowed with the differential $\p$ inherited from $\cF_*$. The Floer differential $\p$ is the twist of this differential by the Maurer-Cartan element $\mathfrak{m}^\cF$. The resulting twisted complex  $FC_*(H,J;\cF)=(\cF_*\otimes C_\bullet,\p)$ is \emph{the Floer complex with coefficients in the DG local system $\cF_*$}. 

\item to a regular homotopy $(H_s,J_s)$ interpolating between regular pairs $(H_\pm,J_\pm)$ at $\pm\infty$, associate a degree $0$ continuation element $\Psi$. This is a cycle in  
$\Hom(FC_*(H_+,J_+;R_*), FC_*(H_-,J_-;R_*))$. 

Given a right $R_*$-module $\cF_*$, we induce tautologically a continuation cycle $\Psi^\cF$ in  
$\Hom(FC_*(H_+,J_+;\cF_*), FC_*(H_-;J_-;\cF_*))$. This 
induces a continuation map $\Psi^\cF:FC_*(H_+,J_+;\cF_*)\to FC_*(H_-;J_-;\cF_*)$ such that $\p^-\Psi^\cF-\Psi^\cF \p^+=0$.

\item to a regular homotopy of homotopies $\{(H^\lambda,J^\lambda)\}$ connecting two regular homotopies $(H^i, J^i)$, $i=0,1$ which interpolate between $(H_\pm,J_\pm)$, associate a degree $1$ parametrized continuation element 
$$\h\in  
\Hom_1(FC_*(H_+,J_+;R_*),FC_*(H_-;J_-;R_*))$$ 
which realizes a chain homotopy between the continuation maps $\Psi_i$ determined by $(H^i, J^i)$, i.e., $\Psi_1 - \Psi_0 = \p^-\h + \h \p^+$. 

Given a right $R_*$-module $\cF_*$, we induce tautologically an element $\h^\cF\in 
\Hom_1(FC_*(H_+,J_+;\cF_*),FC_*(H_-;J_-;\cF_*))$ which realizes a chain homotopy between the continuation maps $\Psi^\cF_i$ determined by $(H^i, J^i)$, i.e.,  $\Psi^\cF_1 - \Psi^\cF_0 = \p^-\h^\cF + \h^\cF \p^+$. 
\end{itemize}

\chapter[Barraud-Cornea and Brown cocycles]{Comparison between the Barraud-Cornea cocycle and the Brown cocycle} \label{sec:BC-Brown}

Understanding the homology of fibrations $F\hookrightarrow E\to X$ was historically a driving question in algebraic topology. The fundamental piece of structure is the Leray-Serre spectral sequence~\cite{Serre1951}, and we warmly recommend McCleary's book~\cite{McC} for a comprehensive panorama of the subject. For us, the case in point is Brown's chain level description of the total space $E$ in terms of chains on $X$ and chains on $F$, see~\cite{Brown1959}. The key ingredient is ``Brown's universal twisting cocycle" 
$$
\varphi\in C^*(X;C_*(\Omega X))
$$
which satisfies the Maurer-Cartan equation 
\begin{equation} \label{eq:MC-Brown}
[\p, \varphi] + \varphi\cup \varphi = 0.
\end{equation} 
Our purpose in this chapter is to explain the precise sense in which Brown's twisting cocycle is equivalent to the Barraud-Cornea twisting cocycle. In the sequel we fix an arbitrary ring of coefficients $A$. 

\section{The Barraud-Cornea cocycle as a Maurer-Cartan element} \label{sec:BC-twisting-MC}

Let $X$ be a closed manifold and $f:X\to \R$ a Morse function together with a choice of regular pseudo-gradient vector field and a choice of collapsing embedded tree as in the previous chapters. Denote $R_*=C_*(\Omega X)$ the dga of cubical chains,  denote $C_\bullet=\langle \Crit(f)\rangle$ the free $\mathbb{Z}$-module generated by the critical points of $f$, viewed as a complex with zero differential, and denote $C_*$ the Morse complex, with underlying $\mathbb{Z}$-module $C_\bullet$ and endowed with the Morse differential.

As explained in~\S\ref{sec:DG-Morse}, the Barraud-Cornea cocycle may be interpreted as a Maurer-Cartan element in $\mathfrak m\in\End_{-1}(R_*\otimes C_\bullet)$. We provide below a slight modification that allows to obtain a Maurer-Cartan element in $\End(R_*\otimes C_*)$. This will be useful to establish a relation with the Brown twisting cocycle.

Given $x,y\in\Crit(f)$ such that $|x|-|y|=1$ we denote $n_{x,y}$ the algebraic number of pseudo-gradient trajectories connecting $x$ to $y$, so that the differential in $C_*$ acts by $\p x = \sum_{|y|=|x|-1} n_{x,y} y$. Let $e$ be the constant loop at the basepoint. We modify the $\mathfrak{m}=(m_{x,y})\in \End_{-1}(R_*\otimes C_\bullet)$ to $\mathfrak{m}'=({m}'_{x,y})\in \End_{-1}(R_*\otimes C_*)$ given by 
$$
{m}'_{x,y}=\left\{\begin{array}{ll} m_{x,y}, & \mbox{ if } |x|-|y|\neq 1, \\
 m_{x,y} - n_{x,y}e, & \mbox{ if } |x|-|y|=1.
 \end{array}\right. 
$$

The following result is proved by a computation similar to that of Proposition~\ref{prop:MC-formula-equation}, which we omit. 
\begin{proposition}
The relation $$\p m_{x,y}=\sum_z (-1)^{|x|-|z|}m_{x,z}m_{z,y}$$ is equivalent to the Maurer-Cartan equation 
$$
D\mathfrak{m}'+ \frac 1 2 [\mathfrak{m}',\mathfrak{m}'] =0,
$$
i.e., $\mathfrak{m}'$ is a Maurer-Cartan element in the dgLa $\End(R_*\otimes C_*)$. \qed
\end{proposition}

From this perspective, the Morse complex enriched with coefficients in $R_*$ is 
$$
(R_*\otimes C_*,\p + \mathfrak{m}'), 
$$
where $\p$ is the tensor differential. 

\section{The Brown cocycle as a Maurer-Cartan element} \label{sec:Brown-twisting-MC}

Brown's twisting cocycle has played a key role both in the development of the theory of the bar-cobar adjunction (see~\cite[\S10]{Brown1959}, as well as Felix-Halperin-Thomas~\cite{FelixHalperinThomas1995} or Loday-Vallette~\cite[\S2]{LodayVallette-AlgOp} for modern developments), and in the theory of homological perturbation (see Gugenheim~\cite{Gugenheim60} and McCleary~\cite[\S 6.4]{McC}, as well as the references therein). Our discussion will focus on the perturbative aspects.
  
Brown's twisting cocycle $\varphi\in C^*(X;C_*(\Omega X))$ acts as 
$$
\varphi:C_*(X)\to C_{*-1}(\Omega X). 
$$
Using the coalgebra structure on $C_*(X)$ with Alexander-Whitney coproduct $\lambda$, and the algebra structure on $C_*(\Omega X)$ with Pontryagin product $\mu$, Brown defines further~\cite{Brown1959,Gugenheim60}
$$
\mathfrak{m}_\varphi = (\mu\otimes 1)(1\otimes \varphi\otimes 1)(1\otimes \lambda)\in \End_{-1}(C_*(\Omega X)\otimes C_*(X)).
$$
The Maurer-Cartan relation~\eqref{eq:MC-Brown} for $\varphi$ turns out to be equivalent to
$$
(\p + \mathfrak{m}_\varphi)^2=0.
$$ 
This is in turn equivalent to saying that $\mathfrak{m}_\varphi$ is a Maurer-Cartan element in the dgLa $\End(C_*(\Omega X)\otimes C_*(X))$, i.e., 
$$
D\mathfrak{m}_\varphi + \frac 1 2 [\mathfrak{m}_\varphi, \mathfrak{m}_\varphi] =0. 
$$
We view therefore $\mathfrak{m}_\varphi$, and hence $\varphi$, as defining a perturbation of the tensor differential on $C_*(\Omega X)\otimes C_*(X)$.

\section{Semi-free resolutions}

The Barraud-Cornea cocycle and the Brown cocycle can be interpreted from a unifying perspective in the context of differential homological algebra. For this we need the following definition from~\cite[\S2]{FelixHalperinThomas1995}. 

\begin{definition} \label{defi:semi-free-resolution}
Let $R_*$ be a DGA. A DG module $T_*$ over $R_*$ is called \emph{semi-free} if it admits a filtration 
$$
0=T(-1)\subset T(0)\subset T(1)\subset \cdots \subset T(k-1)\subset T(k)\subset \cdots
$$  
such that, for all $k\ge 0$, the module $T(k)/T(k-1)$ is $R_*$-free on a basis of cycles. A \emph{semi-free $R_*$-resolution} of an $R_*$-module $M_*$ is a semi-free module $T_*$ with a quasi-isomorphism $T_*\stackrel\simeq\longrightarrow M_*$. 
\end{definition}

We refer to~\cite[\S2]{FelixHalperinThomas1995} for terminology and material regarding semi-free resolutions, which play in differential graded homological algebra the role played by projective resolutions in the study of modules over commutative rings. 
In particular, any two semi-free resolutions of the same $R_*$-module are chain homotopy equivalent.

\begin{example}
Let $R_*$ be a DGA over a ring $A$ and $C_*$ a chain complex which is free over $A$ with a distinguished basis $\cB$. Let $\mathfrak{m}\in\End_{-1}(R_*\otimes C_*)$ be a Maurer-Cartan element which is defined by a strictly lower triangular matrix with respect to the degree in the basis $\cB$, i.e., $\mathfrak{m}_{x, y}=0$ if $x,y\in \cB$ and $|x|\leq |y|$. Then the $R_*$-module $(R_*\otimes C_*,\p + \mathfrak{m})$ is semi-free. Indeed, the filtration $T(k)=R_*\otimes C_{*\le k}$ satisfies the requirements in the previous definition.\footnote{Note that this is precisely the filtration that will be used in Definition~\ref{defi:canonical_filtration} to define the spectral sequence for the twisted complex.} In particular, by Sections \ref{sec:BC-twisting-MC} and \ref{sec:Brown-twisting-MC}:
\begin{itemize}
\item Given a manifold $X$, the Barraud-Cornea twisted complex with coefficients in $C_*(\Omega X)$ is semi-free. 
\item Given a topological space $X$, the Brown cocycle $\mathfrak{m}_\varphi$  gives rise to a semi-free $C_*(\Omega X)$-module~\cite{Brown1959}. 
\end{itemize} 
\end{example} 

The key fact is that the Barraud-Cornea cocycle and the Brown cocycle both provide (semi-free) $C_*(\Omega X)$-resolutions of the trivial $C_*(\Omega X)$-module $\mathbb{Z}$. This is the content of the main theorems of Barraud and Cornea from~\cite[Theorem~1.1.d]{BC07} (enhanced to $\Z$-coefficients by our Theorem~\ref{thm:fibration}, see also Theorem~A) and of Brown from~\cite[Theorem~4.2]{Brown1959}, which provide quasi-isomorphisms\footnote{Brown actually proves a chain homotopy equivalence, and this can also be proved in the Morse setup with more work.} 
$$
(C_*(\Omega X)\otimes C_*(f),\p + \mathfrak{m}')\stackrel\simeq\longrightarrow C_*(\cP_{\star\to X}X)
$$
and
$$
(C_*(\Omega X)\otimes C_*(X),\p + \mathfrak{m}_\varphi)\stackrel\simeq\longrightarrow C_*(\cP_{\star\to X}X).
$$
Here $\cP_{\star\to X}X$ is the space of paths in $X$ starting at the basepoint. This space is contractible, so that $C_*(\cP_{\star\to X}X)$ is chain homotopy equivalent as a $C_*(\Omega X)$-module to the trivial $C_*(\Omega X)$-module $\mathbb{Z}$. 

\begin{definition}
Let $C_*$ be a complex of free $\mathbb{Z}$-modules and consider a Maurer-Cartan element $\mathfrak{m}\in\End_{-1}(C_*(\Omega X)\otimes C_*)$. We say that $\mathfrak{m}$ is a \emph{universal cocycle} if $(C_*(\Omega X)\otimes C_*,\p + \mathfrak{m})$ is a semi-free resolution of 
the trivial $C_*(\Omega X)$-module $\mathbb{Z}$. 
\end{definition}

\begin{definition} Two Maurer-Cartan elements $\mathfrak{m}\in\End_{-1}(C_*(\Omega X)\otimes C_*)$ and $\mathfrak{m}'\in\End_{-1}(C_*(\Omega X)\otimes C'_*)$ are said to be \emph{equivalent} if the complexes $(C_*(\Omega X)\otimes C_*,\p + \mathfrak{m})$ and $(C_*(\Omega X)\otimes C'_*,\p + \mathfrak{m}')$ are chain homotopy equivalent. 
\end{definition}

We therefore have:

\begin{theorem} \label{thm:equivalence-BC-Brown}
The Barraud-Cornea cocycle and the Brown cocycle are universal, and therefore equivalent. 
\end{theorem} 

\begin{proof} We have already seen that these two cocycles are universal. 
On the other hand, any two universal cocycles are equivalent because any two semi-free resolutions of the trivial $C_*(\Omega X)$-module $\mathbb{Z}$ are chain homotopy equivalent.  
\end{proof}

We close this chapter with two questions for further study.

\begin{question}
Realize the above chain homotopy equivalence by a map that preserves the filtrations. 
\end{question} 

\begin{question}
The Brown cocycle $\varphi$ is a Maurer-Cartan element in the convolution algebra $\Hom(C,A)$, with $C=C_*(X)$ the DG coalgebra of chains on $X$ and $A=C_*(\Omega X)$ the DG algebra of chains on the space of based Moore loops. See~\cite{Brown1959,Gugenheim60,LodayVallette-AlgOp}. In the previous discussion we rephrased $\varphi$ as a Maurer-Cartan element $\mathfrak{m}_\varphi$ in the dgLa $\End(C_*(\Omega X)\otimes C_*(X))$, bypassing in this way the definition of the dgLa in terms of the coalgebra structure on $C_*(X)$ and allowing for a comparison with the Barraud-Cornea cocycle. We find it an interesting question to reinterpret the Barraud-Cornea cocycle as a Maurer-Cartan element in a convolution algebra $\Hom(C,A)$,  with $C$ the Morse complex endowed with a structure of coalgebra up to homotopy.
\end{question}

\begin{remark} \label{rmk:remark_on_Tor}
From the perspective adopted in this section, the Morse chain complex $C_*(f;\cF)$ is a model for the derived tensor product with the trivial $C_*(\Omega M)$-module $\Z$: 
$$
C_*(f;\cF)\simeq \cF\stackrel{L}\otimes_{C_*(\Omega M)} \Z,
$$
or, in another notation,
$$
H_*(f;\cF)\simeq \mathrm{Tor}_*^{C_*(\Omega M)} (\cF,\Z).
$$
These identifications should be compared with~\cite[Definition~3.1.1]{Malm-thesis}.

\end{remark}

\chapter[Properties of twisted complexes]{Algebraic properties of twisted complexes}
\label{sec:algebraicsetup}

In this chapter we single out and discuss a number of algebraic properties of twisted complexes, as defined in the previous chapter.

\section{Algebraic setup}\label{sec:sub-algebraicsetup}
We recap the notation from the previous chapter in a slightly more abstract setting, which we adopt for the rest of~\S\ref{sec:algebraicsetup}. 

Our construction uses the following data. 
 
\textbf{(1) A differential graded algebra (DGA) $R=(R_*,\partial)$,} with unit $1\in R_0$ over an ungraded ring $A$. The degree of a homogeneous element $a\in R_*$ is denoted by $|a|$.  We work in homological convention and assume that the differential on $R_*$ has degree $-1$. 

Our main example is $R_*=C_*(\Omega X)$, the DGA of based Moore loops on a manifold $X$.

\textbf{(2) A differential graded (DG) right $R_*$-module $\cF=(\cF_*,\p)$.} This means that $(\cF_*,\p)$ is a chain complex endowed with a right $R_*$-module action 
$$
\cF_*\otimes R_*\to \cF_*, \qquad \alpha\otimes a\mapsto \alpha\cdot a
$$ 
which is a chain map of degree $0$. Having in mind the fundamental example $R_*=C_*(\Omega X)$, we call such an object  \emph{derived, or DG, local system}, and we refer to the Introduction for a discussion of this terminology. 

Our privileged source of examples is provided by fibrations over $X$, see~\S\ref{sec:fibrations}. Upon choosing a lifting function, the cubical chains on the fiber inherit the structure of a right $C_*(\Omega X)$-module.

\textbf{(3) A Maurer-Cartan element (or twisting cocycle) $\mathfrak{m}$.} Given $\cC_\bullet= (\cC_p)_{p\ge 0}$ a graded free $A$-module which is \emph{based}, i.e., endowed with a preferred basis $\calb$, we see $\mathfrak m$ as a matrix $(m_{x,y})$, $x,y\in\calb$ consisting of an element $m_{x,y}\in R_{|x|-|y|-1}$ for any two elements $x, y\in \calb$ with $|x|>|y|$ (and set $m_{x,y}=0$ if $|x|\le |y|$), and satisfying the {\it Maurer-Cartan (MC) relation}
\begin{equation}\label{MC} 
\partial m_{x,y}\, =\, \sum_{|x|>|z|>|y|}(-1)^{|x|-|z|}m_{x,z}\cdot m_{z,y}.
\end{equation} 
We implicitly assume that, for any $x,y\in\calb$ with $|x|>|y|$, the set $\{z\in\calb\, : \, m_{x,z} \mbox{ and } m_{z,y} \mbox{ are nonzero}\}$ is finite. 
 
Given this data, the \textbf{twisted chain complex on $\cC_\bullet$ with DG-coefficients in $\cF$} is  
$$
C_*=C_*(\cC_\bullet,\cF_*)=\cF_*\otimes_A \cC_\bullet,
$$
with differential (of degree $-1$)
\begin{equation}\label{diffDGmodule}
\partial (\alpha\otimes x) = \partial \alpha \otimes x + (-1)^{|\alpha|}\sum_y\alpha\cdot m_{x,y}\otimes y.
\end{equation} 
Here $\alpha\in\cF$, $x\in\calb$ and the differential is extended by $A$-linearity in the second variable.  
The homology of $C_*$ is called \textbf{homology with DG-coefficients in $\calf$} and is denoted $H_*(\cC_\bullet; \calf)$.

\section[Spectral sequence and lifted complex]{Canonical filtration, spectral sequence and lifted complex}\label{sec:spectral-sequence}

Unless otherwise mentioned we will suppose that the DGA $R$ and the DG-local systems $\cF$ considered in this book are supported in nonnegative degrees.  
This condition ensures that the spectral sequences that appear in the sequel converge.

\begin{definition} \label{defi:canonical_filtration} The \emph{canonical filtration} on the complex $(C_*,\partial)$ is given by 
$$
F_p(C_*)\, =\, \bigoplus_{i\leq p}\cF\otimes \cC_i,\qquad p\ge 0.
$$
\end{definition}

Since $\cF$ is supported in a range of degrees that is bounded from below,    
the spectral sequence $(E_{p,q}^r, d^{r})$ associated to this filtration converges to $H_*(C_*)$.  We will now describe its first pages.

Its 0-th page is 
$$
E_{p,q}^{0}\, =\, F_{p}(C_{p+q})/F_{p-1}(C_{p+q})\,  = \, \cF_q\otimes \cC_p
$$
and $d^0:E_{p,q}^{0}\ri E_{p,q-1}^0$ is given by 
$d^0(\alpha\otimes x)\, =\, \partial \alpha \otimes x$.
As a consequence, the first page is  
$$ 
E_{p,q}^1\, =\, H_{q}(\cF_*\otimes \cC_p)\, =\, H_q(\cF_*)\otimes \cC_p.
$$
The differential $d^1:E_{p,q}^1\ri E_{p-1, q}^1$ is given by
$$
d^1(\hat{\alpha}\otimes x)\, =\, (-1)^q\sum_{|y|=|x|-1}\hat{\alpha} \cdot \hat m_{x,y}\otimes y,
$$ 
where $\hat{\alpha}\in H_q(\cF)$, $\hat{m}_{x,y}$ is the image of $m_{x,y}$ in $H_0(R_*)$ and $\hat{\alpha}\cdot \hat m_{x,y}$ is the class of the cycle $\alpha\cdot m_{x,y}$ in $H_q(\cF)$. 

In order to describe the second page of the spectral sequence, we need to introduce the following complex.

\begin{definition}\label{def:algebraic_lifted_complex}
  The \emph{lifted complex of $\cC_\bullet$} is the free graded left $H_0(R_*)$-module 
  $$
  \widetilde{\cC}_p= H_0(R_*)\otimes_A \cC_p
  $$ 
  endowed with the differential  
      \[\delta(x) = \sum_{|y|=|x|-1}\hat{m}_{x,y}y.\]
\end{definition}

The terminology is motivated by the discussion in~\S\ref{sec:lifted-homology}: for the fundamental example where $R_*=C_*(\Omega X)$ and $\mathfrak{m}$ is the Barraud-Cornea cocycle, this is precisely the lifted Morse complex from~\cite{damian2012}. We may now describe the second page of the spectral sequence.

\begin{lemma}\label{spectral} 
  The  second page of the spectral sequence associated to $C_*=C_*(\cC_\bullet,\calf_*)$ is given by 
  \[E_{p,q}^2 = H_p(\widetilde{\cC}_*; H_q(\calf)),\] where
  $H_q(\calf)$ carries the induced right $H_0(R_*)$-module structure.
\end{lemma} 

\begin{proof}
  Recall that $E_{p,q}^2$ is the homology of $E_{p,q}^1= H_q(\calf_*)\otimes_{A} \cC_p$, whose differential is given by:
  $$d^1(\hat{\alpha}\otimes x)\, =\, (-1)^q\sum_{|y|=|x|-1}\hat{\alpha} \cdot \hat{m}_{x,y}\otimes_{A} y.$$
  Up to sign change this is exactly the differential of the complex 
  $$
  H_q(\calf_*)\otimes_{H_0(R_*)}\left(H_0(R_*)\otimes_A \cC_p\right)\, =\, H_q(\calf_*)\otimes_{H_0(R_*)}\widetilde{\cC}_p,
  $$ 
  whose homology is the claimed one.
\end{proof}

\section{Functoriality with respect to the DG-module}\label{sec:alg-functoriality}

Our construction satisfies the following functoriality property.

\begin{proposition}\label{functoriality1bis} Let $\calf$ and $\calf'$ be two DG-modules over a DGA  $R_*$ endowed with a twisting cocycle  $(m_{x,y})$ associated to a free based complex $\cC_{\bullet}$. Let $\Gamma:\calf\ri\calf'$ be a morphism of DG-modules over $R_*$. Then $\Gamma$ yields a  functorial chain map 
$$
\widetilde{\Gamma}: C_*(\cC_\bullet,\calf) \ri C_*(\cC_\bullet,\calf').
$$ 
If $\Gamma$ is a quasi-isomorphism then so is $\widetilde\Gamma$.
\end{proposition} 

\begin{proof}
We define $\widetilde\Gamma$ by $ \widetilde{\Gamma}(a\otimes x) = \Gamma(a)\otimes x$. 
It is easy to check that $\widetilde{\Gamma}$ is a chain map which preserves the canonical filtrations, hence $\widetilde{\Gamma}$ induces a morphism between the corresponding spectral sequences. If the map $\Gamma$ is a quasi-isomorphism then $\widetilde{\Gamma}$ induces an isomorphism between the first pages of the spectral sequences, hence an isomorphism between the limits of the two spectral sequences. 
\end{proof}


\begin{example}
The degree 0 homology  
$H_0(\calf)$ is a right $H_0(R)$-module 
and we can view $H_0(\calf)$ as a DG-module $\calf'$ over $R_*$  
by setting $\calf'_0=H_0(\calf)$, $\calf'_k=0$ for all $k\neq 0$ and using the canonical projection $R_*\to H_0(R_*)$, which is an algebra map. There is an obvious morphism $\calf\to \calf'$ defined by the canonical projection $\pi:\calf_0\ri \calf'_0=H_0(\calf)$ (and $0$ elsewhere). Proposition~\ref{functoriality1bis} 
yields a natural map
\begin{equation} \label{eq:twisted-to-lifted}
\pi_* : H_*(\cC_\bullet ; \calf) \ri H_*(\widetilde{\cC}_\bullet ; H_0(\calf)).
\end{equation}
The last assertion of Proposition~\ref{functoriality1bis} implies that,  
if $H_k(\calf)=0$ for $k\neq 0$, then $\pi_*$ is an isomorphism. In fact, we will see in the proof of Proposition~\ref{prop:H0} below (see~\eqref{eq:deg0}) that $\pi_*$ is always an isomorphism in degree 0.
\end{example}

In the previous statement we used the same DGA for both modules $\calf$ and $\calf'$, but sometimes we will have to deal with different modules over different DGAs. The following remark will be helpful.

\begin{remark}\label{rem:pullback} (push-forward of twisting cocycle = pullback of DG-module structure)
Let $R_*$ be a DGA endowed with $(m_{x,y})$ as above, $\Phi : R_*\ri R'_*$ a DGA-map and $\calf'$ a DG-module over $R'_*$. We may build twisted complexes from this data by either:
\begin{enumerate}[leftmargin=* ,parsep=0cm,itemsep=0cm,topsep=-2mm]
\item pushing forward the  twisting cocycle  into $R'_*$ by defining $m'_{x,y} = \Phi(m_{x,y})$ and using this new twisting cocycle to define $C_*(\cC_\bullet,\calf')$, or
\item pulling back via $\Phi$ the $R'_*$-module structure on $\calf'$ into an $R_*$-module structure denoted $\Phi^*\calf'$, and using it to define the twisted complex $C_*(\cC_\bullet,\Phi^*\calf')$ using the original cocycle $(m_{x,y})$. 
\end{enumerate}
As one can easily see these two twisted complexes coincide. We denote their common homology by $H_*(\cC_\bullet ; \Phi^*\calf')$.
\end{remark}

\begin{example}\label{ex:pullback}
Let $\phi:X\to Y$ be a continuous map between pointed topological spaces. It naturally induces a  map $\Omega\phi:\Omega X\to\Omega Y$ hence a DGA morphism between the cubical chain complexes  $R_*=C_*(\Omega X)$ and  $R_*'=C_*(\Omega Y)$.
\end{example}

 Using this remark we get the following generalization of Proposition~\ref{functoriality1bis}.
 
 \begin{proposition}\label{functoriality1tiers} Let $\calf$, resp. $\calf'$, be right DG-modules over the DGAs $R_*$, resp. $R'_*$, and let $$(\Gamma,\Phi): \calf\ri \calf'$$ be a morphism of DG-modules, i.e.,  $\Gamma: \calf\ri \calf'$ is a morphism of complexes, $\Phi:R_*\ri R'_*$ is a morphism of DGAs and 
   \begin{equation}
     \label{eq:Gamma-Phi}
     \Gamma(\alpha\cdot a)\,= \, \Gamma(\alpha)\cdot \Phi(a) \gol\gol \forall a\in R_*, \, \forall \alpha\in \calf.
   \end{equation}
  Let $(m_{x,y})$ be a twisting cocycle on $R_*$ and $(m'_{x,y})= \Phi(m_{x,y})$ be its push-forward in $R'_*$. Then $(\Gamma,\Phi)$ induces a functorial map
 $$
 \widetilde{\Gamma}: C_*(\cC_\bullet,\calf) \ri C_*(\cC_\bullet,\calf').
 $$ 
Morerover, if $\Gamma$ is a quasi-isomorphism then so is $\widetilde{\Gamma}$.
\end{proposition}

\begin{proof}
 Our assumptions imply that $\Gamma : \calf\ri \Phi^* \calf'$ is a morphism of DG-modules over $R_*$, which therefore induces a functorial map 
 $$
 \widetilde{\Gamma}: C_*(\cC_\bullet,\calf) \ri C_*(\cC_\bullet,\Phi^*\calf')
 $$
 by Proposition~\ref{functoriality1bis}. The identification 
 $ C_*(\cC_\bullet,\Phi^*\calf')\, \equiv\, C_*(\cC_\bullet,\calf')$ 
 discussed in Remark~\ref{rem:pullback} yields the result.
\end{proof}

\begin{example}
In the particular case where $\calf_*=R_*$,  $\calf'_*=R'_*$ and $\Gamma=\Phi$, condition (\ref{eq:Gamma-Phi})  in Proposition~\ref{functoriality1tiers} is automatically satisfied and we obtain a map $\widetilde{\Phi}:C_*(\cC_\bullet,R_*)\ri C_*(\cC_\bullet,R'_*)$ which is a quasi-isomorphism whenever $\Phi$ is a quasi-isomorphism. 
\end{example}

\section{A criterion for quasi-isomorphism between twisted complexes}

The goal of this section is to show that certain maps between twisted complexes are quasi-isomorphisms if the induced maps between the lifted complexes (in the sense of Definition~\ref{def:algebraic_lifted_complex}) are quasi-isomorphisms. To state this result, we consider two complexes $\cC_\bullet$, $\cD_\bullet$, with preferred bases $\calb^\cC$, $\calb^\cD$, twisted by a DG-module $\calf$ over a DGA $R_*$ using respective twisting cocyles $m^{\cC}_{(\cdot,\cdot)}$ and $m^{\cD}_{(\cdot,\cdot)}$. 

Recall from Proposition~\ref{prop:continuation-DG} that a collection of elements $\nu_{x,y}\in R_{|x|-|y|}$ for $x\in\calb^\cC$, $y\in\calb^\cD$ such that 
\begin{equation}\label{MC-invariance}
 \partial\nu_{x,y}= \sum_{z\in
    \calb^\cC}m^\cC_{x,z}\cdot\nu_{z,y} - \sum_{w\in
    \calb^\cD}(-1)^{|x|-|w|}\nu_{x,w}\cdot m^\cD_{w,y}\end{equation}
gives rise to a chain map $\Psi: C_*(\cC_\bullet,\calf)\ri C_*(\cD_\bullet,\calf)$ defined by $\Psi(\alpha\otimes x) = \alpha \sum_{y\in \calb^\cD}\nu_{x,y}\otimes y$.
Our result applies to such maps, which include the continuation maps of~\S\ref{sec:DG-Morse}.

\begin{proposition}\label{functoriality2}  Let $(\widetilde{\cC}_*,\delta^\cC)$ and $(\widetilde{\cD}_*,\delta^\cD)$ be the respective lifted complexes of $\cC_\bullet$ and $\cD_\bullet$ (See Definition~\ref{def:algebraic_lifted_complex}). Let  $\widetilde{\Psi} : \widetilde{\cC}_*\ri\widetilde{\cD}_*$ be the morphism of $H_0(R_*)$-complexes defined by $\widetilde{\Psi}(x)=\sum_{|y|=|x|}n_{x,y}y$, where $n_{x,y}$ denotes the class of $\nu_{x,y}$ in $H_0(R_*)$ for any $x\in \calb^\cC$ and $y\in \calb^\cD$ with $|x|=|y|$.  
Then, if $\widetilde{\Psi}$ is a quasi-isomorphism, so is $\Psi$.
\end{proposition}

\begin{proof}  Note that $\Psi$ clearly preserves the filtrations of the twisted complexes. It induces therefore a  morphism between the two corresponding spectral sequences. In order to prove that it is an isomorphism in the limit we show that it induces an isomorphism on the second page.  

We saw in~\S\ref{sec:spectral-sequence} that the first page of the spectral sequence for $C_*(\cC_\bullet,\calf)$ is 
$$
E_{p,q}^1\cong H_q(\calf)\otimes_{H_0(R_*)}\widetilde{\cC}_p.
$$ 
Since $n_{x,y}$ represents $\nu_{x,y}$ in $H_0(R_*)$, one can easily see by looking at page $0$ that $\Psi$ acts on the first page as
$$\Psi^{(1)}(\hat{\alpha}\otimes x)\, =\, \sum_{|x|=|y|}\hat{\alpha} \cdot n_{x,y}\otimes y.$$
Thus $\Psi^{(1)}= \Id\otimes \widetilde{\Psi}$ and we deduce that $\Id\otimes \widetilde{\Psi}$  induces in homology the map $\Psi^{(2)}$ from the second page. 
That $\Psi^{(2)}$ is an isomorphism then follows from Lemma~\ref{change} below,  applied with $H=H_0(R_*)$ to the free left $H$-modules $\widetilde \cC_*$ and $\widetilde \cD_*$.
\end{proof}

\begin{lemma}\label{change} Let $\widetilde \cC_*$ and $\widetilde \cD_*$ be complexes of free left modules over a ring 
$H$. Given a quasi-isomorphism $\widetilde{\Psi}: \widetilde \cC_*\ri \widetilde \cD_*$, the map  
$$\Id\otimes \widetilde{\Psi} :M\otimes_H \widetilde \cC_*\ri M\otimes_H\widetilde \cD_*$$ is a quasi-isomorphism for any right $H$-module $M$.
\end{lemma}

\begin{proof}
This is a consequence of the existence of the spectral sequence for a change of coefficients (see for instance \cite[Th\'eor\`eme~5.5.1]{Godement}).
In the case of $M\otimes_H \widetilde \cC_*$, the spectral sequence $(F_{p,q}^r)$ converges towards $H_*(M\otimes_H \widetilde \cC_*)$ and its second page is 
$$
F_{p,q}^2\, =\, \mbox{Tor}_p^H(M;H_q(\widetilde \cC_*)).
$$
The spectral sequence is associated to the bicomplex $P_*\otimes_H \widetilde \cC_*$, where $P_*$ is a free $H$-resolution of $M$, and it is functorial. As such, the morphism $\Id\otimes \widetilde{\Psi}$ naturally induces a morphism between the corresponding spectral sequences. By assumption it is an isomorphism at the second page and therefore it induces an isomorphism between the limits.
\end{proof}

\section{The degree $0$ homology of the twisted complex} \label{sec:degree-zero}

The next result shows that the degree $0$ homology of the twisted complex is the tensor product of the degree $0$ homologies of the DG-module $\calf$ and of the lifted complex $\widetilde \cC_\bullet$.

\begin{proposition}\label{prop:H0}
We have 
$$H_0(\cC_\bullet; \calf)  = H_0(\calf) \otimes_{H_0(R_*)} H_0(\widetilde{\cC}_\bullet),$$ where $\widetilde{\cC}_\bullet$ is the lifted complex of $\cC_\bullet$ from Definition~\ref{def:algebraic_lifted_complex}.
\end{proposition}

\begin{proof} Using the spectral sequence from~\S\ref{sec:spectral-sequence} we get 
  \begin{equation}
H_0(\cC_\bullet; \calf) = E_{0,0}^\infty = E_{0,0}^2 = H_{0}(\widetilde{\cC}_\bullet; H_0(\calf)).\label{eq:deg0}
\end{equation}

The term on the right hand side is computed by the change of coefficients spectral sequence $F_{p,q}^r$ from the proof of Lemma~\ref{change}:
\begin{align*}
  H_{0}(\widetilde{\cC}_\bullet; H_0(\calf)) &= F_{0,0}^\infty = F_{0,0}^2 = \mathrm{Tor}_0^{H_0(R_*)}(H_0(\calf),H_0(\widetilde{\cC}_\bullet)) \\
  &=  {H_0(\calf)}\otimes_{H_0(R_*)}H_0(\widetilde{\cC}_\bullet).
\end{align*}
\end{proof}

\chapter[DG Morse homology: construction]{Morse homology with DG-coefficients: construction}\label{section:Morse} 

\section[Local coefficients as DG local coefficients]{(Non-DG) local coefficients as DG local coefficients}\label{sec:local-coeff}

In this chapter we explain how the standard homology with (non-DG) local coefficients can be interpreted as a particular case of homology with DG coefficients.

As a first step towards this goal, we recall the definition of the lifted Morse complex from~\cite{Latour,damian2012}. This section is mainly of a motivational nature and for this reason we keep the technicalities to a minimum. These can be recovered from the detailed discussion in~\S\ref{sec:construction-enriched-morse}. 

\subsection{The lifted Morse complex}\label{sec:lifted-homology}


Let $X$ be a manifold endowed with a Morse function $f$ and a negative generic  pseudo-gradient $\xi$ (we use descending flows in this book meaning that $df(\xi) \leq 0$).  As usual in Morse theory ``generic'' means here that $\xi$ is Morse-Smale, i.e., its stable and unstable manifolds intersect transversely. For each $x\in \Crit(f)$ we fix a lift $\tilde{x}$ in the universal cover $\widetilde{X}$. 

We consider the complex of free $\Z[\pi_1(X)]$-modules $(C_*(\widetilde{X}),\delta)$ spanned by these lifts. Its differential $\delta$ is given by the lifts of the gradient trajectories. More precisely, after fixing an orientation of each unstable manifold, any Morse trajectory between critical points of consecutive indices $x$ and $y$ is naturally oriented.   It lifts  to a trajectory from $\tilde{x}$ to $g\tilde{y}$ in $\widetilde{X}$ for some $g\in \piu(X)$: we will count it  as $\pm g$ in the definition of $\delta$, depending on its orientation.

It is well-known that the definition of the usual Morse complex relies on the fact  that the broken trajectories $\lambda\#\lambda'$ between points of consecutive indices $|x|=|z|+1=|y|+2$ are the boundary points of a compact manifold of dimension $1$ whose interior is formed by the unbroken trajectories from $x$ to $y$; this is still valid for their lifts to the universal cover and therefore $\delta^2=0$. 

We will call $(C_*(\widetilde{X}),\delta)$  {\it lifted complex} and name its homology {\it lifted homology}. As for the usual Morse theory, this homology  does not depend on the choices of $f$, $\xi$, nor on the lifts of the critical points and on the orientations of the unstable manifolds~\cite[\S2.16]{Latour}: it is actually known to be isomorphic to the singular homology of the universal cover $H_*(\widetilde{X};\Z)$, and this is a Morse theoretic manifestation of the well-known fact that $H_*(\widetilde{X};\Z)\simeq H_*(X;\Z[\pi_1(X)])$. In fact the unstable manifolds of the critical points yield a CW-decomposition of the manifold $X$ (see for instance~\cite[Theorem~4.9.3]{Audin-Damian_English}) and it is easy to infer  that the complex defined by its lifting  to the universal cover $\widetilde{X}$ coincides with the lifted complex.

Note that this lifted Morse complex is an example of what we defined as  (algebraic) lifted complex in Definition~\ref{def:algebraic_lifted_complex}. To see this we set $R_*=C_*(\Omega X)$, so that $H_0(R_*)=\Z[\pi_1(X)]$. We fix a spanning tree from a basepoint in $X$ to the critical points of $f$ (such a tree will also play a key role in~\S\ref{section:Morse}), and we also fix a lift in $\tilde X$ of the basepoint. These two choices allow on the one hand to see Morse trajectories between critical points as based loops by collapsing the tree, and on the other hand they determine lifts to $\tilde X$ of all the critical points of $f$ by transporting along the branches of the tree the chosen lift of the basepoint. Given a critical point $x$, we denote $\tilde x$ its lift. If $m_{x,y}\in R_0$ is a Morse trajectory from $x$ to $y$ viewed as a $0$-chain on based loops, and if we denote its image in $H_0(R_0)=\Z[\pi_1(X)]$ by $\hat m_{x,y}$, then the endpoint of the lift of $m_{x,y}$ starting at $\tilde x$ is precisely $\hat m_{x,y}\tilde y$, so that the differential in the geometric lifted complex described in this section is identified with the differential in the algebraic lifted complex from Definition~\ref{def:algebraic_lifted_complex}.

\subsection{Homology with local coefficients}

Let $M$ be a right $\Z[\pi_1(X)]$-module. By definition, the homology with local coefficients in $M$ is the homology of the complex $$
(M\otimes_{{\Z}[\pi_1(X)]}C_*(\widetilde{X}), \Id\otimes\delta),
$$ 
where $(C_*(\widetilde{X}), \delta)$ is the lifted Morse complex. Our goal here is to reinterpret it as a twisted complex.

We consider the $\Z$-module $\cC_\bullet$ spanned by the critical points of $f$ and the DGA $R_*$ defined by the cubical chains on $\Omega X$ modulo the degenerate ones. We also let $\calf$ be the DG right $R_*$-module defined by $\calf_0=M$ and $\calf_k=0$ for $k\neq 0$. The action of $R_*$ in positive degree is trivial, and the action of $R_0=C_0(\Omega X)$ factors through that of $\Z[\pi_1(X)]$ via the canonical projection $C_0(\Omega X)\to \Z[\pi_1(X)]$. We will now define an appropriate twisting cocycle, and this construction can be seen as a baby case of the construction of the Barraud-Cornea cocycle discussed in~\S\ref{section:Morse}. 

Take a basepoint $\star$ on $X$ and fix an embedded rooted tree $\caly$ whose root is the basepoint $\star$ and such that  the  critical points of $f$ are external vertices. Choose a lift of the basepoint   to the universal cover $\widetilde{X}$ and use $\caly$ to fix a lift of each critical point.  Notice that   any trajectory $\lambda$ between critical points $x,y$ of $f$ yields a based  loop   $\gamma_x\#\lambda\# \gamma_y^{-1}$   by concatenation  with the paths to the basepoint along the tree  $\caly$.

If $x$ and $y$ have consecutive indices then we define $m_{x,y}\in R_0= C_0(\Omega X)$ as the algebraic sum of all loops associated to all the  Morse trajectories between $x$ and $y$ -- these trajectories are naturally oriented (see the  proof of Proposition~\ref{orientationsMorse}). It is easy to see that the projection of this sum in $H_0(R_*)= \Z[\pi_1(X)]$ is exactly the $(x,y)$ entry of the matrix given by the differential $\delta$ of the lifted complex $C_*(\widetilde{X})$ described above.

Now take a  broken trajectory $\lambda\#\lambda'$ between points of consecutive indices $|x|=|z|+1=|y|+2$. Connect it by a segment of paths from $x$ to $y$ to $\lambda\#\gamma_z^{-1}\#\gamma_z\#\lambda'$. Concatenating all these segments with the ones given by the  $1$-manifold of unbroken trajectories from $x$ to $y$ we get disjoint segments of paths from $x$ to $y$ connecting all the broken trajectories modified by inserting $\gamma^{-1}\#\gamma$ at their intermediate point. All these segments are oriented in the same way as the unbroken Morse trajectories between $x$ and $y$. Now using $\gamma_x$ and $\gamma_y$ we transform these segments of paths into segments of based loops in $\Omega X$.  Their sum defines a $1$-chain $m_{x,y}\in C_1(\Omega X)$ and one can easily see that by construction 
$$\partial m_{x,y}\, =\, -\sum_{z}m_{x,z}\cdot m_{z,y}.$$ 
The sign $``-"$  is given by the fact that with our  conventions the orientations of a broken trajectory $\lambda\#\lambda'$ as a product and as a boundary of the $1$-manifold above are opposite. These conventions are detailed in~\S\ref{subsec: orientation-conventions} and in the proof of Proposition~\ref{orientationsMorse} from~\S\ref{subsec:Latour-cells}. 

For $|x|-|y|\geq 3$ we can choose $m_{x,y}$ arbitrarily; since the module $\calf$ vanishes in nonzero degrees  the Maurer-Cartan relation of Remark~\ref{MC'} is  automatically satisfied and the differential $\partial$ of the twisted complex will be independent of these choices. Relation~(\ref{MC'}) was checked above for the case $|x|-|y|=2$, while the case $|x|-|y|=1$ is obvious. Therefore, we have defined a twisting cocycle.

By construction, the twisted complex $C_*(\cC_\bullet,R_*)$ associated to this data coincides with $(M\otimes_{{\Z}[\pi_1(X)]}C_*(\widetilde{X}), \Id\otimes\delta)$,
  where the lifted complex $C_*(\widetilde{X})$ is constructed using the lifts of the critical points given by $\caly$ and orientations  of the gradient lines between critical points of consecutive indices which are the same as those of the $0$-chains $m_{x,y}$. 
As a consequence, this twisted complex recovers the homology of $X$ with local coefficients in $M$.





\section{The Morse complex with DG-coefficients}\label{sec:construction-enriched-morse}

In this section we explain the construction of the Morse complex with DG-coefficients. We essentially follow~\cite[\S2.4]{BC07}, with two additions: we construct the Barraud-Cornea twisting cocycle over $\Z$, and we discuss the uniqueness of representing chain systems (Proposition~\ref{prop:uniqueness}).  

The setting is that of~\S\ref{sec:DG-Morse}. We now recall it and give details. We let $X$ be a closed manifold with a fixed basepoint $\star$, we pick a Morse function $f:X\ri\R$ and a generic pseudo-gradient $\xi$ associated to $f$. We consider the free graded $\Z$-module $\cC_\bullet=\langle\Crit(f)\rangle$ and the DGA $R_*$ of cubical chains on the space of Moore loops $\Omega X$.

The space of Moore loops $\Omega X$ consists of pairs $(L,\gamma)$ with $L\ge 0$ and $\gamma:[0,L]\to X$ a loop based at $\star\in X$. This space has the same homotopy type as the space of based loops parametrized  by $I=[0,1]$.  For each degree $k$,  $R_k$ is by definition the quotient  $$R_k\, =\, \frac{C_k(\Omega X)}{C_{k}^{0}(\Omega X)} $$ where   $C_k(\Omega X)$ is spanned by the continuous maps $\sigma :I^k\ri \Omega X$ (``cubes") and $C^0_k(\Omega X)$ is spanned by the degenerate cubes, i.e., the cubes  that factor through a face $I^j$ for some $j<k$. The differential in the cubical complex is given by an alternating sum over the faces of a cube. It is known that the resulting \emph{cubical homology} is isomorphic to singular homology. \footnote{Note that without taking the quotient by the degenerate cubes the cubical homology of a point would be infinitely generated.} The algebra structure of   $R_*$ is given by  the canonical Pontryagin product induced by the concatenation of loops $\#$: 
 $$
 (\sigma\cdot\sigma')(x,y)\, =\, \sigma(x)\#\sigma'(y),\qquad (x,y)\in I^{k+\ell}=I^k\times I^\ell.
 $$
 It is easy to check that $R_*$ is a DGA. We refer to it as {\it the algebra of cubical chains}, or {\it cubical algebra}. We will often abuse notation and write $R_*=C_*(\Omega X)$ for short.

The definition of the twisting cocycle $m_{x,y}$ will take the rest of this section.

In order to define it denote by $\call(x,y)$ the space of (unparametrized) trajectories of $\xi$ between $x,y\in \Crit(f)$ endowed with the usual topology; it is a smooth manifold of dimension $|x|-|y|-1$. A choice of orientations on the unstable manifolds $W^u(x)$ associated to the critical points of $f$ naturally determines an orientation on $\call(x,y)$.

It is a classical fact in Morse theory that $\call(x,y)$ admits a compactification $\overline{\call}(x,y)$ given by the space of broken trajectories  
$$\overline{\call}(x,y)\, =\, \call(x,y) \cup \bigcup_{z_1,\cdots, z_{k-1}\in Crit(f)}\call(x,z_1)\times\call(z_1,z_2)\times\cdots\times\call(z_{k-1},y), $$ 
with 
$$
\p \overline{\call}(x,y) = \bigcup_{|x|>|z|>|y|} \, \overline{\call}(x,z) \times \overline{\call}(z,y).
$$

The following theorem was proved by various authors (Latour~\cite[Proposition~2.11]{Latour}, Qin~\cite[Theorem~3.4]{Qin-Lizhen}).

\begin{theorem}\label{corners1} The space $\overline{\call}(x,y)$ admits the structure of manifold with boundary and corners whose interior is $\call(x,y)$. Near a  boundary point $$\lambda=(\lambda_1,\lambda_2,\ldots\lambda_k)\in  \call(x,z_1)\times\call(z_1,z_2)\times\cdots\times\call(z_{k-1},y)$$ the space $\overline{\call}(x,y)$ is locally diffeomorphic to (a neighborhood of $(\lambda,0)$ in) $\call(x,z_1)\times\call(z_1,z_2)\times\cdots\times\call(z_{k-1},y)\times [0,\infty)^k$. \qed
\end{theorem}

In addition, once we fix an orientation of the unstable manifolds,  $\overline{\call}(x,y)$ becomes oriented, cf.~\cite[2.15]{Latour}  (see also the proof of Proposition~\ref{orientationsMorse} below). In particular, $\overline{\call}(x,y)$ possesses a well-defined fundamental class $$
[\overline{\call}(x,y)]\in H_{|x|-|y|-1}(\overline{\call}(x,y),\partial \overline{\call}(x,y)).
$$

  Before going further we need to fix some orientation conventions.
  
  \subsection{Orientation conventions}\label{subsec: orientation-conventions} 

 We will denote by $\ori(V)$ the orientation of an oriented vector space $V$ and by $\coori(V)$ the co-orientation of a co-oriented vector space $V$. Given an oriented manifold $Y$, we will similarly denote by $\ori(Y)$ the orientation of its tangent space (which depends on the point). Similarly, we will denote by $\coori(Y)$ the coorientation of a cooriented submanifold $Y$.

We will use further similar notations whose meaning can also be inferred from the context. For instance, 
given oriented vector spaces $W$ and $Z$, let $(\ori(W), \ori(Z))$
be the orientation of $W\oplus Z$ represented by a basis obtained by concatenating a positively oriented basis of $W$ and a positively oriented basis of $Z$ in this order. Similarly, if $W\subset Z$ are subspaces of $V$ such that $W$ is oriented and $Z$ is cooriented, let $(\ori \, W, \coori \, Z)$ be the orientation of $W\oplus V/Z $ defined by a basis of $W$ followed by a basis of $V/Z$.

  If $W=\mathrm{span}(u)$ is one-dimensional, we will simply write $u$ instead of $\ori \, W$, the orientation defined by $u$.

  Given $y\in \Crit(f)$ we denote $W^u(y)$ its \emph{unstable manifold} with respect to $\xi$ and $W^s(y)$ its \emph{stable manifold} with respect to $\xi$. We denote $S^{s}(y)$ its \emph{stable sphere} defined by 
$$S^s(y)\, =\, W^s(y) \cap f^{-1}(f(y)+\epsilon)$$
for some small $\epsilon >0$. We fix an orientation for $W^u(y)$, and this yields a co-orientation for $W^s(y)$. We co-orient $S^s(y)\subset W^s(y)$ at a point $p$ using the exterior  normal which is the opposite of the vector field $\xi(p)$. 
This choice gives therefore 
\begin{equation}\label{eq:coorientation-sphere-stable} \coori(S^s(y))\, =\, -(\xi, \coori(W^s(y)). \end{equation} 

We also use the following conventions. We orient the boundary $\partial Y$ of some manifold $Y$ using ``its outward normal",  meaning any vector $n$ pointing towards the exterior of the boundary: 
\begin{equation}\label{eq:exterior-normal} 
(n, \ori(\partial Y))\, =\, \ori(Y).\end{equation} 
We orient a transverse intersection between an oriented manifold $X$ and a co-oriented manifold $Y$ as follows:
\begin{equation}\label{eq:transverse-orientation} 
(\ori(X\cap Y), \coori(Y))\,  =\, \ori(X).
\end{equation} 

We will also need  to consider the compactifications $\overline{W}^u(x)$ of the unstable manifolds. These appeared first in Latour \cite{Latour} and were also considered in Barraud-Cornea~\cite{BC07}, see also~\cite{Qin-Lizhen}. We will refer to them as {\it Latour cells} and discuss them in the following subsection. 

\subsection{Latour cells}\label{subsec:Latour-cells} 

\begin{definition} \label{Latour-cell} For $x\in \Crit(f)$ the \emph{unstable Latour cell} is defined by 
$$\overline{W}^u(x)\, =\, W^u(x)\cup\bigcup_{y\in {Crit}(f)}\overline{\call}(x,y)\times W^u(y).$$\end{definition} 

There is a natural topology on these spaces which makes them compact. Moreover Latour~\cite[Proposition~2.11]{Latour} and Qin~\cite[Theorem~3.4]{Qin-Lizhen} prove the following theorem (see also Barraud-Cornea~\cite{BC07} and Audin-Damian~\cite{Audin-Damian_English} for slightly weaker statements).

\begin{theorem}\label{corners2} The space $\overline{W}^u(x)$ is a  compact  manifold with boundary and corners of dimension $|x|$ whose interior is $W^u(x)$. Near a  boundary point $$(\lambda, p)= (\lambda_1,\lambda_2,\ldots\lambda_{k}, p)\in  \call(x,z_1)\times\call(z_1,z_2)\times\cdots\times\call(z_{k-1},y )\times W^u(y) $$ the space $\overline{W}^u(x)$ is locally diffeomorphic to (a neighborhood of $(\lambda,p)$ in) $\call(x,z_1)\times\call(z_1,z_2)\times\cdots\times\call(z_{k-1},y)\times W^u(y)\times [0,\infty)^k$. 

Moreover $\overline{W}^u(x)$ is homeomorphic to the closed disk $\overline{D}^{|x|}.$ \qed
\end{theorem}

There is a natural extension 
$$
\phi_x: \overline{W}^u(x)\ri X
$$ 
of the inclusion $W^u(x) \hookrightarrow X$ defined by 
$$\phi_x(\lambda, p) = p.$$

\begin{figure}
  \centering
  \includegraphics{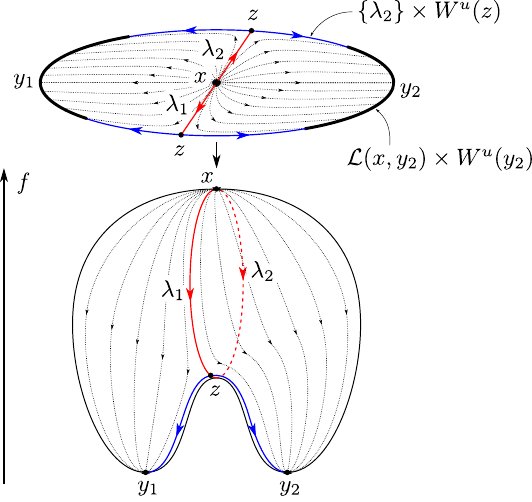} 
  \caption{The Latour cell  $\ol W^u (x)$ of the maximum
    of the height function on a deformed $2$-sphere. }
  \label{fig:MorseLatour}
\end{figure}
%

This actually defines a CW-decomposition of $X$ and provides a proof of the fact that  (lifted) Morse homology is isomorphic to singular homology of $X$ (resp. of its universal cover $\widetilde{X}). $ In fact, with our orientation conventions the Morse differential is the opposite of the cellular differential: this is a consequence of the relation~\eqref{eq:sign-difference-bord-cellule} below by considering  the particular case $|x|-|z|=1$.

\begin{proposition}[{\cite[2.15]{Latour}}] \label{orientationsMorse} 
Let $x,z,y$ be three critical points of $f$ with $|x|>|z|>|y|$. The orientation induced from $\overline{\call}(x,y)$ by the outward normal vector on the $1$-codimensional stratum  $\call(x,z)\times\call(z,y)$ of the boundary differs from the product orientation by $(-1)^{|x|-|z|}$. 
\end{proposition} 

The same result was reproved in~\cite[Theorem~3.6]{Qin-Lizhen} and~\cite[Proposition~5.1]{Zhou-MB}. Since Latour uses in~\cite{Latour} positive gradients and different orientation conventions we choose to write a self-contained proof for this statement. 

\begin{proof}
Let $x,y,z$ be critical points of $f$ as in the statement of the proposition. We will use the identification 
\begin{equation}\label{eq:multiliaisons-intersection}
\overline{\call}(x,y) = \overline{W}^u(x)\cap S^s(y),
\end{equation} 
meaning that we see $\overline{\call}(x,y)\subset \overline{W}^u (x)$ as the preimage of $S^s(y)$ by the map $\phi_x$.
Now $\phi_x$ is transverse to $S^s(y)$ on the interior $W^u(x)$ and also along  any stratum of the boundary of $\overline{W}^u(x).$  In particular $\phi_x\pitchfork S^s(y)$ on the component  $\call(x,z)\times W^u(z)$ of the  stratum of maximal dimension. Therefore near a point $(\lambda, p)\in \call(x,z)\times W^u(z)$, we have that  $\overline{\call}(x,y)$ is a genuine boundary submanifold of $\overline{W}^u(x)$.\footnote{ By ``boundary submanifold" we mean a transverse intersection between a manifold with boundary and a manifold without boundary.} 
As such, it is clear that an outward normal vector $n$ of $\overline{W}^{u}(x)$ at $(\lambda,p)$  is also outward normal for $\overline{\call}(x,y)$. We thus have by~\eqref{eq:exterior-normal}: 
\begin{equation}\label{eq:normal-orientation-cellule} 
(n, \ori\, \partial\wu(x))\, =\, \ori \, \wu(x)
\end{equation}
and 
\begin{equation}\label{eq:normal-orientation-multiliaisons}
(n, \ori\,  \partial\linecall(x,y))\, =\, \ori\, \linecall(x,y).
\end{equation} 
 Let us now denote by $\xi^-$ the outward normal vector to $\overline{W}^u(x)$ at $(\lambda,z)$ obtained by continuously extending a tangent vector\footnote{For instance, one can take $\xi(\lambda(t))/\|\xi(\lambda(t))\|$ for a given norm.} to the curve $\lambda$ at the point $\lambda(t)$ as $t\to+\infty$.  
We have 
\begin{align*}
(n, \ori\, \call(x,z), \ori \, W^u(z)) \, & =\, (\xi^-, \ori\, \call(x,z), \ori\, W^u(z)) \\
& = \, (-1)^{|x|-|z|-1}(\ori\, \call(x,z), \xi^-, \ori\, W^u(z)).
\end{align*}
The first equality above should be seen as an equivalence between orientations of the boundary at the points $(\lambda,p)$ and $(\lambda,z)$.

\begin{figure}
  \centering
  \includegraphics{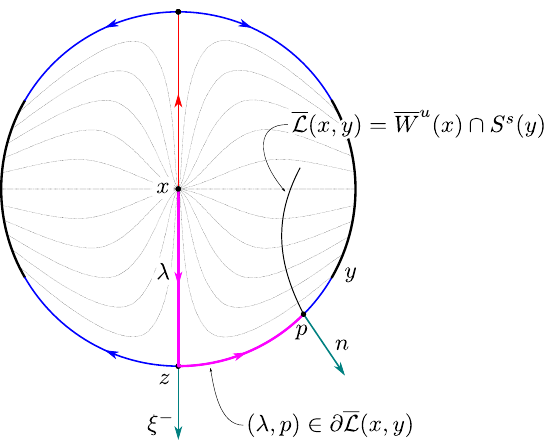}   
  \caption{Comparing orientations.}
  \label{fig:orientation}
\end{figure}

Now $\ori\, W^u(z)= \coori \, W^s(z)$ and, using~\eqref{eq:coorientation-sphere-stable}, we obtain 
\begin{align*}
(n, \ori\, \call(x,z), \ori \, W^u(z)) & =\, (-1)^{|x|-|z|-1}(\ori\, \call(x,z), \xi^-, \coori\, W^s(z)) \\
& =\, (-1)^{|x|-|z|}(\ori\, \call(x,z), \coori\, S^s(z)).
\end{align*}
For the transverse intersection $\call(x,z)= W^u(x)\pitchfork S^s(z)$ we infer from our orientation convention~\eqref{eq:transverse-orientation} that
\begin{align*}
(n, \ori\, \call(x,z), \ori \, W^u(z))\, & =\, (-1)^{|x|-|z|}(\ori\, W^u (x))\, \\
& =\, (-1)^{|x|-|z|}(\ori\, \wu(x)).
\end{align*}
Combining this with~\eqref{eq:normal-orientation-cellule} we get 
\begin{equation}\label{eq:sign-difference-bord-cellule} 
\ori (\call(x,z)\times W^u(z))\, =\, (-1)^{|x|-|z|}(\ori \, \partial \wu(x)).
\end{equation} 
Let us prove that the orientations of $\partial \linecall(x,y)$ and $\call(x,z)\times\call(z,y)$ differ by the sign claimed in our statement. Recall that $\linecall(x,y)= \wu(x)\pitchfork S^s(y)$, so 
\begin{equation}\label{eq:or-coor} 
(\ori \, \linecall(x,y), \coori \, S^s(y)) \, =\, \ori \, \wu(x), 
\end{equation} 
and using~\eqref{eq:normal-orientation-multiliaisons}
\begin{equation}\label{eq:final1} 
(n, \ori \, \partial\linecall(x,y), \coori \, S^s(y)) \, =\, \ori\,  \wu(x).
\end{equation}

On the other hand $\call(y,z)=W^u (z)\pitchfork S^s(y)$ and therefore 
\begin{align}
(n, \ori\, \call(x,z), & \ori\, \call(z,y), \coori\, S^s(y)) \nonumber \\
 & =\, (n, \ori\, \call(x,z), \ori \, W^u(z)) \nonumber \\
 & =\, (-1)^{|x|-|z|}(n, \ori \,  \partial \wu(x)) \nonumber \\
 & =\, (-1)^{|x|-|z|}\ori \,  \wu(x), \label{eq:final2}
 \end{align}
using~\eqref{eq:sign-difference-bord-cellule} and~\eqref{eq:normal-orientation-cellule}. Finally,~\eqref{eq:final1} and~\eqref{eq:final2} yield 
$$\ori\, \partial \linecall(x,z)\, =\, (-1)^{|x|-|z|}\ori(\call(x,z)\times\call(z,y))$$
as claimed. 
\end{proof}

\begin{remark}\label{orientation-rule} The above proof yields the following orientation rule  derived from~\eqref{eq:coorientation-sphere-stable} and~\eqref{eq:or-coor} 
\begin{equation}\label{eq:orientation-rule} 
\left(\ori \, \overline{\call}(x,y), -\xi(p), \ori\, W^u(y)\right) \, =\, \ori\, \wu(x),
\end{equation}
where $\xi(p)$ is the gradient vector field at some point $p\in \wu(x)$ lying on a gradient line $\lambda\in \overline{\call}(x,y)$. 
\end{remark}

We will use this result to prove the following (see \cite[Lemma~2.2]{BC07}). 

\begin{proposition}[Existence] \label{representingchain} 
Let $C_*(\overline{\call}(x,y))$, $x,y\in\Crit(f)$ be the complex of cubical chains on the space $\overline{\call}(x,y)$ of broken trajectories. There exists a collection 
$\{s_{x,y}\}$ with $s_{x,y}\in C_{|x|-|y|-1}(\overline{\call}(x,y))$ satisfying the following properties:  
\begin{enumerate}
\item $s_{x,y}$ is a cycle relative to the boundary and represents the fundamental class $[\overline{\call}(x,y)]$, 
\item  the following equation holds
\begin{equation}\label{eq:representingchain}
\partial s_{x,y}\, =\, \sum_z(-1)^{|x|-|z|}s_{x,z}\times s_{z,y},
\end{equation}
where the product of chains is defined via the inclusions $\overline{\call}(x,z)\times\overline{\call}(z,y)\, \subset \, \partial \overline{\call}(x,y)\, \subset\, \overline{\call}(x,y)$.
\end{enumerate}
\end{proposition}

\begin{proof} Barraud and Cornea prove this result in~\cite{BC07} over $\mathbb{Z}/2$. For arbitrary coefficients the proof follows exactly the same lines, taking in addition into account orientations. The key idea is to build the chain representatives $s_{x,y}$ inductively over $\ell=|x|-|y|-1\ge 0$: at the induction step, an arbitrary representative of the fundamental class is modified by a chain supported on the boundary so that it satisfies the Maurer-Cartan equation~\eqref{eq:representingchain}.

For $\ell=0$ the moduli spaces $\overline{\call}(x,y)$ are compact, $0$-dimensional, and oriented. We define $s_{x,y}$ to be the unique $0$-chain which represents the fundamental class. 

Let $\ell\ge 1$. Assuming that we have constructed $\{s_{x,y}\}$ for $|x|-|y|-1\le \ell-1$, we construct $\{s_{x,y}\}$ for $|x|-|y|-1=\ell$ as follows. 

The space $\p \overline{\call}(x,y)$, oriented as the boundary of $\overline{\call}(x,y)$, carries a fundamental class $[\p \overline{\call}(x,y)]$. Orient each top-dimensional stratum $\overline{\call}(x,z)\times \overline{\call}(z,y)\subset \p \overline{\call}(x,y)$ as the boundary of $\overline{\call}(x,y)$. As such, it carries a fundamental class 
$$
[\overline{\call}(x,z)\times \overline{\call}(z,y)]\in H_{|x|-|y|-2}\left(\overline{\call}(x,z)\times \overline{\call}(z,y),\p \big( \overline{\call}(x,z)\times \overline{\call}(z,y)\big)\right).
$$ The chain $(-1)^{|x|-|z|}s_{x,z}\times s_{z,y}\in C_{|x|-|y|-2}(\overline{\call}(x,z)\times \overline{\call}(z,y))$ is a cycle rel boundary and, by the induction assumption and Proposition~\ref{orientationsMorse}, it represents the fundamental class $[\overline{\call}(x,z)\times \overline{\call}(z,y)]$  seen as a class rel boundary. 

On the other hand, by the induction assumption using~\eqref{eq:representingchain} one can easily check that the chain 
$$\sum_z (-1)^{|x|-|z|}s_{x,z}\times s_{z,y}\in  C_{|x|-|y|-2}(\p \overline{\call}(x,y))$$ is a genuine cycle (not only a cycle rel boundary).
Therefore  the previous paragraph implies that it represents the fundamental class $[\p \overline{\call}(x,y)]$. 

We now pick any cubical chain representative $s'_{x,y}$ of the fundamental class $[\overline{\call}(x,y)]$. In particular, $s'_{x,y}$ is a cycle rel boundary. Since the fundamental class $[\p \overline{\call}(x,y)]$ is the image of $[\overline{\call}(x,y)]$ under the boundary homomorphism $H_{|x|-|y|-1}(\overline{\call}(x,y), \partial  \overline{\call}(x,y))\ri H_{|x|-|y|-2}( \partial  \overline{\call}(x,y))$, we infer that it is represented by $\p s'_{x,y}$. Thus $\p s'_{x,y}$ and $\sum_z (-1)^{|x|-|z|}s_{x,z}\times s_{z,y}$ are homologous as representatives of $[\p \overline{\call}(x,y)]$, i.e., there exists $p_{x,y}\in C_{|x|-|y|-1}(\p \overline{\call}(x,y))$ such that 
$$
\p s'_{x,y} - \sum_z (-1)^{|x|-|z|}s_{x,z}\times s_{z,y} =\p p_{x,y}. 
$$
We then set $s_{x,y}=s'_{x,y}-p_{x,y}$. 
\end{proof}


\begin{definition}\label{chainsystem} Following Barraud and Cornea we will call the family $(s_{x,y})$ a \emph{representing chain system} for  the Morse moduli spaces.
\end{definition} 

\begin{proposition}[Uniqueness] \label{prop:uniqueness} 
Any two representing chain systems $(s'_{x,y})$ and $(s_{x,y})$ are homologous in the following sense: there exists a family $(\kappa_{x,y})$ of chains $\kappa_{x,y}\in C_{|x|-|y|}(\overline{\call}(x,y))$ such that:
\begin{enumerate}
\item $\kappa_{x,x}$ is the constant $0$-chain for all $x$ and $\kappa_{x,y} = 0$ for $|x] =|y|$, $x\neq y$. 
\item for all $x,y$ we have 
\begin{equation}\label{eq:sprime-s-homologous1}
\p \kappa_{x,y} = \sum_z s'_{x,z}\times \kappa_{z,y} - (-1)^{|x|-|z|}\kappa_{x,z}\times s_{z,y}.
\end{equation}
\end{enumerate}
\end{proposition}

In~\eqref{eq:sprime-s-homologous1} the product of chains should be understood via the inclusion $\overline{\call}(x,z)\times\overline{\call}(z,y)\, \subset \, \partial \overline{\call}(x,y)\, \subset\, \overline{\call}(x,y)$ for $|x|>|z|>|y|$, while $s'_{x,y}\times \mathit{point}$ and $\mathit{point} \times s_{x,y}$ are identified with $s'_{x,y}$ and $s_{x,y}$. Also, the right hand side does not involve $\kappa_{x,y}$ since $s'_{x,x}=0$ and $s_{y,y}=0$. Thus Equation~\eqref{eq:sprime-s-homologous1} reads equivalently 
$$
\p \kappa_{x,y} = s'_{x,y}-s_{x,y} + \sum_{|x|>|z|>|y|} s'_{x,z}\times \kappa_{z,y} - (-1)^{|x|-|z|}\kappa_{x,z}\times s_{z,y}.
$$

\begin{proof}[Proof of Proposition~\ref{prop:uniqueness}] We construct the chains $\kappa_{x,y}$ inductively over $\ell=|x|-|y|\ge 0$. For $\ell=0$ we set $\kappa_{x,x}$ to be the constant $0$-chain at $x$ and $\kappa_{x,y}= 0$ for $x\neq y$. 

  Let now $\ell\ge 1$. Assuming that $\{\kappa_{z,w}\}$ have been constructed for $|z|-|w|\le \ell -1$, we will construct $\{\kappa_{x,y}\}$ for $|x|-|y|=\ell$. A direct computation shows that the right hand side of (\ref{eq:sprime-s-homologous1}), 
  denoted 
 $$
 r_{x,y}=\sum_z s'_{x,z}\times \kappa_{z,y} - (-1)^{|x|-|z|}\kappa_{x,z}\times s_{z,y},
 $$ 
 is a cycle in $C_{|x|-|y|-1}(\overline{\call}(x,y))$, i.e., 
$$
\p r_{x,y}=0. 
$$
Indeed, by the induction assumption we have
\begin{align*}
\p (\sum_z & s'_{x,z}\times \kappa_{z,y} - (-1)^{|x|-|z|}\kappa_{x,z}\times s_{z,y}) \\
= \, & \sum_{z,u} (-1)^{|x|-|u|} s'_{x,u}\times s'_{u,z} \times \kappa_{z,y} \\
& + (-1)^{|x|-|z|-1} \sum_{z,u}s'_{x,z}\times s'_{z,u}\times \kappa_{u,y} - (-1)^{|z|-|u|} s'_{x,z}\times \kappa_{z,u}\times s_{u,y} \\
& - (-1)^{|x|-|z|} \sum_{z,u} s'_{x,u}\times \kappa_{u,z}\times s_{z,y} - (-1)^{|x|-|u|}\kappa_{x,u}\times s_{u,z}\times s_{z,y} \\
& - (-1)^{|x|-|z|} (-1)^{|x|-|z|} \sum_{z,u} (-1)^{|z|-|u|} \kappa_{x,z}\times s_{z,u} \times s_{u,y} \\
= \, & 0.
\end{align*}
We now write this cycle as a sum $r_{x,y}=\sum_\alpha r_{x,y}^\alpha$ over connected components $\overline{\call}^\alpha(x,y)$ of $\overline{\call}(x,y)$, so that each $r_{x,y}^\alpha\in C_{|x|-|y|-1}(\overline{\call}^\alpha(x,y))$ is a cycle. We construct $\kappa_{x,y}=\sum_\alpha \kappa_{x,y}^\alpha$ component-wise, with $\kappa_{x,y}^\alpha\in C_{|x|-|y|}(\overline{\call}^\alpha(x,y))$. 
\begin{itemize} 
\item For a connected component $\overline{\call}^\alpha(x,y)$ with empty boundary we have $r_{x,y}^\alpha = {s'}_{x,y}^\alpha - s_{x,y}^\alpha$. Since ${s'}_{x,y}^\alpha$ and $s_{x,y}^\alpha$ both represent the fundamental class of $\overline{\call}^\alpha(x,y)$ they are necessarily homologous and we choose $\kappa_{x,y}^\alpha$ such that $\p \kappa_{x,y}^\alpha=r_{x,y}^\alpha$. 
\item For a connected component of $\overline{\call}^\alpha(x,y)$ with non-empty boundary we have $H_{|x|-|y|-1}(\overline{\call}^\alpha(x,y))=0$. Since $r_{x,y}^\alpha$ is a cycle we infer the existence of $\kappa_{x,y}^\alpha$ such that $\p \kappa_{x,y}^\alpha=r_{x,y}^\alpha$. 
\end{itemize}
\end{proof}

\begin{remark}
Proposition~\ref{prop:uniqueness} is a particular case of the comprehensive Invariance Theorem~\ref{independence} proved in~\S\ref{sec:invariance-continuation}. Although it deviates slightly from the overall philosophy of the book, which would require to construct the continuation cocycle from a Morse problem on $X\times [0,1]$, the direct proof that we presented above has the merit of showing explicitly how the continuation cocycle is built up in a concrete example.
\end{remark}

In order to define a twisting cocycle in $R_* = C_*(\Omega X)$ we need the following lemma (see also~\cite[\S2.2.1]{BC07}).  

\begin{lemma}\label{lecture}  There exists a family of continuous maps
$$q_{x,y} : \overline{\call}(x,y)\ri \Omega X$$ such that:
\begin{enumerate}
\item If $|x|-|y|=1$, then for any $\lambda\in \call(x,y)= \overline{\call}(x,y)$ the homotopy class $[q_{x,y}(\lambda)]\in\pi_{1}(X)$ coincides with the one that is assigned to $\lambda$ in the lifted Morse complex, constructed using a fixed lift of the tree $\caly$ to the universal cover $\widetilde{X}$ 
(see~\S\ref{sec:lifted-homology}). 
\item For  any $(\lambda, \lambda')\in \overline{\call}(x,z)\times\overline{\call}(z,y)$ we have 
  $$q_{x,y}(\lambda,\lambda')\, =\, q_{x,z}(\lambda)\# q_{z,y}(\lambda'),$$
  where $\#$ stands for the concatenation of paths. 
\end{enumerate}
\end{lemma}

\begin{proof} For $x,y\in \Crit(f)$ consider the natural map
$$\Gamma: \overline{\call}(x,y)\ri \calp_{x\ri y}X $$ which sends each broken orbit $\lambda$ to the path from $x$ to $y$ parametrised by the values of $f$. More precisely $\Gamma(\lambda) =\gamma :[0,f(x)-f(y)]\ri X$ is defined by 
$$\gamma(t) \, =\, \lambda \cap f^{-1}(f(x)-t).$$
We clearly have 
$$\Gamma(\lambda,\lambda')\, =\, \Gamma(\lambda)\#\Gamma(\lambda').$$
Consider now the projection 
$$
p:X\ri X/\caly,
$$ 
where $\caly$ is the chosen tree that connects the basepoint to the critical points.  We use this projection in order to avoid the conjugation with the paths $\gamma_x$ from~\S\ref{sec:lifted-homology}. 
Since $\caly$ is contractible the projection $p$ is a homotopy equivalence; pick a homotopy inverse 
$$
\theta : (X/\caly,\star)\ri (X,\star).
$$ 
The homotopy between $\theta \circ p$ and $\Id$ can be assumed without loss of generality to preserve the basepoint $\star$.  We define $$q_{x,y} \, =\, \theta\circ p\circ \Gamma.$$
These maps clearly satisfy condition 2 since $\Gamma$ does and $\theta$ and $p$ act point-wise. Condition 1 is also fulfilled by construction, since $\theta\circ p$ is homotopic to the identity. Indeed, given $\lambda \in \overline{\call}(x,y)=\call(x,y)$ we denote $\gamma_x$, $\gamma_y$ the branches of $\caly$ between the root $\star$ and $x$, resp. $y$, and consider $\lambda_\star= \gamma_x\#\Gamma(\lambda)\#\gamma_y^{-1}\in \Omega X$. We then have 
$$
[q_{x,y}(\lambda)]\, =\,  [\theta \circ p\circ \Gamma(\lambda)] \, =\, [\theta \circ p (\lambda_\star)]\, =\, [\lambda_\star],  
$$ 
and $[\lambda_\star]$  is by  definition the homotopy class assigned to $\lambda$ in the lifted complex. 
\end{proof}

We are now in position to define the twisting cocycle. We denote 
$$
\Xi=(f,\xi,o, s_{x,y},\cY,\theta)
$$ 
the data consisting of the following objects, as above: 
\begin{itemize}
\item the Morse function $f$. 
\item the Morse-Smale negative pseudo-gradient vector field $\xi$. 
 \item the orientation $o = (o_x)_{x\in \crit(f)}$ of the unstable manifolds of the critical points of $f$. 
\item the representing chain system $(s_{x,y})$, depending in particular on the choice   of the orientation $o$. 
\item the tree $\caly$ with root at the basepoint $\star$.
\item the map $\theta :(X/\caly, \star) \ri (X, \star)$ homotopy inverse to the projection $p:X\ri X/\caly$.  
\end{itemize}

\begin{definition}[twisting cocycle and twisted Morse complex] \qquad 

\emph{The Barraud-Cornea twisting cocycle $(m_{x,y})$ associated to $\Xi$} is defined by    
\begin{equation}\label{eq:twistingcocycleMorse} 
m_{x,y}\, =\, (q_{x,y})_*s_{x,y} \in C_{|x|-|y|-1}(\Omega X).
\end{equation} 
\emph{The twisted Morse complex with coefficients in a right DG-module $\calf$ over $C_*(\Omega X)$} is defined using equation~\eqref{diffDGmodule} and is denoted 
$$
C_*(X,\Xi;\calf). 
$$
\end{definition} 

The definition is sustained by the fact that the family $(m_{x,y})$ satisfies the MC equation~\eqref{MC}, cf.  Proposition~\ref{representingchain}~(2) and Lemma~\ref{lecture}~(2): 
\begin{align*}
  \partial m_{x,y} & = \partial (q_{x,y})_*s_{x,y}  =  (q_{x,y})_* \partial s_{x,y} = (q_{x,y})_*\left(\sum_{z} (-1)^{|x|-|z|}s_{x,z}\times s_{z,y}\right) \\
                   &= \sum_{z} (-1)^{|x|-|z|}(q_{x,y})_*(s_{x,z}\times s_{z,y}) \\
                   &= \sum_{z} (-1)^{|x|-|z|}(q_{x,z})_*s_{x,z}\cdot (q_{z,y})_*s_{z,y} \\
                   &= \sum_{z} (-1)^{|x|-|z|}m_{x,z}\cdot m_{z,y}. 
\end{align*}

We include in the notation of the twisted Morse complex the auxiliary data $\Xi$ in order to stress the dependence at chain level on all the choices constituting~$\Xi$. 


Using Lemma~\ref{lecture}~(1) we immediately infer that the cocycle $(m_{x,y})$ is compatible with the lifted complex,  meaning that for $|x|=|y|+1$ the homology classes $\hat m_{x,y}$ of $m_{x,y}$ in $H_0(\Omega X)= \Z[\pi_1X]$ are the entries of the matrix of the lifted differential. 

\begin{remark}\label{rem:spectral-sequence-Morse}  Once the twisted Morse complex has been defined, we can realize it as the limit of a spectral sequence as in the algebraic section~\S\ref{sec:spectral-sequence}. According to Lemma~\ref{spectral} the second page of this spectral sequence is 
$$
E_{pq}^2\, =\, H_{p}(C_*(\widetilde{X}); H_q(\calf))\,:= \, H_p(X; H_q(\calf))
$$
i.e., the homology of $X$ with local coefficients in $H_q(\calf)$ as defined in~\S\ref{sec:local-coeff}.
\end{remark} 

\section[Manifolds with boundary]{DG-Morse homology for manifolds with boundary}\label{sec:DG-for-boundary}  In the case of a manifold with boundary $(X,\p X)$ we proceed as follows. We start by fixing   in a collar  neighborhood $V= [-\delta, 0]\times \p X$ of the boundary $\p X\equiv \{0\}\times \p X$ a vector field $\xi$  which is transverse to $\{0\}\times \p X$. We choose $f$ on $V$ to be strictly decreasing along the flow lines  of $\xi$ and then we extend it to a Morse function on $X$. Finally we extend $\xi$ on $X$ to a (negative) pseudo-gradient for $f$. See \cite[\S3.5]{Audin-Damian_English}  for a similar discussion on the definition of usual Morse homology for manifolds with boundary.  We choose the other elements needed to construct the enriched complex in a similar way and get a set of data $\Xi$ as in the closed case.  There is no  difference in the construction of the representing chain system: the gradient lines between critical points stay away from the boundary and thus the structure of the trajectory spaces $\overline\call(x,y)$ is the same.  Therefore we may  define the  enriched complex  $C_*(X;\Xi;\calf)$ analogously. However its homology will depend on the choice of  the direction of $\xi$ (inwards or outwards) along the   boundary: 

It is convenient to write $\p X$ as the disjoint union $\p_{-}X\sqcup \p_+ X$ of its components along which $\xi$ respectively points inwards and outwards; we may think of $X$ as a cobordism.  We use the notation 
$$H_* (X, \p_+ X; \calf)$$
for the homology of $C_*(X,\Xi;\calf)$, 
motivated by the fact that, in the case where $\calf$ is the trivial local system (viewed as DG-module), we get the usual singular homology $H_*(X,\p_+X)$ with integer coefficients, as shown in  \cite{Audin-Damian_English}. However, to justify the notation above, we still have to prove   that this homology only depends on the cobordism $(X,\p_-X,\p_+X)$ and on the DG-module $\calf$ over $C_*(\Omega X)$. This is the purpose of the next chapter \S\ref{sec:invariance-homology}. 

Finally note that, as in the case of closed manifolds (Remark \ref{rem:spectral-sequence-Morse}), we have a spectral sequence which converges to $H_*(X,\p_+ X;\calf)$ and whose second page is in this case $$E_{pq}^2\, =\, H_p(X, \p_+X; H_q(\calf)).$$

\chapter[DG Morse homology: invariance]{Morse homology with DG-coefficients: invariance}\label{sec:invariance-homology} 

The previous construction of the enriched Morse complex with coefficients in a $C_*(\Omega X)$-module $\calf$ depend on all the choices involved in the set of auxiliary data $\Xi=(f,\xi, o, s_{x,y},\cY,\theta)$. 
The goal of this chapter is to show that the homology of the enriched complex does not depend on any of these choices,  see Theorem~\ref{independence} below.  There is no difference between the closed case and the boundary case in the proofs below. 

\section[Invariance for (lifted) Morse homology]{Invariance for the usual and for the lifted Morse homology} \label{sec:invariance-for-usual}

We recall in this section the proof of invariance in the classical setting of Morse homology, where the auxiliary data consists of $f$, $\xi$, the orientation of the unstable manifolds $o$, and --- in the case of lifted Morse homology --- the choice of lifts of the critical points of $f$ to the universal cover $\widetilde{X}$.

 Let $(f_0, \xi_0)$, $(f_1,\xi_1)$ be two Morse-Smale pairs on $X$  and fix orientations $o_0$ respectively $o_1$ of the corresponding unstable manifolds. The proof of the invariance of usual Morse homology, as discussed for example in~\cite[\S3.4]{Audin-Damian_English}, involves the choice of a homotopy $(f_t)_{t\in [-\epsilon, 1+\epsilon] }$ between $f_0$ and $f_1$, supposed to be stationary for $|t|\leq \epsilon$ and for $|t-1|\leq \epsilon$. Here $\epsilon>0$ is some small positive real number. One then defines a function $$F:[-\epsilon, 1+\epsilon] \times X\ri \R$$ by the formula $$F(t,x)\, = \, f_t(x)+g(t), $$
 where $g:[-\epsilon, 1+\epsilon] \ri \R$ has exactly two critical points (a maximum at $t=0$ and a minimum at $t=1$) and satisfies $g'(t)\ll 0$ for $t\in [\epsilon, 1-\epsilon]$, so that 
 $$
 \Crit(F)\, =\, \left(\{0\}\times \Crit(f_0)\right) \cup \left(\{1\}\times\Crit(f_1)\right).
 $$
We define a negative pseudo-gradient $\xi$ for $F$ by setting it to be equal to $\xi_0-\nabla g$ on $[-\epsilon, \epsilon]\times X$, equal to $\xi_1-\nabla g$ on $[1-\epsilon, 1+\epsilon]\times X$, and by interpolating between these two vector fields on the remaining part of the product space, where we have in particular $dF(\xi) <0$. By construction the Morse complex associated to $(F,\xi)$  satisfies  
$$
C_{k}(F)\, =\, C_{k-1}(f_0)\oplus C_k(f_1).
$$ 
 In order to define the differential we first have to orient the unstable manifolds of $(F,\xi)$. If $x\in \{0\}\times \Crit(f_0)$ then notice that at a point $q$ of a Morse neighborhood of $x$ in $[-\epsilon, 1+\epsilon]\times  X$ the tangent space $T_qW_F^{u}(x)$ is the product  $\R \times T_qW^{u}_{ f_0}(x)$ where the factor $\R$ stands for the tangent space to the interval $[-\epsilon, 1+\epsilon]$. We choose the natural orientation 
\begin{equation}\label{eq:orientation-produit}\ori\, W^{u}_{F}(x)\, =\, \left(\tfrac\partial{\partial t}, \ori \, W^{u}_{f_0}(x)\right).\end{equation}
For $x\in \{1\}\times \Crit(f_1)$ we have $W^u_F(x)=W^u_{f_1}(x)$ and we keep the same orientation  
\begin{equation}\label{eq:orientation-produit2}\ori\, W^{u}_{F}(x)\, =\,  \ori \, W^{u}_{f_1}(x).\end{equation}
It is easy to check (one may also apply the more general Lemma~\ref{orientation-difference} below) that  
  the Morse differential $\partial: C_{k+1}(F)\ri C_{k}(F)$ has the form
 $$\partial\, =\, \left( \begin{array}{rl} -\partial_{f_0}&0\\
 \Psi^F&\partial_{f_1}\end{array}\right).$$

The relation $\partial^2=0$ implies that $\Psi^F: C_*(f_0,\xi_0,o_0)\ri C_*(f_1,\xi_1,o_1)$ is a 
chain map. It has the following two properties: 
\begin{enumerate}
\item For $(f_0,\xi_0,o_0)=(f_1,\xi_1,o_1)$,  if we
  take $F(t,x)= f_0(x)+g(t)$ and $\xi(t,x)=\xi_0(x)- g'(t)\frac{\partial}{\partial t}$, we get
\begin{equation}
  \label{eq:usual-Id}\Psi^{F} =-\mathrm{Id}.
\end{equation}
Indeed one easily notices that for this choice of the pair $(F,\xi)$ the  gradient lines joining critical points of consecutive indices and lying on the slices $\{0\}\times X$ or $\{1\}\times X$  only depend on $t$: there is one of them from each critical point  $x_0=(0,x)\in \{0\}\times \Crit(f_0)$ to its correspondent $x_1=(1,x)\in \{1\}\times \Crit(f_0)$; we will denote it by  $c_x$.  With our  conventions $c_x$ is equipped with the negative sign orientation. Indeed, the orientation rule (\ref{eq:orientation-rule}) implies  here that  
  $$
  \left(\ori\, c_x, -\frac{\partial}{\partial t }, \ori\, W^u_F(x_1)\right)\, =\, \ori\, W^u_F(x_0),
  $$ 
  and using~\eqref{eq:orientation-produit} and~\eqref{eq:orientation-produit2} we get 
  $$
  \left(\ori\, c_x, -\frac{\partial}{\partial t }, \ori\, W^u_{f_0}(x)\right)\, =\, \left(\frac{\partial}{\partial t}, \ori\, W^u_{f_0}(x)\right),
  $$ which yields a negative sign for $c_x$. 
  
  We denote  the function $F$  above by $\Id $
  to emphasize that it corresponds to the trivial homotopy from $f_0$
  to itself. The discrepancy between the notation $\Id$ and the map $-\Id$ above will be resolved in~\S\ref{sec:invariance-enriched}, where the continuation cocycle will be defined with a minus sign (see~\eqref{eq:enriched-morphism}). 
 
 \item (See the third step in the proof of~\cite[Theorem~3.4.2 ]{Audin-Damian_English}) The chain maps $\Psi^F$, $\Psi^G$ and
  $\Psi^H$ produced by {\it arbitrary} homotopies from $f_0$ to $f_1$,
  from $f_1$ to $f_2$ respectively from $f_0$ to $f_2$ have the
  property that $\Psi^G\circ\Psi^F$ and $\Psi^H$ are chain homotopic, and in
  particular they are equal in homology.
\end{enumerate}
From these two properties it is obvious that the particular case $f_0=f_2$ and $H=\Id $ gives the invariance. 

This proof adapts to the case of lifted homology, showing that the latter is also invariant with respect to the choice of the pair $(f, \xi,o)$. Actually, the lifted complex also depends on a lift $l$ of the critical points of the Morse function to the universal cover, therefore one also needs to prove the independence of the homology with respect to $l$. To this end one considers the lifted complex on the universal cover of $[-\epsilon, 1+\epsilon]\times X$, denoted $[-\epsilon, 1+\epsilon]\times \widetilde{X}$, associated to $(F,\xi,o_F, l_0\cup l_1)$,  where $o_F$ is the orientation of the unstable manifolds given by~\eqref{eq:orientation-produit} and~\eqref{eq:orientation-produit2}, and $l_0\cup l_1$ are the obvious lifts of the critical  points of $F$ to the universal cover given by the lifts $l_i$ for $f_i$.  This lifted complex gives rise to a morphism of complexes 
$$\widetilde{\Psi}^F: \widetilde{C}_*(f_0,\xi_0,o_0,l_0)\ri \widetilde{C}_*(f_1,\xi_1,o_1,l_1)$$
which also satisfies  properties 1 and 2 above by an analogous argument.   Based on these two properties, the invariance of  lifted homology is straightforward. 

 \rmk\label{rmk:op1} If we change all the orientations of the unstable manifolds into the opposite ones, the orientation rule \eqref{eq:orientation-rule} implies that  the (lifted)  Morse complex stays unchanged $C_*(f,\xi,o,l)=C_*(f,\xi,-o,l)$. However,    the invariance morphism 
$$\widetilde{\Psi}^{F^{\mathrm{op}}}: \widetilde{C}_*(f_0,\xi_0,o_0,l_0)\ri \widetilde{C}_*(f_1,\xi_1,-o_1,l_1)$$
has an opposite sign 

$$\widetilde{\Psi}^{F^{\mathrm{op}}}\, =\, -\widetilde{\Psi}^{F}.$$
Indeed, the orientation of the unstable manifold $W^u_F$ changes  in~\eqref{eq:orientation-produit2}, but remains the same  in \eqref{eq:orientation-produit}, so that the rule  \eqref{eq:orientation-rule} yields  opposite orientations for the trajectories which define $\widetilde{\Psi}^{F^\mathrm{op}}$. 

We get the same result when we change $o_0$ into $-o_0$ while leaving $o_1$ unchanged. 
\kmr 

\section{Invariance for DG  Morse homology}\label{sec:invariance-enriched} 


Take two sets of data $\Xi_0=(f_0, \xi_0,  o_0,   s_{x,y}^0, \caly_0, \theta_0)$, $\Xi_1=(f_1, \xi_1,  o_1, s_{x,y}^1, \caly_1, \theta_1)$, and 
define the associated 
twisted complexes with coefficients in the same DG-module $\calf$ over $C_*(\Omega X)$.  We assume that the root of both trees is the basepoint $\star$.  Take a homotopy from $f_0$ to $f_1$ and define a function $$F: [-\epsilon, 1+\epsilon]\times X\ri \R$$ and a pseudo-gradient $\xi$ as in \S\ref{sec:invariance-for-usual}. Then pick a homotopy  $(\caly_t)_{t\in [0, 1]}$ between $\caly_0$ and $\caly_1$ consisting of trees with common root $\star$ and put 
$$\caly =\bigcup_{t\in [0, 1]}(\{t\}\times \caly_t)\, \subset \, [0, 1]\times X.$$  
Next, choose a homotopy inverse $\Theta$   of the canonical projection  $$p: [0, 1]\times X\ri \bigcup_{t\in [0,1]} \{t\}\times X/\caly_t,$$
that coincides with $\theta_0$ on $\{0\}\times X$ and with $\theta_1$ on $\{1\}\times X$. 

 Finally, orient the Latour cells $\overline{W}^u_F(x)$ using~(\ref{eq:orientation-produit}--\ref{eq:orientation-produit2}) and choose a representing chain system  $(s_{x,y}^F)$ compatible with (the fundamental classes given by)  these orientations as indicated in the following lemma. 
 
 \begin{lemma}\label{orientation-difference} (i) With our orientation conventions we have 
 $$\ori\, \overline{\call}_F(x,y) \, =\, (-1)^{|x|-|y|} \ori\, \overline{\call}_{f_0}(x,y)$$
 for any  $x,y\in\{0\}\times\Crit(f_0)$, and 
 $$\ori\, \overline{\call}_F(x,y) \, =\,  \ori\, \overline{\call}_{f_1}(x,y)$$
 for any  $x,y\in\{1\}\times\Crit(f_1)$.\\
(ii) There exists a representing chain system of $F$ associated to these orientations such that 
$$s_{x,y}^{F,0} \, =\, (-1)^{|x|-|y|} s_{x,y}^{0}\qquad\mbox{and} \qquad s_{x,y}^{F,1} \, =\,  s_{x,y}^{1},$$
where $s_{x,y}^{F,i}$ is the chain  associated to the critical points $x,y \in \{i\}\times \Crit(f_i)$ for $i=0,1$. 
\end{lemma} 

\begin{proof}(i)  Consider first the case $x,y\in\{0\}\times\Crit(f_0)$. Recall that  $\overline{\call}_F(x,y)$ is  oriented as the intersection between the oriented manifold $\overline{W}_F^{u}(x)$ and the co-oriented manifold $S_F^{s}(y)$ (see~\eqref{eq:multiliaisons-intersection} in~\S\ref{subsec:Latour-cells}). 
Let us compare the orientations of  $\overline{\call}_F(x,y)$  and  $\overline{\call}_{f_0}(x,y)$ at some point $p\in \call(x,y)$ lying on a gradient line with gradient vector $\xi(p)$. Following our orientation convention~\eqref{eq:transverse-orientation} from~\S\ref{subsec: orientation-conventions} we have
$$\left(\ori \,  \overline{\call}_F(x,y), \coori \, S^{s}_{F}(y)\right)\, =\, \ori\,  \overline{W}^{u}_F(x),$$
which becomes using~\eqref{eq:coorientation-sphere-stable} 
\begin{equation*}
 \left(\ori \,  \overline{\call}_F(x,y), -\xi(p), \coori \, \overline{W}^{s}_{F}(y)\right)\, =\, \ori\,  \overline{W}^{u}_F(x)\end{equation*}
and then 
\begin{equation}
\left(\ori \,  \overline{\call}_F(x,y), -\xi(p), \ori \, \overline{W}^{u}_{F}(y)\right)\, =\, \ori\,  \overline{W}^{u}_F(x),\label{eq:adhoc}
\end{equation}

as stated in the orientation rule~\eqref{eq:orientation-rule} 

Now using~\eqref{eq:orientation-produit} we infer 
$$\left(\ori \,  \overline{\call}_F(x,y), -\xi(p), \tfrac\partial{\partial t}, \ori \, \overline{W}^{u}_{f_0}(y)\right)\, =\, \left(\tfrac\partial{\partial t}, \ori\,  \overline{W}^{u}_{f_0}(x)\right),$$
which implies 
$$(-1)^{|x|-|y|} \left(\tfrac\partial{\partial t}, \ori \,  \overline{\call}_F(x,y), -\xi(p), \ori \, \overline{W}^{u}_{f_0}(y)\right)\, =\, \left(\tfrac\partial{\partial t}, \ori\,  \overline{W}^{u}_{f_0}(x)\right)$$
and therefore 
$$(-1)^{|x|-|y|} \left( \ori \,  \overline{\call}_F(x,y), -\xi(p), \ori \, \overline{W}^{u}_{f_0}(y)\right)\, =\,  \ori\,  \overline{W}^{u}_{f_0}(x).$$
On the other hand the relation~\eqref{eq:orientation-rule} for $f_0$ instead of $F$ writes 
$$\left( \ori \,  \overline{\call}_{f_0}(x,y), -\xi(p), \ori \, \overline{W}^{u}_{f_0}(y)\right)\, =\,  \ori\,  \overline{W}^{u}_{f_0}(x).$$
Comparing with the above we get 
$$(-1)^{|x|-|y|} \ori \,  \overline{\call}_F(x,y)\, =\, \ori \,  \overline{\call}_{f_0}(x,y),$$ which proves the first relation between the orientations.

For the proof of the second relation consider two critical points $x,y\in \{1\}\times \Crit(f_1)$ and use an analogous argument. The difference is that in this case 
$$\overline{W}_F^{u}(x)\, =\, \overline{W}_{f_1}^{u}(x)$$ as oriented manifolds (and the same for $y$), so the vector $\tfrac\partial{\partial t}$ does not show up in the orientation computation which therefore leads to 
$$ \ori \,  \overline{\call}_F(x,y)\, =\, \ori \,  \overline{\call}_{f_1}(x,y).$$
This finishes the proof of (i).\\
(ii) Let us point out that  if  $X$ has a non-empty boundary, $(F,\xi)$ is defined on the manifold with boundary and corners   $[-\epsilon,1+\epsilon]\times X$. However  the gradient lines between critical points are far from its boundary $\p([-\epsilon, 1+\epsilon]\times X)$ and therefore we have the same structure for the trajectory spaces $\overline{\call}_F(x,y)$.  We may therefore construct  a representing chain system $(s_{x,y}^F)$ for the Morse moduli spaces of $(F,\xi)$ using the inductive method  of  Proposition~\ref{representingchain}. We do it  in a slightly different way:  we construct it first  for any pair of  critical points $x,y$ belonging to the same slice $\{i\}\times X$ for $i=0,1$. The first part of the lemma ensures that we may choose 
$$s_{x,y}^{F,0} \, =\, (-1)^{|x|-|y|} s_{x,y}^{0}\qquad\mbox{and}\qquad s_{x,y}^{F,1} \, =\,  s_{x,y}^{1}$$
for $x,y \in \{0\}\times \Crit(f_0)$, respectively $x,y \in \{1\}\times \Crit(f_1)$, in the construction from Proposition~\ref{representingchain} which produces the representing chain system. Then, for $x\in \{0\} \times X$ and $y\in \{1\}\times X$ critical points of $F$, the same inductive method enables us to complete the construction of the representing chain system $(s^F_{x,y})$. The proof of our lemma is now complete.
\end{proof}


\begin{proposition}\label{invariance}
The continuation data $\Xi= (F,\xi,  o_F, s_{x,y}^F,\caly,\Theta)$ defines a quasi-isomorphism of complexes 
$$\Psi^{\Xi}: C_*(X,\Xi_0;\calf)\ri C_*(X, \Xi_1;\calf).$$

\end{proposition}

\begin{remark} \label{rmk:on_invariance}
Proposition~\ref{invariance} will be superseded by the Invariance Theorem~\ref{independence} from the next section, which states in particular that $\Psi^\Xi$ is a chain homotopy equivalence whose chain homotopy type is determined by $\Xi_0$ and $\Xi_1$. We find it worthwhile to state and prove Proposition~\ref{invariance} separately for two reasons: it gives an immediate flavor of invariance, and its proof is considerably simpler than that of Theorem~\ref{independence}. This simplification is made possible by the use of invariance for lifted Morse homology. The proof of Theorem~\ref{independence} circumvents the use of lifted Morse homology, it follows a more standard pattern in Morse-Floer theory, but it is also more involved.   
\end{remark}

\begin{proof}[Proof of Proposition~\ref{invariance}] We use the algebraic recipe from Proposition~\ref{functoriality2} to define the morphism $\Psi^{\Xi}$. Recall that 
$$\Crit(F)\, =\, \left(\{0\}\times\Crit(f_0)\right) \cup \left(\{1\}\times\Crit(f_1)\right).$$
For a point $x\in\{i\}\times\Crit(f_i)$, $i=0,1$, we denote by $|x|$ the Morse index of $x$ as a critical point of $f_i$. We will use the following notation:
\begin{itemize}
\item For $x\in \{0\}\times\Crit(f_0)$ and
  $y\in \{1\}\times \Crit(f_1)$, the chain
  $s^F_{x,y}\in C_{|x|-|y|}( \overline \call_F(x,y))$ will be denoted $\sigma_{x,y}$,
\item For $x, y\in \{i\}\times \Crit(f_i)$ for the same $i=0,1$, the chain  $s^{F}_{x,y}\in C_{|x|-|y|-1}( \overline \call_F(x,y))$ will be denoted $s^{F,i}_{x,y}$, as in the previous lemma.
\end{itemize}
With this notation, 
(\ref{eq:representingchain})  gives in particular 
\begin{equation}\label{eq:invarianceforchains} 
\partial \sigma_{x,y}=\sum _{z\in \Crit(f_0)}(-1)^{|x|-|z|}s^{F,0}_{x,z}\times\sigma_{z,y}-\sum _{w\in \Crit(f_1)}(-1)^{|x|-|w|}\sigma_{x,w} \times s^{F,1}_{w,y}.
\end{equation} 
The negative sign between the sums in the right hand side comes from the fact that the indices  of $x$ and $z$ as critical points of $F$ are respectively $|x|+1$ and $|z|+1$, whereas the index of $w$ is $|w|$. 
Now we apply item (ii) of Lemma~\ref{orientation-difference}  to ~\eqref{eq:invarianceforchains} which becomes
\begin{equation}\label{eq:newinvarianceforchains} 
\partial \sigma_{x,y}=\sum _{z\in \Crit(f_0)}s^{0}_{x,z}\times\sigma_{z,y}-\sum _{w\in \Crit(f_1)}(-1)^{|x|-|w|}\sigma_{x,w} \times s^{1}_{w,y}.
\end{equation} 

This relation looks very much like relation~\eqref{eq:DGcont} in Proposition~\ref{prop:continuation-DG} (or, equivalently, relation~\eqref{MC-invariance} from Proposition~\ref{functoriality2}) --  the algebraic Maurer-Cartan equation which enables one to define a morphism between twisted complexes. In order to precisely get this relation we need to convert~\eqref{eq:newinvarianceforchains} into a relation on $\Omega X$. To this end, we follow the line of proof of Lemma~\ref{lecture} and construct a continuous map $q_{x,y} : \overline\call_F(x,y)\ri \Omega X$ as follows. 

We first define a map 
$$\Gamma^F:  \overline\call_F(x,y)\ri \calp_{x,y}([0, 1]\times X)$$ as in the proof of Lemma~\ref{lecture}, by parametrizing each gradient trajectory as a Moore path by the values of $F$. We then recall the projection 
$$p: [0, 1] \times X\ri \bigcup_{t\in [0, 1]}\{t\} \times X/\caly_t$$
which collapses the trees $\caly_t$ to the basepoint in their respective slices $\{t\}\times X$, and the choice of a homotopy inverse for $p$ denoted 
$$\Theta : \bigcup_{t\in [0, 1]}\{t\} \times X/\caly_t\ri [0, 1] \times X.$$
The map $\Theta$ can --- and will --- be chosen such that it sends every slice $\left(\{t\}\times X/\caly_t, \{t\}\times  \{\star\}\right)$ onto $\left(\{t\}\times X, \{t\}\times \{\star\}\right)$
(all the trees   $\caly_t$, $t\in [0,1]$ are assumed to have the same root $\star$ in $X$). Finally, we consider the projection 
$$\pi : [0, 1] \times X\ri X, $$
and define the map  $q_{x,y}: \overline\call_F(x,y)\ri \Omega X $ by 
\begin{equation}\label{eq:invariancelecture}
q_{x,y} \, =\, \pi\circ\Theta\circ p\circ \Gamma^F.\end{equation} 

Note that $\caly$ is contractible and therefore admits a lift to the universal cover $[0, 1] \times \widetilde{X}$. A fixed lift defines in particular a lift for the critical points of $F$ and therefore a lifted complex $\widetilde{C}_*(F,\xi)$. By construction we have (compare to Lemma~\ref{lecture}):

\begin{lemma}\label{lectureforinvariance} The maps $q_{x,y}$ satisfy  the following properties:
  \begin{enumerate}
  \item If $x\in  \mbox{ Crit}(f_0)$ and $y\in \Crit(f_1)$ have the same index then, for any $\lambda\in \call_F(x,y)$, the homotopy class $g=[q_{x,y}(\lambda)]$ in $\pi_1(X)$ is exactly the one assigned to $\lambda$ in the lifted Morse complex $\widetilde{C}_*(F,\xi)$. 
  \item For any $(\lambda,\lambda')\in \overline{\call}_F(x,z)\times \overline\call_F(z,y)$ we have 
    $$q_{x,y}(\lambda,\lambda')\, =\, q_{x,z}(\lambda)\# q_{z,y}(\lambda').$$
    
\item  If $x,y\in \{i\} \times \Crit(f_i)$ for the same  $i=0$ or $i=1$, then $q_{x,y}$ is exactly the map constructed in Lemma~\ref{lecture} for $(f_i,\xi_i)$. \qed
\end{enumerate}
\end{lemma}

  Now denote $\nu_{x,y}= -q_{x,y,*}(\sigma_{x,y})$ and $m_{x,y}^{(i)}=(q_{x,y})_*s^{i}_{x,y}$ for $i=0,1$. Remark that item 3 of the  above statement  implies that $m_{x,y}^{(i)}$ is the twisting cocycle associated to $(f_i, \xi)$. We infer
from~\eqref{eq:newinvarianceforchains} using item 2:
    \begin{equation}\label{eq:enriched-morphism} \partial \nu_{x,y}=\sum _{z\in \Crit(f_0)}m^{(0)}_{x,z}\cdot\nu_{z,y}-\sum _{w\in \Crit(f_1)}(-1)^{|x|-|w|}\nu_{x,w} \cdot m^{(1)}_{w,y}.\end{equation} 
  This is exactly the relation~\eqref{MC-invariance} corresponding to~\eqref{eq:DGcont} in Proposition~\ref{prop:continuation-DG}. We  deduce the existence of a canonical morphism $${\Psi}^{\Xi}: C_*(X,\Xi_0;\calf)\ri C_*(X,\Xi_1;\calf)$$
  defined by 
  \begin{equation}\label{eq:definition-continuation}
  \Psi^{\Xi}(\alpha\otimes x)\, =\, \alpha\sum_{y\in \Crit(f_1)}\nu_{x,y}\otimes y.
  \end{equation} 
    Item 1 of the previous lemma shows that it is compatible with $\widetilde{\Psi}=-\widetilde\Psi^F$  from~\S\ref{sec:invariance-for-usual} in the sense of Proposition~\ref{functoriality2}. Moreover, since we know that $-\widetilde\Psi^F$ is a quasi-isomorphism  by the invariance of the lifted homology, we may apply Proposition~\ref{functoriality2}  and infer that ${\Psi}^{\Xi}$ is a quasi-isomorphism.  This finishes the proof of Proposition~\ref{invariance}. 
\end{proof}

  \begin{remark}\label{change-of-sign}
   The reason why we chose the minus sign in front of $q_*$ to define $\nu_{x,y}$   above will be clear in~\S\ref{sec:invariance-continuation}: we want $\Psi^{\Xi}$ to induce the identity in homology for the constant homotopy; recall from~\S\ref{sec:invariance-for-usual} that at the level of usual and lifted homology our orientation conventions gave $\widetilde{\Psi}^{\Id} =-\Id.$ \end{remark} 
  

\smallskip

We may conclude therefore that the enriched Morse homology only depends on the manifold $X$ and the DG-module $\calf$, which justifies the use of the notation $H_*(X;\calf)$ (resp. $H_*(X,\p_+X; \calf)$ in the boundary case): given two different sets of auxiliary data $\Xi_0$ and $\Xi_1$, we constructed a morphism  of complexes $\Psi^{\Xi}$ which induces an isomorphism in homology. In the next section we show that the  ``continuation" morphism is a homotopy equivalence whose chain homotopy type only depends on $\Xi_0$ and $\Xi_1$. 

\section{Invariance of the continuation morphism} \label{sec:invariance-continuation}

The goal of this section is to prove the following 
\begin{theorem}[{\bf Invariance}] \label{independence} 
1)  Let $\Xi_0$, $\Xi_1$ be two sets of data for the construction of the enriched Morse complex with coefficients in a given DG-module $\calf$ over $C_*(\Omega X)$. The continuation morphism 
$$
\Psi^\Xi: C_*(X,\Xi_0;\calf)\ri C_*(X, \Xi_1;\calf)
$$ 
from Proposition~\ref{invariance} is a homotopy equivalence, and its chain homotopy type does not depend on the choice of continuation data $\Xi$ between $\Xi_0$ and $\Xi_1$. The map $\Psi^\Xi$ induces in particular an isomorphism in homology.

2) Given another set of data $\Xi_2$ and denoting by $\Psi_{ij}$, $i,j\in\{0,1,2\}$ the continuation maps determined by continuation data $\Xi_{ij}$ between $\Xi_i$ and $\Xi_j$, we have that $\Psi_{00}$ is homotopic to the identity and $\Psi_{12}\circ\Psi_{01}$ is homotopic to $\Psi_{02}$. In particular we have in homology
\[\Psi_{00} =\Id\quad\text{ and }\quad \Psi_{12}\circ\Psi_{01}\, =\, \Psi_{02}.\]
\end{theorem}

As noted in Remark~\ref{rmk:on_invariance}, this statement implies that of Proposition~\ref{invariance}. 

\begin{proof}[Proof of Theorem~\ref{independence}] We will break the proof into two steps which are the counterparts of properties 1 and 2 from~\S\ref{sec:invariance-for-usual}.

\noindent\underline{Step 1. The case $\Xi=\Id$.}

Consider a set of data  $\Xi_0$ and  the constant homotopy  $$(f_t, \caly_t, \theta_t=\Theta|_{\{t\}\times X/\caly_t})= (f_0, \caly_0, \theta_0).$$
  We consider on $[0,1] \times X$ the Morse function given by $$F(t,x)=f_0(x) +g(t)$$ and the pseudo-gradient 
  $$
  \xi(t,x)=\xi_0(x)-g'(t)\frac{\partial}{\partial t}.
  $$ 
  We orient the unstable manifolds of $F$ according to~\eqref{eq:orientation-produit} and~\eqref{eq:orientation-produit2} and denote $o_F$ this set of orientations. 
   
  \begin{proposition}\label{Id=id} There exists a representing  chain system  $(s_{x,y}^F)$  for the triple $(F,\xi, o_F)$  satisfying condition (ii) of Lemma~\ref{orientation-difference} and such that, for $$\Xi= (F, \xi,  o_F, s_{x,y}^F, \caly,\Theta),$$ the chain map $\Psi^\Xi$ constructed in Proposition~\ref{invariance} is chain homotopic to the identity. We will denote $\Xi= \Id$ (with an abuse of notation since $\Xi$ is not necessarily unique). 
\end{proposition}

\begin{proof}[Proof of Proposition~\ref{Id=id}]
The morphism $\Psi$ was defined previously using the chains $\nu_{x,y}\, =\, - q_{x,y,*}(\sigma_{x,y})$, where  $$\sigma_{x,y}\, =\, s_{x,y}^F$$ for $x\in \{0\} \times X$ and $y\in \{1\}\times X$ critical points of $F$. Since both $x$ and $y$ are also critical points of $f_0$   we denote by $x_i$ the critical points of  $F$ on $\{i\}\times X$ for $i\in \{0,1\}$ in order to avoid any confusion.  We denote by $\overline{\call}_\Id(x_0,y_1)$ the spaces of trajectories in this particular case, motivated by the notation $F=\Id$ from~\S\ref{sec:invariance-for-usual}. Let us describe next  the  choice of the representing chain system $(\sigma_{x_0,y_1}) = (s_{x_0,y_1}^F)$. (For critical points lying on the same slice $\{i\}\times X$ the choice of the representing chain system is imposed by Lemma~\ref{orientation-difference}.) 

 If $|x|=|y|$ then $$ \overline{\call}_\Id(x_0,y_1) = \emptyset$$ except for the case $x=y$, when $ \overline{\call}_\Id(x_0,x_1)$ consists of a single gradient line that only depends on $t$ and connects the maximum of $g$ to its minimum. In~\S\ref{sec:invariance-for-usual} we denoted this gradient line by $c_x$ and we proved the following:  

\begin{lemma}\label{constant-chain} With our orientation conventions the chain $$\sigma_{x_0,x_1}\in C_{0}(\overline{\call}_\Id(x_0,x_1))$$ is the opposite of the constant chain. Therefore $\nu_{x_0,x_1}=-q_{x_0,x_1,*}(\sigma_{x_0,x_1})$ is a constant chain in $C_{0}(\Omega X)$. \qed
\end{lemma}

 If $|x|>|y|$ there is a natural projection 
$$\pi: \overline{\call}_\Id(x_0,y_1) \ri  \overline{\call}_{f_0}(x,y),$$ and also an inclusion 
$$i :  \overline{\call}_{f_0}(x,y)\ri  \overline{\call}_\Id(x_0,y_1),$$ 
the latter being defined by $$i(\lambda)\, =\, (\lambda, c_y).$$
We therefore have
$$\pi\circ i \, =\, \Id.$$

\begin{lemma}\label{special-rcs}  There exists a representing chain system $(\sigma_{x_0,y_1})$  on \break $ \overline{\call}_\Id(x_0,y_1)$ such that, for $|x|>|y|$, we have
$$\pi_*(\sigma_{x_0,y_1})\, =\, 0. $$
\end{lemma}

The heuristic explanation for this statement is the following: in the cubical complex the degenerate cubes are by definition equal to zero. Whenever $|x|>|y|$, the cubical chain $\pi_*(\sigma_{x_0,y_1})$ has dimension $|x|-|y|$ and lives on a space of dimension $|x|-|y|-1$. Thus it ``wants'' to be degenerate, and the proof consists in showing that we can indeed choose $\sigma_{x_0,y_1}$ such that $\pi_*(\sigma_{x_0,y_1})$ is degenerate, and therefore zero.

\begin{proof}[Proof of Lemma~\ref{special-rcs}] We proceed by induction on $|x|-|y|$. For $|x|-|y|=1$ the trajectories in $\overline{\call}_{f_0}(x,y)$ are isolated points and therefore $$\pi_*(\sigma_{x_0,y_1}) \in C_1(\overline{\call}_{f_0}(x,y))$$ is formed by (constant) degenerate $1$-chains and therefore it equals $0$.

Suppose that we constructed $(\sigma_{x_0,y_1})$ for $|x|-|y|\leq k-1$ and consider the case $|x|-|y|=k$. As in the proof of Proposition~\ref{invariance} we find  a chain 
$$\sigma'_{x_0,y_1} \, \in \, C_{|x|-|y|}(  \overline{\call}_\Id(x_0,y_1))$$
 which satisfies relation~\eqref{eq:newinvarianceforchains}. In the current setting this writes
\begin{equation}\label{eq:newinvarianceforchains-bis} 
\partial \sigma'_{x_0,y_1}=\sum _{z\in \Crit(f_0)}s^0_{x_0,z_0}\times\sigma_{z_0,y_1}-\sum _{w\in \Crit(f_0)}(-1)^{|x|-|w|}\sigma_{x_0,w_1} \times s^0_{w_1,y_1}.
\end{equation} 
By the induction hypothesis we have 
$$\partial (\pi_*(\sigma'_{x_0,y_1})) \, =\, \pi_*(\partial \sigma'_{x_0,y_1})\, =\, 0.$$
We use here the fact that, in the above sum, the terms 
$s^{0}_{x_0,y_0}\times\sigma_{y_0,y_1}$ and $\sigma_{x_0,x_1} \times s^{0}_{x_1,y_1}$ occur with opposite signs and have the same image $-s_{x,y}^{0}$ through the projection $\pi_*$. 

We infer that $\pi_*(\sigma'_{x_0,y_1})$ is a cycle in  $C_{|x|-|y|}(  \overline{\call}_{f_0}(x,y))$. Since the dimension of $\overline{\call}_{f_0}(x,y)$ is $|x|-|y|-1$, it follows that there is some 
$$b_{x,y}\in C_{|x|-|y|+1}(  \overline{\call}_{f_0}(x,y))$$
such that 
$$\partial b_{x,y}\, =\, \pi_*(\sigma'_{x_0,y_1}).$$
We then define $$\sigma_{x_0,y_1}\, =\, \sigma'_{x_0,y_1}-\partial i_*(b_{x,y}).$$
Since we only added a boundary to $\sigma'_{x_0,y_1}$  we infer that $\sigma_{x_0,y_1}$ still satisfies~\eqref{eq:newinvarianceforchains-bis} and represents the fundamental class of $\overline{\call}_\Id(x_0,y_1)$. Moreover 
$$\pi_*(\sigma_{x_0,x_1})\, =\, \pi_*(\sigma'_{x_0,x_1})- \pi_*i_*(\partial b_{x,y})\, =\, 0$$
as claimed.
\end{proof}  

We now continue the proof of Proposition~\ref{Id=id} and show that, for this choice of representing chain system, we obtain a chain homotopy to the identity. 
  
 Following our usual pattern, we evaluate the representing chain system in $\Omega X$. Define continuous maps
  $\overline{q}_{x_0,y_1}: \overline{\call}_\Id(x_0,y_1)\ri \Omega X$ by
   $$\ol q_{x_0,y_1}(\lambda)\, =\,
   \begin{cases}\star  &\mbox{if}\quad x= y, \\
     q_{x,y} \circ \pi(\lambda) &\mbox{otherwise},
   \end{cases}
   $$
 where $q_{x,y}: \overline{\call}_{f_0}(x,y) \ri \Omega X$ is the evaluation defined in Lemma~\ref{lecture} and $\star$ is the constant Moore  loop parametrized by a single point $\{0\}$. Set 
 $$\overline{\nu}_{x_0,y_1}\, =\, - \overline{q}_{x_0,y_1,*}(\sigma_{x_0,y_1}).$$
 It is clear from the definition that the evaluations $\overline{q}_{x_0,y_1}$ satisfy the concatenation relation 
 \begin{equation}\label{eq:concatenation-relation} \overline{q}(\lambda, \lambda') \, =\, \ol q (\lambda)\# \ol q (\lambda'),
 \end{equation} 
 which together with~\eqref{eq:newinvarianceforchains-bis}  implies that the family $(\overline{\nu}_{x_0,y_1})$ satisfies~\eqref{eq:enriched-morphism} and therefore defines an automorphism of $C_*(X, \Xi_0;\calf)$. On the other hand Lemmas~\ref{constant-chain} and~\ref{special-rcs} imply that this morphism is nothing but the identity. 
 
 This however does not yet conclude the proof of Proposition~\ref{Id=id}, since the evaluation maps $\overline{q}_{x_0,y_1}$ do not coincide with the evaluation maps $q_{x_0,y_1}$ defined by \eqref{eq:invariancelecture}, that we used  to define the continuation cocycle and then the  continuation morphism by  \eqref{eq:definition-continuation}.  The next lemma shows that they are actually homotopic.
 
 \begin{lemma}\label{evaluations-homotopic} There is a homotopy $q^s_{x_0,y_1}: \overline{\call}_\Id(x_0,y_1)\ri \Omega X$ between  $\overline{q}$ and $q$ such that, for each $s\in [0,1]$, we have 
 \begin{equation}\label{eq:concatenation1} 
 q_{x_0,y_1}^s(\lambda,\lambda') \, =\, q_{x,z}(\lambda) \#q^s_{z_0,y_1}(\lambda')
 \end{equation} for any $(\lambda,\lambda')\in \overline{\call}_\Id(x_0,z_0)\times \overline{\call}_\Id(z_0,y_1)=\overline{\call}_{f_0}(x,z)\times \overline{\call}_\Id(z_0,y_1)$, and 
 \begin{equation}\label{eq:concatenation2} 
 q_{x_0,y_1}^s(\lambda,\lambda') \, =\, q^s_{x_0,w_1}(\lambda) \#q_{w,y}(\lambda')
 \end{equation} for any $(\lambda,\lambda')\in \overline{\call}_\Id(x_0,w_1)\times \overline{\call}_\Id(w_1,y_1)=\overline{\call}_\Id(x_0,w_1)\times \overline{\call}_{f_0}(w,y)$
 \end{lemma} 
 
 \begin{proof}[Proof of Lemma~\ref{evaluations-homotopic}]
 First note that the evaluation $\overline{q}$ satisfies the relations~\eqref{eq:concatenation1} and~\eqref{eq:concatenation2} as a consequence of the fact that the evaluations $q_{x,y}$ defined  in Lemma~\ref{lecture} on $\overline{\call}_{f_0}(x,y)$ do satisfy a concatenation relation. 
 
 Then recall that  the evaluations $q_{x_0,y_1}: \overline{\call}_\Id(x_0,y_1)\ri \Omega X  $ were defined in (\ref{eq:invariancelecture}) 
 by the formula 
 $$
 q_{x_0,y_1}\, =\, \pi\circ \Theta\circ p \circ \Gamma_{x_0,y_1}^{\Id},
 $$ where $\pi: [0,1]\times X\ri X$ is the projection, $p:[0,1]\times X\ri \bigcup_{t\in [0,1]} \{t\} \times X/\caly_t$ is the projection on the quotient space and $\Theta: \bigcup_{t\in[0,1]} \{t\} \times X/\caly_t \ri [0,1]\times X$ is a homotopy inverse for $p$ which is the same on every slice $\{t\} \times X$ (recall that the homotopy of trees $(\caly_t)$ is constant). Therefore we have
 $$\pi\circ\Theta\circ p \, =\, \theta_0\circ p_0\circ \pi$$
 where $p_0: X\ri X/\caly_0$ is the projection  and $\theta_0: X/\caly_0\ri X$ is its chosen homotopy inverse. 
 
 Recall also that 
 $$\Gamma^{\Id}_{x_0,y_1}:  \overline{\call}_\Id(x_0,y_1)\ri \calp_{x_0\ri y_1}([0,1]\times X)$$
 was defined by parametrizing the broken orbits using the Morse function $F(t,x) =f_0(x) +g(t)$. More precisely, assuming w.l.o.g. that $g(0)= A>0$ and $g(1)=0$, the map 
 $$\Gamma^{\Id}_{x_0,y_1}(\lambda) : [0, f_0(x)-f_0(y) +A] \ri [0,1]\times X$$ was defined by 
 $$\Gamma^{\Id}_{x_0,y_1}(\lambda)(r) \, =\, \lambda\cap (f_0+g)^{-1}(-r +f_0(x) +A).$$
 For $s\in (0,1]$ we define a homotopy $\Gamma^s : \overline{\call}_\Id(x_0,y_1)\ri \calp_{x_0\ri y_1}([0,1]\times X)$ by parametrizing the broken orbits using the values of $f_0+sg$, namely 
 $$\Gamma^s(\lambda) : [0, f_0(x)-f_0(y) +sA] \ri [0,1]\times X$$ is given  by 
 $$\Gamma^s(\lambda)(r) \, =\, \lambda\cap(f_0+sg)^{-1}(-r +f_0(x) +sA).$$
 This formula is not valid for $s=0$ since the values of $f_0$ do not strictly decrease along the constant orbits $c_z\in \call_\Id(z_0,z_1)$ which may appear in the broken orbits of $\overline{\call}_\Id(x_0,y_1)$. However, we may extend the composition $\pi\circ \Gamma^s : \overline{\call}_\Id(x_0,y_1)\ri \calp_{x\ri y} X$ to $s=0$ as we will now show. Denote $$\Gamma^s(\lambda)(r) \, =\, (\delta_s(r), \gamma_s(r))\, \in \, [0,1]\times X.$$
 We therefore have $$f_0(\gamma_s(r))+ sg(\delta_s(r)) \, =\, -r +f_0(x) +sA.$$
 Since $\gamma_s(r) \in \pi(\lambda)\in \overline{\call}_{f_0}(x,y)$ we can write 
 $$\pi\circ\Gamma^s(\lambda)(r)\, =\, \gamma_s(r)\, =\, \pi(\lambda) \cap f_{0}^{-1}\left( -r +f_0(x) + sA-sg(\delta_s(r))\right)$$
 for $r\in [0, f_0(x)-f_0(y)+sA]$. When $s$ tends to $0$ the right hand term converges towards 
$$  \pi(\lambda) \cap f_{0}^{-1}( -r +f_0(x)),$$
 which is exactly the parametrization of $\pi(\lambda)$ via the values of $f_0$, i.e.,  
 $$\Gamma_{x,y}(\pi(\lambda)): [0, f_0(x)-f_0(y)]\ri \calp_{x\ri y}X.$$
 We may therefore define a continuous map $$Q: [0,1]\times  \overline{\call}_\Id(x_0,y_1)\ri  \Omega X$$ by the formulas
 \[Q(s,\cdot)= q^s=
 \begin{cases}
    \theta_0\circ p\circ \pi \circ \Gamma^s \, =\, \pi\circ \Theta\circ p\circ \Gamma^s \quad & \text{for } s\in (0,1],\\
 \theta_0\circ p \circ \Gamma_{x,y}\circ \pi \, =\, q_{x,y}\circ\pi =\overline{q} \quad & \text{for } s=0.            
 \end{cases}\]
 By construction $q^0=\overline{q}$ and $q^1=q$. Moreover, the maps $q^s$ satisfy the relations~\eqref{eq:concatenation1} and~\eqref{eq:concatenation2}; we have already noticed it for $s=1$, whereas the case $s>0$ is analogous to the case $s=1$. The proof of the lemma is now complete. 
 \end{proof} 
 
 To finish the proof of Proposition~\ref{Id=id} it suffices now to prove 
 
 \begin{lemma}\label{homotopic-lectures}  The chains 
 $$
 \nu_{x_0,y_1}= -q_{x_0,y_1,*}(\sigma_{x_0,y_1}) \qquad \mbox{and} \qquad \overline{\nu}_{x_0,y_1}= -\overline{q}_{x_0,y_1,*}(\sigma_{x_0,y_1})
 $$ 
 satisfy the homotopy relation~\eqref{eq:alg-homotopy-eqn}. Therefore the morphism $\Psi^\Id$ defined by $\nu_{x_0,y_1}$ is chain homotopic to the one defined by $\overline{\nu}_{x_0,y_1}$ (which was proved to be the identity). 
\end{lemma}

\begin{proof}[Proof of Lemma~\ref{homotopic-lectures}] Let $(\sigma_{x_0,y_1})$ be the 
representing chain system from Lemma~\ref{special-rcs} with $\sigma_{x_0,y_1}\in C_{|x|-|y|}(\overline{\call}_\Id(x_0,y_1))$, and denote by $$\cals_{x_0,y_1} \in C_{|x|-|y|+1}([0,1]\times \overline{\call}_\Id(x_0,y_1))$$ the chain defined by 
$$\cals_{x_0,y_1}\, =\, \id_{[0,1]}\times\sigma_{x_0,y_1}.$$
We therefore have 
$$\p \cals_{x_0,y_1}\, =\, \{1\}\times \sigma_{x_0,y_1}-\{0\}\times \sigma_{x_0,y_1}-\id\times \p\sigma_{x_0,y_1},$$
which using~\eqref{eq:newinvarianceforchains} becomes
\begin{align*} 
\p & \cals_{x_0,y_1} \\
& = \{1\}\times \sigma_{x_0,y_1}-\{0\}\times \sigma_{x_0,y_1}\\
&\quad -\id\times \left(  \sum _{z\in \Crit(f_0)}s^{f_0}_{x_0,z_0}\times\sigma_{z_0,y_1}-\!\!\!\sum _{w\in \Crit(f_0)}(-1)^{|x|-|w|}\sigma_{x_0,w_1} \times s^{f_0}_{w_1,y_1}  \right)\\
&= \{1\}\times \sigma_{x_0,y_1}-\{0\}\times \sigma_{x_0,y_1}\\
                               &\quad - \sum _{z\in \Crit(f_0)}\id\times s^{f_0}_{x_0,z_0}\times\sigma_{z_0,y_1}+\!\!\!\sum _{w\in \Crit(f_0)}(-1)^{|x|-|w|}\cals_{x_0,w_1} \times s^{f_0}_{w_1,y_1}.\end{align*}
                             
We denote by $I$ the map 
\begin{align*}
\overline{\call}_{f_0}(x_0,z_0)\times [0,1]& \times \overline{\call}_{\Id} (z_0,y_1)\\
& \ri [0,1] \times   \overline{\call}_{f_0}(x_0,z_0)\times \overline{\call}_{\Id} (z_0,y_1)\subset \p\overline{\call}_{\Id}(x_0,y_1)
\end{align*} 
which switches the first two variables. We get 
\begin{align*}\p \cals_{x_0,y_1} = & \{1\}\times \sigma_{x_0,y_1}-\{0\}\times \sigma_{x_0,y_1}\\
& - \sum _{z\in \Crit(f_0)}\!\!\!\!(-1)^{|x|-|z|-1} I_* (s^{f_0}_{x_0,z_0}\times\id\times\sigma_{z_0,y_1})\\
& +\sum _{w\in \Crit(f_0)} \!\!\!\!(-1)^{|x|-|w|}\cals_{x_0,w_1} \times s^{f_0}_{w_1,y_1},
\end{align*} 
and therefore 
\begin{align*}
\p \cals_{x_0,y_1} = & \{1\}\times \sigma_{x_0,y_1}-\{0\}\times \sigma_{x_0,y_1}\\
&+ \sum _{z\in \Crit(f_0)}(-1)^{|x|-|z|} I_* (s^{f_0}_{x_0,z_0}\times\cals_{z_0,y_1})\\
& +\sum _{w\in \Crit(f_0)}(-1)^{|x|-|w|}\cals_{x_0,w_1} \times s^{f_0}_{w_1,y_1}.
\end{align*}
Now consider the homotopy $$Q: [0,1]\times  \overline{\call}_\Id(x_0,y_1)\ri  \Omega X$$  between  $\overline{q}$ and $q$  from Lemma~\ref{evaluations-homotopic} 
and apply $Q_*$ to the relation above. Setting $$h_{x_0,y_1}=Q_*(\cals_{x_0,y_1})\in C_{|x|-|y|+1}(\Omega X), $$ using~\eqref{eq:concatenation2} and taking into account that  $Q\circ I\, =\, (q,Q)$ by~\eqref{eq:concatenation1}, 
we infer 
\begin{align*}
\p h_{x_0,y_1}= \overline{\nu}_{x_0,y_1}-\nu_{x_0,y_1}
 & + \sum _{z\in \Crit(f_0)}(-1)^{|x|-|z|} m^{f_0}_{x_0,z_0} h _{z_0,y_1} \\ 
 & + \sum _{w\in \Crit(f_0)}(-1)^{|x|-|w|} h_{x_0,w_1}  m^{f_0}_{w_1,y_1},
\end{align*}
which is relation~\eqref{eq:alg-homotopy-eqn} as claimed.
\end{proof} 
The proof of Proposition~\ref{Id=id} is now complete. 
\end{proof}

Lemma~\ref{homotopic-lectures} is valid in the following more general framework with a similar proof. Let $(\ol \calm(x_0,y_1))$ be a family of moduli spaces indexed by the critical points of two Morse functions $f_0$ and $f_1$, and assume given a representing chain system $(\sigma_{x_0,y_1})$ which satisfies equation~\eqref{eq:newinvarianceforchains}. Let $q_{x_0,y_1}, \overline{q}_{x_0,y_1} : \ol\calm(x_0,y_1)\ri \Omega Y$ be two collections of evaluation maps which satisfy the concatenation relations~\eqref{eq:concatenation1} and~\eqref{eq:concatenation2} above, so that $\nu_{x_0,y_1}= -q_{x_0,y_1,*}(\sigma_{x_0,y_1})$ and $\overline{\nu}_{x_0,y_1} = -\overline{q}_{x_0,y_1,*}(\sigma_{x_0,y_1})$ are continuation cocycles that satisfy~\eqref{eq:DGcont}. 

\begin{lemma} \label{lem:homotopic-evaluations} Assume that the evaluation maps $q$ and $\overline{q}$ are homotopic in the following sense: there exists a collection of homotopies $(q^\tau_{x_0,y_1})_{\tau\in [0,1]}$ such that $q^0_{x_0,y_1}=\ol q_{x_0,y_1}$, $q^1_{x_0,y_1}=q_{x_0,y_1}$, and for each $\tau\in[0,1]$ the family of maps $(q^\tau_{x_0,y_1})$ satisfies the concatenation relations~\eqref{eq:concatenation1} and~\eqref{eq:concatenation2}. Then the cocycles $(\nu_{x_0,y_1})$ and $(\overline{\nu}_{x_0,y_1})$ satisfy the algebraic homotopy equation~\eqref{eq:alg-homotopy-eqn}, and therefore yield homotopic continuation maps. \qed
\end{lemma}

\noindent\underline{Proof of Theorem~\ref{independence}, Step 2: composition of continuation maps.} 

We now prove that continuation maps are compatible with composition. Take three sets of data $\Xi_0$, $\Xi_1$ and $\Xi_2$ which correspond to three Morse functions $f_0$, $f_1$ and $f_2$. For a given DG-module $\calf$ and some arbitrary choices of continuation data $\Xi$, $\Xi'$ and $\Xi''$  consider the continuation morphisms between the corresponding enriched complexes 
$$\Psi^\Xi: C_*(X, \Xi_0; \calf)\ri C_*(X, \Xi_1; \calf), $$
$$\Psi^{\Xi'}: C_*(X, \Xi_1; \calf)\ri C_*(X, \Xi_2; \calf), $$
and $$\Psi^{\Xi''}: C_*(X, \Xi_0; \calf)\ri C_*(X, \Xi_2; \calf). $$
\begin{proposition}\label{end-independence} 
The maps $\Psi^{\Xi'}\circ \Psi^\Xi$ and $\Psi^{\Xi''}$ are chain homotopic. In particular, the following relation holds in homology:
$$\Psi^{\Xi'}\circ \Psi^\Xi\, =\, \Psi^{\Xi''}.$$
\end{proposition} 

\begin{proof} First we construct a representing chain system adapted to this situation. We proceed as in  \cite{Audin-Damian_English} (third step of Theorem~3.4.2). Take $g: [-\epsilon, 1+\epsilon] \ri \R $ as before with a maximum at $0$   and a minimum at $1$. Then pick a Morse function $$K : [-\epsilon, 1+\epsilon]\times[-\epsilon, 1+\epsilon] \times X \ri \R$$ of the form  $$K(\tau,t, x)\, =\, k_{\tau,t}(x) +g(\tau)+g(t),$$ and a negative generic pseudo-gradient $\xi _K$,  where $k_{\tau,t}$ interpolates between $f_0$, $f_1$ and $f_2$ (as explained below),  such that the following conditions are satisfied:
\begin{itemize} 
\item for $t\in [-\epsilon, \epsilon] $ we have  
$$K(\tau,t,x)\, =\, F(\tau, x) + g(t),$$
where $F= f_\tau+g(\tau)$ is the Morse function on $[-\epsilon, 1+\epsilon]\times X$ which was used to define $\Psi^\Xi$. We define the pseudo-gradient by $\xi_K(\tau,t,x) = \xi_F(\tau,x)-g'(t)\frac{\partial}{\partial t}$ in this region. 
\item for $t\in [1-\epsilon, 1+\epsilon] $ we have
$$K(\tau,t,x)\, =\, f_2(x)+g(\tau)+g(t).$$
We  take $\xi_K(\tau,t,x)= \xi_{f_2}(x)-g'(\tau)\frac{\partial}{\partial \tau} -g'(t)\frac{\partial}{\partial t} $ in this region. 
\item for $\tau\in [-\epsilon, \epsilon]$ we have 
$$K(\tau,t, x)\, =\, H(t,x)+g(\tau),$$ 
where $H: [-\epsilon, 1+\epsilon]\times X\ri \R$ is the Morse function interpolating between $f_0$ and $f_2$ which was used to define $\Psi^{\Xi''}$. We define  $\xi_K(\tau,t,x) =\xi_H(t,x) - g'(\tau)\frac{\partial}{\partial \tau} $ in this region. 
\item for $\tau\in [1-\epsilon, 1+\epsilon]$ we have
$$K(\tau,t, x)\, =\, G(t,x)+g(\tau),$$ 
where $G: [-\epsilon, 1+\epsilon]\times X\ri \R$ is the Morse function interpolating between $f_1$ and $f_2$ which was used to define $\Psi^{\Xi'}$. Finally, we take  the pseudo-gradient $\xi_K(\tau,t,x)=\xi_G(t,x)- g'(\tau)\frac{\partial}{\partial \tau} $ in this region. 
\end{itemize} 

The function $g$ is chosen to have negative enough slope on $[\epsilon,1-\epsilon]$ so that the only critical points of $K$ are the critical points of $f_0$, $f_1$ or $f_2$ on $\{(i,j)\} \times X$ for $i, j \in \{0, 1\}$. More precisely
\begin{align*}
\Crit(K)=\{(0,0)\}\times & \Crit(f_0) \ \sqcup \ \{(1,0)\}\times \Crit(f_1) \\
& \sqcup \ \{(0,1)\}\times\Crit(f_2)\ \sqcup \ \{(1,1)\}\times\Crit(f_2).
\end{align*}
The index of a critical point of $K$ of the form $(i,j,x)$ is equal to $i+j$ plus the index of $x$ as a critical point of the corresponding Morse function $f_0$, $f_1$ or $f_2$.  

 We orient the unstable manifolds of the critical points of $K$ in a manner analogous to the one described in~\eqref{eq:orientation-produit} and~\eqref{eq:orientation-produit2}, namely
\begin{itemize}
\item for $x\in \{(0,0)\}\times \Crit(f_0)$, 
$$\ori \, W^u_K (x)\, =\, \left(\tfrac{\partial}{\partial \tau}, \tfrac{\partial}{\partial t}, \ori\,   W^u_{f_0}(x)\right).$$
\item for $x\in \{(1,0)\}\times \Crit(f_1)$, 
$$\ori \, W^u_K (x)\, =\, \left( \tfrac{\partial}{\partial t}, \ori\,   W^u_{f_1}(x)\right).$$
\item for $x\in \{(0,1)\}\times \Crit(f_0)$, 
$$\ori \, W^u_K (x)\, =\, \left(\tfrac{\partial}{\partial \tau},  \ori\,   W^u_{f_0}(x)\right).$$
\item for $x\in \{(1,1)\}\times \Crit(f_2)$, 
$$\ori \, W^u_K (x)\, =\, \ori\, W^u_{f_2}(x).$$
\end{itemize} 
 Denote this set of orientations by $o_K$. We will choose a representing chain system $(s_{x,y}^K)$ for $(K,\xi, o_K)$ associated with this orientation in  Lemma~\ref{change-of-sign-bis} below. The fact that the manifold $[-\epsilon, 1+\epsilon]\times[-\epsilon, 1+\epsilon]\times X$  has boundary and corners is again not an issue since all the gradient lines between critical points stay away from the boundary. We will use different notations for the chains of the system $(s_{x,y}^K)$ depending on the type of critical points which are involved. As previously, we will use the Morse indices of the functions $f_0$, $f_1$ and $f_2$ for the critical points of $K$. 
\begin{itemize}
\item for $x,y \in \{(0,0)\}\times X$ critical points of $f_0$ we will  use the notation 
$$s_{x,y}^{K,0}\in C_{|x|-|y|-1}(\overline{\call}_{K}(x,y))=C_{|x|-|y|-1}(\overline{\call}_{f_0}(x,y))$$
\item for $x,y \in \{(1,1)\}\times X$ critical points of $f_2$ we will  use the notation 
$$s_{x,y}^{K,2}\in C_{|x|-|y|-1}(\overline{\call}_{K}(x,y))=C_{|x|-|y|-1}(\overline{\call}_{f_2}(x,y))$$
\item for $x\in \{(0,0)\} \times X$ critical point of $f_0$ and $y\in \{(1,0)\} \times X$ critical point of $f_1$ we will  use the notation 
$$\sigma_{x,y}^{K,0,1}\in C_{|x|-|y|}(\overline{\call}_{K}(x,y))=C_{|x|-|y|}(\overline{\call}_{F}(x,y))$$
\item for $x\in \{(1,0)\} \times X$ critical point of $f_1$ and $y\in \{(1,1)\} \times X$ critical point of $f_2$ we will  use the notation 
$$\sigma_{x,y}^{K,1,2}\in C_{|x|-|y|}(\overline{\call}_{K}(x,y))=C_{|x|-|y|}(\overline{\call}_{G}(x,y))$$
\item for $x\in \{(0,0)\} \times X$ critical point of $f_0$ and $y\in \{(0,1)\} \times X$ critical point of $f_2$ we will use the notation
$$\sigma_{x,y}^{K,0,2}\in C_{|x|-|y|}(\overline{\call}_{K}(x,y))=C_{|x|-|y|}(\overline{\call}_{H}(x,y))$$
\item for $x\in \{(0,1)\} \times X$ critical point of $f_2$ and $y\in \{(1,1)\} \times X$ critical point of $f_2$ we will use the notation 
$$\sigma_{x,y}^{K,2,2}\in C_{|x|-|y|}(\overline{\call}_{K}(x,y))=C_{|x|-|y|}(\overline{\call}_{Id}(x,y))$$
\item finally,  for $x\in \{(0,0)\} \times X$ critical point of $f_0$ and $y\in \{(1,1)\} \times X$ critical point of $f_2$ we will denote 
$$S_{x,y}^{K}\in C_{|x|-|y|+1}(\overline{\call}_{K}(x,y))$$
the corresponding chain.

\end{itemize} 

With this notation and taking into account the differences between the Morse indices of the critical  points of $K$ considered as critical points of $f_i$, the defining relation of the representing chain system~\eqref{eq:representingchain} writes here for $S_{x,y}^K$: 
\begin{align*} 
& \p S_{x,y}^K  \\
& = \!\!\!\sum_{z\in \Crit(f_0)} (-1)^{|x|-|z|} s_{x,z}^{K,0}\times S_{z,y}^K+  \!\!\!\sum_{w\in \Crit(f_2)} (-1)^{|x|-|w|+2} S_{x,w}^{K}\times s_{w,y}^{K,2}\label{eq:rcs-homotopies}\\
&+\!\!\! \!\!\! \sum_{u\in \Crit(f_1)} (-1)^{|x|-|u|+1} \sigma_{x,u}^{K,0,1}\times \sigma_{u,y}^{K,1,2}+\!\!\!\!\!\!\sum_{v\in \Crit(f_2)} (-1)^{|x|-|v|+1} \sigma_{x,v}^{K,0,2}\times \sigma_{v,y}^{K,2,2}.
\end{align*}
We now make precise the choice of the representing chain system $(s_{x,y}^K)$.
We proceed as  in Lemma~\ref{orientation-difference}: we extend the representing chain systems of $f_0$, $f_2$, $F$ etc., with   some sign changes due to different orientations.  

\begin{lemma}\label{change-of-sign-bis} Denote by $\sigma^{F}_{x,y}$,  $\sigma^G_{x,y}$ and $\sigma^H_{x,y}$ the representing chain systems for the functions $F$, $G$ and respectively $H$. Denote $\sigma^\Id_{x,y}$ the representing chain system constructed in the proof of Proposition~\ref{Id=id} for  the constant homotopy from $f_2$ to $f_2$.
 
There exists a representing chain system $(s_{x,y}^K)$ for $(K,\xi, o_K)$ such that: 
$$ s_{x,z}^{K,0}= s_{x,z}^{f_0},\gol  s_{w,y}^{K,2} =  s_{w,y}^{f_2},\gol  \sigma_{x,u}^{K,0,1}= (-1)^{|x|-|u|}\sigma_{x,u}^{F},$$
$$ \sigma_{u,y}^{K,1,2} =\sigma_{u,y}^{G},\gol\sigma_{x,v}^{K,0,2}=(-1)^{|x|-|v|+1} \sigma_{x,v}^{H},\gol \sigma_{v,y}^{K,2,2}= \sigma_{v,y}^{\Id}.
$$
\end{lemma}
\begin{proof}  We proceed as in the proof of Lemma~\ref{orientation-difference}. By construction many of the trajectory spaces $\overline{\call}_K(x,y)$ coincide with those of $f_i$, $F$, $G$ or $H$.  But their orientations may differ and the sign difference corresponds exactly to the one between the corresponding chains in the statement of our lemma. Let us check this in each case. 

 Recall the orientation rule~\eqref{eq:orientation-rule} for the spaces of broken orbits of $K$: 
\begin{equation}
\left(\ori \,  \overline{\call}_K(x,y), -\xi(p), \ori \, W^{u}_{K}(y)\right)\, =\, \ori\,  \overline{W}^{u}_K(x),\label{eq:adhoc-bis}
\end{equation}
where $\xi(p)$ is the gradient vector at a point $p$ on a broken orbit in $\overline{\call}_K(x,y)$. We apply it for each situation to deduce the sign differences between the orientations of the representing chains: 

1. If $x,z\in \{(0,0)\}\times X$ are both critical points of $f_0$ note that analogously to~\eqref{eq:orientation-produit} we have 
$$\ori \, \overline{W}_K^u(x)\, =\, \left(\tfrac\partial{\partial \tau}, \tfrac\partial{\partial t}, \ori \, \overline{W}^{u} _{f_0}(x)\right),$$
the same being true for $W^u_K(z)$ and $W^{u}_ {f_0}(z)$. Inserting this in~\eqref{eq:adhoc-bis} we get 
$$\left(\ori \,  \overline{\call}_K(x,z), -\xi(p), \, \tfrac\partial{\partial \tau}, \tfrac\partial{\partial t}, \ori\, W^{u}_{f_0} (z)\right)\, =\, \left( \tfrac\partial{\partial \tau},\tfrac\partial{\partial t}, \ori\,  \overline{W}^{u}_{f_0}(x)\right),$$
which, using again the orientation rule for $\overline{\call}_{f_0}(x,z)$, yields 
$$\ori \,  \overline{\call}_K(x,z)\, =\, \ori \,  \overline{\call}_{f_0} (x,z).$$
One may therefore use the construction  procedure from Lemma~\ref{representingchain} to choose $$s^{K,0}_{x,z} \, =\, s^{f_0}_{x,z}.$$

\noindent 2. If $w,y\in \{(1,1)\}\times X$ are both critical points of $f_2$ then $$\ori\, \overline{W}_K^u(w)= \ori\, \overline{W}_{f_2}^u(w),$$ and the same is true  for $y$. By~\eqref{eq:adhoc-bis} we obtain 
$$\ori \,  \overline{\call}_K(w,y)\, =\, \ori \,  \overline{\call}_{f_2} (w,y)$$ and we may therefore choose 
$$s^{K,2}_{w,y} \, =\, s^{f_2}_{w,y}.$$

\noindent 3. If $x\in \{(0,0)\}\times X$ is a critical  point of $f_0$ and $u\in \{(1,0)\}\times X$ is a critical point  of $f_1$ we have 
\begin{align*} \ori \, \overline{W}_K^u(x)& =\left(\tfrac\partial{\partial \tau}, \tfrac\partial{\partial t}, \ori \, \overline{W}^{u} _{f_0}(x)\right)\, =\, -  \left(\tfrac\partial{\partial t}, \tfrac\partial{\partial \tau}, \ori \, \overline{W}^{u} _{f_0}(x)\right)\\
&=-\left( \tfrac\partial{\partial t}, \ori \, \overline{W}^{u} _{F}(x)\right)\end{align*}
and  $$\ori\, W^u_{K}(u)\, =\, \left(  \tfrac\partial{\partial t}, \ori\, W^u_{f_1}(u) \right) \, =\, \left( \tfrac\partial{\partial t}, \ori\, W^u_{F}(u) \right).$$
By~\eqref{eq:adhoc-bis} this yields 
$$\ori \,  \overline{\call}_K(x,u)\, =\, (-1)^{|x|-|u|+2}\, \ori \,  \overline{\call}_{F} (x,u),$$ and we may therefore choose 
$$\sigma^{K,0,1}_{x,u}\, =\, (-1)^{|x|-|u|}\sigma^F_{x,u}.$$

\noindent 4. If $u\in \{(1,0)\}\times X$ is a critical  point of $f_1$ and $y\in \{(1,1)\}\times X$ is a critical point  of $f_2$  we have
$$ \ori \, \overline{W}_K^u(u)\,  =\, \left(\tfrac\partial{\partial t}, \ori \, \overline{W}^{u} _{f_1}(x)\right)\, =\,  \ori \, \overline{W}_G^u(u)$$
and $$\ori \, W^u_K(y)\, =\, \ori \, W^u_{f_2}(y)\, =\, \ori\, W^u_G(y).$$
By~\eqref{eq:adhoc-bis} we get
$$\ori \,  \overline{\call}_K(u,y)\, =\, \ori \,  \overline{\call}_{G} (u,y),$$
so we may choose 
$$\sigma^{K,1,2}_{u,y}\, =\, \sigma^G_{u,y}.$$

\noindent 5. If $x\in \{(0,0)\}\times X$ is a critical  point of $f_0$ and $v\in \{(0,1)\}\times X$ is a critical point  of $f_2$  we have
\begin{align*} \ori \, \overline{W}_K^u(x)& =\left(\tfrac\partial{\partial \tau}, \tfrac\partial{\partial t}, \ori \, \overline{W}^{u} _{f_0}(x)\right)\\
&=\left( \tfrac\partial{\partial t}, \ori \, \overline{W}^{u} _{H}(x)\right)\end{align*}
and  $$\ori\, W^u_{K}(v)\, =\, \left(  \tfrac\partial{\partial \tau}, \ori\, W^u_{f_2}(v) \right) \, =\, \left( \tfrac\partial{\partial \tau}, \ori\, W^u_{H}(v) \right).$$
By~\eqref{eq:adhoc-bis} this implies again  
$$\ori \,  \overline{\call}_K(x,v)\, =\, (-1)^{|x|-|v|+1}\,  \ori \,  \overline{\call}_{f_0} (x,z),$$ hence we may choose 
$$\sigma^{K,0,2}_{x,v} \, =\, (-1)^{|x|-|v|+1}\sigma^H_{x,v}.$$

\noindent 6. Finally, if $v\in \{(0,1)\}\times X$  and $y\in \{(1,1)\}\times X$ are both critical points  of $f_2$ we get 
$$ \ori \, \overline{W}_K^u(v)\,  =\,   \left(  \tfrac\partial{\partial \tau}, \ori\, \overline{W}^u_{f_2}(v) \right)\, =\,  \ori \, \overline{W}_\Id^u(v)$$
and $$\ori\,  W^u_K(y)\, =\, \ori\, W^u_{f_2}(y) \, =\, \ori \, W^u_\Id(y).$$ Once again by~\eqref{eq:adhoc-bis} we obtain
$$\ori \,  \overline{\call}_K(v,y)\, =\, \ori \,  \overline{\call}_{\Id} (v,y),$$
which enables us to choose 
$$\sigma^{K,1,2}_{v,y}\, =\, \sigma^{\Id}_{v,y}$$
using the recipe of Lemma~\ref{representingchain}. 

To finish the proof of Lemma~\ref{change-of-sign-bis} it remains now to construct $S_{x,y}^K$. This is done inductively as in the proof of Lemma~\ref{representingchain}. 
\end{proof}

We now get back to the proof of Proposition~\ref{end-independence}. From Lemma~\ref{change-of-sign-bis} we get the relation
\begin{align}
  \begin{split}
\p S_{x,y}^K & = \sum_{z\in \Crit(f_0)} (-1)^{|x|-|z|} s_{x,z}^{f_0}\times S_{z,y}^K+  \sum_{w\in \Crit(f_2)} (-1)^{|x|-|w|} S_{x,w}^{K}\times s_{w,y}^{f_2}\label{eq:rcs-homotopies}\\
&\, \, \, \, \, \,  -  \sum_{u\in \Crit(f_1)} \sigma_{x,u}^{F}\times \sigma_{u,y}^{G}+\sum_{v\in \Crit(f_2)}  \sigma_{x,v}^{H}\times \sigma_{v,y}^{\Id}.
\end{split}
\end{align}
As previously we convert this relation into one in $C_*(\Omega X)$. To this purpose we choose a family of trees $(\caly_{\tau,t})_{\tau,t\in [0,1]^2}$  with fixed root $\star$,  which is constant for $t=1$ (where our homotopy of functions is constant between $f_2$ and itself) and which extends  the homotopies of trees already defined  by $\Xi$, $\Xi'$ and $\Xi''$  on $ (\tau,t)\in \p ([0,1]^2)$. Then we consider the projection 
$$p: [0,1]^2\times X \ri \bigcup_{(\tau, t) \in [0,1]^2} \{(\tau,t)\} \times X/\caly_{\tau, t} = [0,1]^2\times X/\caly,$$ where $\caly= \bigcup_{\tau,t}\caly_{\tau,t}$, and finally choose a homotopy inverse for $p$ denoted
$$\Theta :  \bigcup_{(\tau, t) \in [0,1]^2} \{(\tau,t)\} \times X/\caly_{\tau, t} \ri [0,1]^2\times X$$ and formed by a family  $\theta_{\tau,t} :\{(\tau,t)\}\times X/\caly_{\tau,t}\ri \{(\tau,t)\}\times X$ which  is constant for $t=1$ and extends what is already defined by $\Xi$, $\Xi'$ and $\Xi''$ on $\p([0,1]^2)$. 

Now, as in the  initial construction of the enriched complex, we parametrize the broken orbits of $\overline{\call}_K(x,y)$ by the values of $K$ and get continuous maps
$$\Gamma_{x,y} : \overline{\call}_K(x,y)\ri \calp_{x\ri y}([0,1]^2\times X)$$ and   evaluations
$$q_{x,y}: \overline{\call}_K(x,y)\ri \Omega  X$$ defined by 
$$q_{x,y}\, =\, \pi\circ\Theta\circ p  \circ \Gamma_{x,y}$$
where $\pi: [0,1]^2\times X\ri X$ is the projection. These evaluations satisfy as usual 
$$q_{x,y}(\lambda,\lambda') \, =\, q_{x,z}(\lambda)\#q_{z,y}(\lambda').$$ By applying $q_{x,y,*}$ to~\eqref{eq:rcs-homotopies} and denoting
$$h_{x,y}\, =\, q_*(S^K_{x,y})\, \in \, C_{|x|-|y|+1}(\Omega X)$$
we therefore get 
\begin{align*} \p h_{x,y} & = \sum_{z\in \Crit(f_0)} (-1)^{|x|-|z|} m_{x,z}^{f_0}\cdot h_{z,y}+  \sum_{w\in \Crit(f_2)} (-1)^{|x|-|w|} h_{x,w}\cdot m_{w,y}^{f_2}\\
&\, \, \, \, \, \,  -  \sum_{u\in \Crit(f_1)} (-\nu^\Xi_{x,u})\cdot  (-\nu_{u,y}^{\Xi'})+\sum_{v\in \Crit(f_2)}  (-\nu_{x,v}^{\Xi''})\cdot (-\nu_{v,y}^{\Id}),
\end{align*}
which simplifies into 
\begin{align*} \p h_{x,y} & = \sum_{z\in \Crit(f_0)} (-1)^{|x|-|z|} m_{x,z}^{f_0}\cdot h_{z,y}+  \sum_{w\in \Crit(f_2)} (-1)^{|x|-|w|} h_{x,w}\cdot m_{w,y}^{f_2}\\
&\, \, \, \, \, \,  -  \sum_{u\in \Crit(f_1)} \nu^\Xi_{x,u}\cdot  \nu_{u,y}^{\Xi'}+\sum_{v\in \Crit(f_2)}  \nu_{x,v}^{\Xi''}\cdot \nu_{v,y}^{\Id}.
\end{align*}

Here  $\nu_{x,y}^\Xi$, $\nu_{x,y}^{\Xi'}$ and $\nu_{x,y}^{\Xi''}$ are the chains in $C_{|x|-|y|}(\Omega X)$ which we used to define the continuation morphisms $\Psi^{\Xi}$, $\Psi^{\Xi'}$ respectively $\Psi^{\Xi''}$ by~\eqref{eq:definition-continuation}; remember that they are obtained by evaluating the corresponding representing chain systems $(\sigma_{x,y})$ by the map $-q_{x,y}$. Remember also that the morphism  $\Psi^\Id$ defined by $\nu^\Id$ is homotopic to the identity at chain level and therefore induces the identity in homology by Proposition~\ref{Id=id} in Step 1; this was the reason of choosing the  ``$-$'' sign in front of the evaluation. 

Remark also that the above relation is quite similar to~\eqref{eq:alg-homotopy-eqn} -- the relation which defines the algebraic homotopy between enriched complexes. Taking all this into account and using the definition of the continuation morphism~\eqref{eq:definition-continuation}, we infer that $\Psi^{\Xi'}\circ \Psi^{\Xi}$ and $\Psi^{\Id} \circ\Psi^{\Xi''}$ are homotopic. Taking into account that $\Psi^{\Id}$ is homotopic to the identity, we infer that $\Psi^{\Xi'}\circ \Psi^{\Xi}$ and $\Psi^{\Xi''}$ are homotopic, and in particular they induce the same maps in homology. This finishes the proof of Proposition~\ref{end-independence}.
\end{proof}

\underline{End of the proof of Theorem~\ref{independence}}. 
We prove item 1. The independence of the chain homotopy type of the continuation map $\Psi^\Xi$ with respect to the choice of $\Xi$ follows by applying Proposition~\ref{end-independence} with $\Xi_0$ arbitrary, $\Xi_1=\Xi_2$, $\Xi$ and $\Xi''$ arbitrary, and $\Xi'$ given by Proposition~\ref{Id=id}. That $\Psi^\Xi$ is a chain homotopy equivalence follows by applying Proposition~\ref{end-independence} with $\Xi_2=\Xi_0$ and $\Xi''$ the continuation data given by Proposition~\ref{Id=id}. 

We prove item 2. The first half of item 2) follows by applying Proposition~\ref{end-independence} with $\Xi_0=\Xi_1=\Xi_2$, $\Xi$ arbitrary and $\Xi'$, $\Xi''$ given by Proposition~\ref{Id=id}. The second half of item 2) follows directly from Proposition~\ref{end-independence}. 
\end{proof} 

This theorem enables us to define Morse homology with DG-coefficients  more accurately than just up to a isomorphism: there is a precise identification $\Psi_{01}$ between the homologies defined with two sets of data $\Xi_0$ and $\Xi_1$. This will be particularly useful in~\S\ref{sec:functoriality} in order to define morphisms between the enriched Morse homologies of two different manifolds $X$ and $Y$ which are associated to continuous maps $\varphi:X\ri Y$. 

 \rmk\label{rmk:op2} The statement of Remark \ref{rmk:op1} is still valid in the DG-setting. The same arguments show that if we replace the set of data $\Xi$ by $\Xi^{\mathrm{op}}$, which coincides with $\Xi$ except for the fact that the orientations of the unstable manifolds are chosen to be opposite, then $C_*(X,\Xi^{\mathrm{op}}; \cF)=C_*(X,\Xi;\cF)$ 

but the identification morphism $\Psi_{0,1}^{\mathrm{op}}: C_*(X, \Xi_{0};\calf)\ri C_*(X, \Xi_1^{\mathrm{op}};\calf)$ differs from $\Psi_{0,1}: C_*(X, \Xi_{0};\calf)\ri C_*(X, \Xi_1;\calf)$ by a sign, consequence  of the change of  orientation in all the trajectory spaces which define this morphism: we therefore have $\Psi_{01}^{\mathrm{op}}=-\Psi_{01}$. 

We get the same effect when we change $\Xi_0$ into $\Xi_0^{\mathrm{op}}$ with  $\Xi_1$ left unchanged. 
\kmr

\section[Independence from the choice of basepoint]{Identification of twisted complexes defined at different basepoints}\label{sec:identification-basepoints} 

Let $\star_0$ and $\star_1$ be two basepoints in $X$ and denote $\Omega_0 X$ and $\Omega_1 X$ the loop spaces based at $\star_0$, respectively at $\star_1$.  Consider two sets of data $\Xi_0$ and $\Xi_1$ defined respectively at these two basepoints. The problem one faces when trying to compare  twisted complexes defined with these two sets of data is that --- unlike in the case of usual local systems --- there is no canonical ring morphism between $C_*(\Omega_0X)$ and $C_*(\Omega_1 X)$ associated to a path $\gamma$ between $\star_0$ and $\star_1$, and therefore we have a priori no means to identify the DG-modules $\calf_0$ and $\calf_1$ over these rings.    To  overcome this difficulty we will consider a situation where $\calf_0$ and $\calf_1$ are related. Namely suppose that there is a continuous map $\eta :X\ri Y$ with values in some topological space $Y$  with  basepoint $\star_Y$ such that 
$$\eta(\star_0)\, =\, \eta(\star_1)\, =\, \star_Y.$$
This defines $\eta_0:\Omega_0 X\ri \Omega Y$ and $\eta_1:\Omega_1 X\ri \Omega Y$. We will work under  the assumption that
\begin{equation}\label{eq:identification-basepoints}
\calf_i\, =\, \eta_i^*\calf \, \, \, \,  \mathrm{for} \, \, \, i=0,1,
\end{equation} 
where $\calf$ is a DG-module over $C_*(\Omega Y)$. 

\begin{proposition}\label{identification} (compare with Theorem~\ref{independence}) Let $\Xi_0$ and $\Xi_1$ be as above and let $\calf_i$ be two DG-modules over $\Omega_iX$ for $i=0,1$ which satisfy~\eqref{eq:identification-basepoints}.  We consider continuation data $\Xi$ from $\Xi_0$ to $\Xi_1$ as before (note that the homotopy of trees moves their respective roots from $\star_0$ to $\star_1$). Each such choice determines a 
continuation map which is a chain homotopy equivalence 
$$\Psi_{01} : C_*(X, \Xi_0; \calf_0) \ri C_*(X,\Xi_1;\calf_1)$$
whose chain homotopy type only depends on $\Xi_i$, $\calf$ and $\eta$.  Moreover:\\
a) If $\star_0=\star_1$  then $\Psi_{01}$ coincides with the continuation map defined in Proposition~\ref{invariance} (for $\calf_0=\calf_1=\eta^*\calf$).\\
b) Given three basepoints $\star_0$, $\star_1$ and $\star_2$ such that $\eta(\star_i)=\star_Y$, three sets of data $\Xi_i$ respectively corresponding to these  points, and three sets of continuation data from $\Xi_0$ to $\Xi_1$, from $\Xi_1$ to $\Xi_2$, and from $\Xi_0$ to $\Xi_2$, with associated continuation maps $\Psi_{01}$, $\Psi_{12}$, respectively $\Psi_{02}$, the composition  
$\Psi_{12}\circ\Psi_{01}$ is chain homotopic to $\Psi_{02}$. In particular we have in homology 
$$\Psi_{12}\circ\Psi_{01}\, =\, \Psi_{02}.$$
\end{proposition} 

\begin{proof} Proceeding exactly as in the proof of Proposition~\ref{invariance} we get a cocycle $(\nu_{x,y})$ which satisfies equation~\eqref{eq:enriched-morphism}:
 $$\partial \nu_{x,y}=\sum _{z\in \Crit(f_0)}m^{(0)}_{x,z}\cdot\nu_{z,y}-\sum _{w\in \Crit(f_1)}(-1)^{|x|-|w|}\nu_{x,w} \cdot m^{(1)}_{w,y}.$$
 The difference here  is that, since $\star_0$ and $\star_1$ may not coincide,  this cocycle is a chain on the space of paths from $\star_0$ to $\star_1$, namely 
 $$\nu_{x,y}\in C_{|x|-|y|}(\calp_{\star_0\ri\star_1}X).$$
 We  use $\eta$ to transform it into a chain over a space of loops and write 
 $$\mu_{x,y} \, =\, \eta_*(\nu_{x,y})\, \in\, C_{|x|-|y|}(\Omega Y).$$
 The equation above then becomes 
 \begin{equation}\label{eq:identificationcocycle} 
\partial \mu_{x,y}=\sum _{z\in \Crit(f_0)}\eta_*(m^{(0)}_{x,z})\cdot\mu_{z,y}-\sum _{w\in \Crit(f_1)}(-1)^{|x|-|w|}\mu_{x,w} \cdot \eta_*(m^{(1)}_{w,y}).
\end{equation} 
To get the desired morphism $\Psi_{01}$  we use the canonical  identification from Remark~\ref{rem:pullback} for $i=0,1$, namely $C_*(X,\Xi_i;\calf_i)\, \equiv\, C_*(X, \eta_*(m^{(i)}_{x,y}); \calf)$, and also the morphism of complexes 
$C_*(X, \eta_*(m^{(0)}_{x,y}); \calf)\ri C_*(X, \eta_*(m^{(1)}_{x,y}); \calf)$ defined by  the cocycle $(\mu_{x,y})$ via Proposition~\ref{prop:continuation-DG} (equation~\eqref{eq:identificationcocycle} is of type~\eqref{eq:DGcont}).

It is obvious from the construction that for $\star_0=\star_1$ we get the usual continuation morphism. In particular $\Psi_{00}$ is homotopic to the identity. The fact that $\Psi_{01}$ only depends on $\Xi_i$, $\eta$ and $\calf$ is proved exactly in the same way as in Theorem~\ref{independence}, using an  analogue of Proposition~\ref{end-independence} in the proof of which we need to convert chains on paths in $X$ into chains on the loop space of $Y$ using $\eta$.  Item b) and the fact that $\Psi_{01}$ is a chain homotopy equivalence are proved as in Theorem~\ref{independence}.  
\end{proof}

\begin{remark}  \label{rmk:F01}
The following particular case of the previous general result is instructive: we consider an embedded path $\gamma : [0,1] \ri X$ between $\star_0$ and $\star_1$ and let $\eta$ be the projection $\eta : X\ri X/\gamma$.  
Denote $\star$ the basepoint of $X/\gamma$ and consider a DG-module $\calf$ over the loop space $\Omega(X/\gamma)$ based at this point. Then, for $\calf_i= \eta_i^*\calf$ as in~\eqref{eq:identification-basepoints}, we get a chain homotopy equivalence 
$\Psi_{01} : C_*(X, \Xi_0; \calf_0) \ri C_*(X,\Xi_1;\calf_1)$
whose chain homotopy type only depends on $\Xi_i$, $\gamma$ and $\calf$. 
\end{remark}

\begin{remark} Consider a non-DG local system $M$ on a connected manifold $X$, viewed as a representation of the fundamental groupoid of $X$. For a choice of basepoint $x$, the local system can be equivalently described as the $\Z[\pi_1(X,x)]$-module given by the fiber at $x$, denoted $M_x$. While this module structure on any fiber $M_x$ is canonical, the identification between the modules corresponding to two different basepoints $\star_0$ and $\star_1$ is not: it depends on the choice of a (homotopy class of) path $\gamma$ from $\star_0$ to $\star_1$. Subsequently, the same is true for the corresponding homologies with local (non-DG) coefficients.

This type of identification can be retrieved as a particular case of the construction from Proposition~\ref{identification}. Assume that $\gamma$ is embedded. Following up on the notation from Remark~\ref{rmk:F01}, choose a homotopy inverse $\chi : (X/\gamma, \star)\ri (X, \star_0)$ for the projection $\eta : X\ri X/\gamma$ and take for $\calf$ the pullback $\chi^*M_0$ of  $M_0=M_{\star_0}$ -- the fiber at $\star_0$ of our local system. Seeing $\calf$ as a DG-module over $C_*(\Omega_0X)$ supported in degree $0$ we claim that  the identification $\Psi_{01}$ given by the previous remark is the same as the one given by the path $\gamma$ between the homologies with local coefficients at $\star_0$ and $\star_1$. Since $\chi$ is a homotopy inverse for $\eta$, the $\Z[\pi_{1}(X,\star_0)]$-action on $\calf_0=\eta_0^*\calf$ is the same as the one on $M_0$, and it is easy to check that the $\Z[\pi_1(X,\star_1)]$-action on $\calf_1=\eta_1^*\calf$ is the same as the one on $M_1=\Phi_\gamma(M_0)$, where $\Phi_\gamma:M_0\stackrel\simeq\longrightarrow M_1$ is the 
identification given by the monodromy of the local system along the path $\gamma$ and $M_1$ is the fiber of that same local system at $\star_1$. 
Then, taking the same Morse function, pseudo-gradient and tree in the data $\Xi_0$ and $\Xi_1$, it is quite straightforward that the two identifications coincide. \end{remark}

\section[Identification of complexes from isotopies]{Canonical identification between twisted complexes defined by an isotopy}\label{sec:identification-isotopy}

We discuss in this section a particular case of continuation morphism in the context of isotopies.

We first need a piece of preliminary notation. Let $\Xi= (f, \xi, o,  s_{x,y},\caly,\theta)$ be a set of data constructed at a basepoint $\star\in X$. Let $\phi$ be a diffeomorphism of $X$ which is isotopic to the identity; we do not require $\phi(\star)=\star$. The critical points of $f\circ\phi^{-1}$ are the images through $\phi$ of the critical points of $f$, and if we choose as a pseudo-gradient for $f\circ \phi^{-1}$ the vector field $\phi_*\xi$ we obviously obtain $\overline{\call}_{f\circ \phi^{-1}}(\phi(x),\phi(y)) \, =\, \phi (\overline{\call}_f(x,y))$. We denote  
$$\Xi^\phi=\left(f\circ\phi^{-1},  \phi_*(\xi),  \phi(o), \phi_*(s_{x,y}), \phi(\caly), \theta^\phi=\phi\circ\theta\circ\phi^{-1}\right).$$
Since $\theta^\phi$ is a homotopy inverse of the projection $\phi\circ p\circ\phi^{-1}\ : X\ri X/\phi(\caly)$, we conclude that $\Xi^\phi$ is a set of Morse data associated to the basepoint $\phi(\star)$.

Let now $\Xi_0$ and $\Xi_1$ be two sets of data constructed at the basepoint $\star\in X$. Let also $(\phi_t)$ be an isotopy on $X$ starting at the identity, and denote $\phi=\phi_1$. Our goal is to describe a continuation cocycle from $\Xi_0$ to $\Xi_1^\phi$. 

We denote $\Phi: [0,1]\times X \ri [0,1]\times X$ the diffeomorphism defined by $\Phi(t,x)\, =\, (t, \phi_t(x))$.
Given $\Xi=(F, \xi,  o_F,  s_{x,y}^F, \caly, \Theta)$ Morse continuation data  from $\Xi_0$ to $\Xi_1$ on $[0,1]\times X$, let  
$\Xi^\Phi\, =\, \left(F\circ\Phi^{-1}, \Phi_*(\xi),  o_{F\circ\Phi}, \Phi_*(s_{x,y}^F), \Phi(\caly), \Phi\circ\Theta\circ\Phi^{-1}\right)$ 
be the induced continuation data from $\Xi_0$ to $\Xi_1^\phi$. 
The following proposition is proved by inspection of the definitions. 

\begin{proposition} 
Given $x\in \Crit(f_0)$ and $y\in \Crit(f_1)$, the continuation data $\Xi$ and $\Xi^\phi$ determine as in Proposition~\ref{invariance} (see~\eqref{eq:invariancelecture}) chains 
$$
\overline{\nu}_{x,y}\, =\, -(\Theta\circ p\circ\Gamma^F)_*(s_{x,y}^F)\, \in \, C_{|x|-|y|}\left(\calp_{(0,\star)\ri(1,\star)} [0,1]\times X\right)$$ 
and 
$$
\overline{\nu}_{x,\phi(y)}\,   \in \, C_{|x|-|y|}\left(\calp_{(0,\star)\ri(1,\phi(\star))} [0,1]\times X\right),
$$
which give rise in turn to continuation cocycles $\nu_{x,y}\, =\, \pi_*(\overline{\nu}_{x,y})$ and $\nu_{x,\phi(y)}\, =\, \pi_*(\overline{\nu}_{x,\phi(y)})$. These chains are related by the formula  
\begin{equation}\label{eq:identification-isotopy}
\overline{\nu}_{x,\phi(y)}\, =\, \Phi_*(\overline{\nu}_{x,y}).
\end{equation} 
\qed
\end{proposition}

%

\begin{remark} Note that we could have chosen $\Xi_0$ and $\Xi_1$ to be associated to different basepoints $\star_i$. This would have given rise to a cocycle
$$\overline{\nu}_{x,y}\,  \in \, C_{|x|-|y|}\left(\calp_{(0,\star_0)\ri(1,\star_1)} [0,1]\times X\right)$$
satisfying an analogous conclusion. 
\end{remark}

\chapter{Fibrations}\label{sec:fibrations} 

 In this chapter we analyze the important example of Hurewicz fibrations and prove Theorem~A from the Introduction. 

\section{Lifting functions}\label{sec:lifting-functions}

Given a topological space $X$ we denote as above $\cP X$ the space of Moore paths 
with free endpoints in $X$, and $\ev_0,\ev_1:\cP X\to X$ the evaluation maps at the initial point, respectively at the endpoint of the path. Recall that by definition, the elements of $\cP X$ are continuous paths $\gamma:[0,a]\to X$ defined on intervals of arbitrary length $a\ge 0$ (see for example~\cite[\S5.1]{CarlssonMilgram}). 
Given a basepoint $\star\in X$  we denote $\cP_{\star \to X} X = \ev_0^{-1}(\star)$ the space of Moore paths starting at $\star$, and $\Omega X=\Omega_\star X$ the space of Moore loops based at $\star$. The latter is a topological monoid with respect to concatenation.

%
A \emph{fibration}, or \emph{Hurewicz fibration}, is a continuous map that is surjective and has the homotopy lifting property with
respect to all spaces. If the homotopy
lifting property only holds for finite-dimensional cells (or, equivalently, for all 
CW-complexes), we speak of a \emph{Serre fibration}. The notion of a Serre fibration is
strictly weaker than that of a Hurewicz fibration~\cite{RBrown1966}. However, any Serre fibration is homotopy equivalent and fiber weakly homotopy equivalent to a Hurewicz fibration~\cite{StromGoodwillieMO}. 
This can be seen by applying the general procedure of turning a map $\pi:E\to X$ into a fibration $\overline \pi = \pr_2\circ \ev_1:E\, _\pi\!\times_{\ev_0} \cP X\to X$. Seen through this equivalence, our discussion for fibrations can be adapted to Serre fibrations. 

Following Hurewicz~\cite{Hurewicz1955}, the homotopy lifting property with respect to all spaces for a map $\pi:E\to X$ is equivalent to the existence of a \emph{lifting function}, i.e., a map 
$$
\Phi : E \, {_\pi}\!\!\times_{\ev_{0}} \cP X\to \cP E
$$
such that $\ev_0\circ \Phi = \mathrm{pr}_1$ and $\pi\circ \Phi = \mathrm{pr}_2$. Moreover, any two lifting functions are homotopic through lifting functions (Fadell~\cite[Proposition~1]{Fadell1959}). 

Brown~\cite{Brown1959} uses a stronger form of lifting function which he calls ``(weakly) transitive" (see also~\cite{Fadell-review-Brown}). Let $\pi:E\to X$ be a fibration and fix a basepoint $\star\in X$. A \emph{ transitive}  (resp. \emph{weakly transitive}) \emph{lifting function} for $\pi$ is a map
$$
\Phi: E \, {_\pi}\!\!\times_{\ev_{0}} \cP X\to E
$$
such that $\pi\circ \Phi = \ev_1\circ \mathrm{pr}_2$, such that $\Phi(e,b)=e$ for any constant path $b\in X$ and any $e\in E$, and such that for any $e\in E$ (resp. for any $e\in \pi^{-1}(\star)$) and any two Moore paths $\gamma,\delta\in \cP X$ such that the starting point of $\gamma$ is $\pi(e)$ (resp. is $\pi(e)=\star$) and the endpoint of $\gamma$ equals the starting point of $\delta$, denoting $\gamma\#\delta$ their concatenation we have 
\begin{equation} \label{eq:liftingfunction}
\Phi(\Phi(e,\gamma),\delta) = \Phi(e,\gamma\#\delta). 
\end{equation}
It is proved in~\cite[Proposition~5.5]{DyerKahn1969} that any fibration is fiber homotopy equivalent to a fibration that admits a  transitive lifting function. For the purpose of studying chains on the total space of a fibration we can therefore assume w.l.o.g. that there exists a transitive lifting function. 

A transitive lifting function gives rise to a map 
$$
\Phi:F\times \Omega X\to F,\qquad F=\pi^{-1}(\star),
$$
such that 
$$
\Phi(\Phi(e,\gamma),\delta)=\Phi(e,\gamma\# \delta)
$$ 
for all $e\in F$ and $\gamma,\delta\in \Omega X$. In other words, the lifting function exhibits $F$ as a topological right module over the topological monoid $\Omega X$. The composition
$$
C_p(F) \otimes C_q(\Omega X)\stackrel{EZ}\longrightarrow C_{p+q}(F\times \Omega X)\stackrel{\Phi_*}\longrightarrow C_{p+q}(F),
$$
with $EZ$ the Eilenberg-Zilber map~\cite[\S VIII.8]{MacLane-Homology} defined on the canonical basis  by $(\sigma\otimes \tau)\mapsto (\sigma, \tau)$ defines a right $C_*(\Omega X)$-module structure on $C_*(F)$. Here we work with non-degenerate cubical chains and arbitrary coefficients. We will sometimes abuse notation and denote this composition $\Phi_*$. 
More generally, the lifting function $\Phi:F\times \cP_{\star\to X} X\to E$ determines the composition $C_p(F)\otimes C_q(\cP_{\star\to X} X) \stackrel{EZ}\longrightarrow C_{p+q}(F\times \cP_{\star\to X} X) \stackrel{\Phi_*}\longrightarrow C_{p+q}(E)$, which we will also denote simply by $\Phi_*$.

We interpret the above $C_*(\Omega X)$-module structure on $C_*(F)$ as a DG-local system on $X$, which we denote $\cF$. Since any two lifting functions are homotopic through lifting functions, we infer that the DG-local system $\cF$ is canonically defined up to homotopy on any path-connected component of $X$.
To summarize: 

\begin{proposition}
	A Hurewicz fibration $\pi:E\to X$ with path-connected base determines a DG-local system on $X$ which is canonically defined up to homotopy. The fiber of the local system is given by cubical chains on the fiber of $\pi$, and the $C_*(\Omega X)$-right module structure is determined by the choice of a lifting function.  \qed
\end{proposition}

\section{Homology with DG-coefficients for fibrations}

We now assume that $X$ is a closed manifold and we choose a Morse function $f:X\to\R$ with a unique minimum at the basepoint $\star$. We also choose regular auxiliary data $\Xi$ for the definition of the Morse chain complex with DG-local coefficients, consisting of a Morse-Smale pseudo-gradient vector field and of a particular  embedded tree  $\caly$  whose branches are gradient trajectories joining the critical points of $f$ to the minimum $\star$ (one for each critical point). The next result is a restatement of Theorem~A from the Introduction. 

\begin{theorem} \label{thm:fibration} Let $E\to X$ be a Hurewicz fibration and let $\cF$ be the associated DG-local system on $X$. For any  Morse function $f:X\to \R$ with a unique minimum at the basepoint $\star$ and for a choice of the tree $\caly$ as above, there is a chain map 
	$$
\Psi:	C_*(X, \Xi;\cF)\longrightarrow C_*(E)
	$$
 which induces an isomorphism between the spectral sequence of the enriched Morse complex and the Leray-Serre spectral sequence of the fibration $E$. In particular $\Psi$ induces an isomorphism in homology 
 $$
 \Psi_* :H_*(X;\cF)\stackrel\simeq\longrightarrow H_*(E). 
 $$
\end{theorem}

\begin{remark}
With a little bit more work one can prove that $\Psi$ is a chain homotopy equivalence, and even a deformation retract. This is essentially equivalent to proving that the Morse complex is a deformation retract of the cubical or singular chain complex. See~\cite{Hutchings-families} for a relevant construction.  
\end{remark}

Theorem~\ref{thm:fibration} was proved with $\Z/2$-coefficients by Charette~\cite{Charette2017}. It is a generalization of the seminal result of Barraud-Cornea~\cite[Theorem~2.14]{BC07} which deals with the path-loop fibration $\mbox{ev}:\cP_{\star \to X} X\to X$. We now provide a few examples to which it applies. 

\begin{example}\label{ex:fibration}
(1) 
The fundamental example $\calf_*=R_*=C_*(\Omega X)$ endowed with right multiplication corresponds to the path-loop fibration 
$$
\Omega X\hookrightarrow \calp_\star X\stackrel{\mathrm{ev}_1}{-\!\!\!\longrightarrow} X
$$
whose total space $\calp_\star X=\calp_{\star\ri X} X$ is the space of paths in $X$ starting at the basepoint $\star$. Since $\calp_\star$ is contractible we obtain 
$$ H_*(X;C_*(\Omega X))\simeq H_*(point).
$$

(2)
Given a subset $C\subset X$, the fibration 
$$
\calp_{C\ri \star}X \hookrightarrow \calp_{C\ri X} X\stackrel{\mathrm{ev}_1}{-\!\!\!\longrightarrow} X
$$ 
whose total space $\calp_{C\ri X} X$ is the space of paths with starting point on $C$ defines on $C_*(\calp_{C\ri \star}X)$ --- the cubical chains on Moore paths starting on $C$ and ending at the basepoint $\star$ --- the right $C_*(\Omega X)$-module structure determined by concatenation. We have 
$$
H_*(X;C_*(\calp_{C\ri \star}X))\simeq H_*(\mathcal{P}_{C\to X}X). 
$$

(3) The loop-loop fibration $$\Omega X\hookrightarrow \call X \ri X$$ gives rise to the $C_*(\Omega X)$-module structure on $C_*(\Omega X)$ induced by the adjoint action of $\Omega X$ onto itself, $\Omega X\times  \Omega X \ri \Omega X$, $\gamma\otimes g \mapsto g^{-1}\gamma g$. Denoting the resulting local coefficients $C_*(\Omega X)^{\mathrm{ad}}$, we have 
$$
H_*(X;C_*(\Omega X)^{\mathrm{ad}})\simeq H_*(\call X).
$$

(4)
Let $X\subset Y$ and $\mathcal{F}=C_*(\Omega Y)$ seen as a right $C_*(\Omega X)$-module via multiplication and the embedding $C_*(\Omega X)\subset C_*(\Omega Y)$, where the basepoint $\star$ belongs to $X$. Then 
$$
H_*(X;\mathcal{F})\simeq H_*(\mathcal{P}_{\star\to X}Y),
$$
where $\mathcal{P}_{\star\to X}Y$ is the space of paths in $Y$ starting at the basepoint $\star$ and ending in $X$. 
\end{example}

\section{Proof of the main theorem on fibrations}\label{sec:proof-of-fibration}

This section is devoted to the proof of Theorem~\ref{thm:fibration}, i.e. Theorem~A in the Introduction. We will need a preliminary result regarding the Latour cells $\overline{W}^u(x)$, $x\in \Crit(f)$ from~\S\ref{subsec:Latour-cells}. We denote as usual by $C_*$ complexes of cubical chains. 

\begin{definition}
  Let $(s_{x,y})$ be a representing chain system for the Morse moduli spaces (see Definition~\ref{chainsystem}). A \emph{compatible representing chain system} for the Latour cells is a collection $\{s_x\in C_{|x|}(\overline{W}^u(x))\, : \, x\in \Crit(f)\}$ satisfying the following properties:
  \begin{enumerate}
  \item each $s_x$ is a cycle rel boundary and represents the fundamental class $[\overline{W}^u(x)]$;
  \item  each $s_x$ satisfies 
	\begin{equation}\label{eq:representingchain-Latourcells}
	\p s_x = \sum_y s_{x,y}\times s_y,
\end{equation} 
	with the product of chains defined via the inclusions $\overline\cL(x,y)\times \overline{W}^u(y)\subset \p \overline{W}^u(x)\subset \overline{W}^u(x)$. 
  \end{enumerate}
\end{definition}

The next result is an analogue of Proposition~\ref{representingchain}. 

\begin{lemma}
	Given a representing chain system $(s_{x,y})$ for the Morse moduli spaces, there exists a compatible representing chain system $(s_x)$ for the Latour cells. 
\end{lemma}

\begin{proof} 	Recall that  our recipe for orienting the moduli spaces of connecting Morse trajectories $\cL(x,y)$ takes as input a choice of orientation for the unstable manifolds $W^u(x)$. (The latter clearly also determines orientations of the Latour cells $\overline{W}^u(x)$.) With these choices, the product orientation on $\overline{\cL}(x,y) \times \overline{W}^u(y)\subset \p \overline{W}^u(x)$ differs from the boundary orientation by the sign $(-1)^{|x|-|y|}$, see~\eqref{eq:sign-difference-bord-cellule}. 
	
	From this point on, the construction proceeds inductively much like in Proposition~\ref{representingchain}. The resulting representing chains $(s'_x)$ will satisfy 
	$$
	\p s'_x = \sum_y (-1)^{|x|-|y|} s_{x,y}\times s'_y.
	$$ 
	To conclude, we set $s_x=(-1)^{|x|} s'_x$.
\end{proof}

\begin{proof}[Proof of Theorem~\ref{thm:fibration}]
 Denote $p:X\to X/\caly$ the canonical projection and let $\theta: X/\caly\ri X$ be a homotopy inverse for $p$.  Consider the  pullback fibration  
$$
\xymatrix{
 F' \ar[r]\ar[d]& F \ar[d]\\
 E' \ar[r]\ar[d]& E \ar[d]\\
 X/\caly \ar[r]^-\theta & X.
}
$$
 The desired isomorphism between $H_*(X;\calf)$ and $H_*(E)$ will be induced by a composition of several chain maps which are quasi-isomorphisms: 
 $$
 C_*(f,\cF) \stackrel\simeq\longrightarrow C_*(f,\cF') \stackrel\simeq\longrightarrow C_*(E') \stackrel\simeq\longrightarrow C_*(E). 
 $$
Our proof will be decomposed in three steps, each step consisting in showing that one of the above maps is a quasi-isomorphism. The DG local system $\cF'$ is the pullback $\theta^*\calf$ over $C_*(\Omega(X/\caly))$ and the complex $C_*(f,\calf')$ will be made explicit in Step 2 below. We actually define a 
quasi-isomorphism between the   enriched complex and the total space of the  pullback fibration over $X/\caly$ (Step 3) and  convert it into the desired quasi-isomorphism using the map $\theta: X/\caly \ri X$ (Step 1 and 2). 
 
\medskip  
\noindent {\it Step~1. Chain homotopy equivalence $C_*(E')\stackrel\simeq\longrightarrow C_*(E)$.} The homotopy equivalence $\theta$ induces a homotopy equivalence between the total space of the pullback fibration $E'=\theta^* E$ and the total space of the original fibration $E$, hence a quasi-isomorphism $C_*(E')\stackrel\simeq\longrightarrow C_*(E)$. 

\medskip 
\noindent {\it Step~2. Identification $C_*(f, \calf')\stackrel\cong\longleftrightarrow C_*(f, \calf)$.}

We first define the complex $C_*(f,\calf')$.
 Let  $q'_{x,y}: \overline{\call}(x,y)\ri \Omega (X/\caly)$ be a family of maps defined  for $x, y \in \Crit(f)$ similarly to the maps $q_{x,y}$ from Lemma~\ref{lecture}, namely 
 $$
 q'_{x,y}\, =\, p\circ \Gamma, 
 $$
 where $\Gamma: \overline{\call}(x,y) \ri \calp_{x\ri y} X$ is defined in Lemma~\ref{lecture}. Let 
 $$
 m'_{x,y}\, =\, q'_{x, y,*}(s_{x,y}) \in C_{|x|-|y|-1}(\Omega(X/\caly)).
 $$
 Obviously the family $(m'_{x,y})$ satisfies the Maurer-Cartan equation
 \begin{equation}\label{eqn:MCprime} 
 \partial m'_{x,y}\, =\, \sum_{z\in Crit(f)} (-1)^{|x|-|z|} m'_{x,z}\cdot m'_{z,y}.
 \end{equation}
 Let $\calf'$ be the DG-module defined by the action of $\Omega (X/\caly)$ on the fiber $F'$ of $ E'$; via the natural identification between $F'$ and $F$ the module $\calf'$ is the pullback $\theta^*\calf$.  Denote by $C_*(f, \calf')$ the enriched complex corresponding to the twisting cocycle $(m'_{x,y})$.

We remark that there is an obvious identification $C_*(f,\calf)\cong C_*(f,\calf')$. Indeed, by definition, the twisting cocycles which define these complexes satisfy the relation 
$$ 
m_{x,y}\, =\, \theta_*(m'_{x,y}).
$$
If $\alpha\in \calf'=\theta^*\calf$ then  $\alpha\cdot m'_{x,y} = \alpha\cdot\theta_*(m'_{x,y})=\alpha\cdot m_{x,y}$, which implies the claimed identification (see Remark \ref{rem:pullback}).

 Step~2 is proved and we now turn to the third and last step of the proof of Theorem~\ref{thm:fibration}.

 \medskip
\noindent {\it Step~3. Quasi-isomorphism $C_*(f,\cF') \stackrel\simeq\longrightarrow C_*(E')$.} 
 
 

We will construct a chain map $\Psi': C_*(f; \calf')\ri C_*(E')$ and show that it is a quasi-isomorphism.

   Consider a weakly transitive lifting function $\Phi$ for 
 the fibration $F'\ri E' \ri X/\caly$. 
 Also remark that since the tree $\caly$ has its  root $\star$ at the unique minimum of $f$ and its branches formed by gradient lines, with one gradient line $l_x$  between $\star$ and each $x\in \Crit(f)$ (considered with its endpoints included), we have that $\overline{W}^u(x)/l_x$ is homeomorphic to a closed disk of dimension $|x|$ and  moreover  the canonical maps $i_x:\overline{W}^u(x)\ri X$ induce a CW-decomposition of $X/\caly$ with cells
   $$j_x: \overline{W}^u(x)/\l_x \ri X/\caly.$$ The map $p: X\ri X/\caly$ is then cellular. 
 Now pick a representing chain system $(s_{x,y})$ for the Morse moduli spaces, and a compatible representing chain system $(s_x)$ for the Latour cells. Again we will use evaluation maps in order to transform the relation~\eqref{eq:representingchain-Latourcells}  into one  in cubic chains. The next statement is a counterpart to Lemma~\ref{lecture}. 
  
  \begin{lemma}\label{lecture-Latourcells} There exists a family of continuous maps 
 $$q_x: \overline{W}^u(x)\ri \calp_{\star\ri X/\caly}(X/\caly)$$
 such that: 

 a) For any  $(\lambda, a) \in \overline{\call}(x,y)\times \overline{W}^u(y)\subset \p \overline{W}^u(x)$ we have 
 \[q_x(\lambda, a)\, =\,  q'_{x,y}(\lambda)\#q_y(a).\]
  
b) The image $\ev_1\circ q_x (\overline{W}^u(x)) $ coincides with the image of the corresponding cell $j_x(\overline{W}^u(x)/l_x)$ in $X/\caly$.   
\end{lemma} 
 
 \begin{proof}[Proof of Lemma~\ref{lecture-Latourcells}] 
Any element of  $\overline{W}^u(x)$ can be identified with a (possibly broken) gradient line emanating from $x$. Indeed, by Definition \ref{Latour-cell}, such an element is either a point $a\in W^u(x)$ or a pair $(\lambda, a)\in  \overline{\call}(x,y)\times W^u(y)$. 
The corresponding broken gradient line is then either the gradient line  from $x$ to $a$, or the broken orbit $(\lambda,l(a))$ with $l(a)$ the gradient line between $y$ and $a$ in $W^u(y)$.
   As in the proof of Lemma~\ref{lecture} we parametrize this orbit by the values of $f$ and define $\Gamma(\lambda, l(a))\in \calp_{x\ri a} X$ by the same formula.  We see it as a map 
 $\Gamma: \overline{W}^u(x)\ri \calp_{x\ri X}(X)$
 and we obviously have 
 $$\Gamma(\lambda,l(a) ) \, =\, \Gamma(\lambda)\#\Gamma(l(a))$$
 for any  $(\lambda, a)\in  \overline{\call}(x,y)\times W^u(y)$. 
We then define the family of continuous maps 
 $$
 q_x\, =\,  p\circ \Gamma,
 $$ 
where $p:X\ri X/\caly$ is the projection. These maps satisfy Property a) by construction.

\medskip
We now prove property b).  Denote $p_x: \overline{W}^u(x) \ri \overline{W}^u(x)/l_x$
  the projection and remark that we have by definition 
   $
   \mbox{ev}_1\circ \Gamma \, =\, i_x,
   $ 
   where $i_x:\overline{W}^u(x)\ri X$ was defined above.  Therefore, 
  $$
  \mbox{ev}_1\circ q_x\, =\, \mbox{ev}_1\circ p \circ \Gamma\, =\, p\circ \mbox{ev}_1\circ \Gamma\, =\, p\circ i_x\, =\, j_x\circ p_x,
  $$
which implies the desired statement and finishes the proof. 
\end{proof} 
 
 \noindent {\it Proof of Theorem~\ref{thm:fibration} - Step 3 continued.} Define  $$
 m_x\, =\, q_{x*}(s_x) \in C_{|x|}(\cP_{\star\to X/\caly}X/\caly)
 $$ 
 for each $x\in \Crit(f)$. The previous lemma together with~\eqref{eq:representingchain-Latourcells} immediately imply
\begin{equation}\label{eq:MC-Latourcells}
\p m_x = \sum_y m'_{x, y}\cdot m_y.
\end{equation} 
Here the multiplication is determined by the left $C_*(\Omega (X/\caly))$-module structure on $C_*(\cP_{\star \to X/\caly}X/\caly)$ given by the concatenation $$\Omega( X/\caly)\times \cP_{\star \to X/\caly}X/\caly\to \cP_{\star \to X/\caly}X/\caly.$$ 
Using the lifting map $
\Phi_*: C_*(F')\otimes C_*(\cP_{\star \to X/\caly}X/\caly) \to C_*(E')
$ for the fibration $F'\ri E'\ri X/\caly$, we define 
$$
\Psi' : C_*(f;\cF')\to C_*(E'),\qquad \Psi'(\sigma \otimes x)=\Phi_*(\sigma\otimes m_x). $$
We check that $\Psi'$ is a chain map: 
\begin{align*}
	\p\Psi'(\sigma\otimes x) & = \p \Phi_*(\sigma\otimes m_x)\\
	& = \Phi_*(\p(\sigma\otimes m_x))\\
	& = \Phi_*(\p \sigma \otimes m_x + (-1)^{|\sigma|}\sigma\otimes \p m_x)\\
	& = \Phi_*(\p \sigma \otimes m_x + (-1)^{|\sigma|}\sigma\otimes \textstyle\sum_y m'_{x,y}\cdot m_y )\\
	& = \Phi_*(\p \sigma\otimes m_x) + (-1)^{|\sigma|}\textstyle\sum_y \Phi_*(\Phi_*(\sigma\otimes m'_{x,y})\otimes m_y)\\
	& = \Phi_*(\p \sigma\otimes m_x) + (-1)^{|\sigma|}\textstyle\sum_y \Phi_*(\sigma\cdot m'_{x,y}\otimes m_y)\\
	& = \Psi'(\p \sigma \otimes x + (-1)^{|\sigma|}\textstyle\sum_y\sigma\cdot m'_{x,y}\otimes y)\\
	& = \Psi' \p(\sigma\otimes x).
\end{align*}
The 5th equality makes use of the concatenation property~\eqref{eq:liftingfunction} for the lifting function.  The 6th equality follows from the fact that the module structure satisfies by definition $\sigma\cdot m'_{x,y}=\Phi_*(\sigma\otimes m'_{x, y})$.

Now let us prove that $\Psi'$ is a quasi-isomorphism. To this end we consider two spectral sequences which respectively converge to the homologies of our complexes: on the domain we have the spectral sequence $E_{pq}^r$ associated to the twisted complex (see~\S\ref{sec:spectral-sequence}), and on the target we have the Leray-Serre spectral sequence $\cE_{pq}^r$ associated to the filtration $C_*(\pi^{-1}(\mbox{Sk}_p(X/\caly))\subset C_*(E')$ by the pre-images of the skeleta $\mbox{Sk}_p(X/\caly)\subset X/\caly$ defined as the union of the cells $j_x(\overline{W}^u(x)/l_x)$ with $|x|\le p$.  

We are now close to conclude the proof of Theorem~\ref{thm:fibration}. Indeed, we 
 defined a chain morphism  $\Psi': C_*(f,\calf')\ri C_*(E')$ and by Lemma~\ref{lem:Psiprime} below it induces a morphism of spectral sequences which is an isomorphism on the first page. Therefore $\Psi'$ is a quasi-isomorphism, which finishes the proof of Step~3 in the proof of Theorem~\ref{thm:fibration}. 
\end{proof}

\begin{lemma} \label{lem:Psiprime} The map $\Psi'$ preserves the filtrations which define the two spectral sequences. Moreover, it induces an isomorphism at the first page $$E_{pq}^1 \stackrel\simeq\longrightarrow \cE_{pq}^1.$$
\end{lemma} 

\begin{proof} Let $x\in \Crit(f)$ and $\sigma\in C_*(F')$. We prove that the projection of $\Psi'(\sigma\otimes x)$ is contained in the singular complex of the corresponding cell $j_x(\overline{W}^u(x)/l_x)$. Using the properties of the lifting function we obtain 
$$\pi_*(\Psi'(\sigma\otimes x)) \, =\, \pi_*(\Phi_*(\sigma\otimes m_x)\, =\, (\pi\circ\Phi)_*(\sigma\otimes m_x)\, =\, \mbox{ev}_{1, *}( m_x),$$
hence 
$$
\pi_*(\Psi'(\sigma\otimes x)) \, =\,  \mbox{ev}_{1, *}( m_x)\, =\, (\mbox{ev}_1\circ q_x)_*( s_x).
$$
Now by Lemma~\ref{lecture-Latourcells}.b we know that the image of $\mbox{ev}_1\circ q_x$ coincides with the image of the cell $\overline{W}^u(x)/l_x$ in $X/\caly$, so our first claim is proved. 

In order to prove the second claim recall that the first pages of our spectral sequences are given by 
$$E_{pq}^1\, =\, H_q(F')\otimes C_p(f)$$ and 
$$\cE_{pq}^1\, =\, H_{p+q}(\pi^{-1}(\mbox{Sk}_p), \pi^{-1}(\mbox{Sk}_{p-1})).$$
It is known  that $\cE_{pq}^1$ is isomorphic to $ H_q(F')\otimes H_{p}(\mbox{Sk}_p, \mbox{Sk}_{p-1})$, and therefore by pullback through the embeddings of the cells $j_x:\overline{W}^u(x)/l_x\ri X/\caly$ it is isomorphic to $$\bigoplus_{x\in \Crit(f), \, |x|=p} H_q(F') \otimes H_p(\overline{W}^u(x)/l_x ,\partial (\overline{W}^u(x)/l_x)).$$
This isomorphism can be described using the lifting function $\Phi$ in the following way: if $\gamma_{\hat{a}}$ is a continuous  family of paths from the basepoint $\star$ to the points $\hat{a} \in \overline{W}^u(x)/l_x$, the maps $\chi_x : F'\times \overline{W}^u(x)/l_x\ri \pi^{-1}(\mbox{Sk}_p)\subset E'$
defined by $$\chi_x (f', \hat{a}) \, =\, \Phi(f', j_x(\gamma_{\hat{a}}))$$ induce the above isomorphism in homology.  On $\overline{W}^u(x)$ we already considered such a family of paths $\gamma_a$ defined by the broken gradient lines $l(a)$ from $\star=x$ to $a$ parametrized by the values of the function $f$. We define for $\hat{a}=p_x(a)$ the path $\gamma_{\hat{a}} = p_x(\gamma_a)$ and we infer 
$$
\chi_x (f', \hat{a}) \, =\, \Phi(f', j_x (\gamma_{\hat{a}}) )\, =\, \Phi(f', j_x\circ p_x(\gamma_a))\, =\, \Phi(f', p\circ i_x(\gamma_a)),
$$
where $i_x: \overline{W}^u(x)\ri X$ is the embedding of the corresponding cell and $p:X\ri X/\caly$ is the canonical projection. Now remark that $i_x(\gamma(a))= \Gamma(a)$, where 
$$
\Gamma :\overline{W}^u(x)\ri \calp_{\star\ri X}X
$$ 
was defined in the proof of Lemma~\ref{lecture-Latourcells} above. We therefore get 
$$
\chi_x (f', p_x(a)) \, =\, \Phi(f', q_x(a)),
$$
 which at chain level gives 
 $$
 \chi_{x,*}(\sigma\otimes p_{x,*}(s_x))\, =\, \Phi_* (\sigma\otimes m_x)\, =\, \Psi'(\sigma\otimes x).
 $$
 In other words, the map induced by $\Psi'$ on the first page is the same as the one induced by $\chi$ in homology via the obvious bijection $$H_p(\overline{W}^u(x)/l_x,\partial (\overline{W}^u(x)/l_x))\, \approx\,  \Z\langle x\rangle.$$ As a consequence, it is an isomorphism. The proof of Lemma~\ref{lem:Psiprime} is complete. 
\end{proof}

\rmk\label{rmk:fibration-boundary}  
There is a version of Theorem \ref{thm:fibration} for manifolds with boundary. 
Using the notation from~\S\ref{sec:DG-for-boundary} and Theorem~\ref{thm:fibration}, if $\p X= \p_-X  \cup \p_+X$ and $E_+=\pi^{-1}(X_+)$ is the total space of the restriction of the fibration to the part of the boundary along which the  negative pseudo-gradient points outwards, we have
$$
H_*(X,\p_+X;\calf)\, \, \simeq H_*(E,E_+).
$$
The proof is analogous to the closed case and is based on the fact that $X$ retracts onto $\p_+ X\cup \bigcup_{x\in\crit(f)}\ol W^u(x)$, with the canonical maps $i_x: \ol W^u(x) \ri X$ defining a CW-decomposition of $\p_+ X\cup \bigcup_{x\in\crit(f)}\ol W^u(x)$ relative to $\p_+X$. 
\kmr

 \section{Pullback of a Hurewicz fibration}\label{sec:pullback-fibration}
 We showed in Theorem~\ref{thm:fibration}
 that, given a Hurewicz fibration $\mathcal{E}$ :  $F\ri E_Y\ri Y$ over a closed manifold $Y$,  there is an isomorphism $\Psi_{\mathcal{E}}: H_*(Y;C_*(F))\ri H_*(E_Y)$ between the enriched homology with coefficients in the chains on the fiber and the singular homology of the total space. Let now $\varphi:X\ri Y$ be a continuous map which preserves the basepoints and consider the pullback fibration $\varphi^*{\mathcal{E}} : F\ri E_X\ri X$. We prove the following straightforward consequence of Theorem~\ref{thm:fibration}:
 \begin{corollary}\label{cor:pullback} There is an isomorphism $\Psi_{\varphi^*\mathcal{E}}: H_*(Y;\varphi^*C_*(F))\stackrel{\simeq}{\longrightarrow} H_*(E_X)$, where $\varphi^*C_*(F)$ denotes $C_*(F)$ as a DG-module over $C_*(\Omega X)$ via the induced map of DG-algebras $\Omega\varphi_*: C_*(\Omega X)\ri C_*(\Omega Y)$. In particular, if $X\hookrightarrow Y$ is an embedding, then $H_*(X, \varphi^* C_*(F))\simeq H_*(E_Y|_X).$  
 \end{corollary} 
 
 \begin{proof} It suffices to show that the module structure of $C_*(F)$ over $C_*(\Omega X)$ defined (as in \S\ref{sec:lifting-functions}) by some transitive lifting function of the pullback fibration $\varphi^*\cale$ is exactly the induced structure $\varphi^*C_*(F)$ described in the statement of the corollary. We then apply Theorem~\ref{thm:fibration} to conclude. 
 
 Writing $E_X=X \, _\varphi\!\times_{\pi_Y}E_Y$  (where $\pi_Y:E_Y\ri Y$ is the projection), notice that any lifting function $\Phi: E_Y\times  \calp_{\star \ri Y}Y\ri E_Y$ for the fibration $\cale$ defines  a lifting function $\Phi^\varphi:E_X \times  \calp_{\star \ri X}X\ri E_X$ by the following formula: 
$$\Phi^\varphi[(\gamma(0),e),\gamma]\, =\, [\mathrm{ev}(\gamma),\Phi(e,\varphi(\gamma)) ],$$
where $\mathrm{ev}$ is the evaluation at the endpoint of $\gamma$. Moreover it is easy to check that, if $\Phi$ is transitive, then $\Phi^\varphi$ has the same property. A choice of such a lifting function $\Phi$ for $\cale$ determines the DG-module structure of $C_*(F)$ over $C_*(\Omega Y)$. If $\gamma \in \Omega X$ then the relation above writes for $e$ in the fiber  $F_{\star_Y}$: 
$$\Phi^\varphi[(\star_X,e),\gamma]\, =\, [\star_X,\Phi(e,\varphi(\gamma)) ],$$
which means that the module structure of $C_*(F)$ over $C_*(\Omega X)$ defined by $\Phi^\varphi$ is precisely the one induced by $\Omega\varphi_*$. 
 \end{proof}

\chapter{Functoriality: general properties}\label{sec:functoriality}

We spell out in this chapter various functoriality properties of Morse homology with DG-coefficients, involving both direct and shriek maps. In the subsequent chapters~\S\S\ref{sec:functoriality-first-definition}-\ref{sec:second-definition} we will give two different (but equivalent) constructions for these maps, and prove their properties.

The following definition was already mentioned in~\S\ref{sec:alg-functoriality}; see Remark~\ref{rem:pullback} and Example \ref{ex:pullback}. 
\begin{definition}[pullback of a local system] \label{defi:inducedF} 
Let $Y$ be a based topological space and $\cF$ a DG local system on $Y$, i.e., $\cF$ is a right $C_*(\Omega Y)$-module. Given a continuous map $\varphi:X\to Y$ of based topological spaces, we induce a DG local system $\varphi^*\cF$ on $X$ by viewing $\cF$ as a right $C_*(\Omega X)$-module via the induced map of DG algebras $(\Omega \varphi)_*:C_*(\Omega X)\to C_*(\Omega Y)$. We call $\varphi^*\cF$ \emph{the pullback of $\cF$ via $\varphi$}.  
\end{definition}

Let $\varphi:X\ri Y$ be a continuous map between manifolds and $\calf$ a right $C_*(\Omega Y)$-module as above. We will suppose unless otherwise mentioned  that 
\begin{equation}\label{eq:preserving-basepoint}
\star_Y\, =\,  \varphi(\star_X),
\end{equation}
with $\star_X$, $\star_Y$ the basepoints of $X$ and $Y$ respectively. Our goal in this chapter is to define the following two functorial maps induced by $\varphi$ between the enriched homologies:
\begin{enumerate}
\item the \emph{direct morphism} defined for compact manifolds $X$ and $Y$
$$
\varphi_*:H_*(X;\varphi^*\cF)\to H_*(Y;\cF).
$$
\item the \emph{shriek morphism} defined under the assumption that $X$ and $Y$ are oriented and closed
$$
\varphi_!:H_*(Y;\cF)\to H_{*+\mathrm{dim}(X)-\mathrm{dim}(Y)} (X;\varphi^*\cF).
$$
The definition of the shriek morphism adapts to the case where $X$ and/or $Y$ have boundary, with the target space of $\varphi_!$ being replaced by the relative homology $H_{*+\mathrm{dim}(X)-\mathrm{dim}(Y)} (X,\p X;\varphi^*\cF)$. We explain this  in Remark \ref{rem:shriek-for-boundary}. 
\end{enumerate}
These maps will be defined at the level of the enriched Morse complexes built from some auxiliary data consisting of a Morse function, a negative pseudo-gradient, a collapsing tree etc. 

We denote $\Xi^X$, $\Xi^Y$ the auxiliary data on $X$, respectively $Y$. In order to get well-defined maps, the definitions have to match with the identifications from Theorem~\ref{independence}, meaning that the following diagrams have to commute up to homotopy:
\begin{equation}\label{welldefined-direct} 
  \xymatrix{
     C_*(X, \Xi_0^X; \varphi^*\calf)\ar[r]^-{\varphi_*}  \ar[d]^{\Psi_{01}^X}& C_*(Y, \Xi_0^Y; \calf)\ar[d]^{\Psi^Y_{01} }
     \\
     C_*(X, \Xi_1^X; \varphi^*\calf) \ar[r]^-{\varphi_*} & C_*(Y, \Xi_1^Y; \calf) }
\end{equation}
for the direct morphism, and 
\begin{equation}\label{welldefined-shriek} 
  \xymatrix{
     C_*(Y, \Xi_0^Y; \calf)\ar[r]^-{\varphi_!}  \ar[d]^{\Psi_{01}^Y}& C_{*+\mathrm{dim}(X)-\mathrm{dim}(Y)}(X, \Xi_0^X; \varphi^*\calf)\ar[d]^{\Psi^X_{01} }
     \\
     C_*(Y, \Xi_1^Y; \calf) \ar[r]^-{\varphi_!} & C_{*+\mathrm{dim}(X)-\mathrm{dim}(Y)}(X, \Xi_1^X; \varphi^*\calf) }
\end{equation}
for the shriek morphism.

\xmpl\label{ex:fibrations} {\bf The case of Hurewicz fibrations.}  

Let $\varphi:X\ri Y$ a continuous map between closed manifolds. Given a Hurewicz fibration $F\ri E_Y\stackrel{\pi_Y}{\longrightarrow} Y$, we consider its pullback $F\ri E_X\ri X$ and the map $\widetilde{\varphi}:E_X\ri E_Y$ induced canonically by $\varphi$. By the isomorphisms from Theorem~\ref{thm:fibration} and Corollary~\ref{cor:pullback} we may see the direct map $\varphi_*$ as a map from  $H_*(E_X)$ to $H_*(E_Y)$.  We will show in \S \ref{sec:direct-for-fibrations} that this map is identified with $\widetilde{\varphi}_*$. 

As for the shriek map, in general there is no
topological map $\widetilde{\phi}_!$ between the singular homologies $H_*(E_Y)$ and $H_*(E_X)$ that we could compare to our map $\phi_!$. However, such a map does exist in the particular case where $X\subset Y$ is a closed submanifold, with $X$ and $Y$ oriented (or, more generally, with $X$ cooriented in $Y$). We describe this map in~\S\ref{sec:shriek-for-fibration}. Riegel proves in~\cite{Riegel} that it coincides with the shriek map with DG coefficients. In~\S\ref{sec:shriek-for-fibration} we state his result and we prove a particular case that is used in~\cite{BDHO-cotangent}. 
\lpmx

The rest of this chapteris organized as follows. In~\S\ref{sec:functoriality-properties} we state the expected properties of direct and shriek maps. In~\S\ref{sec:comments-homotopy} we further explain the meaning of the {\sc (homotopy)} property. In~\S\ref{sec:sketch-two-constructions} we sketch our two equivalent constructions for both the direct and for the shriek maps. The details are given in~\S\ref{sec:functoriality-first-definition} for the first construction, and in~\S\ref{sec:second-definition} for the second one.

\section{Properties}\label{sec:functoriality-properties} 

We state the expected properties of direct and shriek maps in the form of a self-contained theorem.

\begin{theorem} \label{thm:f*!}
A continuous map between smooth closed manifolds $\varphi:X\to Y$ induces in homology a canonical map 
$$
\varphi_*:H_*(X;\varphi^*\cF)\to H_*(Y;\cF),
$$
and also, under the assumption that $X$ and $Y$ are oriented, a canonical map 
$$
\varphi_!:H_*(Y;\cF)\to H_{*+\dim(X)-\dim(Y)}(X;\varphi^*\cF), 
$$
with the following properties: 
\begin{enumerate} 
\item {\sc (Identity)} We have $\mathrm{Id}_*=\mathrm{Id}$ and $\mathrm{Id}_!=\mathrm{Id}$. 
\item {\sc (Composition)} Given maps $X\stackrel{\varphi}\longrightarrow Y \stackrel{\psi}\longrightarrow Z$ and a DG local system $\cF$ on $Z$, noting that       
$
(\psi\varphi)^*\cF=\varphi^*\psi^*\cF
$
we have 
$$
(\psi\varphi)_* = \psi_*\varphi_* : H_*(X;\varphi^*\psi^*\cF)\to H_*(Z;\cF)
$$
and 
$$
(\psi\varphi)_! = \varphi_!\psi_! :  H_*(Z ;\cF)\to H_{*+\mathrm{dim}(X)-\mathrm{dim}(Z)} (X;\varphi^*\psi^*\cF).
$$ 
\item {\sc (Homotopy)} Two homotopic maps induce the same direct and shriek morphisms. 
\item {\sc (Spectral sequence)} The morphisms $\varphi_*$ and $\varphi_!$ are limits of morphisms between the spectral sequences associated to the corresponding enriched complexes, given at the second page by 
$$ \varphi_{p,*}: H_{p}( X ; \varphi^*H_q(\calf) )\ri H_p( Y; H_q(\calf))$$ and
$$ \varphi_{p,!}:H_p(Y; H_q(\calf))\ri H_{p+\mathrm{dim}(X)-\mathrm{dim}(Y)}(X; \varphi^*H_q(\calf)),$$
i.e., the usual direct and shriek maps induced by $\varphi$ in (Morse)  homology with coefficients in $H_q(\calf)$. 
\end{enumerate}
\end{theorem}

\begin{corollary} \label{cor:homotopy_equivalence}
Let $\varphi:X\to Y$ be a homotopy equivalence and let $\cF$ be a DG local system on $Y$. 
The canonical maps
$$
\varphi_*:H_*(X;\varphi^*\cF)\to H_*(Y;\cF) \quad \mbox{and} \quad \varphi_!:H_*(Y;\cF)\to H_*(X;\varphi^*\cF)
$$
are isomorphisms. 
\end{corollary}

\begin{proof}
This is a straightforward consequence of conditions 1., 2. and 3. from Theorem~\ref{thm:f*!}. 
\end{proof}

As a matter of fact, much more is true: 

\begin{corollary} \label{cor:homotopy_equivalence_shriek_announcement}
Let $\varphi:X\to Y$ be an orientation preserving homotopy equivalence between closed oriented manifolds and let $\cF$ be a DG local system on $Y$. The canonical maps $\varphi_!:H_*(Y;\cF)\to H_*(X;\varphi^*\cF)$ and $\varphi_*:H_*(X;\varphi^*\cF)\to H_*(Y;\cF)$ are isomorphisms inverse to each other.  \qed
\end{corollary}

We will prove this result as Corollary~\ref{cor:homotopy_equivalence_shriek}. 

\begin{remark}
1. If $\varphi$ is a homeomorphism then $(\varphi^{-1})_*=\varphi_*^{-1}$ and $(\varphi^{-1})_!=\varphi_!^{-1}$. This is an immediate consequence of properties 1. and 2. above. Of course, if $\varphi$ is orientation preserving then it follows from the previous result that we actually  have $\varphi_!=\varphi_*^{-1}$. 

2. The shriek map can also be constructed if the manifolds $X$ and $Y$ have boundary (and are oriented). In this case it takes the form 
$$
\varphi_!:H_*(Y,\p Y;\cF)\to H_{*+\dim(X)-\dim(Y)}(X,\p X;\varphi^*\cF). 
$$
The functoriality properties are directly similar to those for the closed case.  
\end{remark}

%

\section{The meaning of the homotopy property} \label{sec:comments-homotopy}  Let us comment on the homotopy invariance property above. As stated it means that for two homotopic maps $\varphi_0$ and $\varphi_1$ the morphisms $$(\varphi_i)_*:H_*(X;\varphi_i^*\cF)\to H_*(Y;\cF)
$$
and  $$(\varphi_i)_!:H_*(Y;\cF)\to H_{*+\mathrm{dim}(X)-\mathrm{dim}(Y)} (X;\varphi_i^*\cF)
$$
are the same for $i=0$ and $i=1$.
For this to make sense we need a canonical identification between the enriched homologies $H_*(X;\varphi_0^*\cF)$ and $H_*(X;\varphi_1^*\cF)$. By Remark \ref{rem:pullback}, this is equivalent to an identification between the homologies of the complexes $C_*(X, \varphi_{i*}(m_{x,y}) ; \cF)$ for $i=0, 1$,  i.e., the enriched complexes on $X$ with coefficients in the  $C_*(\Omega Y)$-module  $\calf$ built using the twisted cocycles $(\varphi_{i*}(m_{x,y}))$.

In this section, we build such an identification $\Psi^\varphi$ from a homotopy $\varphi$ between $\varphi_0$ and $\varphi_1$ (Proposition~\ref{prop:homotopy-identification}). We then prove that $\Psi^\varphi$ does not depend on the choice of homotopy $\varphi$ (Proposition~\ref{prop-indep-homotopy-identifications}).

With these identifications in hand, the property of homotopy invariance should be understood as 
the fact that the following diagrams should be commutative {\it in homology}:

\begin{equation}
\xymatrix
@C=20pt
{
C_*(X, \varphi_{0*}(m_{x,y}) ; \cF) \ar[d]_{\Psi^\varphi} \ar[r]^-{=} &C_*(X, m_{x,y}; \varphi_{0}^*\cF) \ar[l] \ar[r]^-{\varphi_{0*}}& C_*(Y;\cF) \\
C_*(X, \varphi_{1*}(m_{x,y});\cF)\ar[r]^-{=}&C_*(X, m_{x,y}; \varphi_{1}^*\cF) \ar[l]  \ar[ur]_{\varphi_{1*}} &  
}\label{diag:homotopy-invariance}
\end{equation}
 
for the direct maps, 
and 
$$
\xymatrix
@C=18pt
{
C_*(X, \varphi_{0*}(m_{x,y}) ; \cF) \ar[d]_{\Psi^\varphi} \ar[r]^-{=} &C_*(X, m_{x,y}; \varphi_{0}^*\cF)\ar[l]& T_{*-\mathrm{dim}(X)+\mathrm{dim}(Y)}(Y;\cF)  \ar[l]^-{\varphi_{0!}} \ar[dl]_{\varphi_{1!}}\\
C_* (X, \varphi_{1*}(m_{x,y});\cF)\ar[r]^-{=}&C_*(X, m_{x,y}; \varphi_{1}^*\cF) \ar[l]  &  
}
$$ 
for the shriek maps.


A  homotopy $(\varphi_t)_{t\in [0,1]}$ yields an identification between the chain complexes $C_*(X, (\varphi_{0*}(m_{x,y}); \calf)$ and $C_*(X, (\varphi_{1*}(m_{x,y}); \calf)$ as follows:

\begin{proposition}\label{prop:homotopy-identification} Let $\varphi =(\varphi_t)_{t\in [0,1]}$ be a homotopy which satisfies condition~\eqref{eq:preserving-basepoint} at the endpoints, i.e.,  
$$\varphi_i(\star_X)\, =\, \star_Y$$ for $i=0,1$. Then there exists a chain homotopy equivalence 
$$\Psi^\varphi : C_*(X, \varphi_{0*}(m_{x,y}); \calf)\ri C_*(X, \varphi_{1*}(m_{x,y}); \calf)$$ such that:\\
(i) If $\varphi = \Id$ is the constant homotopy at $\varphi_0$ then $\Psi^\Id$ is homotopic to the identity and in particular $\Psi^\varphi=\Id$ in homology. \\
(ii) If two homotopies $\varphi$ and $\varphi'$ are homotopic with fixed endpoints then $\Psi^\varphi$ and $\Psi^{\varphi'}$ are homotopic.\\
(iii) If $\varphi_{01}$ is a homotopy between $\varphi_0$ and $\varphi_1$ and $\varphi_{12}$ is a homotopy between $\varphi_1$ and $\varphi_2$, then denoting $\varphi_{02}$ the concatenation of $\varphi_{01}$ and $\varphi_{02}$ we have that $\Psi^{\varphi_{12}}\circ \Psi^{\varphi_{01}}$ and $\Psi^{\varphi_{02}}$ are homotopic. Therefore 
$$\Psi^{\varphi_{12}}\circ \Psi^{\varphi_{01}}\, =\, \Psi^{\varphi_{02}}$$ in homology. 
 In particular $\Psi^\varphi$ is always an homotopy equivalence and thus a quasi-isomorphism. 
\end{proposition} 
\begin{proof} Take $\Xi$ to be the data on $X$ which produces $(m_{x,y})$ and proceed as in~\S\ref{sec:invariance-enriched} where we defined  the continuation morphism; here we consider  $\Psi^\Id$ corresponding to the trivial homotopy $\Id$ between $\Xi$ and $\Xi$.  The construction started with a representing chain system: if $f$  is the Morse function of $\Xi$ we defined a chain $$\sigma_{x,y} \in C_{|x|-|y|} (\overline{\call}_\Id(x,y))$$ for $x\in \{0\}\times \Crit(f)$ and $y\in \{1\}\times \Crit(f)$ which satisfies  equation~\eqref{eq:newinvarianceforchains}, i.e.,  
\begin{equation}\label{eq:newinvarianceforchains-2} 
\partial \sigma_{x,y}=\sum _{z\in \Crit(f)}s_{x,z}\times\sigma_{z,y}-\sum _{w\in \Crit(f)}(-1)^{|x|-|w|}\sigma_{x,w} \times s_{w,y},
\end{equation} 
where $(s_{x,y})$ is the representing chain system defined by $\Xi$. The evaluation we use now is different from~\eqref{eq:invariancelecture} of~\S\ref{sec:invariance-enriched}: instead of projecting $[0,1]\times X$ onto $X$ we send it to $Y$ via the homotopy $\varphi: [0,1]\times X\ri Y$; more precisely define 
$$Q_{x,y} : \overline{\call}_{\Id}(x,y) \ri \Omega Y$$ by 
$$Q_{x,y}\, =\, \varphi \circ \Theta\circ p \circ \Gamma_{x,y}^\Id$$
using the notation of~\eqref{eq:invariancelecture}; the condition~\eqref{eq:preserving-basepoint} guarantees that $Q_{x,y}$ has values in $\Omega Y$. Denoting $\nu^\varphi_{x,y} = -Q_{x,y,*}(\sigma_{x,y})\in C_{|x|-|y|}(\Omega Y)$ we infer from~\eqref{eq:newinvarianceforchains-2} that 
\begin{equation}\label{eq:homotopy-invariance-cocycle} 
\partial \nu^\varphi_{x,y}=\sum _{z\in \Crit(f)}\varphi_{0_*}(m_{x,z})\times\nu^\varphi_{z,y}-\sum _{w\in \Crit(f)}(-1)^{|x|-|w|}\nu^\varphi_{x,w} \times \varphi_{1_*}(m_{w,y}).
\end{equation} 

This is exactly the algebraic equation~\eqref{eq:DGcont} which gives rise to a chain map 
$$\Psi^\varphi : C_*(X, (\varphi_{0*}(m_{x,y}); \calf)\ri C_*(X, (\varphi_{1*}(m_{x,y}); \calf)$$
by setting 
$$\Psi^\varphi(\alpha\otimes x) \, =\, \sum_{y\in\Crit(f)}\alpha\nu^\varphi_{x,y}\otimes y.$$
Let us now check properties (i), (ii)  and (iii). 

As a preliminary remark, note that the algebraic homotopy class of the identification morphism $\Psi^\varphi$ does not depend on the choice of the representing chain system $\sigma_{x,y} \in C_{|x|-|y|} (\overline{\call}_\Id(x,y))$. Indeed, Proposition~\ref{prop:uniqueness} yields a homotopy cocycle $\kappa_{x,y} \in  C_{|x|-|y|+1} (\overline{\call}_\Id(x,y))$ between two representing chain systems $(\sigma_{x,y})$ and $(\sigma'_{x,y})$  which satisfies equation~\eqref{eq:sprime-s-homologous1} (one has to choose  $\kappa_{x,y}=0$ for $x,y$ critical points belonging to the same slice $\{i\}\times X$ in order to get this equation).  Applying the evaluation $Q$ above to~\eqref{eq:sprime-s-homologous1}, we get the algebraic homotopy equation~\eqref{eq:alg-homotopy-eqn} which implies that the morphisms $\Psi^\varphi$ and $\Psi^{\varphi'}$ are homotopic. 

 (i)  If $\varphi=\Id $ is the constant homotopy then $\varphi:[0,1]\times X\ri Y$ satisfies 
$\varphi\, =\, \varphi_0\circ \pi$, 
where $\pi:[0,1]\times X\ri X$ is the projection, and therefore $Q_{x,y}\, =\, \varphi_0\circ q_{x,y}$, 
where $q_{x,y}$ is the usual evaluation we used for the continuation morphism corresponding to the trivial homotopy between $\Xi$ and itself. We get 
$\nu^\varphi_{x,y}\, =\, \varphi_{0*}(\nu^\Id_{x,y})$ and, after the identification  
$C_*(X, \varphi_{0*}(m_{x,y}); \calf)\, =\, C_*(X, m_{x,y}; \varphi_{0}^*\calf)$ from Remark~\ref{rem:pullback}, we see that $\Psi^\Id$ corresponds to the continuation morphism between $C_*(X, m_{x,y}; \varphi_{0}^*\calf)$ and itself, which was proved to be the identity in Proposition~\ref{Id=id}. 

(ii) Consider two homotopies $\varphi$ and $\varphi'$ which are homotopic with fixed endpoints via a family $(\varphi_{\tau,t})_{(\tau,t)\in [0,1]^2}$.  Take $\Xi$ to be a set of data of the enriched complex on $X$ and construct as in the proof of Proposition \ref{end-independence} (the second step of Theorem~\ref{independence}), a representing chain system $S^\Id_{x,y} \in C_{|x|-|y|+1}\overline{\call}(x,y)$ corresponding to the trivial homotopy  of homotopies $\Xi_{\tau,t}=\Xi$.  We keep the same notation $S_{x,y}$ as in Proposition~\ref{end-independence}, with the superscript $\Id$ emphasising that we use a trivial homotopy of homotopies for $\Xi$ over $[0,1]^2\times X$.

As in the first part of the proof we use a different evaluation which maps this chain into $C_{|x|-|y|+1}(\Omega Y)$. More precisely, if $\Phi:[0,1]^2\times X\ri Y$ is obtained from $(\varphi_{\tau,t})$ in the  obvious way, we define $$Q_{x,y}: \overline{\call}(x,y) \ri \Omega Y$$ by 
$$
Q_{x,y} \, =\, \Phi\circ\Theta\circ p \circ \Gamma_{x,y}.
$$
Arguing as in the proof of~\ref{independence} we get that the cocycle $h_{x,y}^\Phi= Q_{x,y,*}(S^\Id_{x,y})$ thus obtained defines a chain homotopy  between $\Psi^\Id\circ\Psi^\varphi$ and $\Psi^{\varphi'}\circ\Psi^\Id$.  Since by point (i) the morphism $\Psi^\Id$ is homotopic to the identity we get the desired result. 

(iii) The argument is quite similar to the one of (ii) above. Start with a homotopy of homotopies connecting the concatenations $\Id\#\varphi_{02}$ and $\varphi_{12}\#\varphi_{01}$. We get a continuous map $\Phi:[0,1]^2\times X\ri Y$ such that $\Phi(\cdot,0,\cdot) = \varphi_{01}$, $\Phi(1,\cdot,\cdot)= \varphi_{12}$, $\Phi(0,\cdot,\cdot)=\Id$ (i.e. the constant homotopy, $\Phi(0,t,x)=\varphi_0(x)$) and $\Phi(\cdot, 1,\cdot)=\varphi_{02}$. The cocycle $h^\Phi_{x,y}$ defined by the same formula as above yields now a homotopy between $\Psi^{\varphi_{12}}\circ\Psi^{\varphi_{01}}$ and $\Psi^{\varphi_{02}}\circ\Psi^\Id$ and the latter is homotopic to $\Psi^{\varphi_{02}}$ by (i).

Finally remark that (i), (ii)  and (iii) imply that $\Psi^\varphi$ is a homotopy equivalence, its homotopy inverse being defined by $\bar{\varphi}=(\varphi_{1-t})$, the reverse homotopy of $\varphi$.
\end{proof}

\rmk The identification morphism $\Psi^\varphi$ was defined at the level of complexes and therefore depends a priori on the chosen   set of data $\Xi$. In order to have a well-defined map induced in homology  we still have to check the compatibility with the continuation morphisms. Namely,  if $\Xi_0$ and $\Xi_1$  are two sets of data we have to prove that the following diagram is commutative in homology : 
\begin{equation}\label{eq:identification-compatibility} 
  \xymatrix{
     C_*(X, \Xi_0 ; \varphi_0^*\calf)\ar[r]^-{\Psi_0^\varphi}  \ar[d]^{\Psi_{01}^0}& C_*(X, \Xi_0; \varphi_1^*\calf)\ar[d]^{\Psi_{01}^1 }
     \\
     C_*(X, \Xi_1; \varphi_0^*\calf) \ar[r]^-{\Psi_1^\varphi} & C_*(X, \Xi_1; \varphi_1^*\calf) }
\end{equation}
where $\Psi_{01}^i$ are the continuation morphisms between the sets of data $\Xi_0$ and $\Xi_1$ for the DG-module $\varphi_i^*\calf$ and $\Psi^\varphi_i$ are the two identification morphisms of the homotopy $\varphi$ respectively defined for the sets of data $\Xi_i$. This is done as usual by considering a set of data  $\Xi$ on  $[0,1]^2\times X$ which will produce a chain homotopy between  $\Psi_{01}^1\circ \Psi_0^\varphi$ and $\Psi_1^\varphi\circ\Psi_{01}^0$. To construct $\Xi$, start with a set of data $\Xi_t$ on $[0,1]\times X$ which defines the continuation morphisms $\Psi_{01}^i$ and extend it on $[0,1]^2\times X$ constantly: $\Xi_{(\tau, t)}=\Xi_t$.   As in the proof of (ii) above (and  in the second step of the proof of Theorem~\ref{independence}) construct a representing chain system $S_{x,y}^\Xi$ from this data. We evaluate it using the application $\Phi:[0,1]^2\times X\ri Y$ defined by $\Phi(\tau, t, x) = \varphi(\tau, x)$. More precisely, the evaluation analogous to the one in the proof of (ii)  $$
Q_{x,y} \, =\, \Phi\circ\Theta\circ p \circ \Gamma_{x,y}
$$
  produces now a cocycle $h^\Phi(x,y)= Q_{x,y,*}(S_{x,y}^{\Xi})\in C_{|x|-|y|+1}(\Omega Y)$ which defines the desired chain homotopy and implies the commutativity of the above diagram in homology. 
\kmr


Let us emphasize that  the morphism of complexes  $\Psi^\varphi$ depends on the homotopy $\varphi$. At this stage we 
only proved that  when two homotopies are homotopic with fixed endpoints,  then the associated morphisms coincide in homology (they are actually homotopic). However, it turns out that in homology  the identification $\Psi^\varphi$ is actually independent on the choice of the homotopy $\varphi$ between $\varphi_0$ and $\varphi_1$. Remarkably, we are only able to prove this statement {\it a posteriori}, i.e., using the fact that direct maps satisfy all the properties  listed in Theorem~\ref{thm:f*!},
including homotopy invariance. Here is the statement : 

\begin{proposition}\label{prop-indep-homotopy-identifications}
Suppose that the direct  maps are well defined and satisfy Properties 1-3 from Theorem~\ref{thm:f*!}. Then, given two homotopic maps $\varphi_0, \varphi_1: (X,\star_X)\ri (Y, \star_Y)$, the map induced in homology by the identification map  $\Psi^\varphi$ defined above is independent  of the choice of the homotopy $\varphi$.  \end{proposition} 

\begin{proof} Suppose first that $\varphi_{1*}$ is a monomorphism in homology. Take two homotopies $\varphi$ and $\varphi'$ between $\varphi_0$ and $\varphi_1$. Then the 
  homotopy property (see diagram (\ref{diag:homotopy-invariance}))
  implies $$\varphi_{1*}\circ \Psi^\varphi\, =\,  \varphi_{1*}\circ\Psi^{\varphi'}$$ and since $\varphi_{1*}$ is injective in homology this gives the equality $\Psi^{\varphi}= \Psi^{\varphi'}$,  also in homology.

Let us now turn to the general case. Let $\varphi=(\varphi_t)$ be a homotopy between the given maps $\varphi_0$ and $\varphi_1$.  Denote by $\tilde{\varphi}= (\tilde{\varphi_t})$ the homotopy of maps $X\ri X\times Y$ given by the graphs $$\tilde{\varphi}_t(x) = (x,\varphi_t(x))$$
Remark that $\tilde{\varphi}_{1*}$ is injective in homology. Indeed, denoting  $\pi_X: X\times Y \ri X$, the composition and identity properties imply that $\pi_{X*}\circ\tilde{\varphi}_{1*}=\Id$. Denote $\pi^Y:X\times Y\ri Y$ and consider the DG-module $\pi_{Y}^*\calf$.  According to the homotopy property there is  a map $\Psi^{\tilde{\varphi}}: C_*(X, \tilde{\varphi}_{0*}(m_{x,y}); \pi_Y^*\calf)\ri C_*(X, \tilde{\varphi}_{1*}(m_{x,y});\pi_Y^*\calf)$ such that  $\tilde{\varphi}_{0*}=\tilde{\varphi}_{1*}\circ \Psi^{\tilde{\varphi}} $ in homology; this map was explicitly constructed above. 

Then,  our particular case above tells us that  $\Psi^{\tilde{\varphi}} $ is independent of $\tilde{\varphi}$ (and a fortiori of $\varphi$) in homology since $\tilde{\varphi}_{1*}$ is mono. On the other hand  by Remark~\ref{rem:pullback} (which we have also used in the diagram~\eqref{diag:homotopy-invariance}), as $\pi_Y\circ\tilde{\varphi} =\varphi$,  there are obvious identifications $$C_*(X, \tilde{\varphi}_{i*}(m_{x,y});\pi_Y^*\calf)=C_*(X, \varphi_{i*}(m_{x,y});\calf)$$  for $i=0,1$ and the maps $\Psi^{\tilde{\varphi}} $ and $\Psi^\varphi$ are identical via these identifications. 
Indeed by definition,  using the same notation we have 
$$\Psi^\varphi(\alpha\otimes x)\, =\, \sum_y\alpha \nu_{x,y}^\varphi\otimes y,$$
where 
$$\nu^{\varphi}_{x,y} \, =\, -Q_{x,y,*}(\sigma_{x,y})$$ and $Q_{x,y}=\varphi\circ \Theta\circ p \circ \Gamma^\Id_{x,y}$. On the other hand  $$\Psi^{\tilde\varphi}(\alpha\otimes x)\, =\, \sum_y\alpha \nu^{\tilde{\varphi}}_{x,y}\otimes y$$
with 
$$\nu^{\tilde{\varphi}}_{x,y} \, =\, -\widetilde{Q}_{x,y,*}(\sigma_{x,y}),$$ where $Q_{x,y}=\tilde{\varphi}\circ \Theta\circ p \circ \Gamma^\Id_{x,y}$. Therefore $\pi_Y\circ \widetilde{Q}_{x,y}=Q_{x,y}$ which implies $\pi_{Y*}(\nu_{x,y}^{\tilde{\varphi}}) =\nu_{x,y}^\varphi$. Now in the formula for $\Psi^{\tilde\varphi}$ the product $\alpha\cdot \nu^{\tilde{\varphi}}_{x,y}$ is defined by the module structure of $\pi_Y^*\calf$, and so it equals $\alpha\cdot \pi_{Y*}(\nu^{\tilde{\varphi}}_{x,y})=\alpha \cdot\nu^{\varphi}_{x,y}$.  This finishes the proof of our proposition. \end{proof} 

Note that we could have used shriek maps instead of direct maps with epimorphisms instead of monomorphisms in the argument above.

\rmk\label{identical-identifications}  Let  $\varphi_0, \varphi_1: X\ri Y$ be homotopic continuous  maps  and $\chi:Y\ri Z$ another continuous map. Consider a DG-module $\calf$ over $C_*(\Omega Z)$. If we denote by $\varphi$  the homotopy between $\varphi_0$ and $\varphi_1$ and $\chi\circ \varphi$ the composed homotopy between $\chi\circ \varphi_0$ and $\chi\circ \varphi_1$, we have a priori two identification isomorphisms between $H_*(X;\varphi_0^*\chi^*\calf)$ and $H_*(X;\varphi_1^*\chi^*\calf)$: the one defined by homotopy $\varphi$ for the DG-module $\chi^*\calf$ and the one defined by the homotopy $\chi\circ \varphi$ for the DG-module $\calf$. We remark that by construction,  $\Psi^\varphi, \Psi^{\chi\circ \varphi}: H_*(X;\varphi_0^*\chi^*\calf)\ri H_*(X;\varphi_1^*\chi^*\calf)$ are equal. Indeed, by definition,  if  $\nu_{x,y}^\varphi\in C_{|x|-|y|}(\Omega Y)$ is the cocycle which defines $\Psi^\varphi$ then $\nu_{x,y}^{\chi\circ \varphi}= \chi_*(\nu_{x,y}^\varphi)\in C_{|x|-|y|}(\Omega Z)$. 
But on the other hand when we write the formula of $\Psi^\varphi$
$$\Psi^\varphi(\alpha\otimes x)\, =\, \sum_y\alpha \nu_{x,y}^\varphi\otimes y, $$
the product $\alpha\cdot \nu_{x,y}^\varphi$ is defined by the module structure of $\chi^*\calf$ so it actually equals $\alpha\cdot \chi_*(\nu_{x,y}^\varphi)=\alpha\cdot \nu_{x,y}^{\chi\circ \varphi}$, which implies the claimed equality. 
\kmr

\begin{remark}\label{rmk:identification-different-basepoints}
One can easily get rid of condition~\eqref{eq:preserving-basepoint} for the endpoints of the homotopy,  which is  impossible to satisfy when $\varphi_0$ and $\varphi_1$ have disjoint images. For that we need to identify the twisted complexes constructed using  $\varphi_0(\star_X)$ and $\varphi_1(\star_X)$ as basepoints. This is done like in~\S\ref{sec:identification-basepoints}, using a map  $\eta: Y\ri Z$ such that $\eta(\varphi_0(\star_X))\, =\, \eta(\varphi_1(\star_X))=\star_Z$ and working under the assumption~\eqref{eq:identification-basepoints}, i.e., $\calf_i= \eta^*\calf$ for $i=0,1$, where $\calf$ is a DG-module over $C_*(\Omega Z)$. 
\end{remark}

\section{Sketch of the two equivalent constructions} \label{sec:sketch-two-constructions}

We will present two constructions both for the direct and for the shriek maps. The details will be given in~\S\ref{sec:functoriality-first-definition} and~\S\ref{sec:second-definition}, and the equivalence of the two constructions will be proved in~\S\ref{sec:two-defi-equivalent}. 
Each of the two approaches has its own advantages: the first one allows for an easier comparison of signs and is suitable for explicit computations (for example, the map induced by an embedding is directly seen to be an inclusion at chain level); the second one is uniform for all maps and prone to generalizations in Floer theory. In order to immediately provide the reader with an intuition, we give in this section a sketch of these constructions which we intend as a reading guide for the subsequent chapters. 

A common point of both constructions is that it is enough to treat smooth maps. Indeed, any continuous map can be approximated by a smooth map, and two such approximations which are close enough in $C^0$-norm are necessarily homotopic.

The first approach proceeds by constructing the direct and the shriek maps in the following steps: 
\renewcommand{\theenumi}{\roman{enumi}}
\begin{enumerate}
\item {\it $X\subset Y$ is codimension $0$ submanifold (with boundary)}. In this case the direct map is constructed by starting with a negative pseudo-gradient on $X$ which points inwards along the boundary, and then extending it to $Y$. The shriek map is constructed by starting with a negative pseudo-gradient on $X$ which points outwards along the boundary, and then extending it to $Y$. 
\item {\it $X\subset Y$ is a closed submanifold}. In this case one starts with Morse data on $X$, extends it to Morse data on a tubular neighborhood of $X$ in $Y$, and then proceeds as in the previous step.  
\item {\it $\varphi:X\to Y$ is an embedding}. We view $\varphi$ as the composition between the diffeomorphism $\varphi:X\to \varphi(X)$ and the embedding $\varphi(X)\to Y$, we define direct and shriek maps for diffeomorphisms and then refer to the previous step. 
\item {\it $\varphi:X\to Y$ is a general smooth map}. We pick an embedding $\chi:X\to D^m$ into an $m$-disc and further consider the embedding $(\chi,\varphi):X\to D^m\times Y$. Since the Morse data on $D^m\times Y$ can be chosen such that the Morse complex is identified to the Morse complex for $Y$, we reduce in this way to the previous case.  
\end{enumerate}
While building from one step to the other we need to check that properties 1-4 from Theorem~\ref{thm:f*!} are indeed satisfied, and in particular that the definitions do not depend on the choice of auxiliary data. The details of the first approach are contained in~\S\ref{sec:functoriality-first-definition}.

The second approach takes as input a smooth map $\varphi:X\to Y$ and proceeds to define $\varphi_*$ and $\varphi_!$ at chain level directly in terms of suitable moduli spaces. 

For the definition of $\varphi_*$ we use the moduli spaces 
$$
\cM^\varphi(x,y') = W^u(x)\cap \varphi^{-1}(W^s(y')) \, \cong W^u(x) \, _\varphi\!\times  W^s(y'),
$$
where $x$ is a critical point of the chosen Morse function on $X$, where $y'$ is a critical point of the chosen Morse function on $Y$, and we choose the negative pseudo-gradients generically such that the intersection is transverse (of dimension $\dim\, \cM^\varphi(x,y')=|x|-|y'|$). The rightmost term denotes  the fiber product of $W^u(x)$ and $W^s(y')$.

For the definition of $\varphi_!$ we use the moduli spaces 
$$
\cM^{\varphi_!}(x',y) = W^s(y)\cap \varphi^{-1}(W^u(x')) \, \cong W^s(y) \, _\varphi\!\times  W^u(x'),
$$
where $x'$ is a critical point of the chosen Morse function on $Y$ and $y$ is a critical point of the chosen Morse function on $X$, and we again work with generic Morse data such that the intersection is transverse (of dimension $\dim\, \cM(x',y)=|x'|-|y|+\dim(X)-\dim(Y)$). The details of this second approach are contained in~\S\ref{sec:second-definition}.

The verification of signs is delicate and at times tedious. We give the full details for the first approach in~\S\ref{sec:functoriality-first-definition}. For the second approach we explain the orientation conventions and we compare them with those of the first approach in~\S\ref{sec:two-defi-equivalent}, where we show that the two definitions are equivalent. Some other details for the second approach are omitted.

Similarly, we give full details for the proof of Theorem~\ref{thm:f*!} in the first approach, and we remain much more sketchy in the second approach (see~\S\ref{sec:properties-f*-revisited}).

\chapter{Functoriality: first definition}\label{sec:functoriality-first-definition}

We develop in this chapter our first approach for the construction of direct and shriek maps, as outlined in~\S\ref{sec:sketch-two-constructions}.

\section{Functoriality for $0$-codimensional submanifolds }\label{sec:funct-0-codim-manifolds} 
Suppose $U\subset Y$ is a codimension $0$ submanifold with boundary $\p U\subset \mathring{Y}$. Consider a DG-module $\calf$ over $C_*(\Omega Y)$. We define $i_*$ and $i_!$ for the inclusion $$i: U\hookrightarrow Y.$$

\emph{1. The direct map}. We start with a Morse-Smale pair $(f_U, \xi_U)$ on $U$ with $\xi_U$ pointing inwards along the boundary (which means $\p_+U=\emptyset$). We then extend it to $f:Y\ri \R$ (if $Y$ has a boundary we have to proceed as in the previous chapter).    We orient the unstable manifolds of the critical points  of $f_U$ in the same way as the ones of $f$.  Suppose that the basepoint $\star$  is in $U$ and take a tree $\caly_U\subset U$ with root in $\star$ and extend it to a tree $\caly\subset Y$. Finally consider $\theta_U: U/\caly_U\ri U$ a homotopy inverse for the projection and extend it to $\theta : Y/\caly \ri Y$ homotopy inverse for $p:Y\ri Y/\caly$. Denote by $\Xi^U$ and $\Xi^Y$ these two sets of data and define a morphism of enriched Morse complexes 
$$i_*:C_*(U, \Xi^U; i^*\calf)\ri C_*(Y,\Xi^Y; \calf) $$
by the formula 
$$i_* (\alpha\otimes x)\, =\, \alpha\otimes x 
$$
for $\alpha\in \calf$ (viewed as a $C_*(\Omega U)$-module on the left hand side and as a $C_*(\Omega Y)$-module on the right hand side) and $x\in \Crit(f_U)\subset \Crit(f)$. Since there is no flow line starting  from a critical point in $U$ and ending in $Y\setminus U$ it is clear that the first complex is a subcomplex of the last one and the inclusion $i_*$ is a morphism of complexes. Moreover it satisfies~\eqref{welldefined-direct}; indeed a homotopy of  Morse-Smale pairs on $U$ with pseudo-gradients pointing inward can be extended to a homotopy on $Y$ making the diagram~\eqref{welldefined-direct} obviously commutative.  We therefore get a morphism in homology
$$
i_*: H_*(U; i^*\calf)\ri H_*(Y ; \calf),
$$
where the target space could be a relative enriched homology if $Y$ has boundary and the pseudo-gradient $\xi$ is chosen accordingly. 

\emph{2. The shriek map}. We repeat the entire  construction except for the fact that the gradient $\xi_U$ is chosen to point outwards along the boundary of $U$. We define 
$$i_! : C_*(Y, \Xi^Y; \calf)\ri C_*(U,\partial U, \Xi^U; i^*\calf)$$ by 
$$i_!(\alpha\otimes x) \, =\,
\begin{cases}
  \alpha \otimes x& \text{ if } x\in U,\\
  0&\text{ if }x\not\in U.
\end{cases}$$
Again, since $\overline{\call}(x,y)= \emptyset$ for $x\not\in U$ and $y \in U$ it is easy to see that $i_!$ is a morphism of complexes (as a matter of fact $C_*(U,\partial U, \Xi^U; i^*\calf)$ can be interpreted as a quotient complex). This morphism fits into the commutative diagram~\eqref{welldefined-shriek} by an argument similar to the one above, and therefore yields a morphism in homology 
$$i_!: H_*(Y; \calf)\ri H_*(U, \p U ; i^*\calf).$$
Denote $i : U\hookrightarrow Y$ and $j: \overline{Y\setminus U}\hookrightarrow Y$ the inclusions. By construction we have an exact sequence of enriched complexes given by 
\begin{equation}\label{eq:exact-sequence1}
\xymatrix
@C=10pt
{0\ar[r] &C_*(U, \Xi^U; i^*\calf)\ar[r]^-{i_*} &C_*(Y, \Xi^Y;\calf)\ar[r]^-{j_!}& C_*(\overline{Y\setminus U},\p U, \Xi^{Y\setminus U}, j^*\calf)\ar[r]&
0.}
\end{equation}
This gives rise to a long exact sequence
\begin{equation}\label{eq:long-exact-sequence1}
\xymatrix
@C=9pt
{\dots H_*(U; i^*\calf)\ar[r]^-{i_*} &H_*(Y;\calf)\ar[r]^-{j_!}& H_*(\overline{Y\setminus U},\p U; j^*\calf)\ar[r]^-\delta&
H_{*-1}(U;i^*\calf)\dots}
\end{equation}
The basepoint should belong to the boundary $\partial U$ in order for all the above homologies to be well defined. 

\section{Functoriality for closed submanifolds}\label{sec:funct-closed-submanifolds} 

Let $X^m$ be a closed submanifold of a compact manifold  $Y^n$. We define the direct and shriek maps for the inclusion $$i:X\hookrightarrow Y.$$  

\emph{1. The direct map}. Consider a tubular neighborhood $U$ of $X$ and denote by $i_U:U\to Y$ the inclusion.  Fix a basepoint $\star\in X$ and consider a DG-module $\calf$ over $C_*(\Omega Y)$. Choose a set of Morse data $\Xi^X$ on $X$ with Morse function $f$. We extend $f$ to $U$ as $\chi\cdot (f\circ \pi)$,  
where $\chi$ is some cutoff function which equals $1$ near $X$ and $0$ near $\p U$ and $\pi:U\to X$ is the projection. Consider a metric on the normal bundle of $X$ in $Y$, assumed to be Euclidean in the Morse neighborhoods of the critical points of $f$, and  denote $h:U\ri \R$ the function which corresponds to $\|\cdot\|^2$. For $A>0$ large enough the function $F= \chi\cdot (f\circ \pi)+A \cdot h$ is Morse, and its critical points are those of $f$ with the same Morse indices. We extend the negative pseudo-gradient $\xi$ of $f$ to $\xi_F$ on $U$ such that $\xi_F=-A\cdot\nabla h$ near $\p U$, so that $\xi_F$ is in particular pointing inwards along $\p U$. We clearly have 
$$\overline{\call}_F(x,y)\, =\, \overline{\call}_f(x,y)$$ for all critical points $x,y$ of $f$.  We also have  $\overline{W}^u_F(x)= \overline{W}^u_f(x) $ and we choose the same  orientation for $\ol W^u_F(x)$ for any $x\in \crit(f)$
\begin{equation}\label{extended-orientation-unstable}
\ori\, \ol W_F^u(x)\, =\, \ori\, \ol W^u_f(x),
\end{equation} so that the orientation rule~\eqref{eq:orientation-rule}   implies that  the orientations of the above trajectory spaces also coincide. Taking the same tree $\caly$ and the same map $\theta$ on $U$ and on $X$ we obtain twisting cocycles $m_{x,y}^F$ and $m_{x,y}^f$ which satisfy 
$$
(i_X)_*(m_{x,y}^f) \, =\, m_{x,y}^F,
$$ 
where $i_X :\Omega X \hookrightarrow \Omega U$  is the inclusion. Therefore the associated enriched Morse complexes coincide, 
$$C_*(X, \Xi ; i^*\calf) \, =\, C_*(U, \Xi^U ; i_U^*\calf).$$
This means that the morphism $$(i_U)_*: H_*(U;i_U^*\calf)\ri H_*(Y;\calf)$$ defined in~\S\ref{sec:funct-0-codim-manifolds} yields a map 
$$i_*: H_*(X;i^*\calf)\ri H_*(Y;\calf),$$
which is by definition the morphism induced by the embedding $i:X\hookrightarrow Y$ in enriched Morse homology. 

\begin{remark}\label{rem:spectral-direct} The map $i_*: C_*(X,\Xi^X;i^*\calf)\ri C_*(Y,\Xi^Y;\calf )$ obviously preserves the filtration given by the Morse indices and therefore induces a morphism between the spectral sequences associated to this filtration. Thus the map $i_* :H_*(X;i^*\calf)\ri H_*(Y;\calf)$ induced in homology can be seen as the limit of the map between the corresponding spectral sequences. One can easily see that on the second page the morphism 
$$i_{pq,*}^2 : H_{p}( X ; i^*H_q(\calf) )\ri H_p( Y; H_q(\calf))$$ is the natural one induced by the inclusion $i$ between the homologies with coefficients in the local system $H_q(\calf)$ (which induces $i^*H_q(\calf)$ on $X$)
\end{remark} 

\rmk\label{rmk:long-exact}The long exact sequence for enriched Morse homology of a manifold with boundary admits the following description. Let $(Y,\p Y)$ be a compact manifold and consider $U= [-\epsilon, 0]\times \p Y$ a collar neighborhood of the boundary $\p Y\equiv \{0\}\times\p Y$ in $Y$.  Take a basepoint $\star_Y= (-\epsilon,\star) \in U$ for $Y$ and $\calf$ a DG-module over $C_*(\Omega_{\star_Y} Y)$. Denote $V= Y\setminus (-\epsilon,0]\times \p Y$ with inclusion $j:V\hookrightarrow Y$, and write the long exact sequence~\eqref{eq:long-exact-sequence1} for a pseudo-gradient pointing outwards $V$ and inwards $Y$,
\begin{equation*}\label{eq:long-exact-sequence2}
\xymatrix
@C=15pt
{
\dots H_*(U; i_U^*\calf)\ar[r]^-{(i_{U})_*} &H_*(Y;\calf)\ar[r]^-{j_!}& H_*(V,\p V; j^*\calf)\ar[r]^-\delta&
H_{*-1}(U;i_U^*\calf) \dots
}
\end{equation*}

The inclusion $j$ defines an isomorphism $$j_!:H_*(V,\p V; j^*\calf)\stackrel{\simeq}{\to} H_*(Y,\p Y; \calf)$$ since a Morse-Smale 
pair on $(V, \p V)$ (with outward gradient) can be extended to one on $(Y,\p Y)$ without adding any critical point. 
Considering the inclusion $i_0:\p Y\ri U$ defined by $y\mapsto (0,y)$ we have  
$$
H_*(U; i_U^*\calf)\simeq H_*(\p Y; i_0^*i_U^*\calf)= H_*(\partial Y;i^*\calf).
$$
Finally we get the long exact sequence of the pair 
\begin{equation}\label{eq:long-exact-sequence2bis}
\xymatrix
@C=10pt
{
\dots H_*(\p Y; i^*\calf)\ar[r]^-{i_*} &H_*(Y;\calf)\ar[r]& H_*(Y,\p Y; \calf)\ar[r]^-\delta&
H_{*-1}(\p Y;i^*\calf)\dots 
}
\end{equation}

When $F\ri E\ri Y$ is a Hurewicz fibration and $\cF=C_*(F)$ with its $C_*(\Omega Y)$-module structure defined using a lifting function as in~\S\ref{sec:fibrations}, we retrieve the long sequence of the pair $(E, E|_{\p Y})$, according to Remark~\ref{rmk:fibration-boundary}. 
\kmr 

\emph{2. The shriek map}. Suppose that the submanifold $X^m\stackrel{i}{\hookrightarrow} Y^n$ is closed and {\it co-oriented}. Starting from a set of Morse data $\Xi$ on $X$ we construct $\Xi^U$ on a tubular neighborhood $U$ in the same way as above except for the fact that the function $h$ corresponds to $-\| \cdot\|^2$. We still have 
$$
\Crit(F)\, = \, \Crit(f),
$$ 
but now the Morse indices differ 
$$|x|_F\, =\, |x|_f+n-m$$
and the gradient $\xi_F$ points outwards $U$. The trajectory spaces are identical 
$$
\overline{\call}_F(x,y)\, =\, \overline{\call}_f(x,y),
$$
but the unstable manifolds are not:  one can identify $W^u_F(x)$ with $W^u_f(x) \times T_xN$ where $N\ri X$ is the normal bundle.  We choose to orient them by 
\begin{equation}\label{eq:orientation-N} \ori \, W^u_F(x) \, =\, \left(\ori\,  W^u_f(x), \ori \, N \right)\end{equation}  and applying the orientation rule~\eqref{eq:orientation-rule}, we obtain that the identification $\overline{\call}_F=\overline{\call}_f$ preserves the orientation. Therefore $s_{x,y}^F=s_{x,y}^f$, hence  $$m^F_{x,y}= m^{f}_{x,y}.$$
Due to the difference of Morse indices we therefore have 
\begin{equation}\label{eq:adhoc-Thom} C_*(X, \Xi ; i^*\calf) \, =\, C_{*+n-m}(U, \Xi^U ; i_U^*\calf),\end{equation}
and so  applying the construction in~\S\ref{sec:funct-0-codim-manifolds} we finally get a chain map 
$$i_! : C_*(Y, \Xi^Y; \calf)\ri C_{*+n-m} (X, \Xi^X; i^*\calf)$$
which yields the shriek morphism in homology
$$i_!:H_*(Y;\calf)\ri H_{*+n-m}(X; i^*\calf).$$

\begin{remark}\label{rem:spectral-shriek} Here again the shriek map $i_!$ at the level of complexes preserves filtrations: it sends the filtration associated to $Y$ to the one associated to $X$ shifted by $n-m$. Therefore $i_!$ can also be seen as the limit of a morphism of spectral sequences which on the second page writes 
$$
i_{pq, !}^2:H_p(Y; H_q(\calf))\ri H_{p+n-m}(X; i^*H_q(\calf)).
$$
This is the usual shriek morphism with local coefficients in $H_q(\calf)$ induced by the inclusion $i:X\hookrightarrow Y$.

\end{remark}

\begin{remark}\label{lem:both-oriented} A particular situation of a co-oriented inclusion $i:X\hookrightarrow Y$ is  the one when $X$ and $Y$ are oriented. By convention we choose to orient the normal bundle  by 
\begin{equation}\label{eq:both-oriented}
\left(\ori\, X, \ori\, N\right) \, =\, \ori\, Y
\end{equation} \end{remark}

\rmk\label{Thom} Assume that $Y\ri X$ is a disk bundle and $X, Y$ are oriented manifolds. Then we may take $U=Y$ in the above construction and write the equality \eqref{eq:adhoc-Thom} as
$$C_*(X, \Xi ; i^*\calf) \, =\, C_{*+n-m}(Y, \Xi^Y ; \calf),$$
where $i:X\hookrightarrow Y$ is the inclusion. One may think of the  shriek map $$i_!: H_{*+n-m}(Y,\p Y; \calf)\cong H_*(X; i^*\calf)$$ induced by $i$,  as the Thom isomorphism for DG coefficients. 
\kmr

\section{Functoriality for embeddings}\label{sec:funct-embeddings}

An embedding $\varphi:X\ri Y$ is the composition between the diffeomorphism $\varphi_X=\varphi : X\ri \varphi(X)$ and the inclusion $i: \varphi(X)\hookrightarrow Y$. It suffices therefore to define the direct and shriek maps for diffeomorphisms and then apply the composition rules
$$\varphi_*\, =\, i_*\circ(\varphi_X)_*$$ and 
$$ \varphi_!\, =\, (\varphi_X)_!\circ i_!$$
in order to define the desired maps. The latter will be defined under the assumption that  $X$ and $Y$ are oriented.

\emph{1. The direct map for diffeomorphisms}. Let $\varphi : (X,\star)\ri (Y,\star) $ be a diffeomorphism. Given a set of data $\Xi^X$ for the enriched Morse complex on $X$, we can transform it through $\varphi$ into a set of data $\varphi_*(\Xi^X)$ on $Y$ by taking $f\circ\phi^{-1}$ as Morse function, $\varphi_*(\xi)$ as gradient, the orientations $\varphi(o)$, the tree $\varphi(\caly)$ etc.  For these choices we obviously have  $$m_{\varphi(x),\varphi(y)}^Y\, =\, \varphi_*(m_{x,y}^X)$$ and therefore $$\alpha\otimes x \mapsto \alpha\otimes \varphi(x) $$
defines a morphism 
$$
\varphi_*: C_*(X,\Xi^X; \varphi^*\calf)\ri C_*(Y,\varphi_*(\Xi^X); \calf)
$$ 
which, {\it by definition}, is the direct map associated to the diffeomorphism $\varphi$ for enriched complexes. Note that it clearly satisfies the compatibility condition~\eqref{welldefined-direct} and therefore it defines an induced map in homology 
$$
\varphi_*: H_*(X;\varphi^*\calf)\ri H_*(Y;\calf).
$$

\emph{2. The shriek map for diffeomorphisms}. Let $X$ and $Y$ be  oriented manifolds and $\varphi:X\ri Y$ a diffeormorphism.  The shriek map $$\varphi_!:H_*(Y;\calf)\ri H_*(X;\varphi^*\calf)$$ is defined by: $$
\mathrm{deg}(\varphi)\cdot\varphi^{-1}_*: H_*(Y;\calf)=H_*(Y;(\varphi^{-1})^*\varphi^*\calf)\ri H_*(X; \varphi^*\calf),
$$
i.e. $\pm\varphi^{-1}_*$ depending on whether $\varphi$ preserves orientations or not. At the level of complexes, with the following appropriate choice of data on $Y$, we have that
$$
\varphi_!:=\mathrm{deg}(\varphi)\cdot\varphi^{-1}_*: C_*(Y,\varphi_*(\Xi^X);\calf)\ri C_*(X,\Xi^X; \varphi^*\calf)
$$
writes 
$$\varphi_!(\alpha\otimes x)\, =\, \mathrm{deg}(\varphi)\cdot\alpha\otimes \varphi^{-1}(x).$$

 \rmk We could have gotten rid of the above sign $\mathrm{deg}(\varphi)$ if we chose the orientation of $Y$ to be the one induced by $\varphi$. However, when $\varphi:X\ri X$ is an orientation reversing diffeomorphism it is unnatural to choose  different orientations for the same underlying manifold $X$. 
\kmr

\rmk An orientation reversing diffeomorphism $\varphi:X\to X$ can be interpreted as an embedding 
between oriented manifolds in two ways: either as the composition $i\circ \varphi$ between the orientation preserving inclusion $i=\Id:X\hookrightarrow X$ and the orientation reversing diffeomorphism $\varphi: X\ri X$, or as the composition $i_X\circ\varphi_X$ between the orientation reversing inclusion $i_X=i:\ol X\hookrightarrow X$ and the orientation preserving diffeomorphism $\varphi_X=\varphi:X\to \ol X$. Here we fix an orientation of $X$ and denote $\ol X$ the manifold $X$ endowed with the opposite orientation. 

Our definition is consistent with these two descriptions: 
\begin{itemize}
\item Writing $\varphi=i\circ\varphi$, the shriek map $\varphi_!$ is defined by the composition 
$$
\xymatrix@C+1pc{ C_*(X,\Xi_X;\calf)\ar[r]^-{i_!=\Id}&C_*(X, \Xi_X;\calf)\ar[r]^-{\varphi_!=-\varphi_*^{-1}}& C_*(X,\varphi^{-1}_*(\Xi_X);\varphi^*\calf).}
$$
\item Writing $\varphi=i_X\circ\varphi_X$, the shriek map $\varphi_!$ is defined as the composition 
$$
\xymatrix@C+.6pc{ C_*(X,i_{X*}(\Xi_X);\calf)\ar[r]^-{i_{X!}=\Id}&C_*(\ol X, \Xi_X;\calf)\ar[r]^-{\varphi_{X!}=\varphi_*^{-1}}& C_*(X,\varphi^{-1}_*(\Xi_X);\varphi^*\calf).}
$$
\end{itemize}
That the two compositions are equivalent means that the diagram
$$
\xymatrix@C+.8pc{ C_*(X,i_{X*}(\Xi_X);\calf)\ar[r]^-{i_{X!}=\Id}&C_*(\ol X, \Xi_X;\calf)\ar[r]^-{\varphi_{X!}=\varphi_*^{-1}}& C_*(X,\varphi^{-1}_*(\Xi_X);\varphi^*\calf)\\
C_*(X,\Xi_X;\calf)\ar[u]_-{\Psi}\ar[rru]_-{-\varphi_*^{-1}}&\, &\, 
  }
$$
is commutative in homology, where $\Psi$ is the invariance morphism on $X$ between the data $\Xi_X$ and $i_{X*}(\Xi_X)$. Now, because the (rank $0$) normal bundle of $\ol X\subset X$ is oriented by the sign $``-"$, so that, by \eqref{eq:orientation-N}, the unstable manifolds have opposite orientations on the domain and on the target, we find that $i_{X*}(\Xi_X)=\Xi_X^{\mathrm{op}}$, where $\Xi_X^{\mathrm{op}}$ differs from $\Xi_X$ only by the orientation of the unstable manifolds, which is chosen to be opposite. 
Therefore $\Psi=\Psi^{\Id^{\mathrm{op}}}$, the continuation morphism associated to the identity map on $X$, where the data at the endpoints differs only by the orientations of the unstable manifolds, which is chosen to be opposite. It then follows from Proposition~\ref{Id=id} and Remark~ \ref{rmk:op2} that $\Psi^{\Id^{\mathrm{op}}}= -\Psi^{\Id}= -\Id$ in homology, i.e., the above diagram is commutative in homology. 
\kmr
\rmk\label{later}
The previous definition implies $\varphi_*\varphi_!=\mathrm{deg}(\varphi)\cdot \Id$ ($=\pm\Id$ since $\varphi:X\ri Y$ is a diffeomorphism). This statement will be generalized for arbitrary maps between oriented manifolds of the same dimension in Proposition~\ref{prop:phi*phi!}. 
\kmr
\section{Properties of $\varphi_*$ and $\varphi_!$ for embeddings}\label{sec:properties-functoriality-embeddings} 

In this section we check that the direct and shriek maps defined for embeddings satisfy conditions 1-4 from Theorem~\ref{thm:f*!}.

{\it 1. Identity.} It is obvious by definition that $\mathrm{Id}_*=\mathrm{Id}$ and that $\mathrm{Id}_!=\mathrm{Id}$. 

{\it 2. Composition.} This is straightforward too. Given two embeddings  $$X\stackrel{\varphi}\longrightarrow Y \stackrel{\psi}\longrightarrow Z,$$
the maps $(\psi\varphi)_*$ and $\psi_*\varphi_*$ take by definition the same value on the generators of the enriched complex of $X$: 
$$\alpha \otimes x \, \mapsto\, \alpha\otimes \psi(\varphi(x)).$$
 
 The same is true for the shriek maps (recall that we only defined them for $X$ and $Y$ oriented): if we choose the orientations induced by $\varphi$ and by $\psi\circ\varphi$ respectively on the submanifolds $\varphi(X)\subset Y$ and $\psi(Y)\subset Z$ then on the generators $\alpha\otimes x$ of the enriched complex of $Z$ the maps $(\psi\varphi)_! $ and $\varphi_!\psi_!$ are both defined as 
 $$\alpha\otimes  x \, \mapsto \, \left\{\begin{array}{ccc} \alpha\otimes\varphi^{-1}(\psi^{-1} (x)), & \mathrm{if}& x\in \psi(\varphi(X)),\\
 \\
 0,&\mathrm{if}& x\not\in \psi(\varphi(X)).\end{array}\right.$$
 
 \noindent{\it 3. Homotopy.}   Let $(\varphi_t)$ be a homotopy   between two embeddings $\varphi_0,\varphi_1: X\ri Y$ which satisfy
 $$\varphi_{i}(\star_X)\, =\, \star_Y.$$ The meaning of the homotopy property was explained in~\S\ref{sec:comments-homotopy}. Denote as in~\S\ref{sec:comments-homotopy} by  $\varphi : [0,1]\times X\ri Y$ the map given by this homotopy and extend it to $[-\epsilon, 1+\epsilon] \times X$ by $\varphi_0$ for $t\leq 0$ and by $\varphi_1$ for $t\geq 1$. We may assume w.l.o.g. that $\varphi$ is smooth. The map $$\Phi: [-\epsilon, 1+\epsilon] \times X \ri [-\epsilon, 1+\epsilon] \times Y$$ defined by $\Phi(t,x)=(t,\varphi_t(x))$ is an embedding and fits into the following commutative diagrams
 $$\xymatrix
 @C=10pt
{
X \ar[d]^-{i_0} \ar[r]^-{\varphi_0} & Y\ar[d]^-{i_0 }\\
[-\epsilon,1+\epsilon]\times X\ar[r]^-{\Phi}&[-\epsilon, 1+\epsilon]\times Y, }
\  \xymatrix 
@C=10pt
{
X \ar[d]^-{i_1} \ar[r]^-{\varphi_1} & Y\ar[d]^-{i_1 }\\
[-\epsilon,1+\epsilon]\times X\ar[r]^-{\Phi}&[-\epsilon, 1+\epsilon]\times Y }
$$ 
where $i_0(x)= (0,x)$ and $i_1(x)= (1,x)$ are the inclusions. This means that if we are able to prove that $i_{0,*}= i_{1,*}$ as well as $i_{0,!}= i_{1,!}$, and moreover that all these maps are quasi-isomorphisms, 
then the composition property we have just proved above implies the claimed homotopy property. 
 We need  to be more precise here since $i_0$ and $i_1$ cannot preserve a basepoint simultaneously.   Starting from the  DG-module $\cF$ over $C_*(\Omega_{\star_Y}Y)$ and denoting  $\pi:[-\epsilon, 1+\epsilon] \times Y \ri Y$ the projection,  the maps on the right hand side of the above diagrams are 
$$ i_{j,*} : H_*(Y;\calf)= H_*(Y;i_j^*\pi_j^*\calf)\ri  H_* ([-\epsilon, 1+\epsilon]\times Y; \pi_j^*\calf)$$ 
and 
$$ i_{j,!} : H_*([-\epsilon, 1+\epsilon] \times Y;\pi_j^*\calf)\ri   H_{*-1} (Y; i_j^*\pi_j^*\calf)= H_{*-1}(Y;\calf)$$
for $j=0,1$,  where we denoted as in~\S\ref{sec:identification-basepoints}  by $\pi_j: \Omega_{(j\,,\,\star_Y)} [-\epsilon,1+\epsilon]\times Y\ri \Omega_{\star_Y}Y$ the maps induced by $\pi$ on the loop spaces. As in \emph{loc. cit.} we identify the twisted complexes based at $i_0(x)= (0,x)$ and $i_{1}(x)=(1,x)$ in order to give a meaning to the equalities between the two direct maps and two shriek maps above. We will make this more precise in the following lemma which proves the equalities between the direct and shriek maps induced by $i_0$ and $i_1$ in a slightly more general setting. The discussion about the equalities $i_{0,*}=i_{1,*}$ and $i_{0,!}= i_{1,!}$ for $i_j:X\ri [-\epsilon, 1+\epsilon]\times X$ is analogous, except for the fact that we replace $\pi_j^*\calf$ by $\Phi^*\pi_j^* \calf$ 
 and therefore we have 
 $$ i_{j,*} : H_*(X;\varphi_j^*\calf)= H_*(X;i_j^*\Phi^*\pi_j^*\calf)\ri  H_* ([-\epsilon, 1+\epsilon]\times X; \Phi^*\pi_j^*\calf)$$ 
and 
$$ i_{j,!} : H_*([-\epsilon, 1+\epsilon] \times X;\Phi^*\pi_j^*\calf)\ri   H_{*-1} (X; i_j^*\Phi^*\pi_j^*\calf)= H_{*-1}(X;\varphi_j^*\calf).$$
 We also need to identify the twisted complexes on $X$ corresponding to the two different DG-modules $\varphi_j^*\calf= i_j^*\Phi^*\pi_j^*\calf$ (both over $C_*(\Omega_{\star_X}X)$). This is done using the homotopy identification from~\S\ref{sec:comments-homotopy} via the homotopy $(\varphi_t)$ or, alternatively, using the homotopy $i_t(x)= (t,x)$ combined with the identification of basepoints from~\S\ref{sec:identification-basepoints} with $\eta=\pi\circ\Phi$. This will also be explained in the proof of the following lemma.

 \begin{lemma}\label{ia=ib} Let $X^n$ be a closed manifold with a basepoint $\star_X$ and denote $D^m$ the closed $m$-disk. 
 
 (i) Given   $a \in \mathring{D}$ we denote $i_a(x)=(a,x)$ the corresponding inclusion of $X$ into $D\times X$.  The maps $i_{a,*}$ and $i_{a!}$ are quasi-isomorphisms.
 
 (ii)  Given $a, b\in \mathring{D}$, let $\eta : D\times X\ri Y$ be a continuous map with values in some based topological space $Y$ such that $$\eta(a,\star_X)=\eta(b,\star_X)=\star_Y.$$ 
 Then $(i_a)_*= (i_b)_*$ and $(i_a)_!=(i_b)_!$.
\end{lemma} 

\begin{proof} \qquad 

(i) We give explicit formulas for $i_{a,*}$ and $i_{a,!}$.  Let $\Xi = (f, \xi, o, s_{x,y}, \caly, \theta)$ be a set of data on $X$.  Take a  function $h: D\ri \R$ with a unique minimum at $a$. Consider the function $(f+h): D\times X\ri \R$ seen as an extension of $f:i_a(X)\ri \R$. According to~\S\ref{sec:funct-closed-submanifolds}, the direct map is defined by completing $f+h$ to an appropriate set of data $\Xi_a$ on $D\times X$ with basepoint $(a,\star_X)$ and taking a DG-module $\calf$ over the loop space of $D\times X$ based at this point. Then $i_{a,*} : C_*(X,\Xi; i_a^*\cF)\ri C_*(D\times X, \Xi_a; \cF)$ is given by the formula  
$$i_{a,*}(\alpha\otimes x)\, =\, \alpha\otimes (a,x).$$
This is obviously a bijection between the two complexes. 

For the shriek map the proof is analogous except for the fact that we take $h$ to have a unique maximum at $a$, we choose an appropriate set of data $\Xi_a$ on $D\times X$ with $f+h$ as Morse function as in \S\ref{sec:funct-closed-submanifolds}  and we get the bijective map $i_{a,!} : C_*(D\times X, \Xi_a;\cF)\ri C_{*-m}(X,\Theta; i_a^*\cF)$ defined by 
$$i_{a,!}(\alpha\otimes (a,x))\, =\, \alpha\otimes x.$$

(ii)  The maps $i_a$ and $i_b$ are homotopic through $ i(t,x)= (\gamma(t),x)$, where $\gamma$ is a path that connects $a$ and $b$ in $D$, and any two homotopies of this type are homotopic with fixed endpoints.  The  claimed equalities of  maps should be understood after, on the one hand the homotopy  identification from~\S\ref{sec:comments-homotopy}, and on the other hand the identification for twisted complexes with different basepoints described in~\S\ref{sec:identification-basepoints}. More precisely, the proof consists  in  checking the commutativity of the following diagrams: 

$$
\xymatrix{
H_*(X ; i_a^*\eta_a^* \cF) \ar[d]_{H_{ab}} \ar[r]^-{i_{a,*}} & H_*(D\times X;\eta_a^*\cF)\ar[d]_{I_{ab} }\\
H_*(X, i_b^*\eta_b^*\cF)\ar[r]^-{i_{b,*}}&H_*(D\times X; \eta_b^*\cF)  }
$$ 
for the direct maps, and 
$$
\xymatrix{
H_{*}(D\times X ; \eta_a^* \cF) \ar[d]_{H_{ab}} \ar[r]^-{i_{a,!}} & H_{*-m}(X;i_a^*\eta_a^*\cF)\ar[d]_{I_{ab} }\\
H_{*}(D\times X ;  \eta_b^*\cF)\ar[r]^-{i_{b,!}}&H_{*-m}(X; i_b^*\eta_b^*\cF)  }
$$ 
for shriek maps. Here $H_{ab}$ represents the homotopy identification and $I_{ab}$ the identification for different basepoints. Following~\S\ref{sec:identification-basepoints} we denoted $\eta_a:\Omega_{a,\star_X}(D\times X)\ri \Omega_{\star_Y} Y$ and $\eta_b:\Omega_{b,\star_X}(D\times X)\ri \Omega_{\star_Y} Y$ the maps defined by $\eta$ on the respective loop spaces. 

We begin with the diagram for direct maps. Let $\Xi = (f, \xi,o, s_{x,y}, \caly, \theta)$ be a set of data on $X$. As above we take a function $h: D\ri \R$ with a unique minimum in $a$ and the function $(f+h): D\times X\ri \R$ seen as an extension of $f:i_a(X)\ri \R$ completed to an appropriate set of data $\Xi_a$ on $D\times X$ with basepoint $(a,\star_X)$.  We thus get that $i_{a,*} : C_*(X,\Theta; i_a^*\eta_a^*\cF)\ri C_*(D\times X, \Xi_a; \eta_a^*\cF)$ is defined by 
$$i_{a}(\alpha\otimes x)\, =\, \alpha\otimes (a,x)$$
for any $\alpha\in \cF$ and $x\in \Crit(f)$. 

Consider now an isotopy $(\phi_t)$ on $D$ such that $\phi=\phi_1$ satisfies $\phi(a)=b$, and consider the function $f+h\circ\phi^{-1}: D\times X\ri \R$ seen as an extension of $f:i_b(X)\ri \R$. Modifying the data $\Xi_a$ by $\phi\times \id_X$ yields a set of data $$\Xi_b\, =\, \Xi_a^{\id_X\times \phi}$$ associated to the basepoint $(b, \star_X)$
which completes $f+h\circ\phi^{-1}$. As above, we get that the direct morphism $i_{b,*} : C_*(X,\Theta; i_b^*\eta_b^*\cF)\ri C_*(D\times X, \Xi_b; \eta_b^*\cF)$ is given by  
$$i_{b}(\alpha\otimes x)\, =\, \alpha\otimes (b,x).$$
We check the commutativity of the first diagram at the level of complexes: 
$$
\xymatrix{
C_*(X ,\Xi; i_a^*\eta_a^* \cF) \ar[d]_{H_{ab}} \ar[r]^-{i_{a,*}} & C_*(D\times X,\Xi_a;\eta_a^*\cF)\ar[d]_{I_{ab} }\\
C_*(X, \Xi; i_b^*\eta_b^*\cF)\ar[r]^-{i_{b,*}}&C_*(D\times X, \Xi_b; \eta_b^*\cF)  }
$$
Let us describe explicitly $H_{ab}$ and $I_{ab}$. To define $H_{ab}$ the recipe given in~\S\ref{sec:comments-homotopy} goes as follows: take the trivial continuation data $\Id$ between $\Xi$ and $\Xi$ and define as in~\S\ref{sec:identification-isotopy} 
\begin{equation}\label{eq:useful-for-continuous}\overline{\nu}_{x,y}\, =\, - (\Theta\circ p\circ \Gamma^\Id)_*(s^\Id_{x,y}) \, \in  \, C_{|x|-|y|}(\calp_{(0,\star_X)\ri (1,\star_X)}[0,1]\times X),\end{equation}
then use the homotopy between $i_a$ and $i_b$ as a map $[0,1]\times X\ri D\times X$ together with $\eta:D\times X\ri Y$ to transform this chain into one in $\Omega Y$. Note that we may use the homotopy $$ \left[(\phi\times\id_X)\circ(\id_{[0,1]}\times i_a)\right](t,x)\, = \, (\phi_t(a), x),$$
where $\phi:[0,1]\times D\ri D$ is the isotopy on $D$ that we considered above.  We therefore get the cocycle
$$\nu_{x,y} \, =\, \left[\eta\circ(\phi\times\id_X)\circ(\id_{[0,1]}\times i_a)\right]_*\overline{\nu}_{x,y}\, \in \, C_{|x|-|y|}(\Omega Y)$$
for any $x,y\in \Crit(f)$ and by definition, we have that  
$$H_{ab}(\alpha\otimes x)\, =\, \sum_y \alpha\nu_{x,y}\otimes y$$
 for $\alpha \in \calf$ and $x\in \Crit(f)$. 
 
Let us now give a formula for  $I_{ab}$. We follow the procedure of~\S\ref{sec:identification-basepoints}. We need a continuation cocycle between a set of data based at $(a,\star_X)$ and one based at $(b, \star_X)$: we pick $\Xi_a$ and $\Xi_b= \Xi_{a}^{\id_X\times\phi}$ as above.  Since these are related through the diffeomorphism $\id_X\times \phi$ (which is isotopic to the identity), the discussion from 
from~\S\ref{sec:identification-isotopy} provides a definition for this cocycle. We start by a continuation procedure using the constant homotopy between $\Xi_a$ and itself. All the critical points of $f+h$ are of the form $(a,x)$ for $x\in \Crit(f)$. As above we  first get a chain 
$$ \overline{\nu}_{(a,x),(a,y)}\, \in\, C_{|x|-|y|}\left(\calp_{(0,a,\star_X)\ri(1,a,\star_X)}[0,1]\times D\times X\right).$$
Since $a$ is the unique critical point of $h$ and we took the constant homotopy for the continuation, the gradient lines from the slice $\{0\}\times D\times X$ to $\{1\}\times D\times X$ which define $\overline{\nu}_{(a,x),(a,y)}$ are constant equal to $a$ on the $D$ component and so are the paths which we obtain after evaluating them through the composition $\Theta\circ p\circ \Gamma$. One can easily check that actually 
$$ \overline{\nu}_{(a,x),(a,y)}\, =\, (\id_{[0,1]}\times i_{a})_*( \overline{\nu}_{x,y}),$$
where $\overline{\nu}_{x,y}$ is the cocycle from the description of $H_{ab}$. Now we use the discussion from~\S\ref{sec:identification-isotopy} 
to see that, according to~\eqref{eq:identification-isotopy}, we have 
$$ \overline{\nu}_{(a,x),(b,y)} \, =\, (\Phi\times \id_X)_*( \overline{\nu}_{(a,x),(a,y)}),$$
where $\Phi: [0,1]\times D\ri [0,1]\times D$ is defined by 
$$\Phi(t,r)\, =\, (t, \phi_t(r)).$$
Then, following the recipe from~\S\ref{sec:identification-basepoints}, we get first the identification cocycle between $\Xi_a$ and $\Xi_b$: 
$$\nu_{(a,x), (b,y)}  = (\pi_{D\times X})_* \overline{\nu}_{(a,x),(b,y)} = (\phi\times \id_X)_*(\overline{\nu}_{(a,x),(a,y)}),$$
which combined with the above yields 
\begin{align*}
\nu_{(a,x), (b,y)}  = (\phi\times \id_X)_* & (\id_{[0,1]}\times i_{a})_*( \overline{\nu}_{x,y}) \\
& \in C_{|x|-|y|}(\calp_{(a,\star_X)\ri(b,\star_X)}(D\times X)).
\end{align*}
Finally we use $\eta$ to transform this chain into one on $\Omega Y$. With the notation of~\S\ref{sec:identification-basepoints} we get 
\begin{align*}
\mu_{(a,x), (b,y)} & = \eta_*(\nu_{(a,x), (b,y)})  \\
& = \eta_*(\phi\times \id_X)_*(\id_{[0,1]}\times i_{a})_*( \overline{\nu}_{x,y})\, \in\, C_{|x|-|y|}(\Omega Y),
\end{align*}
and therefore the basepoint identification morphism $I_{ab}$ is given by 
$$I_{ab}(\alpha\otimes(a,x))\, =\, \sum_{y\in \Crit(f)}\alpha\mu_{(a,x),(b,y)}\otimes (b,x).$$
According to the formulas above \begin{equation}\label{eq:mu=nu}\mu_{(a,x), (b,y)}\, =\, \nu_{x,y},\end{equation}
and therefore the diagram for the direct maps commutes at the level of complexes. This implies that $i_{a,*}=i_{b,*}$.

For the equality of the shriek maps $i_{a,!}=i_{b,!}$ the proof is quite similar. The only difference is that we choose the function $h:D\ri \R$ to have a unique maximum at $a$. We thus have 
$$i_{a,!}(\alpha \otimes (a,x))\, =\, \alpha\otimes x$$
and 
$$i_{b,!}(\alpha \otimes (b,x))\, =\, \alpha\otimes x.$$
The formulas of $H_{ab}$ and $I_{ab}$ are the same and therefore using~\eqref{eq:mu=nu} the diagram 
$$
\xymatrix{
C_*(D\times X, \Xi_a ; \eta^*_a \cF) \ar[d]_{H_{ab}} \ar[r]^-{i_{a,!}} & C_{*-m}(X, \Xi;i_a^*\eta^*_a\cF)\ar[d]_{I_{ab} }\\
C_*(D\times X, \Xi_b ; \eta^*_b\cF)\ar[r]^-{i_{b,!}}&C_{*-m}(X, \Xi; i_b^*\eta^*_b\cF)  }
$$ 
commutes, which leads to the desired conclusion. 
\end{proof}

The homotopy property for embeddings is now established. \\
\\
{\it 4. Spectral sequence.} We have already proved in Remarks~\ref{rem:spectral-direct} and~\ref{rem:spectral-shriek} that this property is satisfied for submanifolds. The generalization to the case of embeddings is straightforward. 

\section{Functoriality for  general smooth maps}\label{sec:funct-direct-general} 

We define in this section direct and shriek maps associated to a smooth map $\varphi : (X,\star_X)\ri (Y,\star_Y)$  between based compact manifolds and a DG-module $\calf$ over $C_*(\Omega_{\star_Y}Y)$. 

\emph{1. The direct map $\varphi_*: H_*(X;\varphi^*\calf)\ri H_*(Y;\calf)$}. We consider an embedding $\chi: X\ri D^m$ into a $m$-disk such that $\chi(\star_X)=0$. Obviously $\varphi^\chi= (\chi,\varphi): X\ri D^m\times Y$ is also an embedding. Denote $D=D^m$, denote $i= i_0:Y\hookrightarrow D\times Y$ the inclusion $y\mapsto(0,y)$, and denote $\pi: D\times Y\ri Y$ the projection. The maps $$i_{*}: H_*(Y;\calf)= H_*(Y; i^*\pi^*\calf)\ri H_*(D\times Y;\pi^*\calf)$$
and $$\varphi^\chi_*: H_*(X;\varphi^*\calf)=H_*(X; (\varphi^\chi)^*\pi^*\calf)\ri H_*(D\times Y; \pi^*\calf)$$ are well defined and the former is an isomorphism by Lemma~\ref{ia=ib}. We define the direct map $\varphi_*$ by the formula \begin{equation}\label{eq:direct-map-general} 
\varphi_*\, =\, (i_{*})^{-1}\circ\varphi^\chi_*,
\end{equation}
and we check that this definition is independent of $\chi$. 

\begin{lemma}\label{chi=rho}
Let $\chi: X\ri D^m$ and $\rho:X\ri D^{m'}$ be two embeddings which map the basepoint $\star_X$ to the centers of the  respective disks. Denote $D'= D^{m'}$ and denote $j=i_0: Y\hookrightarrow D'\times Y$ the inclusion of  $Y\times \{0\}$.   Then 
$$(i_{*})^{-1}\circ\varphi^\chi_*\, =\, (j_{*})^{-1}\circ\varphi^{\rho}_*.$$
\end{lemma}

\begin{proof}
Denote by $\varphi_*$ the left-hand side and by $\overline{\varphi}_*$ the right-hand side of the equality above; we thus claim that $\varphi_*=\overline{\varphi}_*$. Also denote by $\varphi^{\chi,\rho}: X\ri D\times D'\times Y$ the map $x\mapsto (\chi(x),\rho(x), \varphi(x))$.  Consider the inclusions $k : D\times Y\hookrightarrow D\times D'\times X$ defined by $(r, y) \mapsto (r,0,y)$ and $l:D'\times Y\ri D\times D'\times Y$ defined by $(r',y)\mapsto (0,r',y)$. We consider the following diagram, in which we use  the  notation $\pi_{(\cdot)}$ for the different  projections $(\cdot)\times Y\ri Y$: 
$$
\xymatrix{\, & H_*(D\times Y; \pi_D^*\calf)\ar[dr]^{k_*}  & \, \\
H_*(Y;\calf)\ar[ur]^-{i_*} \ar[dr]_-{j_*}&H_*(X; \varphi^*\calf) \ar[u]^{\varphi^\chi_*}\ar[d]_(.4){\varphi^{\rho}_*}\ar@<.5ex>[l]^{\overline{\varphi}_*}\ar[r]^(0.4){\varphi^{\chi,\rho}_*}\ar@<-.5ex>[l]_{\varphi_{*}}& H_*(D\times D'\times Y;\pi_{D\times D'}^*\calf) \\
\, &H_*(D'\times Y;\pi_{D'}^*\calf)\ar[ur]_{l_*}&\, }
$$ 

Note that the upper left and lower left triangles are commutative by the definitions of $\varphi_*$ and $\overline{\varphi}_*$. Also remark that, by the composition property and the homotopy invariance, the two triangles on the right side also commute. Indeed, the embeddings $k\circ \varphi^\chi(x)=(\chi(x), 0,\varphi(x))$ and $\varphi^{\chi,\rho}(x)= (\chi(x), \rho(x), \varphi(x))$ are homotopic through embeddings by $$H_{t}(x)\, =\, \left[(\chi(x), t\rho(x), \varphi(x))\right]_{t\in [0,1]}.$$ 
Note that  $H_{t}^*\pi^*_{D\times D'} \calf= \varphi^*\calf$ is constant and therefore the identification from~\S\ref{sec:comments-homotopy} is not needed. A similar argument works for the lower right triangle. 

From all this we infer that 
$$k_*\circ i_*\circ \varphi_*\, =\, k_*\circ \varphi^{\chi}_*\, =\, \varphi^{\chi,\rho}_*\, =\, l_*\circ \varphi^\rho_*\, =\, l_*\circ j_*\circ \overline{\varphi}_*.$$ On the other hand  we obviously have  $k\circ i=l\circ j$ and by Lemma~\ref{ia=ib}, these maps are isomorphisms in homology. This implies the desired conclusion $\varphi_*=\overline{\varphi}_*$. 
\end{proof} 

We also need to verify that the definition~\eqref{eq:direct-map-general}  matches  with the one that we have already given  in the case where $\varphi$ is an embedding. 

\begin{lemma}\label{verif-embedding} Let $\varphi:(X,\star_X)\ri (Y, \star_Y)$  and $\chi:(X, \star_X)\ri (D,0)$ be embeddings. Then, with the above notation, we have 
$$\varphi_*\, =\,  (i_*)^{-1}\varphi^\chi_*,$$ 
where $\varphi_*$ is defined as in~\S\ref{sec:funct-embeddings}.
\end{lemma} 

\begin{proof} We have a homotopy of embeddings $(X,\star_X)\ri (D\times X, (0,\star_X))$ between $i\circ\varphi$ and $\varphi^\chi$ given by 
$$h_t(x)\, =\, (t\chi(x), \varphi(x)).$$

By the homotopy property we get that $i_*\circ\varphi_*\, =\, \varphi^\chi_*$. Note that, because $h_t^*\pi_D^*\calf= \varphi^*\calf$ is constant, the homotopy identification from~\S\ref{sec:comments-homotopy} for $(h_t)_*: H_*(X;\varphi^*\calf)\ri H_*(D\times X; \pi_D^*\calf)$ is not needed. This completes the proof of the lemma. 
\end{proof} 

Finally, for $\varphi_*$ to be well defined we need to prove  that its  definition satisfies the compatibility with the continuation maps~\eqref{welldefined-direct}. For this purpose we give a description of $\varphi_*$ at the level of twisted complexes. We start with sets of data $\Xi^X$  and $\Xi^Y$ on $X$,  respectively $Y$. We complete $\Xi^X$  to a set of data on $D\times Y$ adapted to the embedding $\varphi^\chi=(\chi,\varphi)$ as in~\S\ref{sec:funct-closed-submanifolds} and~\S\ref{sec:funct-embeddings}: we first extend to a tubular neighborhood $U$ of $\varphi^\chi(X)$ (with gradient pointing inwards), and then to the whole of $D\times Y$. Denote by $\Xi^{D\times Y}_{(\chi,\varphi)}$ this set of data. If $f$ is the Morse function on $X$, then for any $x\in \crit(f)$ and $\alpha \in \calf$ we have 
\begin{equation}\label{eq:phi-chi-direct} \varphi^\chi_*(\alpha \otimes x) \, =\, \alpha\otimes \varphi^\chi(x)\end{equation} 
by the definition of 
$$ 
\varphi^\chi_*: C_*(X,\Xi^X;\varphi^*\calf={\varphi^\chi}^* \pi^*\calf)\ri C_*(D\times Y, \Xi^{D\times Y}_{(\chi,\varphi)}; \pi^*\calf).
$$

For the inclusion $i: Y\ri D\times Y$, $y\mapsto (0,y)$, the map $i_*: C_*(Y, \Xi^Y;\calf)\ri C_*(D\times Y, \Xi^D\times \Xi^Y; \pi^*\calf)$ is defined by 
$$i_*(\alpha \otimes y)\, =\, \alpha \otimes (0,y),$$ 
where $y$ is critical point of the Morse function in $\Xi^Y$ and $\alpha \in \calf$. Here $\Xi^D=(h, -\nabla h, o_{\{0\}}=+, \caly=\{0\}, \theta=\id)$ is the set of Morse data on the disc with Morse function $h(x)=\|x\|^2$, and $\Xi^D\times \Xi^Y$ is the set of data obtained canonically on $D\times Y$ from  $\Xi^D$ and $\Xi^Y$. We have already observed in Lemma~\ref{ia=ib} that $i_*$ is an isomorphism at the level of complexes. However, in order to write the composition $(i_*)^{-1}\circ \varphi^\chi_*$ at the level of complexes we need to identify the twisted complexes on $D\times Y$ corresponding to the sets of data $\Xi^{D\times Y}_{\chi,\varphi}$ and $\Xi^{D}\times \Xi^Y$. Denote by $\Psi^{D\times Y} : C_*(D\times Y, \Xi_{(\chi,\varphi)}; \pi^*\calf)\ri C_*(D\times Y, \Xi^D\times \Xi^Y;\pi^*\calf)$ a continuation morphism. Then a formula for $\varphi_*: C_*(X, \Xi^X; \varphi_*\calf)\ri C_*(Y,\Xi^Y;\calf)$ is 
\begin{equation}\label{eq:formula-direct-complexes} 
\varphi_* \, =\,( i_*)^{-1}\circ\Psi^{D\times Y}\circ\varphi^\chi_*.
\end{equation} 

The compatibility diagram~\eqref{welldefined-direct} writes here as follows, with the notation from above and from~\eqref{welldefined-direct}:
$$
{\scriptsize 
 \xymatrix{
     C_*(X, \Xi^X_0; \varphi^*\calf)\ar@/^2pc/[rrr]_{\varphi_*} \ar[r]_-{\varphi^\chi_*}  \ar[d]^{\Psi_{01}^X} &C_*(D\times Y, \Xi^{D\times Y}_{\chi,\varphi,0}; \pi^*\calf)\ar[r]_-{\Psi^{D\times Y}}  \ar[d]^{\Psi_{01}^{D\times Y}}&C_*(D\times Y, \Xi^D\times\Xi^Y_0; \pi^*\calf)\ar[r]_-{(i_*)^{-1}}  \ar[d]^{\Psi_{01}^{D\times Y}}&C_*(Y, \Xi_0^Y; \calf)\ar[d]^{\Psi^Y_{01}}\\
     C_*(X, \Xi_1^X; \varphi^*\calf)\ar@/^-2pc/[rrr]^{\varphi_*} \ar[r]^-{\varphi^\chi_*} & C_*(D\times Y, \Xi^{D\times Y}_{\chi,\varphi,1}; \pi^*\calf) \ar[r]^-{\Psi^{D\times Y}}&C_*(X, \Xi^D\times\Xi_1^Y; \pi^*\calf) \ar[r]^-{(i_*)^{-1}}&C_*(Y, \Xi_1^Y; \calf) }
     }
     $$

This diagram is commutative: the left and right rectangles are commutative in homology since the compatibility with the continuation maps was already checked  in~\S\ref{sec:funct-embeddings} for $\varphi^\chi_*$ and $i_*$, which are embeddings. The middle rectangle is also commutative since it involves continuation maps and we proved in~\S\ref{sec:invariance-continuation}  that these only depend in homology on the domain and on the target. 

This finishes the verification of the fact that the direct map $\varphi_*$ is well defined. 

\medskip 

\emph{2. The shriek map $\varphi_!: H_*(Y;\calf)\ri H_{*+\mathrm{dim}(X)-\mathrm{dim}(Y)}(X;\varphi^*\calf)$}. We proceed in an analogous way under the assumption that $X$ and $Y$ are {\it oriented} manifolds.  By definition 
\begin{equation}\label{eq:shriek-general} 
\varphi_!\, =\, \varphi^\chi_!\circ (i_!)^{-1}.
\end{equation} 
 Note that,  according to~\S\ref{sec:funct-embeddings},  the assumption that is needed in order for the right-hand side to be  well defined is that the embedding $\varphi^\chi:X\ri D\times Y$ is co-oriented (the inclusion $i:Y\ri D\times Y$ is obviously cooriented). This is the reason why we supposed that $X$ and $Y$ are oriented. 
 
 The proofs of the shriek versions of Lemmas~\ref{chi=rho} and~\ref{verif-embedding}, as well as the verification of the compatibility with the continuation maps~\eqref{welldefined-shriek}, are completely analogous to the case of direct maps. We omit the details and conclude that the map $\varphi_!$ is well defined. 
 
 \section[The maps $\varphi_*$ and $\varphi_!$ in the general smooth case]{Properties of $\varphi_*$ and $\varphi_!$ in the general smooth case} 
 
 In this section we check that the direct and the shriek maps defined for general smooth maps satisfy conditions 1-4 from~\S\ref{sec:functoriality-properties}.

 \noindent {\it 1. Identity.} This is obvious and was already remarked in~\S\ref{sec:properties-functoriality-embeddings}. 
 
 {\it 2. Composition.} We only treat the case of direct maps, as the argument for shriek maps is analogous. 
Let $\varphi:(X,\star_X)\ri (Y,\star_Y)$ and $\psi: (Y, \star_Y)\ri (Z, \star_Z)$ be smooth maps. Let $\calf $ be a DG-module over $C_*(\Omega_{\star_Z})$.   In order to define $\varphi_*$ and $\psi_*$ we choose as in the previous section embeddings $\chi: (X,\star_X)\ri(D,0)$ and $\rho: (Y, \star_Y)\ri (D',0)$. Consider then  the following diagram of direct maps in twisted homology:  
 
$$
\xymatrix
 @C=15pt
{\, & \, &H_*(D\times D'\times Z;\pi_{D\times D'}^*\calf)  \\
\, &H_*(D\times Y;\pi_D^*\psi^*\calf)\ar[ur]^-(0.4){(\Id\times\psi^\rho)_*}&H_*(D'\times Z;\pi_{D'}^*\calf)\ar[u]^-{k_*}\\
H_*(X; \psi^*\phi^*\calf)\ar[ur]^-{\varphi^\chi_*}\ar[r]^-{\varphi_*}\ar@/^3pc/[urur]^{(\psi\circ\varphi)^{\chi,\rho}_*}&H_*(Y;\psi^*\calf)\ar[u]^-{i_*}\ar[ur]^-{\psi^\rho_*}\ar[r]^-{\psi_*}&H_*(Z;\calf)\ar[u]^-{j_*}\ar@/_5pc/ [uu]_{l_*}
}$$ 
Most of the notations are similar to those of the preceding section: the maps $i$, $j$, $k$ and $l$ are inclusions defined by $x\mapsto (0,x)$; the maps $\pi_{(\cdot)}$ are projections  $(\cdot)\times Y\ri Y$ or $(\cdot)\times Z\ri Z$. As above we denote $\varphi^\chi(x)= (\chi(x), \varphi(x))$,  then $\psi^\rho(y)=(\rho(y),\psi(y))$, and finally $$(\psi\circ\varphi)^{(\chi,\rho)}(x)=(\chi(x), \rho(x), \psi(\varphi(x)))$$

This diagram is commutative. Indeed, the two lower triangles are commutative by definition of the direct maps $\varphi_*$ and $\psi_*$. The other sub-diagrams only contain embeddings and they are commutative at the level of maps (as the reader may immediately check), and therefore also commutative in homology using the composition property for embeddings. Let us now check the claimed formula $(\psi\circ\varphi)_*=\psi_*\circ\varphi_*$. By commutativity of the diagram we have
$$
\psi_*\circ\varphi_*\, =\, j^{-1}_*\circ \psi^\rho_*\circ i^{-1}_*\circ \varphi^\chi_*\, =\,  j^{-1}_*\circ k^{-1}_*\circ(\Id\times\psi^\rho)_*\circ \varphi^\chi_*\, =\, l^{-1}_*\circ(\psi\circ\varphi)_*^{(\chi,\rho)}.
$$
The last expression is exactly the definition of $(\psi\circ\varphi)_*$ corresponding to the embedding $(\chi,\rho):X\ri D\times D'$. Note that the product of $D=D^m$ and $D'=D^{m'}$ is not exactly  the disc $D^{m+m'}$ as required by the definition of $(\psi\circ\varphi)_*$, but the two are homeomorphic with diffeomorphic interiors and this is sufficient. (We could also have defined direct maps by replacing $D^m$ with $[-1,1]^m$ endowed with a vector field which points outwards.) The composition property for direct maps is therefore established.

To prove the same property for shriek maps we proceed in a similar manner by changing the direction of the arrows in the diagram above. The details are straighforward and left to the reader. 

{\it 3. Homotopy.} Using the work we have done so far, the homotopy property is quite easy to establish. Take a homotopy $(\varphi_t)$ between two smooth maps $\varphi_0, \varphi_1: (X, \star_X)\ri (Y, \star_Y)$. Take the map $\varphi : [-\epsilon, 1+\epsilon]\times X\ri Y$ obtained from this homotopy by extending it with $\varphi_0$ for $t\in [-\epsilon, 0)$ and with $\varphi_1$ for $t\in (1, 1+\epsilon]$. Then, using the notation of~\S\ref{sec:properties-functoriality-embeddings}, we have $\varphi_0= \varphi\circ i_0$ and $\varphi_1= \varphi\circ i_1$. This implies the homotopy property as a consequence of the composition property that we have just established  and  of the equalities $i_{0,*}=i_{1,*}$ and $i_{0,!}=i_{1,!}$, which we proved in Lemma~\ref{ia=ib}, item (ii)  (in our case the map $\eta$ from that lemma is $\varphi$).

{\it 4. Spectral sequence.} Consider a smooth map $\varphi: (X, \star_X)\ri (Y,\star_Y)$ together with a DG-module $\calf$ over $C_*(\Omega_{\star_Y}Y)$ and then take an embedding $\chi : (X,\star_X)\ri (D,0)$  defining an embedding $\varphi^\chi: X\ri D\times Y$ as above. By definition we have $\varphi_*\, =\, i^{-1}_*\circ \varphi^\chi_*$ and $\varphi_!= \varphi^\chi_! \circ i^{-1}_!$. Now, both the direct and shriek maps of $i^{-1}$ and $\varphi^\chi$ are defined on the level of (twisted) complexes and they preserve the filtrations of these, yielding thus morphisms between the corresponding spectral sequences. At the second page $E_{p,q}^2$ these morphisms are given by 
$$(i_*)^{-1}\circ \varphi_*: H_p(X; \varphi^*H_q(\calf))\ri H_p(Y; H_q(\calf))$$
and $$\varphi^\chi_!\circ( i_!)^{-1}: H_p(Y; H_q(\calf))\ri H_{p+\mathrm{dim}(X)-\mathrm{dim}(Y)}(X; \varphi^*H_q(\calf)).$$
These morphisms are the Morse versions of the direct and shriek maps between homologies with local coefficients defined by some smooth map $\varphi:X\ri Y$. This is done for example in~\cite[\S4.6]{Audin-Damian_English} for direct maps in homology  with integer coefficients;  the generalization to local coefficients, as well as to shriek maps, is immediate. 

\section{Functoriality for continuous maps}\label{sec:funct-for-continuous}

Any continuous map $\varphi: (X, \star_X) \ri (Y,\star_Y)$ can be approximated by a smooth map $\overline{\varphi}$ which preserves the basepoints, and any two such approximations are homotopic by a homotopy which also preserves the basepoints.  Let $\ol \varphi$ be such a map and choose a homotopy between $\ol\varphi$ and $\phi$ which is $\calc^0$- close to the constant homotopy. We define $\varphi_* : H_*(X; \varphi_*\calf)\ri H_*(Y;\calf)$ by the composition 
\begin{equation}\label{def:direct-for-continuous}
\xymatrix{  H_*(X; \ol\varphi^*\calf) \ar[dr]^-{\ol\varphi_*}&\, \\
H_*(X;\varphi^*\calf)\ar[u]_{\Psi}\ar[r]_-{\varphi_*}&H_*(Y;\calf)}\end{equation} 
and similarly 
\begin{equation}\label{def:shriek-for-continuous}
\xymatrix{  H_{*+\mathrm{dim}(X)-\mathrm{dim}(Y)}(X; \ol\varphi^*\calf)\ar[d]_{\Psi^{-1}} &H_*(Y;\calf)\ar[l]^-{\ol\varphi_!}\ar[dl]^-{\varphi_!}\\
H_{*+\mathrm{dim}(X)-\mathrm{dim}(Y)}(X;\varphi^*\calf)&\, }\end{equation} 
where $\Psi$ is the identification isomorphism associated to the homotopy between $\varphi$ and $\ol\varphi$. Recall that the isomorphism $\Psi$ was defined in Proposition~\ref{prop:homotopy-identification} of \S\ref{sec:comments-homotopy}. In the statement of   Proposition~\ref{prop:homotopy-identification} the applications were smooth, but actually it is valid with exactly the same proof for continuous ones.  

Let us prove that the definition is independent of the choices of $\ol\varphi$ and the homotopy connecting $\varphi$ to $\ol\varphi$. For direct maps we use  the following diagram 
$$\xymatrix{  H_*(X; \ol\varphi^*\calf) \ar[dr]^-{\ol\varphi_*}\ar@/_4pc/ [dd]_{\Psi''}&\, \\
H_*(X;\varphi^*\calf)\ar[u]_{\Psi}\ar[d]{\Psi'}&H_*(Y;\calf)\\
H_*(X; \ol\varphi'^*\calf)\ar[ur]_{\ol\varphi'_*}&\, 
}$$
The isomorphism $\Psi''$ is the identification isomorphism associated to the homotopy defined by the concatenation of the reversed homotopy between $\varphi$ and $\ol\varphi$ with the homotopy between $\varphi$ and $\ol\varphi'$. Let us show that the diagram above is commutative. The relation $\Psi'=\Psi''\circ\Psi$ is the property (iii) of Proposition~\ref{prop:homotopy-identification}. We also have  $\ol\varphi_*= \ol\varphi'_*$; this is the homotopy property satisfied by the smooth maps $\ol\varphi$ and $\ol\varphi'$. We infer that $\ol\varphi_*\circ \Psi =\ol\varphi'_*\circ\Psi'$ which means that $\varphi_*$ is well defined. 
The proof for the shriek map is analogous. 

Let us now check that the properties 1-4 from~\S\ref{sec:functoriality-properties} are satisfied. We will only do it for direct maps, the proofs for shriek maps being similar. \\
\\
{\it 1. Identity: $Id_*=Id$}.  Obvious. 

{\it 2. Composition: $\varphi_*\circ\zeta_*=(\varphi\circ\zeta)_*$}. We need the following naturality property for identification morphisms: 
\begin{lemma}\label{lem:naturality-for-identification} Let $\varphi_0, \varphi_1: Y\ri Z$ two continuous maps homotopic through a homotopy $\varphi$, and $\zeta:X\ri Y$ another continuous map. Denote by $\varphi\circ\zeta$ the homotopy between $\varphi_0\circ\zeta$ and $\varphi_1\circ\zeta$ obtained by composition. Let $\calf$ be a DG-module over $C_*(\Omega Z)$. Then the following diagram is commutative : 
$$\xymatrix{H_*(X;\zeta^*\varphi_0^*\calf)\ar[r]^-{\zeta_*}\ar[d]^{\Psi^{\varphi\circ\zeta}}&H_*(Y; \varphi_0^*\calf)\ar[d]^{\Psi^\varphi}\\
H_*(X;\zeta^*\varphi_1^*\calf)\ar[r]^-{\zeta_*}&H_*(Y; \varphi_1^*\calf)}$$
\end{lemma} 

\begin{proof} Suppose first that $\zeta$ is a smooth embedding. We will prove the commutativity of the diagram at the level of complexes. Recall the definition of the direct maps for embeddings from Section \S\ref{sec:funct-embeddings}. We start with a set of data $\Xi^X=(f,\ldots)$ on $X$ and we choose an adapted set of data $\Xi^Y_\zeta$ on $Y$, at first on $\zeta(X)$ via $\zeta$ (with Morse function $f\circ\zeta^{-1}$), then on a tubular neighbourhood of $\zeta(X)$ with inward pointing gradient and finally on the whole $Y$. For these choices, if ${\mathcal{G}}$ is a DG-module over $C_*(\Omega Y)$ then the map $\zeta_*:C_*(X, \Xi^X;\zeta^*{\mathcal{G}})\ri C_*(Y, \Xi^Y_\zeta; {\mathcal{G}})$ is given by the formula 
$$\zeta_*(\alpha\otimes x)\, =\, \alpha \otimes \zeta(x)$$
for any $\alpha \in {\mathcal{G}}$ and $x\in \crit(f)$. 
Now recall the definition of $\Psi^\varphi$ from the proof of Proposition~\ref{prop:homotopy-identification}: at the level of complexes we have 
$$\Psi(\alpha\otimes x)\, =\, \sum_y\alpha\nu^\varphi_{x,y}\otimes y.$$
The cocycle $\nu^{\varphi}_{x,y}$ was obtained as follows: given a set of data $\Xi^Y$ on $Y$ consider the set of data $\Xi^{Y,\Id}$ on $[0,1]\times Y$ which defines the trivial continuation map between $C_*(Y,\Xi^Y)$ and itself. It yields as in the relation \eqref{eq:useful-for-continuous} of \S\ref{sec:properties-functoriality-embeddings} a cocycle 
$$\overline{\nu}^Y_{x,y}\, =\, - (\Theta\circ p\circ \Gamma^\Id)_*(s^\Id_{x,y}) \, \in  \, C_{|x|-|y|}(\calp_{(0,\star_Y)\ri (1,\star_Y)}[0,1]\times Y)$$
which we transform into a cocycle on $\Omega Z$ by 
$$\nu^\varphi_{x,y}\, =\, \varphi_*(\ol\nu^Y_{x,y}).$$
In a similar way the cocycle $\nu_{x,y}^{\zeta\circ\varphi}$ which defines the identification map $\Psi^{\varphi\circ\zeta}$ is given by 
$$\nu_{x,y}^{\zeta\circ\varphi}\, =\, \varphi_*\circ(\id_{[0,1]}\times \zeta)(\ol \nu_{x,y}^X),$$
where $\ol \nu_{x,y}^X\in  \, C_{|x|-|y|}(\calp_{(0,\star_X)\ri (1,\star_X)}[0,1]\times X)$ is obtained from the trivial continuation data $\Xi^{X,\Id}$ as above. This implies immediately 
$$\ol\nu^Y_{\zeta(x),\zeta(y)}= (\Id_{[0,1]}\times\zeta)_*(\ol\nu^X_{x,y})$$ and also $\ol\nu^Y_{\zeta(x),z}=0$ if $z$ is a critical point of the Morse function on $Y$ which is outside of $ \zeta(X)$. 

Now remember that in order to express $\zeta_*$, we chose the adapted set of data $\Xi^Y_\zeta$ on $Y$ and one can easily see that for this choice the set of data $\Xi^{Y,\Id}$  on $[0,1]\times Y$ which corresponds to the trivial continuation is adapted to $\Xi^{X,\Id}$ for the embedding $\Id_{[0,1]}\times \zeta:[0,1]\times X\ri [0,1]\times Y.$ We infer that 
$$\nu_{x,y}^{\varphi\circ\zeta}\, =\, \nu_{\zeta(x),\zeta(y)}^\varphi$$
and $\nu^\varphi_{\zeta( x), z}=0$ for the other critical points $z$ not belonging to $\zeta(X)$. This is equivalent to $\Psi^\varphi\circ\zeta_*=\zeta_*\circ\Psi^{\varphi\circ\zeta}$, i.e. the fact that the diagram in our statement is commutative. \\
\\
Consider now the case where $\zeta$ is a general smooth map. We will use the definition of $\zeta_*$ given in \S\ref{sec:funct-direct-general}: We take an embedding $\chi:X\ri D$ into some disc of large dimension and set $\zeta_*=(i_*)^{-1}\circ\zeta^\chi_*$, where $\zeta^\chi=(\chi,\zeta):X\ri D\times Y$ and $i:Y\hookrightarrow D\times Y$ is the inclusion $y\mapsto (0,y).$ We get the following diagram: 
$$\xymatrix{H_*(X;\zeta^*\varphi_0^*\calf)\ar@/^2pc/ [rr]^{\zeta_*}\ar[r]^-{\zeta^\chi_*}\ar[d]^{\Psi^{\varphi\circ\zeta}=\Psi^{\varphi\circ\pi\circ\zeta^\chi}}&H_*(D\times Y;\pi^*\varphi_0^*\calf )\ar[d]^-{\Psi^{\varphi\circ \pi}}&H_*(Y; \varphi_0^*\calf)\ar[l]_-{i_*}\ar[d]^{\Psi^\varphi=\Psi^{\varphi\circ\pi\circ i}}\\
H_*(X;\zeta^*\varphi_1^*\calf)\ar@/_2pc/ [rr]_{\zeta_*}\ar[r]^-{\zeta^\chi_*}&H_*(D\times Y;\pi^*\varphi_1^*\calf)&H_*(Y; \varphi_1^*\calf)\ar[l]_-{i_*}}$$
which is commutative. 
Indeed the  upper and lower part are commutative by definition of $\zeta_*$ and the left and right square also commute by the particular case of embeddings we have just proved. It follows that $\Psi^\varphi\circ\zeta_*=\zeta_*\circ\Psi^{\varphi\circ\zeta}$ as claimed. \\
\\
Finally suppose that $\zeta$ is  continuous. Choose a smooth approximation $\ol\zeta$. Denote by $\Upsilon$ the homotopy between $\zeta$ and $\ol\zeta$.   Once again we will use a commutative diagram: 
$$\xymatrix{   H_*(X;\zeta^*\varphi_0^*\calf)\ar@/_4pc/ [ddd]_{\Psi^{\varphi\circ\zeta}}\ar[d]^-{\Psi^{\varphi_0\circ\Upsilon}}\ar[dr]^-{\zeta_*}&\, \\
H_*(X;\ol\zeta^*\varphi_0^*\calf)\ar[r]^-{\ol\zeta_*}\ar[d]^{\Psi^{\varphi\circ\ol\zeta}}&H_*(Y; \varphi_0^*\calf)\ar[d]^{\Psi^\varphi}\\
H_*(X;\ol\zeta^*\varphi_1^*\calf)\ar[r]^-{\ol\zeta_*}&H_*(Y; \varphi_1^*\calf)\\
H_*(X;\zeta^*\varphi_1^*\calf)\ar[u]_-{\Psi^{\varphi_1\circ\Upsilon}}\ar[ur]_-{\zeta_*}&\, }$$

To show that it is commutative remark first that the lower and upper triangles are commutative by definition of $\zeta_*$ and then that the right square is also commutative by the previous proof of the naturality for smooth maps. Finally, for the commutativity of the left part, remark that the two-parameter family $\varphi_\tau\circ\Upsilon_t$ yields a homotopy of homotopies with fixed endpoints between the concatenation of the homotopies $\varphi_0\circ\Upsilon_t$ and $\varphi_t\circ\ol\zeta$ and the concatenation of the homotopies   $\varphi_t\circ\zeta$  and $\varphi_1\circ\Upsilon_t$. We may therefore apply 
Proposition~\ref{prop:homotopy-identification} parts (ii) and (iii). 

We infer that the whole diagram above is commutative and in particular we have the claimed relation $\Psi^\varphi\circ\zeta_*=\zeta_*\circ\Psi^{\varphi\circ\zeta}$, which finishes the proof of the lemma. 
\end{proof}

We are now ready to prove the composition property. Let $\zeta:X\ri Y$ and $\varphi:Y\ri Z$ be continuous maps and $\calf$ a DG-module over $C_*(\Omega Z)$. Let $\ol\zeta$, $\ol\varphi$ be smooth approximations of $\zeta$ respectively $\varphi$; obviously $\ol\varphi\circ\ol\zeta$ is a smooth approximation of $\varphi\circ\zeta$. Denote by $\Phi$ the homotopy between $\varphi$ and $\ol\varphi$,  by $\Upsilon$ the homotopy between $\zeta$ and $\ol\zeta$ and by $\Phi\circ\Upsilon$ the homotopy $\Phi_t\circ\Upsilon_t$. Consider the following diagram:
$$\xymatrix{ H_*(X;\ol\zeta^*\ol\varphi^*\calf) \ar[dr]^-{\ol\zeta_*}\ar@/^3pc/ [ddrr]^-{(\ol\varphi\circ\ol\zeta)_*}&\, &\, \\
H_*(X; \ol\zeta^*\varphi^*\calf)\ar[u]^-{\Psi^{\Phi\circ\ol\zeta}}\ar[dr]^-{\ol\zeta^*}&H_*(Y;\ol\varphi^*\calf)\ar[dr]^-{\ol\varphi_*} &\, \\
H_*(X;\zeta^*\varphi^*\calf)\ar[u]^-{\Psi^{\varphi\circ\Upsilon}}\ar[r]^-{\zeta_*}\ar@/^4pc/ [uu]^-{\Psi^{\Phi\circ\Upsilon}}&H_*(Y;\varphi_*\calf)\ar[r]^-{\varphi_*} \ar[u]^-{\Psi^\Phi}&H_*(Z;\calf)}$$

Again this diagram is commutative: The two lower triangles are commutative by definition of $\zeta_*$ resp. $\varphi_*$. The parallelogram above them is commutative by the previous lemma. The commutativity of the right upper part is the composition property for smooth maps. Finally, the two parameter family $\Phi_\tau\Upsilon_t$ provides a homotopy of homotopies with fixed end points between $\Phi\circ\Upsilon$ and the concatenation of $\varphi\circ\Upsilon$ with $\Phi\circ\ol\zeta$ and we apply Proposition~\ref{prop:homotopy-identification}, items (ii) and (iii) to get the commutativity of the left part of the diagram. We infer that 
$$(\ol\varphi\circ\ol\zeta)_*\circ\Psi^{\Phi\circ\Upsilon}\,=\, \varphi_*\circ\zeta_*$$
and the right hand side is by definition $(\varphi\circ\zeta)_*$, so the composition property is proved. 
 
{\it 3. Homotopy: $\varphi_{0,*}=\varphi_{1,*}\circ\Psi^\varphi$.} The argument  is quite similar to  the one which proves  that $\varphi_*$ is well defined. To a homotopy $\varphi$ between continuous maps $\varphi_0, \varphi_1:X\ri Y$  we associate the following diagram  : 
$$\xymatrix{ H_*(X; \ol\varphi_0^*\calf) \ar[ddr]^-{\ol\varphi_{0,*}}\ar@/_4pc/ [dddd]_{\Psi}&\, \\
H_*(X;\varphi_0^*\calf)\ar[dr]_{\varphi_{0,*}}\ar[u]_{\Psi_0}\ar[dd]_{\Psi^\varphi}&\, \\
\, &H_*(Y;\calf)\\
H_*(X;\varphi_1^*\calf)\ar[ur]^{\varphi_{1,*}}\ar[d]_{\Psi_1}&\, \\
H_*(X; \ol\varphi_1^*\calf)\ar[uur]_{\ol\varphi_{1,*}}&\, 
}$$
Here $\Psi$ is the identification isomorphism associated to the homotopy between $\ol\varphi_0$ and $\ol\varphi_1$ obtained by concatenation. The desired homotopy property is 
$$\varphi_{0,*}\, =\,  \varphi_{1,*}\circ \Psi^\varphi.$$

To show it, note that the upper and lower right triangles are commutative by definition of the maps $\varphi_{0,*}$ and $\varphi_{1,*}$ and that $\Psi_1\circ\Psi^\varphi=\Psi\circ\Psi_0$, by the property (iii) of 
Proposition~\ref{prop:homotopy-identification}. Note also that $\ol\varphi_{0,*}=\varphi_{1,*}\circ\Psi$ by the homotopy property for smooth maps. The conclusion is now straightforward.

{4. \it Spectral sequence.} By definition $\varphi_*=\ol\varphi_*\circ\Psi$ and both maps $\ol\varphi_*$ and $\Psi$ are defined at the level of complexes and they preserve the filtrations given by the index of critical points. Therefore we may also define $\varphi_*: C_*(X,\Xi^X;\varphi^*\calf)\ri C_*(Y,\Xi^Y;\calf)$ by the composition 
$$\xymatrix{  C_*(X, \Xi^X; \ol\varphi^*\calf) \ar[dr]^-{\ol\varphi_*}&\, \\
C_*(X, \Xi^X;\varphi^*\calf)\ar[u]_{\Psi}\ar[r]_-{\varphi_*}&C_*(Y, \Xi^Y;\calf)}$$
and it obviously also preserves the filtration.  The  morphisms between  the corresponding spectral sequences write as follows at the second page : 

$$\xymatrix{  H_p(X,  \ol\varphi^*H_q(\calf)) \ar[dr]^-{\ol\varphi^2_{p,q}=\ol\varphi_*}&\, \\
H_{p}(X, \varphi^*H_q(\calf))\ar[u]_{\Psi^{2}_{p,q}}\ar[r]_-{\varphi_{p,q}^2}&H_p(Y, H_{q}(\calf))}$$
We know that $\ol\varphi_{p,q}^2$ equals the direct map $\ol\varphi_*$ in homology since  smooth maps satisfy the spectral sequence property.  Then, by construction, $\Psi^2_{p,q}$ is the identification map for the usual (Morse and therefore singular) homologies with local coefficients in $\varphi^*H_q(\calf)$ resp. $\ol\varphi^*H_q(\calf)$. The homotopy invariance for continuous maps in singular homology with local coefficients implies that $\varphi_*=\ol\varphi_*\circ \Psi^{2}_{p,q}$ and therefore $\varphi_*=\varphi^{2}_{p,q}$, as claimed. 

\rmk\label{indep-homotopy-identifications-for-continuous} With all these properties satisfied, the proof of Proposition~\ref{prop-indep-homotopy-identifications} is the same for continuous maps. This means that for homotopic continuous maps $\varphi_0,\varphi_1:X\ri Y$ the identification isomorphism $$\Psi: H_*(X;\varphi_0^*\calf)\ri H_*(Y;\varphi_1^*\calf)$$ does not depend on the homotopy between $\varphi_0$ and $\varphi_1$. 
\kmr

 \section{The direct map for Hurewicz fibrations}\label{sec:direct-for-fibrations}

Let $\varphi:X\ri Y$ be a continuous map between closed manifolds. Consider a Hurewicz fibration $\mathcal{E}$ :  $F\ri E_Y\ri Y$ and the pullback fibration $\varphi^*{\mathcal{E}} : F\ri E_X\ri X$. We know by Theorem~\ref{thm:fibration} that $H_*(Y;C_*(F))\simeq H_*(E_Y)$ and by Corollary~\ref{cor:pullback} that $H_*(X;\varphi^*C_*(F))\simeq H_*(E_X)$. 

Via these isomorphisms the direct map $\varphi_*:H_*(X;\varphi^*C_*(F))\ri H_*(Y;C_*(F))$ induces a map  between $H_*(E_X)$ and $H_*(E_Y)$. The following result shows that this map is the  expected one: 
\begin{proposition}\label{prop:direct-for-fibrations} Let $\widetilde{\varphi}: E_X\ri E_Y$ be the map canonically  induced by~$\varphi$. Then the map defined by $\varphi_*$ coincides with $\widetilde{\varphi}_*$, the map induced by $\widetilde{\varphi}$ in singular homology, via the isomorphisms above. 
\end{proposition}

\begin{corollary}\label{cor:direct-map-induced-fibrations}  Assume $X\subset Y$ is a submanifold and let $\varphi:X\hookrightarrow Y$, $\widetilde{\varphi}:E_{Y}|_X\hookrightarrow E_Y$ be the inclusions. The direct map $\varphi_*:H_*(X;\varphi^*C_*(F))\ri H_*(Y;C_*(F))$ is identified with the direct map $\widetilde{\varphi}_*:H_*(E_{Y}|_X)\to H_*(E_Y)$. 
\end{corollary}

\begin{proof} This follows directly from Proposition~\ref{prop:direct-for-fibrations} and Corollary~\ref{cor:pullback}. 

The corollary also admits a direct proof that relies exclusively on Corollary~\ref{cor:pullback}. Construct on $Y$ a Morse function $f$ as follows: first choose a Morse function on $X$, then extend it to $Y$ such that in a tubular neighborhood of $X$ the extension is given by adding a fiberwise positive definite quadratic form. The critical points of the restriction $f|_X$ form a subcomplex, and the direct map in DG Morse homology is induced by the inclusion of that subcomplex. The Latour cells of critical points of $f|_X$, seen in $X$ or in $Y$, coincide, and we can choose chain representatives for these Latour cells so that they also coincide when seen in $X$ or in $Y$. Working with compatible transitive lifting functions for the fibration $E_Y$ and $E_{Y}|_X$ as in Corollary~\ref{cor:pullback}, we obtain a chain level commutative diagram 
$$
\xymatrix{
C_*(f|_X;\varphi^*C_*(F))\ar[r]^-{\varphi_*} \ar[d]_\simeq & C_*(f;C_*(F))\ar[d]^\simeq \\
C_*(E|_X)\ar[r]_-{\tilde\varphi_*} & C_*(E),
}
$$ 
in which the vertical maps are the quasi-isomorphisms constructed in the proof of the Fibration Theorem~\ref{thm:fibration}. The conclusion follows by passing to homology.  
\end{proof}

\begin{proof}[Proof of Proposition~\ref{prop:direct-for-fibrations}] Denote by $\calf$ the DG-module $C_*(F)$ over $C_*(\Omega Y)$. Our goal is to show that the following diagram is commutative: \begin{equation}\label{eq:direct-for-fibrations} \xymatrix{H_*(X; \varphi^*\calf)\ar[r]^{\varphi_*}\ar[d]_-{\simeq}^{\Psi_{\varphi^*\cale}}& H_*(Y;\calf)\ar[d]^-{\Psi_{\cale}}_-{\simeq}\\
H_*(E_X)\ar[r]^-{\widetilde{\varphi}_*}   &H_*(E_Y)}\end{equation} 

The isomorphisms $\Psi_\cale$ and $\Psi_{\varphi^*\cale}$ involve lifting functions for the respective fibrations. More precisely, if $\Phi: F\times  \calp_{\star \ri Y}Y\ri E_Y$  is the transitive lifting function for  $\cale$ that defines the module structure of $\calf$ and the isomorphism $\Psi_\cale$, then,  as pointed out in the proof of Cor.\ref{cor:pullback},   the isomorphism $\Psi_{\varphi^*\cale}$ is defined by the induced lifted function $\Phi^\varphi: F\times  \calp_{\star \ri X}X\ri E_X$ given by the following formula: 
$$\Phi^\varphi(f,\gamma)\, =\, [\mathrm{ev}(\gamma),\Phi(f,\varphi(\gamma)) ],$$
where $\mathrm{ev}$ is the evaluation at the endpoint of $\gamma$ and we wrote  $E_X$ as $X \, _\varphi\!\times_{\pi_Y}E_Y$  (where $\pi_Y:E_Y\ri Y$ is the projection).


In the writing of $E_X$ above, the map $\widetilde{\varphi}:E_X\ri E_Y$ is the projection on the second factor, so that we have 
\begin{equation}\label{eq:lifting-functions} 
\widetilde{\varphi}\circ\Phi^\varphi (f,\gamma)\, =\, \Phi(f,\varphi(\gamma)).
\end{equation}

  We decompose our proof into three steps: \\
 \\
 {\it Step 1: $\varphi : X\ri Y$ is an embedding}. We may suppose w.l.o.g. that $X\subset Y$ is a submanifold and $\varphi$ is the inclusion. Take a set of data $\Xi^X=(f, \xi, \ldots, )$ on $X$ and extend it to $\Xi^Y=(F,\ldots,)$ as in the definition of the direct map for submanifolds in~\S\ref{sec:funct-closed-submanifolds}. We therefore have $$\varphi_*(\alpha\otimes x)\, =\, \alpha\otimes x$$
for all $\alpha\in C_*(F)$ and $x\in \Crit(f)$. We also may suppose that $\Xi^X$ and $\Xi^Y$ are as in Theorem \ref{thm:fibration} (Morse function with a unique minimum at the basepoint, trees formed by gradient lines). For this particular choice the tree $\caly^X$ is contained in $\caly^Y$ and the homotopy inverse of the projection $\theta^Y:Y/\caly^Y\ri Y$ is an extension of $\theta^X:X/\caly^X\ri X$. We therefore have
\begin{equation}\label{eq:commutation-trees}
\xymatrix{X/\caly^X\ar[d]^-{\theta^X}\ar@{^{(}->}[r]^-{\varphi}&Y/\caly^Y\ar[d]^-{\theta^Y}\\
X\ar@{^{(}->}[r]^-{\varphi}&Y}
\end{equation} 

Denote  by  $E_Y^\theta$, $E_X^\theta$ the total spaces of the pullbacks by $\theta^Y$ resp. $\theta^X$ of the fibrations $\cale$ and $\varphi^*\cale$. We remind that the isomorphism $\Psi_\cale: H_*(Y;\calf)\ri H_*(E_Y)$  in the proof of Theorem \ref{thm:fibration} is induced by a composition of quasi-isomorphisms: 
\begin{equation}\label{eq:composition-quasi-iso}\xymatrix{C_*(Y, \Xi^Y; \calf)\ar[r]^-{\cong} &C_*(Y/\caly^Y ; \theta^{Y*}\calf)\ar[r]^-{\simeq} &C_*(E_Y^\theta)\ar[r]^-{\simeq}_-{\widetilde{\theta}^Y_*} &C_*(E_Y)}\end{equation}
(we will remind the definition of the second complex below) and the same for $X$. We will show that the diagram \eqref{eq:direct-for-fibrations} is commutative at the level of complexes for each of the three horizontal arrows above.  More precisely we claim that the following diagram is commutative
$$\xymatrix{C_*(X, \Xi^X; \varphi^*\calf)\ar[r]^-{\cong}\ar[d]^-{\varphi_*} &C_*(X/\caly^X ; \varphi^*\theta^{Y*}\calf)\ar[r]^-{\simeq} \ar[d]^-{\varphi_*}&C_*(E_X^\theta)\ar[r]^-{\simeq}_-{\widetilde{\theta}^X_*}\ar[d]^{\widetilde{\varphi}_*} &C_*(E_X)\ar[d]^-{\widetilde{\varphi}_*}\\
C_*(Y, \Xi^Y; \calf)\ar[r]^-{\cong} &C_*(Y/\caly^Y ; \theta^{Y*}\calf)\ar[r]^-{\simeq} &C_*(E_Y^\theta)\ar[r]^-{\simeq}_-{\widetilde{\theta}^Y_*} &C_*(E_Y)} $$ We start with rightmost square  whose commutativity is an immediate consequence of the diagram \eqref{eq:commutation-trees}; from it we get a commutative diagram for the total spaces of the corresponding pullback fibrations 
$$\xymatrix{E_X^\theta\ar[d]^-{\widetilde{\theta}^X}\ar[r]^-{\widetilde{\varphi}}& E^\theta_Y\ar[d]^-{\widetilde{\theta}^Y}\\
E_X\ar[r]^-{\widetilde{\varphi}}&E_Y}$$
which induces the desired diagram for direct maps in singular homology. 

Let us now analyse the leftmost square: The complex $C_*(Y/\caly^Y; \theta^{Y*}\calf)$ is by definition the enriched complex defined using the critical points of the Morse function  $F$  of the data $\Xi^Y$ on $Y$ and the twisting cocycle $m'_{x,y}\in C_*(\Omega(Y/\caly^Y))$. We recall that the cocycle $(m'_{x,y})$ was defined in the proof of Th.\ref{thm:fibration} in \S\ref{sec:proof-of-fibration} in the same way as the Barraud-Cornea cocycle $(m_{x,y})$ which defines $C_*(Y,\Xi^Y;\calf)$ except for the fact that we omit $\theta^Y:Y/\caly^Y\ri Y$ from the evaluation map; therefore the relation between the two is 
\begin{equation}\label{eq:two-cocycles}
m_{x,y}\, =\, \theta^Y_*(m'_{x,y})
\end{equation} 

 The same is valid for the corresponding complex $C_*(X/\caly^X; \theta^{X*}\varphi^*\calf)$ and the Morse function $f$ of the data $\Xi^X$. Given the compatibility between $\Xi^ X$ and $\Xi^Y$ we have $\Crit(f)\subset \Crit(F)$  and moreover $\varphi_*(m'_{x,y}(f))=m'_{x,y}(F)$ for $x,y\in \crit(F)$, i.e. the second complex is a subcomplex of the first via the inclusion $\varphi$. 
 
 The first arrow of \eqref{eq:composition-quasi-iso} is the usual identification of Rem.\ref{rem:pullback} (due to \eqref{eq:two-cocycles} here) between the enriched complex on $Y$ with cocycle pushed forward by $\theta_*^Y$ and the complex on $Y/\caly^Y$ with module pulled back by $\theta^{Y*}$. In view of the above, using \eqref{eq:commutation-trees}, it is clear that the diagram $$\xymatrix {
 C_*(X, \Xi^X; \varphi^*\calf)\ar[r]^- {\varphi_*}\ar[d]^-{\cong}  & C_*(Y, \Xi^Y; \calf)\ar[d]^-{\cong} \\
 C_*(X/\caly^X; \theta^{X*}\varphi^*\calf)\ar[r]^-{\varphi_*}&C_*(Y/\caly^Y ; \theta^{Y*}\calf)}
 $$
is commutative. 

We turn now our attention to the middle arrow of \eqref{eq:composition-quasi-iso} which corresponds to the third and most important step of the proof of Th.\ref{thm:fibration}: the quasi-isomorphism between the enriched complex $C_*(Y/\caly^Y;\theta^{Y*}\calf)$ defined above and singular (cubic) homology of the total space $E_Y^\theta$ of the fibration $\theta^{Y*}\cale$. Let us show that $\varphi$ induces a commutative diagram in this case too, i.e. that the middle square above is commutative. Denote by $\Psi^Y$ this quasi-isomorphism and by $\Psi^X: C_*(X/\caly^X; \theta^{X*}\varphi^*\calf)\ri C_*(E^\theta_X)$ its analogue for $X$.  Recall the definition of $\Psi^Y$ from \S\ref{sec:proof-of-fibration} (it was denoted by $\Psi'$ there): 
$$\Psi^Y(\alpha\otimes x)\, =\, \Phi_*(\alpha\otimes m^Y_x)$$
for any $\alpha \in \theta^{Y*}\calf$ and $x\in \Crit(F)$. Here $\Phi: F\times \calp_{\star\ri Y/\caly^Y}Y/\caly^Y\ri E_Y^\theta$ is a lifting function for $\theta^{Y*}\cale: F\ri E_Y^\theta\ri Y/\caly^Y$. The family of chains $m^Y_x\in C_{|x|}( \calp_{\star\ri Y/\caly^Y}Y/\caly^Y)$ is the one constructed in \S\ref{sec:proof-of-fibration}. 

Consider now the analogous  quasi-isomorphism for the pullback fibration $\varphi^*\theta^{Y*}\cale= \theta^{X*} \varphi^*\cale$.  Taking the lifted function $\Phi^\varphi$ which is derived from $\Phi$ as explained in the beginning  of the present proof, we have $$\Psi^X(\alpha\otimes x)\, =\,  \, \Phi^\varphi_*(\alpha\otimes m^X_x)$$
for $\alpha \in \varphi^*\theta^{Y*}\calf$ and $x\in \Crit(f)$. Again, the compatibility between the data $\Xi^X$ and $\Xi^Y$ implies that $\Crit(f)\subset \Crit(F)$ and, by construction, $m_x^Y=\varphi_*(m_x^X)$. Now we are ready to prove the commutativity of the diagram
$$\xymatrix {
 C_*(X/\caly^X; \theta^{X*}\varphi^*\calf)\ar[r]^-{\varphi_*}\ar[d]^-{\Psi^X}_-{\simeq}&C_*(Y/\caly^Y ; \theta^{Y*}\calf)\ar[d]^-{\Psi^Y}_-{\simeq}\\
 E_X^\theta\ar[r]^{\widetilde{\varphi}_*}&E_Y^\theta}
 $$
 where $\varphi_*(\alpha\otimes x) =\alpha\otimes x$. Indeed, using~\eqref{eq:lifting-functions} we have
 \begin{align*}
 \Psi^Y\circ\varphi_*(\alpha\otimes x)=\Phi_*(\alpha\otimes m_x^Y) & =\Phi_*(\alpha\otimes\varphi_*(m_x^X)) \\
 & =\widetilde{\varphi}_*(\Phi^\varphi_*(\alpha\otimes m_x^X))=\widetilde{\varphi}_*\circ\Psi^X(\alpha\otimes x).
 \end{align*}
 This finishes the proof of Step 1.

 {\it Step 2:  $\varphi:X\ri Y$ is a general smooth map}. By definition (see~\S\ref{sec:funct-direct-general}), we have $\varphi_*=(i_*)^{-1}\varphi^\chi_*$ where $\chi: X\ri D$ is an embedding into a disk, $\varphi^\chi=(\chi,\varphi):X\ri D\times Y$, and $i:Y\hookrightarrow D\times Y$ is the inclusion $y\mapsto(0,y)$.  On the other hand considering  the pullback of the projection $\pi:D\times Y\ri Y$ which is  $$\pi^*\cale: F\ri D\times E_Y \stackrel{\Id\times \pi}{\longrightarrow}D\times Y$$ we have two homotopic maps $i\circ \varphi$ and $\varphi^\chi: X\ri D\times Y$ which yield the same map in singular homology
$$\widetilde{\varphi^\chi}_*\, =\,  (\widetilde{i\circ\varphi})_*: H_*(E_X)\ri H_*(D\times E_Y)$$
(we have $\varphi^{\chi *}\pi^*\cale= (i\circ \varphi)^*\pi^*\cale=\varphi^*\cale$.) Noting that $\widetilde{i}_*: H_*(E_Y)\ri H_*(D\times E_Y)$ is an isomorphism we get $\widetilde{\varphi}_*=(\widetilde{i}_*)^{-1}\widetilde{\varphi^\chi}_*$.  Since  the diagram \eqref{eq:direct-for-fibrations} is commutative for the embeddings $\varphi^\chi$ and $i$ by Step 1, it is also commutative for  $\varphi$:
$$ \xymatrix{H_*(X; \varphi^*(C_*(F)))\ar[r]^{\varphi^\chi_*}\ar[d]_-{\simeq}^{\Psi_{\varphi^*\cale}}\ar@/^2pc/[rr]^{\varphi_*}&  H_*(D\times Y;\pi^*C_*(F)) \ar[d]_-{\simeq}^{\Psi_{\pi^*\cale}} &H_*(Y;C_*(F))\ar[l]_-{i_*} \ar[d]^-{\Psi_{\cale}}_-{\simeq}\\
H_*(E_X)\ar[r]^-{\widetilde{\varphi^\chi}_*} \ar@/_2pc/[rr]^{\widetilde{\varphi}_*}&H_*(D\times E_Y)  &H_*(E_Y)\ar[l]_-{\widetilde{i}_*}}$$

{\it Step 3: $\varphi:X\ri Y$ is continuous}. We will use the following lemma: 
\begin{lemma}\label{lem:homotopy-invariance-fibrations} Let $\varphi_0, \varphi_1: X\ri Y$ two homotopic continuous maps and $\cale: F\ri E_Y\ri Y$ a Hurewicz fibration over $Y$. Denote by $\calf=C_*(F)$ the associated DG-module over $C_*(\Omega Y)$ and by $\Psi$ the identification isomorphism $\Psi:H_*(X; \varphi_0^*\calf)\ri H_*(X;\varphi_1^*\calf)$. Also denote by $E_X^i$ the total spaces of the pullback fibrations $\varphi_i^*\cale$ for $i=0,1$ and by $\widetilde{\varphi}^i : E_X^i \ri E_Y$ the maps induced by $\varphi^i$. Then the following diagram is commutative: 
$$\xymatrix{H_*(E_X^0)\ar[r]^-{\widetilde{\varphi}_{0*}}&H_*(E_Y)&H_*(E_X^1)\ar[l]_-{\widetilde{\varphi}_{1*}}\\
H_*(X;\varphi_0^*\calf)\ar[rr]^-{\Psi}_-{\simeq}\ar[u]^-{\Psi_{\varphi_0^*\cale}}_-{\simeq}&\, & H_*(X;\varphi_1^*\calf)\ar[u]^-{\Psi_{\varphi_1^*\cale}}_-{\simeq}}$$

\end{lemma}
\begin{proof} Remark first that the proof is immediate if $\varphi_i$ are smooth. Indeed, in this case we may complete our diagram as follows: 

\begin{equation}\label{eq:particular-case}\xymatrix{H_*(E_X^0)\ar[r]^-{\widetilde{\varphi}_{0*}}&H_*(E_Y)&H_*(E_X^1)\ar[l]_-{\widetilde{\varphi}_{1*}}\\
H_*(X;\varphi_0^*\calf)\ar[r]^-{\varphi_{0*}}\ar@/_2pc/[rr]^-{\Psi}_-{\simeq}\ar[u]^-{\Psi_{\varphi_0^*\cale}}_-{\simeq}&H_*(Y;\calf)\ar[u]^-{\Psi_\cale}_-{\simeq}& H_*(X;\varphi_1^*\calf)\ar[u]^-{\Psi_{\varphi_1^*\cale}}_-{\simeq}\ar[l]_-{\varphi_{1*}}}\end{equation} 

The two squares are commutative by Step 2, the lower part is also commutative by homotopy invariance, so the whole diagram is commutative and this particular case is proved. . 

To prove the general case, denote by $D$ the interval $[-\epsilon, 1+\epsilon]$, by $\varphi :D\times X\ri Y$ the homotopy between $\varphi_0$ and $\varphi_1$ (we may suppose that this homotopy is stationary on $[-\epsilon, 0]$ and on $[1, 1+\epsilon]$) and by $i_j: X\hookrightarrow D\times X$ the inclusions $i_j(x)=(j,x)$ for $j=0,1$. We write the above diagram for the homotopic maps $i_0$ and $i_1$ and for the pullback fibration $\varphi^*\cale:F\ri E_{D\times X}\ri D\times X$ instead of $\cale$, which we complete as follows: 
$$\xymatrix{\, &E_Y&\, \\
H_*(E_X^0)\ar[r]^-{\widetilde{i}_{0*}}\ar[ur]^-{\widetilde{\varphi}_{0*}}&H_*(E_{D\times X})\ar[u]^-{\widetilde{\varphi}_*}&H_*(E_X^1)\ar[l]_-{\widetilde{i}_{1*}}\ar[ul]_-{\widetilde{\varphi}_{1*}}\\
H_*(X;i_0^* \varphi^*\calf)\ar[rr]^-{\Psi}_-{\simeq}\ar[u]^-{\Psi_{i_0^*\varphi^*\cale}}_-{\simeq}&\, & H_*(X;i_1^*\varphi^*\calf)\ar[u]^-{\Psi_{i_1^*\varphi^*\cale}}_-{\simeq}}$$
Noting that $\varphi\circ i_j=\varphi_j$ we infer that the two upper triangles are commutative. The lower rectangle is also commutative since it concerns the  smooth homotopic maps $i_0$ and $i_1$. We point out that  here we are in the framework of Remark \ref{rmk:identification-different-basepoints}, the identification isomorphism being defined for maps which do not preserve the basepoint. Also, we used  Remark \ref{identical-identifications} which asserts that the isomorphism $\Psi$ above  is the same as the one defined for the homotopic maps $\varphi_0$ and $\varphi_1$.
Therefore the exterior part of the diagram  diagram is also commutative, and the lemma is proved.

\end{proof} 

We complete now the proof of Step 3. Let $\varphi:X\ri Y$ be a continuous map. The direct map $\varphi_*$ is defined by \eqref{def:direct-for-continuous}  using $\ol\varphi:X\ri Y$,  a smooth approximation of $\varphi$. Denote by $E_X$ resp $\ol E_X$ the total spaces of the pullback fibrations $\varphi^*\cale$ resp. $\ol\varphi^*\cale$. Consider the following diagram: 
$$\xymatrix{H_*(E_X)\ar[r]^-{\widetilde{\varphi}_{*}}&H_*(E_Y)&H_*(\ol E_X)\ar[l]_-{\widetilde{\ol\varphi}_{*}}\\
H_*(X;\varphi^*\calf)\ar[r]^-{\varphi_{*}}\ar@/_2pc/[rr]^-{\Psi}_-{\simeq}\ar[u]^-{\Psi_{\varphi^*\cale}}_-{\simeq}&H_*(Y;\calf)\ar[u]^-{\Psi_\cale}_-{\simeq}& H_*(X;\ol\varphi^*\calf)\ar[u]^-{\Psi_{\ol\varphi^*\cale}}_-{\simeq}\ar[l]_-{\ol\varphi_{*}}}$$
This diagram looks like  \eqref{eq:particular-case} for the homotopic maps $\varphi$ and $\ol\varphi$ but now our goal is to prove that the left square is commutative. We know that the right square is commutative by Step 2, that the lower part of the diagram is commutative by  the definition \eqref{def:direct-for-continuous} of $\varphi_*$, and finally that the outer part of the diagram is commutative, by the previous lemma. We infer thus the commutativity of the left square which was our purpose \eqref{eq:direct-for-fibrations} and the proof of Proposition~\ref{prop:direct-for-fibrations} is now complete. 
\end{proof} 

 \section{The shriek map for Hurewicz fibrations}\label{sec:shriek-for-fibration}

Let $Y$ be a closed manifold and $X\subset Y$ a closed submanifold, let $\varphi:X\hookrightarrow Y$ be the inclusion, and assume that both $X$ and $Y$ are oriented (or, more generally, that $X$ is cooriented in $Y$). Let $F\to E\stackrel\pi\to Y$ be a Hurewicz fibration and denote $\tilde\varphi:E|_X\hookrightarrow E$ the inclusion. 

In this situation there is a ``topological" shriek map 
$$
\tilde\varphi_!:H_*(E)\to H_{*+\dim(X) -\dim(Y)}(E|_X)
$$
defined as follows: choose a tubular neighborhood $U$ of $X$ in $Y$ and consider the Thom isomorphism $\cap\, \tau : H_*(U,\p U)\stackrel{\simeq}{\longrightarrow} H_{*+\mathrm{dim}(X)-\mathrm{dim}(Y)}(X)$ defined by the cap product with the Thom class. Define $\tilde\varphi_!$ as the composition
$$\xymatrix{H_*(E)\ar[r]& H_* (E, E|_{Y\setminus U})\ar[r]^-{\mathrm{exc}}& H_*(E|_U,E|_{\p U})\ar[r]^-{\cap\, \pi^*\tau}& H_{*+\mathrm{dim}(X)-\mathrm{dim}(Y)}(E_X),}$$
where $\mathrm{exc}$ is the excision isomorphism. 

The following result is proved by Riegel~\cite{Riegel} and plays an important role in his approach to string topology operations in the framework of Morse homology with DG coefficients. 

\begin{proposition}[{Riegel~\cite{Riegel}}] \label{prop:shriek-for-fibrations} The shriek maps 
$$
\varphi_!:H_*(Y;C_*(F))\to H_{*+\dim(X)-\dim(Y)}(X;\varphi^*C_*(F))
$$
and 
$$
\tilde\varphi_!:H_*(E)\to H_{*+\dim(X) -\dim(Y)}(E|_X)
$$
are identified via the isomorphisms from Theorem~\ref{thm:fibration} and Corollary~\ref{cor:pullback}. \qed
\end{proposition}

The next result can be understood as a corollary of Proposition~\ref{prop:shriek-for-fibrations}. We use it in~\cite{BDHO-cotangent} and give here an independent proof. 

\begin{lemma} \label{lem:shriek-to-point} Let $F\hookrightarrow E\to Y$ be a fibration with base a connected and oriented $n$-manifold with boundary. Let $Y_{n-1}$ be the $(n-1)$-skeleton of $Y$ for some CW-decomposition relative to the boundary $\p Y$, with a single $n$-dimensional cell. We have a canonical isomorphism $H_*(E,E|_{Y_{n-1}})\simeq H_{*-n}(F)$, and the composition
$$
H_*(E)\to H_*(E,E|_{Y_{n-1}})\to H_{*-n}(F)
$$
is the shriek map induced by the inclusion $\star \hookrightarrow Y$. 
\end{lemma}

\begin{proof}
Denote $e^n$ the unique $n$-cell in the CW-decomposition, which possesses a canonical orientation determined by the orientation of $Y$. Then 
$$
H_*(E,E|_{Y_{n-1}})\simeq H_*(E|_{e^n},E|_{\p e^n})\simeq H_*(F\times (e^n,\p e^n))\simeq H_{*-n}(F).
$$ 
The first isomorphism holds by excision, the second isomorphism holds because $E$ is trivial over $e^n$, and the third one is the K\"unneth isomorphism. 

The composition $H_*(E)\to H_*(E,E|_{Y_{n-1}})\to H_{*-n}(F)$ clearly does not depend on the choice of CW-decomposition as in the statement. We choose it to be given by a Morse function $f$ with a unique maximum $M$ at the basepoint. It follows from the definition that, given any $\alpha\in C_*(F)$, the chain level shriek map for the inclusion $\star\hookrightarrow Y$ sends $\alpha\otimes M$ to $\alpha$, and all the other DG Morse chains $\alpha\otimes p$, $p\in\Crit(f)$, $p\neq M$, to $0$. In other words the chain level shriek map can be described as the following composition: first project $C_*(F)\otimes\langle\Crit(f)\rangle$ onto $C_*(F)\otimes \langle M\rangle$, then send $\alpha\otimes M$ to $\alpha$. Analyzing the proof of the Fibration Theorem~\ref{thm:fibration}, one sees that this composition is identified in homology with $H_*(E)\to H_*(E,E|_{Y_{n-1}})\stackrel\simeq\longrightarrow H_{*-n}(F)$. 
\end{proof}

\rmk\label{rmk:shriek-fundamental-class} The proof of Lemma~\ref{lem:shriek-to-point} shows that, if $E$ is smooth finite dimensional and orientable, then the DG shriek map of the inclusion $\star\hookrightarrow Y$ is identified with the usual shriek map of the inclusion $F\hookrightarrow E$. In particular it maps the fundamental class $[E,\p E]$ to the fundamental class $[F]$. \kmr

\chapter{Functoriality: second definition} \label{sec:second-definition}

We develop in this chapter our second approach for the construction of direct and shriek maps, as outlined in~\S\ref{sec:sketch-two-constructions}.

\section{Direct map induced in homology}\label{sec:funct-direct-alex}

We construct in this section the direct map $\varphi_*$ associated to a smooth map $\varphi:X\to Y$. The generalization to the case of continuous maps works as in~\S\ref{sec:funct-for-continuous}.




Denote by $\star_X\in X$, $\star_Y\in Y$ the basepoints.  We work under the standing assumption $\varphi(\star_X)=\star_Y$. Let $\Xi^X$ and $\Xi^Y$ be sets of data on $X$ and $Y$ in these basepoints.   We construct the direct map at the level of complexes $\varphi_*: C_*(X, \Xi^X;\varphi^*\calf)\ri C_*(Y,\Xi^Y;\calf)$ under the following transversality hypothesis: 

Denote $f$ and $g$ the chosen Morse functions on $X$, respectively $Y$. For all $x\in \Crit(f)$ and $y'\in \Crit(g)$ we have 
$$
\varphi|_{W^u(x)} \pitchfork W^s(y').
$$
This can be achieved  by perturbing the pseudo-gradient vector field on $Y$, as in the proof of the classical theorem of Smale (which corresponds to the case $\varphi=\Id$).

Under this assumption we define 
$$
\cM^\varphi(x,y') = W^u(x)\cap \varphi^{-1}(W^s(y')) \cong W^u(x) \, _\varphi\!\times  W^s(y'),
$$
where the second notation stands for the fiber product of $W^u(x)$ and $W^s(y')$.  
Because the intersection is transverse this is a smooth manifold of dimension 
$$
\dim\, \cM^\varphi(x,y')=|x|-|y'|. 
$$
We consider its  compactification $\ol \cM^\varphi(x,y')\cong \overline{W}^u(x) \, _\varphi\!\times \overline{W}^s(y')$. In this notation for the fiber product we omit the inclusions of the stable and unstable manifolds into $X$ for the sake of simplicity.  This is a manifold with boundary and corners such that 
$$
\p \ol \cM^\varphi(x,y') = \bigcup_y \ol \cL(x,y)\times \ol\cM^\varphi(y,y') \cup \bigcup_{x'} \ol\cM^\varphi(x,x')\times\ol \cL(x',y').
$$

The manifold $\ol \calm^\varphi(x,y')$ is also oriented. Using our convention \eqref{eq:transverse-orientation} 
from \S\ref{subsec: orientation-conventions} the orientation of its interior  is given by 
\begin{equation}\label{eq:orientation-direct-first} \left(\ori\, \calm^\varphi(x,y'), \coori \,  W_g^s(y')\right)\, =\, \ori \,  W_f^u(x)\end{equation} 
and we extend this orientation to the boundary. 
One may also write 
\begin{equation}\label{eq:orientation-direct-second} \left(\ori\, \calm^\varphi(x,y'), \ori\,   W_g^u(y')\right)\, =\, \ori \,  W_f^u(x)\end{equation} 

In order to construct  a representing chain system from this data we need to investigate the difference between the product orientation and the boundary orientation of each of these two types of boundary strata. We respectively have: 
\begin{lemma}\label{lem:orientation-difference-direct} The orientation differences are given by the following signs: 
$$
\ori \, \, \partial\ol\calm^\varphi(x,y') \, =\, (-1)^{|x|-|y|} \left(\ori\, \ol\call_f(x,y), \ori\, \ol \calm^\varphi(y,y')\right)
$$
and 
$$
\ori\, \, \partial\ol\calm^\varphi(x,y')\, =\, (-1)^{|x|-|y'|-1}\left(\ori\, \ol\calm^\varphi(x,x'), \ori\, \ol\call_g(x',y')\right).
$$
\end{lemma} 
\begin{proof}
We prove the first relation. 

Consider a boundary  point of the form  $(\lambda, \ol\lambda)\in \ol\call_f(x,y) \times\ol \calm^\varphi(y,y')$. Near this point we might view $\ol\calm^\varphi(x,y')$ as a subset of $\ol W^u_f(x)$ via the transverse intersection $\varphi|_{\ol W^u_f(x)}\pitchfork W^s(y')$ (we omit again the inclusions of the Latour cells in $X$ for simplicity). Writing $\ol\lambda= (\lambda_a,\lambda'_{\varphi(a)})\in \ol\calm^\varphi(y,y')=\ol W_f^u(y) \, _\varphi\!\times \ol W_g^s(y')$ (where the half-infinite broken trajectory $\lambda_a$ ends at $a\in X$ and the half-infinite broken trajectory $\lambda'_{\varphi(a)}$ starts at $\varphi(a)\in Y$), the point $(\lambda,\ol\lambda)$ corresponds thus to $(\lambda,\lambda_a)\in \ol W^u_f(x)$ in this identification. 

Now let $n$ be the outward normal vector at this point. Our orientation convention \eqref{eq:exterior-normal} yields 
\begin{equation}\label{eq:normal-orientation-direct} 
\left( n, \ori \, (\partial\ol\calm^\varphi(x,y'))\right)\, =\, \ori\, \ol\calm^{\varphi}(x,y').
\end{equation}
Obviously, we may also see $n$ as the outward normal of $\ol W^u_f(x)$ at $(\lambda, \lambda_a)$. We therefore  have  
$$
\left(n,\ori\, \partial \ol W_f^u(x) \right)\, =\,  \ori\, \ol W_f^u(x).
$$

\begin{figure}
  \centering
  \includegraphics{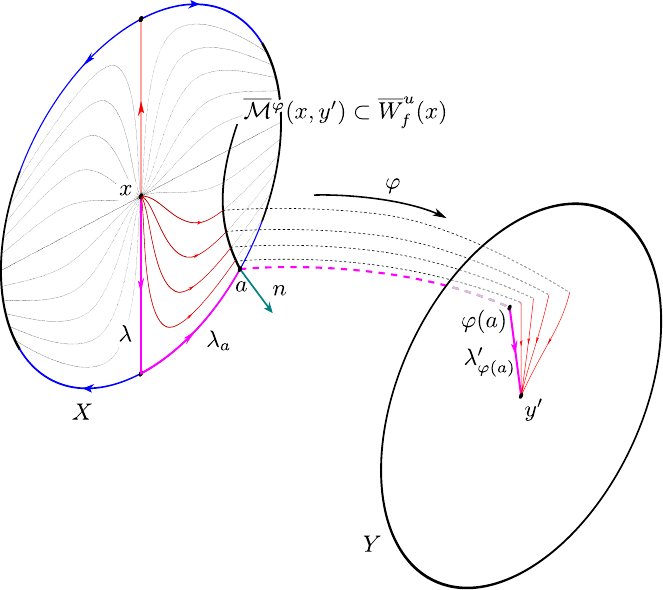}
  \caption{
    The exterior normal of $\ol\calm^\varphi(x,y')$ at
    $(\lambda, (\lambda_a,\lambda'_{\varphi(a)}))$, seen in $\ol W_f^u(x)$.
  }
  \label{fig:directmap}
\end{figure}
%

Now recall the relation \eqref{eq:sign-difference-bord-cellule}   from \S\ref{subsec:Latour-cells} 
$$\ori\, \, \partial\ol W^u_f(x)\, =\, (-1)^{|x|-|y|}\left(\ori\, \ol\call_f(x,y),\ori\, \ol W^u_f(y)\right),$$
which implies here
$$\left(n, \ori\, \ol\call_f(x,y),\ori\, \ol W^u_f(y)\right)\, = (-1)^{|x|-|y|}\ori\, \ol W^u_f(x).$$
Using the extension to the boundary of \eqref{eq:orientation-direct-first} in  both sides we infer 
\begin{align*}\big(n, \ori\, \ol\call_f(x,y),\ori \, & \ol \calm^\varphi (y, y'), \coori\,  W^u_g(y')\big)\\
&= \, (-1)^{|x|-|y|}\left(\ori\, \ol\calm^\varphi(x,y'), \coori\, W^u_g(y')\right)\end{align*} 
which by removing   $\coori \,  W^s_g(y')$ leads to 
$$\left( n, \ori \, \ol\call_f(x,y),\ori\, \ol\calm^\varphi(y,y')\right)\, =\, (-1)^{|x|-|y|}\ori\, \ol\calm^{\varphi}(x,y').$$
Combined with~\eqref{eq:normal-orientation-direct} this gives the first relation of our statement. 

To prove the second relation, denote again by $n$  the outward normal at some point $(\ol\lambda, \lambda')\in \ol\calm^\varphi(x,x')\times \ol\call_g(x',y')$. Relation~\eqref{eq:normal-orientation-direct} is of course still valid. However since  we cannot consider $\ol\calm^\varphi(x,y')$ as a subset of $\ol W^u_f(x)$ near this boundary point we have to describe $n$ more  precisely.

With the same notation $\ol\lambda=(\lambda_a,\lambda'_{\varphi(a)})$, the point $(\lambda,\ol\lambda)$ corresponds to $(a,\varphi(a))\in \ol\calm^\varphi(x',y)$ seen as $ \ol W_f^u(x) \, _\varphi\!\times \ol W_g^s(y')$. Therefore we may consider $n=\gamma'(0)$ where $\gamma=(\gamma_X,\gamma_Y): (-\epsilon,0]\ri  \ol W_f^u(x) \, _\varphi\!\times \ol W_g^s(y')$ is a path pointing outwards the boundary such that $\gamma(0)=(a,\varphi(a))$ ; note that for $t<0$ we have $\varphi(\gamma_X(t))=\gamma_Y(t)$.  Since only the orientation of the interior of  $\ol\calm^\varphi(x',y)$ was defined explicitly defined (by \eqref{eq:orientation-direct-first})  we prefer to write \eqref{eq:normal-orientation-direct} at a interior point $\gamma(t)$ using a diffeomorphic copy of $\partial \ol\calm^\varphi(x,y')$ in a collar neighbourhood. Near such a point we may view $\ol\calm^\varphi(x,y')$ as a subset of $\ol W_f^u(x)$ and  use $\gamma_X'(t)$ instead of $n$. Therefore \eqref{eq:normal-orientation-direct} combined with \eqref{eq:orientation-direct-first} yields in this point 
\begin{align*}\big(\gamma_X'(t), \ori  \, \partial \ol\calm^\varphi(x,y'),& \coori\, W^s_g(y')\big)\\
&=\left(\ori\, \calm^\varphi(x,y'), \coori \, \ol W_g^s(y')\right)=\ori \,  W^u_f(x)
\end{align*}
which  we write equivalently 
\begin{align*}\big( \ori  \, \partial \ol\calm^\varphi(x,y'), & \gamma_X'(t),\coori\, W^s_g(y')\big)\\
&=(-1)^{|x|-|y'|-1}\left(\ori\, \calm^\varphi(x,y'), \coori \, \ol W_g^s(y')\right)\\
& =(-1)^{|x|-|y'|-1} \ori \,  W^u_f(x).
\end{align*}

Now we may replace $\gamma_X'(t)$ by $\gamma_Y'(t)=\varphi_*(\gamma_X'(t))$ in the relation above ; this is because  in the transverse intersection $W^u_f(x)\pitchfork W^s(y')$ we make use of the isomorphism $\varphi_*: TW^u_f(x)/ T\calm^\varphi(x,y')\cong TY/ TW_g^s (y')$ to write  the orientation rule \eqref{eq:orientation-direct-first}.  
 Remark that by construction the  vector $\ol n=\gamma_Y'(0)$ can be viewed as the exterior normal to the stable manifold $\ol W_g^s(y')$ at $(\lambda'_{\varphi(a)}, \lambda')$. Remark also that arguing as in the proof of Proposition~\ref{orientationsMorse}  at the boundary point $(\lambda'_{\varphi'(a)}, x')$ of $\ol W^s_g(y')$, the normalized gradient vector $\frac{\xi}{\|\xi \|}$ (extended by continuity at that point) is directed {\it inwards}. Therefore, we may take $\ol n  =-\frac{\xi}{\|\xi\|}$ at this point. Now on the one hand we have by the above 
 \begin{align*}\big( \ori  \, \partial \ol\calm^\varphi(x,y'), & \gamma_Y'(t),\coori\, W^s_g(y')\big)\\
&=(-1)^{|x|-|y'|-1}\left(\ori\, \calm^\varphi(x,y'), \coori \, \ol W_g^s(y')\right)\\
& =(-1)^{|x|-|y'|-1} \ori \,  W^u_f(x).
\end{align*} 
and on the other hand, by putting $t=0$ we get 
\begin{align*}
\big(\ori\, & \calm^{\varphi}(x,x'), \ori\, \call_g(x',y'), \gamma_Y'(t), \coori \, W^s_g(y')\big) \\
& =\,  \left( \ori\, \calm^{\varphi}(x,x'), \ori\, \call_g(x',y'), \ol n,  \coori \, W^s_g(y')\right) \\
& =\,  \left( \ori\, \calm^{\varphi}(x,x'), \ori\, \call_g(x',y'), -\xi, \coori \,  W^s_g(y')\right) \\
& =\, \left( \ori\, \calm^{\varphi}(x,x'), \ori\, \call_g(x',y'), \coori \, S^s_g(y')\right)\\
& =\,  \left( \ori\, \calm^{\varphi}(x,x'), \ori\, \ W_g^u(y')\right) \\
& =\,  \ori\, \ol W^u_f(x).
\end{align*}
In the second equality we changed the point in $\ol W^s(y')$ from  $(\lambda'_{\varphi(a)}, \lambda')$ to $(\lambda'_{\varphi'(a)}, x')$ so that the exterior normal becomes    $\ol n=-\xi$ (we may omit the norm which has no effect on the orientations). The third equality uses our coorientation  convention for stable spheres~\eqref{eq:coorientation-sphere-stable} and the relation  \eqref{eq:multiliaisons-intersection}. The last equality uses~\eqref{eq:orientation-direct-second}.

We get a sign difference of $(-1)^{|x|-|y'|-1}$ between the two computations,  which gives exactly the second relation of Lemma~\ref{lem:orientation-difference-direct}. 
\end{proof}

Proceeding inductively as  in Proposition~\ref{representingchain} we get a representing chain system $(\tau_{x,y'})$ for $\ol\calm^\varphi(x,y')$ which satisfies the following equation : 
$$
\p\tau_{x,y'}\, =\, (-1)^{|x|-|y|}\sum_{y\in \crit(f)}s_{x,y}\times \tau_{y,y'}-(-1)^{|x|-|y'|}\sum_{x'\in \crit(g)} \tau_{x,x'}\times s_{x',y'}.
$$ Instead of it we would have liked to have an  equation which is similar to~\eqref{eq:newinvarianceforchains}, in order to define a morphism between twisted complexes. For this purpose we set
\begin{equation}\label{eq:correction-rcs}
\sigma_{x,y'}\, =\, (-1)^{|x|-|y'|-1}\tau_{x,y'}.
\end{equation} 

It is then easy to check that $(\sigma_{x,y'})$ indeed satisfies an equation of the form \eqref{eq:newinvarianceforchains}, i.e., 
$$
\p\sigma_{x,y'}\, =\, \sum_{y\in \crit(f)}s_{x,y}\times \sigma_{y,y'}-\sum_{x'\in \crit(g)} (-1)^{|x|-|x'|}\sigma_{x,x'}\times s_{x',y'}.
$$
In the sequel we will refer to $(\sigma_{x,y'})$ as a \emph{corrected representing chain system on $\ol\calm^\varphi(x,y')$}.

 Now we construct continuous evaluation  maps 
 $$
 q^\varphi_{x,y'}: \ol \cM^\varphi(x,y')\to \Omega Y.
 $$ 
 For this purpose we take two homotopies $H^X: [0,1]\times X\ri X$ and $H^Y:[0,1]\times Y\ri Y$ which respectively join $\Id_X$ and $\Id_Y$ to the compositions $\theta^X\circ p^X$ and $\theta^Y\circ p^Y$ (recall that $\theta$ is a homotopy inverse of the projection $p$ in a set of data $\Xi$).  We may and  will construct these homotopies using a similar procedure on $X$ and on $Y$ : take a homotopy of trees $\caly_t$ between the base point $\star$ and the given tree $\caly$,   
 let   $p_t$ be the projection which collapses the tree  $\caly_t$  and choose homotopy inverses $\theta_t$ of $p_t$  for $t\in [0,1]$ such that $\theta_0=\Id$ and $\theta_1 =\theta$. The homotopy $H_t$ is then defined as $\theta_t \circ p_t$. We will use this definition in the proof of Proposition~\ref{id=continuation} below. 
 
 We define the evaluation maps $q^\varphi_{x,y'}$ as follows. We view an element of $\ol \cM^\varphi(x,y')$ as consisting of a half-infinite broken  trajectory $\lambda$  in $X$ flowing from the critical point $x$ to a certain point $a$, coupled with a half-infinite  broken trajectory $\lambda'$ in $Y$ starting at the point $\varphi(a)$ and flowing into $y'$. We start from the evaluations $q^X$  and $q^Y$ defined in Lemma~\ref{lecture}  (see also Lemma~\ref{lecture-Latourcells} in~\S\ref{sec:proof-of-fibration}).  Note that, by definition,   $q^X(\lambda)\in \calp_{\star\ri \theta^X\circ p^X(a)}X$ and $q^Y(\lambda')\in \calp_{\theta^Y\circ p^Y(\varphi(a))\ri \star}Y$. We define 
\begin{equation} \label{eq:qx, yprimephi}
q_{x,y'}^\varphi:\ol \calm^\varphi(x,y')\ri \Omega Y
\end{equation}
by 
\begin{equation}\label{eq:formula-evaluation-direct}
q_{x,y'}^\varphi(\lambda, \lambda')\, =\,  \varphi(q^X(\lambda)) \#\varphi(H^X(1-t, a))\#H^Y(t,\varphi(a))\#q^Y(\lambda').\end{equation} 


In other words we use the homotopies $\ol{H}^X=H^X(1-\cdot,\cdot)$ and $H^Y$ to connect the endpoint of $\varphi(q^X(\lambda))$ with the starting point of $q^Y(\lambda')$, getting thus a loop in $Y$. See Figure~\ref{fig:Mphixyprime} below. 

\begin{figure}
  \centering
  \includegraphics[width=\textwidth]{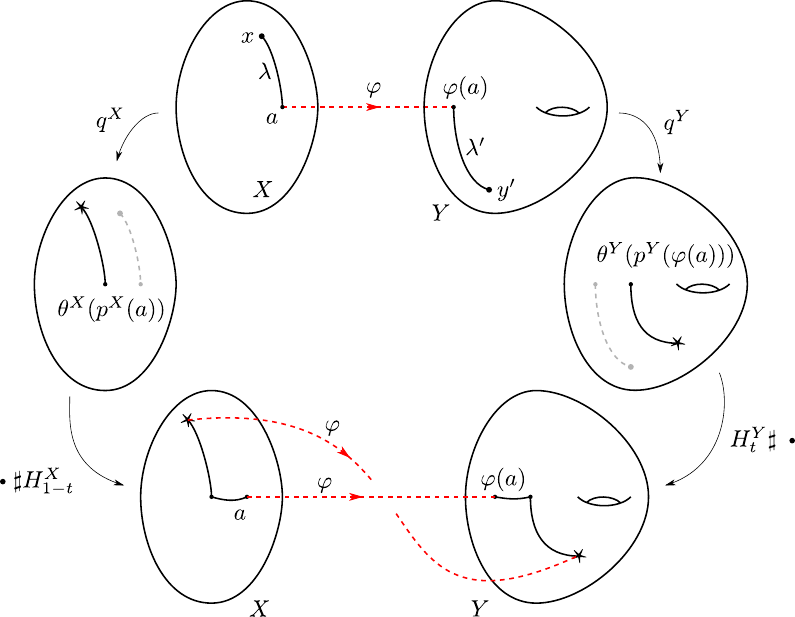}   
  \caption{Definition of the evaluation map $q_{x,y'}:\ol\calm^\varphi(x,y')\ri \Omega Y$.}
  \label{fig:Mphixyprime}
\end{figure}
%

Note that on the boundary $\partial\ol\calm^\varphi(x,y')$ the maps $q_{x,y'}^\varphi$ satisfy
$$q_{x,y'}^\varphi(\lambda,\ol\lambda) \, =\, \varphi(q_{x,y}^X(\lambda))\# q_{y,y'}^\varphi(\ol\lambda)$$
for all $(\lambda,\ol\lambda) \in \ol \call(x,y)\times \ol \calm^\varphi(y,y')$ and
$$q_{x,y'}^\varphi(\ol\lambda,\lambda') \, =\, (q_{x,x'}^\varphi(\ol\lambda))\# q^Y_{x',y'}(\lambda')$$ 
for all $(\ol\lambda,\lambda') \in \ol \calm^\varphi(x,x')\times \ol \call(x',y')$.
Taking into account these properties, the construction of a representing chain system for the compactified moduli spaces $\ol\cM^\varphi(x,y')$ leads via the evaluations $q^\varphi_{x,y'}$  to a collection of chains 
$$
\nu_{x,y'}=-q^\varphi_{x,y',*}(\sigma_{x,y'})\in C_{|x|-|y'|}(\Omega Y)
$$
such that 
$$
\p \nu_{x,y'}= \sum_y (\Omega \varphi)_*(m_{x,y})\nu_{y,y'} - \sum_{x'}(-1)^{|x|-|x'|}\nu_{x,x'}m_{x',y'}.
$$
The reason for inserting a minus sign in the definition of $\nu_{x,y'}$ lies in our orientation conventions and will be clear in the proof of Proposition~\ref{id=continuation}. At this stage one may notice that when $\varphi=\Id$,  the $0$-chain  $\tau_{x,x}$ of the initial  representing chain system has positive sign  since the orientation rule \eqref{eq:orientation-direct-second} writes 
$$\left(\ori \, \ol \calm^\Id(x,x), \ori \, \ol W^u_f(x)\right)\, =\, \ori \, \ol W^u_f(x)$$ Therefore the corrected $0$-chain $\sigma_{x,x}$  has negative sign and as for the continuation maps (see  Remark~\ref {change-of-sign}) we use the evaluations to turn this minus sign into plus.   

We now define 
\begin{equation}\label{eq:direct-second-definition}
\varphi_* : \cF\otimes \Crit(f)\to \cF\otimes \Crit(g), \qquad \varphi_*(\alpha\otimes x) = \sum_{y'} \alpha\cdot \nu_{x, y'}\otimes y'.
\end{equation}

A direct verification shows that this is a chain map, see also~\S\ref{sec:invariance-homology}. This map depends \emph{a priori} on the representing chain system for $\ol \calm^\varphi(x,y')$ and on the homotopies $H^X$ and $H^Y$. We will prove independence in homology with respect to these choices, as well as the compatibility with the continuation maps~\eqref{welldefined-direct}, in Corollary~\ref{cor:alex=mihai} below. Our strategy is to show that the direct map defined above is the same as the one defined in~\S\ref{sec:functoriality-first-definition}, so that it inherits all the properties of the latter.

\rmk\label{rmk:direct-for-boundary} The definition of $\varphi_*$ generalises to the case when $X$ and $Y$ have boundary. We need to take a (negative) gradient on $X$ pointing inwards in order to prevent $\calm^\varphi(x,y')$ to contain boundary points of $X$, which would change  the compactification of this moduli space. Therefore $\varphi_*$ is defined on the absolute homology $H_*(X;\varphi^*\calf)$. We also suppose $\varphi(\mathrm{int}\, X)\subset \mathrm{int}\, Y$, so that $\varphi(W^u(x))\cap \p Y=\emptyset$   and in particular $\calm^\varphi(x,y')$ does not intersect the preimage of a neighbourhood of $\p Y$, which again implies that $\ol\calm^\varphi(x,y')$ is the same as in the case without boundary. One may also use the first definition of  the direct map given in \S\ref{sec:functoriality-first-definition} to  generalise it under the same hypothesis. The  map $\varphi_*$ may take values in $H_*(Y,\p Y;\calf)$ but it factors as $H_*(X;\varphi^*\calf)\ri H_*(Y;\calf)\ri H_*(Y, \p Y;\calf)$. 
\kmr

\section{The two definitions of direct maps coincide} \label{sec:two-defi-equivalent}

We prove in this section that the two definitions of the direct map from~\S\ref{sec:functoriality-first-definition} and~\S\ref{sec:funct-direct-alex} coincide, i.e., the maps induced in homology are the same, see Proposition~\ref{alex=mihai}. Throughout this section we will denote the first map by  $\ol\varphi_*$ and the second one  by $\varphi_*$. 

The key result towards the proof is the following:

\begin{proposition} \label{id=continuation}  
Let  $\Xi_0$ and $\Xi_1$ be two sets of data defined on $(X,\star)$ and $\calf$ a DG-module over $C_*(\Omega X)$. The map $\Id_*: C_*(X,\Xi_0;\calf)\ri C_*(X,\Xi_1;\calf)$ (defined by the second definition) 
does not depend on choices at homology level, and it coincides with the continuation map $\Psi_{01}$.
\end{proposition} 

\begin{proof} We have proved in~\S\ref{sec:invariance-continuation} that the continuation map $\Psi_{01}$ does not depend in homology on the choice of continuation data. It is therefore enough to show that, for any choice of the construction data (representing chain system on $\ol\calm^\Id(x,y')$ and homotopies), the map $\Id_*: C_*(X,\Xi_0;\calf)\ri C_*(X,\Xi_1;\calf)$ is chain homotopic to the continuation map $\Psi_{01}$ defined using a specific choice of continuation data. 

Let us describe this choice of continuation data. We begin the description with the choice of the Morse function and negative gradient on $[-\epsilon,1+\epsilon]\times X$. Let $\rho:\R\to [0,1]$ be a smooth increasing cutoff function such that $\rho(s)=0$ for $s\le \frac12 + \eps$ and $\rho(s)=1$ for $s\ge 1-\eps$ for some small enough $\eps>0$. We moreover suppose that $\rho$ is strictly increasing on $[\frac{1}{2}+\eps, 1-\eps]$. Denote by $(f_i,\xi_i)$ the Morse-Smale pairs corresponding to $\Xi_i$, $i=0,1$, and consider the homotopies 
$$
f_t = \left\{\begin{array}{ll} \rho(1-t) f_0, & t\in [-\eps,\frac12],\\
\rho(t)f_1, & t\in[\frac12,1+\eps],
\end{array}\right.
$$
$$
\xi_t=\left\{\begin{array}{ll} \rho(1-t) \xi_0, & t\in [-\eps,\frac12],\\
\rho(t)\xi_1, & t\in[\frac12,1+\eps].
\end{array}\right.
$$

Note that both $f_t$ and $\xi_t$ vanish for $t\in [\frac{1}{2}-\epsilon, \frac{1}{2}+\epsilon]$.
As in~\S\ref{sec:invariance-homology}, let $g:[-\eps,1+\eps]\to [0,1]$ be a smooth Morse function which has exactly two critical points, a nondegenerate maximum at $0$ and a nondegenerate minimum at $1$, and which is therefore strictly decreasing on $(0,1)$. We consider on $[-\eps,1+\eps]\times X$ the Morse function $F(t,x)=f_t(x)+g(t)$ and the negative pseudo-gradient $\xi(t,x)=\xi_t(x)-g'(t)\p_t$. We define the map $\Psi_{01}$ using the pair $(F,\xi)$. 

We claim that, for all $x\in\Crit(f_0)$, $y\in\Crit(f_1)$, we have a homeomorphism of moduli spaces 
\begin{equation} \label{eq:LF-MId}
\cL_F(x,y)\simeq  \cM^{\Id}(x,y). 
\end{equation}

As a preparation for the proof of this claim, recall the following general fact: let $v$ be a vector field on some manifold, $h$ be a diffeomorphism of intervals on the real line, and $\alpha,\beta$ be parametrized curves in that manifold related by $\alpha(s)=\beta(h(s))$. Then $\beta$ solves $\dot \beta(s)=v(\beta(s))$ if and only if $\alpha$ solves $\dot\alpha(s)=h'(s)v(\alpha(s))$. In other words, $\beta$ is an integral curve of $v$ if and only if $\alpha$ is an integral curve  of $h'v$. 

We now construct the homeomorphism~\eqref{eq:LF-MId}. Let $\tilde \gamma=(a,\gamma):\R\to [-\eps,1+\eps]\times X$ be a representative of an equivalence class in $\cL_F(x,y)$. 
\begin{itemize} 
\item The component $a$ solves $a'(s)=-g'(a(s))$ and $\lim_{s\to -\infty} a(s)=0$, $\lim_{s\to+\infty} a(s)=1$, hence $a(s)\in (0,1)$ for all $s\in\R$ and $a$ is strictly increasing. As a consequence, there exist unique real numbers  $\bar s_-<\bar s< \bar s_+$ such that $a(\bar s)=\frac 1 2$ and $a(\bar s_{\pm})= \frac 1 2 \pm \eps$. \item The component $\gamma$ solves 
$$
\dot\gamma(s)=\left\{\begin{array}{ll} \rho(1-a(s))\xi_0(\gamma(s)), & s\le \bar s,\\
\rho(a(s))\xi_1(\gamma(s)),& s\ge \bar s.
\end{array}\right.
$$
\end{itemize}

Note that $\gamma$ is constant on the interval $[\bar s_-,\bar s_+]$ and denote this constant by $c\in X$.  Let $h_0:(-\infty,\bar s_-]\to (-\infty,0]$ be the diffeomorphism defined as the unique primitive of the (nonzero)  function $s\mapsto \rho(1-a(s))$ such that $h_0(\bar s_-)=0$. Let $h_1:[\bar s_+,\infty)\to [0,\infty)$ be the diffeomorphism defined as the unique primitive of the function $s\mapsto \rho(a(s))$ such that $h_1(\bar s_+)=0$. 
Then 
\begin{equation}\label{eq:gamma}
\gamma(s)=\left\{\begin{array}{ll} \alpha_0(h_0(s)), & s\le \bar s_-,\\
c, & s\in [\bar s_-,\bar s _+],\\
\alpha_1(h_1(s)),& s\ge \bar s_+
\end{array}\right.
\end{equation} 
for some uniquely determined parametrized curves $\alpha_0:(-\infty,0]\to X$, $\alpha_1:[0,\infty)\to X$ such that $\dot\alpha_0=\xi_0(\alpha_0)$, $\dot\alpha_1=\xi_1(\alpha_1)$, $\lim_{s\to-\infty}\alpha_0(s)=x$, $\lim_{s\to\infty}\alpha_1(s)=y$. By construction we have $\alpha_0(0)=\alpha_1(0)=c$, so that $(\alpha_0,\alpha_1)$ defines an element in $\cM^{\Id}(x,y)$. 

The map~\eqref{eq:LF-MId} associates to $[\tilde \gamma]$ the pair $(\alpha_0,\alpha_1)$. It is easy to check that the map is independent of the choice of representative $\tilde\gamma$. The inverse map associates to a pair $(\alpha_0,\alpha_1)$ the class $[\tilde\gamma]$ with $\tilde\gamma$ defined as above, and both maps are continuous.

Moreover, the homeomorphism~\eqref{eq:LF-MId} extends to a homeomorphism between the compactified moduli spaces 
\begin{equation} \label{eq:LF-MId-bar}
R:\ol{\cL}_F(x,y)\simeq  \ol{\cM}^{\Id}(x,y)
\end{equation}
such that the following compatibility conditions are satisfied:
\begin{equation}\label{eq:R1} 
R(\lambda ,\bar \lambda)\, =\, (\lambda, R(\bar \lambda))
\end{equation} 
 for any $(\lambda, \bar\lambda )\in \ol \call_{f_0}(x,z)\times \ol\call_F(z,y)$, and 

\begin{equation}\label{eq:R2} 
R(\bar \lambda,  \lambda)\, =\, (R(\bar \lambda), \lambda)
\end{equation} 
 for any $( \bar\lambda, \lambda )\in  \ol\call_F(x,w)\times \call_{f_1}(w,y)$. 
 
 Starting with a representing chain on $\ol\calm^\Id(x,y)$ we want to use $R^{-1}$ to transport it into one on $\ol\call_F(x,y)$. We have to take into account the fact that in general $R$ does not preserve the orientations. More precisely : 
 
 \begin{lemma}\label{lem:R-orientation} The homeomorphism $R:\ol\call_F(x,y)\ri \ol \calm^\Id(x,y)$ changes the orientation by the sign $(-1)^{|x|-|y|-1}$.
 \end{lemma} 
 
 \begin{proof} According to \eqref{eq:orientation-direct-second} the orientation of $\ol\calm^\Id(x,y)$   is given by 
 $$\left(\ori\, \ol\calm^\Id(x,y), \ori \, \ol W^u_{f_1}(y)\right)\, =\, \ori\, \ol W^u_{f_0}(x).$$
 
 Given $x\in \crit(f_0)$ and $y\in \crit(f_1)$, we apply~\eqref{eq:multiliaisons-intersection} to get  
 $$\left(\ori\, \ol \call_F(x,y), \coori \, S^s_F(y)\right) \, =\, \ori \, \ol W^u_F(x).$$
 Denoting as before by $t$ the variable on $[0,1]$, recall the orientation rule~\eqref{eq:orientation-produit} for $x\in \crit(f_0)$: 
$$\ori\, \ol W^u_F(x)\, =\, \left(\tfrac\partial{\partial t}, \ori \, W^u_{f_0}(x)\right).$$ For $y\in \crit(f_1)$ the vector 
 $\tfrac\partial{\partial t}$ points inwards the stable manifold $\ol W^s_F(y)$ along a level of $F$ by construction. Therefore, by our convention 
 \eqref{eq:coorientation-sphere-stable}, we have $$\coori\, S^s_F(y)\, =\, \left(-\tfrac\partial{\partial t}, \coori\, \ol W_F^s(y)\right)$$
 and we infer 
 \begin{align*} 
 \ori\, \ol W^u_F(x) & =\, \left(\ori\, \ol \call_F(x,y), \coori \, S^s_F(y)\right) \\
& = \, \left(\ori\, \ol \call_F(x,y), -\tfrac\partial{\partial t}, \coori \, \ol W^s_F(y) \right) \\
& =\, (-1)^{|x|-|y|-1}\left(\tfrac\partial{\partial t}, \ori\, \ol \call_F(x,y), \coori \, \ol W^s_F(y) \right) \\
& =\, (-1)^{|x|-|y|-1}\left(\tfrac\partial{\partial t}, \ori\, \ol \call_{F}(x,y), \ori \, \ol W^u_{F}(y)\right) \\
& =\, (-1)^{|x|-|y|-1}\left(\tfrac\partial{\partial t}, \ori\, \ol \call_{F}(x,y), \ori \, \ol W^u_{f_1}(y)\right).
\end{align*}

This implies 
$$\left(\ori\, \ol\call_F(x,y), \ori\, \ol W^u_{f_1}(y) \right) \, =\, (-1)^{|x|-|y|-1}\ori\, \ol W^u_{f_0}(x).$$
 The map  $R$ can be viewed as applying the projection  $[0,1]\times X\to X$ to the trajectories of $\ol\call_F(x,y)$. It therefore preserves the vertical tangent  vectors, in particular those of $W^u_{f_0}(x)$ and $W^u_{f_1}(y)$. Denoting by  $\ori\, R(\ol\call_F(x,y))$ the orientation induced by $R$ on the target, we obtain  
 $$\left(\ori\, R(\ol\call_F(x,y)), \ori\, \ol W^u_{f_1}(y) \right) \, =\, (-1)^{|x|-|y|-1}\ori\, \ol W^u_{f_0}(x).$$
Comparing this orientation relation with the above orientation relation for $\ol\calm^\Id(x,y)$ we obtain
 $$\ori\, R(\ol\call_F(x,y))\, =\, (-1)^{|x|-|y|-1}\ori\, \ol\calm^\Id(x,y).$$
This proves the lemma. 
 \end{proof}

If $(\tau_{x,y})$ is the chosen  representing chain system on $\ol \calm^\Id(x,y)$, the previous lemma allows us to choose $(-1)^{|x|-|y|-1}R^{-1}_*(\tau_{x,y})$ as representing chain system on $\ol \call_F(x,y)$. So, if $\sigma_{x,y}=(-1)^{|x|-|y|-1}\tau_{x,y}$ is the corrected representing chain system (see \eqref{eq:correction-rcs}) then $ R^{-1}_*(\sigma_{x,y})$ is a representing chain system on $\ol\call_F(x,y)$ and it is this one that we will use to define the continuation map. 

To this purpose we have to evaluate it by   $q_{x,y}: \overline\call_F(x,y)\ri \Omega X $ from~\eqref{eq:invariancelecture}. To define this map we need a family of trees $\caly_t$ and a family of expanding maps $\theta_t: X/\caly_t\ri X$, i.e., homotopy inverses of the collapsing maps $p:X\ri X/\caly_t$.    Recall that  $q_{x,y}^\Id:\ol \calm^\Id(x,y)\ri \Omega X$,  the evaluation map from~\eqref{eq:qx, yprimephi}, was defined using homotopies  $H^X_0$ and $H^X_1$ between $\Id$ and $\theta_0\circ p_0$, resp. between $\Id$ and $\theta_1\circ p_1$, and these homotopies were constructed by collapsing and expanding paths  of trees joining $\star$ and $\caly_0$, resp. $\star$ and $\caly_1$. Denote by $(\caly^{\tau}_0)$ the homotopy between $\star$ and $\caly_0$ and by $(\theta^\tau_0): X/\caly^\tau_0\ri X$ the expanding maps which yield  the homotopy $H_0^X=(\theta^\tau_0\circ p^\tau_0)$ between $\Id$ and $\theta_0\circ p_0$. We may suppose  that this homotopy is parametrized by $\tau\in [\frac{1}{2}, \frac{1}{2}+\eps]$.  Use the analogous notation $\caly^\tau_{1}$, $p^\tau_1$ and $\theta^\tau_1$ for the maps defining the homotopy $H^X_1$, also parametrized by the interval $\tau \in [\frac{1}{2},\frac{1}{2}+\eps]$. We choose the following family of trees $(\caly_t)$ on $X$:
 $$\caly_t :\, =\, \left\{\begin{array}{ll}
 \caly_0, &t\in [0,\frac{1}{2}-\eps],\\
 \\
 \caly_0^{1- t}, &t\in[\frac{1}{2}-\eps, \frac{1}{2}],\\
 \\
 \caly_1^{t}, &t\in [\frac{1}{2},\frac{1}{2}+\eps],\\
 \\
 \caly_1, & t\in [\frac{1}{2}+\epsilon, 1],
 \end{array}\right.$$
  and accordingly the expanding maps 
 $$\theta_t :\, =\, \left\{\begin{array}{ll}
 \theta_0, &t\in [0,\frac{1}{2}-\eps],\\
 \\
 \theta_0^{1- t}, &t\in[\frac{1}{2}-\eps, \frac{1}{2}],\\
 \\
 \theta_1^{t}, &t\in [\frac{1}{2},\frac{1}{2}+\eps],\\
 \\
 \theta_1, & t\in [\frac{1}{2}+\epsilon, 1].\end{array}\right.$$ 
 By definition (see~\eqref{eq:invariancelecture}) the evaluation map  $q_{x,y} :\ol\call_F(x,y)\ri \Omega X$ is given by 
 $$
 q_{x,y} \, =\, \pi\circ \Theta\circ p \circ \Gamma,
 $$ 
 where  $\Gamma$ is the reparametrization of a trajectory by the values of $F$, the map $p:[0,1]\times X\ri \bigcup_{t\in [0,1]}\{t\}\times X/\caly_t$ is the collapsing map, the map $\Theta$ is its homotopy inverse and $\pi :[0,1]\times X\ri X$ is the canonical projection.

We now compare the maps  $q_{x,y}^\Id\circ R$ and $q_{x,y}$. If they were equal, then the continuation cocycle 
$$\nu_{x,y} \, =\, -q_{x,y,*}(R^{-1}_*(\sigma_{x,y}))\, =\, -q^\Id_{x,y,*}(\sigma_{x,y})$$ would be equal to the cocycle which defines the direct map $\Id_*$ and therefore we would have $\Psi_{01}=\Id_*$ at chain level and the proof would be finished. But in fact these maps are only  homotopic.  More precisely: 

\begin{lemma}\label{lem:homotopic-evaluations-bisbis} (compare to Lemma~\ref{lem:homotopic-evaluations-bis}) There is a homotopy $$q^s:\ol\call_F(x,y)\ri \Omega X$$ between the maps 
$q_{x,y}^\Id\circ R$ and $q_{x,y}$ which satisfies 
 \begin{equation}\label{eq:first-homotopic-condition-bisbis}
q^s(\lambda ,\bar \lambda)\, =\, q^{f_0}_{x,z}( \lambda)\#q^s(\bar \lambda)
\end{equation} 
 for any $(\lambda, \bar\lambda )\in \ol \call_{f_0}(x,z)\times \ol\call_F(z,y)$,  
  and 
\begin{equation}\label{eq:second-homotopic-condition-bisbis}
q^s(\bar \lambda,  \lambda)\, =\, q^s(\bar \lambda)\#q_{w,y}^{f_1}(\lambda)
\end{equation} 
 for any $( \bar\lambda, \lambda )\in  \ol\call_F(x,w)\times \call_{f_1}(w,y)$. 
\end{lemma}

Having only a homotopy instead of an equality suffices however for our purposes. Indeed, we can apply  Lemma~\ref{lem:homotopic-evaluations} to the representing chain system $R_*^{-1}(\sigma_{x,y})$ --- which satisfies an equation of the form~\eqref{eq:newinvarianceforchains} ---, and  to the  evaluation maps $\pi\circ q^{\varphi^\chi}$ and $q^\varphi\circ\Pi: \ol\calm^{\varphi^\chi}\ri \Omega Y$ which are homotopic through a homotopy satisfying~\eqref{eq:concatenation1} and~\eqref{eq:concatenation2} (these correspond to equations~\eqref{eq:first-homotopic-condition-bisbis} and~\eqref{eq:second-homotopic-condition-bisbis}  in the statement above). Lemma~\ref{lem:homotopic-evaluations} implies that the two cocycles obtained from these evaluations define maps which coincide in homology, which means that $\Psi_{01}=\Id_*$ as claimed. 

 \begin{proof}[Proof of Lemma~\ref{lem:homotopic-evaluations-bisbis}]

We actually show that for any $ \lambda\in \ol \call_F(x,y)$ the loops $q_{x,y}^\Id\circ R(\lambda)$ and $q_{x,y}(\lambda)$ are equal modulo reparametrization, and then use  the fact that any two positive reparametrizations of an interval of the real line are homotopic.  Let us get into more detail. Take as previously $\tilde{\gamma}=(a,\gamma)$ a representative for $\lambda$ and write $R(\lambda)=(\alpha_0,\alpha_1)\in \ol\calm^\Id(x,y)$. By definition we have
$$q_{x,y}^\Id\circ R (\lambda)\, =\, q_{x,y}^\Id(\alpha_0, \alpha_1)\, =\, q^{f_0}(\alpha_0)\#H_0^X(1-t, c)\#H_1^X(t,c)\#q^{f_1}(\alpha_1)$$ and 
$$q_{x,y}(\lambda) \, =\, q_{x,y}(\tilde{\gamma})\, =\, \pi\circ \Theta\circ p\circ \Gamma(\tilde{\gamma}).$$
The map $q_{x,y}$ commutes with concatenations, so 
$$q_{x,y}(\tilde{\gamma})\, =\, q_{x,y}(\tilde{\gamma}|_{(-\infty,\bar s_-]})\#q_{x,y}(\tilde{\gamma}|_{[\bar s_-,\bar s_+]})\#q_{x,y}(\tilde{\gamma}|_{[\bar s_+,+\infty)}),$$
and upon using the formula~\eqref{eq:gamma} for $\tilde{\gamma}$ this relation becomes 
$$q_{x,y}(\tilde{\gamma})\, =\, q_{x,y}(a(s), \alpha_0(h_0(s)))\#q_{x,y}(a(s),p)\#q_{x,y}(\alpha_1(h_1(s))),$$
where $s$ belongs respectively to $(-\infty, \bar s_-]$, $[\bar s_-,\bar s_+]$ and $[\bar s_+, +\infty)$ in the three paths above. Let us analyse them separately. We use the sign $\sim$  to indicate that two paths are equivalent modulo reparametrization. 

For $s\in (-\infty, \bar s_-]$ the path $\tilde{\gamma}$ stays in $[0, \frac{1}{2}-\eps]\times X$, where the trees $\caly_t$ and the expanding maps $\theta_t$ are constant equal to $\caly_0$ resp. $\theta_0$. We therefore  get 
\begin{align*}
q_{x,y}(a(s), & \, \alpha_0(h_0(s)))  \\  
& =\pi\circ \Theta\circ p \circ \Gamma(a(s), \alpha_0(h_0(s))) \\
& = \theta_0\circ p_0\circ \pi\circ \Gamma(a(s), \alpha_0(h_0(s)))\sim \theta_0\circ p_0\circ \Gamma(\alpha_0)= q^{f_0}(\alpha_0)
\end{align*}
since $\pi(\Gamma(a, \alpha_0))$ is a reparametrization of $\Gamma(\alpha_0)$ (we use the values of different Morse functions to parametrize the two of them). 
Similarly, for $s\in [\bar s_+, +\infty)$ we are in $[\frac{1}{2}+\epsilon] \times X$ and we get 
$$
q_{x,y}(a(s), \alpha_1(h_1(s)))\, \sim\, q^{f_1}(\alpha_1).
$$
In the intermediate interval $[\bar s_-, \bar s_+]$ we analyze first the case $s\in [\bar s_-, \bar s]$  for which $\tilde \gamma$ takes values in the slice $[\frac{1}{2}-\eps, \frac{1}{2}]\times X$ and has its second coordinate  $\gamma$ constant equal to $c$. On this interval we have 
$$
q_{x,y}(a(s),c))\, =\, \pi\circ \Theta \circ p\circ \Gamma(a(s),c).
$$
Now $\Gamma(a(s), c)$ is the parametrization of $(a(s),c)$ using the values of the Morse function $F$ on the product (which equals $g$ on this slice); since $a$ is strictly decreasing we may parametrize the same path using its projection on $[0,1]$, which means that $(\Gamma(a(s)), c)\sim (t,c)$, the latter being defined for $t\in [\frac{1}{2}-\eps,\frac{1}{2}]$. Our choice of trees and expanding maps therefore yields 
\begin{align*}
q_{x,y}(a(s),c))\, \sim\, \pi\circ \Theta \circ p (t,c) & =\pi (t, \theta_0^{1-t}\circ p_0^{1-t}(c))\\
& =\theta_0^{1-t}\circ p_0^{1-t}(c)=H^X_0(1-t,c).
\end{align*}
Analogously, on the interval $s\in [\bar s, \bar s _+]$ we find that 
$$q_{x,y}(a(s),c))\, \sim\, H^X_1(t,c),$$ where $t\in [\frac{1}{2}, \frac{1}{2}+\eps]$. 

Putting all this together we finally infer that 
$$q_{x,y}^\Id\circ R(\lambda)\, \sim\, q_{x,y}(\lambda)$$
for any $\lambda$. Moreover, by construction the reparametrizations vary continuously with respect to $\lambda$, which implies that the maps $q_{x,y}^\Id\circ R$ and $q_{x,y}$ are homotopic as claimed. 

We still have to check that the homotopy may be chosen to satisfy the conditions~\eqref{eq:first-homotopic-condition-bisbis} and~\eqref{eq:second-homotopic-condition-bisbis} on the boundary of $\ol \call_F(x,y)$. This is straightforward. Indeed, for $(\lambda,\bar\lambda)\in \ol\call_{f_0}(x,z)\times \ol\call_F(z,y)$ we have 
$$q_{x,y} (\lambda,\bar \lambda)\, =\, q_{x,z}^{f_0}(\lambda)\#q_{z,y}(\bar \lambda) $$
and using~\eqref{eq:R1} we also have 
$$q_{x,y}^\Id (R (\lambda,\bar\lambda)) \, =\, q_{x,z}^{f_0}(\lambda) \#q_{z,y}^\Id (R (\bar\lambda)).$$
Moreover, the equivalence by reparametrization  $\sim$ from above  equals the identity on the first concatenation factor $q_{x,z}^{f_0}(\lambda)$. Therefore, in order to get the desired property~\eqref{eq:first-homotopic-condition-bisbis} we have to insure that the homotopy between these two equivalent loops is also the identity on this factor, which is easy. 

The same applies to $(\bar\lambda, \lambda)\in \ol\call_F(x,w)\times \ol \call_{f_1}(w,y)$, which proves~\eqref{eq:second-homotopic-condition-bisbis}. As a matter of fact, interpolating between  two equivalent loops of the form $\gamma\#\gamma_i$  or $\gamma_i\#\gamma$ (for $i=0,1$) with homotopies which are the identity on the first, resp. second factor, reduces to the following: given a strictly increasing  reparametrization of the source $\chi:[0,a]\ri [0,b] $ for $b<a$, we deform it into $\Id_{[0,a]}$ by $\chi_s$,  keeping $\chi'_s(\tau)=1$ on the intervals where $\chi'(\tau)=1$. This can be achieved by linear interpolation and this recipe depends continuously on some parameter $\lambda$ if $\chi$ does so (in our situation $\lambda\in \ol\call_F(x,y)$). 
 
 This finishes the proof of Lemma~\ref{lem:homotopic-evaluations-bisbis} ... 
\end{proof} 

...as well as the proof of Proposition~\ref{id=continuation}. 
\end{proof}

\smallskip 

We are now in position to prove that the two definitions we gave for the direct maps coincide. 

\begin{proposition}\label{alex=mihai} Let $\varphi:(X, \star_X)\ri (Y, \star_Y)$ be a smooth map,  $\Xi^X$ and $\Xi^Y$ two sets of data defined respectively on $(X,\star_X)$ and $(Y, \star_Y)$, and $\calf$ a DG-module over $C_*(\Omega_{\star_Y}Y)$. Then the  map $\varphi_* : C_*(X,\Xi^X; \varphi^*\calf)\ri C_*(Y,\Xi^Y;\calf)$ defined by~\eqref{eq:direct-second-definition}  coincides in homology with the one defined in~\S\ref{sec:funct-direct-general}, denoted here $\ol\varphi_*$.

\end{proposition} 

This statement should be understood as follows: for any  choice   (representing chain system on $\ol\calm(x,y')$, homotopies $H^X$ and $H^Y$) needed to define $\varphi_*$, this map coincides in homology with the one defined in~\S\ref{sec:functoriality-first-definition} between the complexes associated to $\Xi^X$ and $\Xi^Y$. Since the latter is well defined we obviously get 

\begin{corollary}\label{cor:alex=mihai} The  map $\varphi_*: H_*(X ;\varphi^*\calf)\ri H_*(Y;\calf)$ given by~\eqref{eq:direct-second-definition}  is well defined and satisfies Properties 1-4 from~\S\ref{sec:functoriality-properties}. \qed 
\end{corollary} 
\begin{proof}[Proof of Proposition~\ref{alex=mihai}]

We break the proof into three steps:

 (a) The two maps coincide in homology for embeddings $\varphi:X\ri Y$. 

(b)   Let  $D= D^m$ be a disc with $0$ as a basepoint and  $\pi: D\times  Y\ri Y$ the projection. We denote by $\Xi^D= (h, -\nabla h, o_{\{0\}}=+, \caly=0, \theta =\Id)$ the usual set of data which completes the Morse function $h(x)=\|x\|^2$.    On $D$ we choose the trivial homotopy $H^D=\Id$, and we choose a homotopy $H^Y$ on $Y$.  Then the map $\pi_*: C_*(D\times Y; \Xi^D\times \Xi^Y;\pi^*\calf)\ri C_*(Y,\Xi^Y;\calf)$ constructed using the homotopies $\Id\times H^Y$ and $H^Y$ coincides in homology with the one defined in 
\S\ref{sec:functoriality-first-definition}.

(c) Consider an arbitrary smooth map $\varphi$ and an embedding $\varphi^\chi=(\chi, \varphi) : (X,\star_X)\ri [D\times Y, (0,\star_Y)]$. Let $\cF$ be a DG-module over $C_*(\Omega_{\star_Y} Y)$. The  maps defined in this section $\varphi_*: C_*(X,  \Xi^X; \varphi^*\calf)\ri C_*(Y,  \Xi^Y;\calf)$, $\pi_*: C_*(D\times Y, \Xi^D\times \Xi^Y;\pi^*\calf)\ri C_*(Y,\Xi^Y;\calf)$ (the direct map of the projection $\pi$), and  $\varphi^\chi_*: C_*(X, \Xi^X ;\varphi^*\calf)\ri C_*(D\times Y, \Xi^D\times \Xi^Y;\pi^*\calf)$ (the direct map of the embedding $\varphi^\chi$), satisfy the following particular case of the composition property in homology : 
$$\pi_*\circ\varphi^\chi_*\, =\, (\pi\circ\varphi^\chi)_*\, =\, \varphi_*.$$
Here the map $\varphi_*$ is defined using arbitrary choices for the representing chain system and the homotopies, whereas for $\pi_*$ and  $\varphi^\chi_*$ we use  some adapted choices of the representing chain systems and the same   homotopies $H^X$, $H^Y$ on $X$, resp. $Y$ and $\Id\times H^Y$ on $D\times Y$.  

Let us show how these three steps imply Proposition~\ref{alex=mihai}. Recall the notation  $\ol\varphi_*$ for the direct map defined in~\S\ref{sec:functoriality-first-definition}. Take an embedding $\chi:(X,\star_X)\ri (D,0)$ for some disc $D$. By definition (see \eqref{eq:direct-map-general}) we have $$\ol\varphi_* \, =\,( \ol i_*)^{-1}\circ\ol \varphi^\chi_*\, =\, \ol\pi_*\circ \ol \varphi^\chi_*.$$ 
We should point out here that the above maps are defined at the chain level: $\ol\pi_*= (i_*)^{-1} : C_*(Y, \Xi^Y; \calf) \ri C_*(D\times Y, \Xi^D\times \Xi^Y; \pi^*\calf)$ by $$\alpha\otimes (0,x)\mapsto\alpha\otimes x $$
and $\ol\varphi^\chi_* : C_*(X, \Xi^X ;\varphi^*\calf)\ri C_*(D\times Y, \Xi^D\times \Xi^Y;\pi^*\calf)$ as in  the proof of (a) below. Now by (a) and (b) we have   $\ol \varphi^\chi_*=\varphi^\chi_*$ and $\ol\pi_*=\pi_*$. Applying (c) we get the sequence of equalities in homology
$$\varphi_*\, =\, (\pi\circ\varphi^\chi)_*\, =\, \pi_*\circ \varphi^\chi_*\, =\, \ol\pi_*\circ \ol \varphi^\chi_*\, =\, \ol\varphi_*.$$
Thus $\varphi_*=\ol\varphi_*$ in homology as desired.

We now prove the three steps above.
 
\smallskip 
\underline{Proof of (a).} We may suppose that $\varphi :X\hookrightarrow Y$ is an inclusion; the generalization to the case of embeddings is straightforward. Proceed as in~\S\ref{sec:funct-closed-submanifolds} to define $\ol\varphi_*$: extend the set of data $\Xi^X$ on $X$  first to a tubular neighborhood $U$  of $X$ in $Y$ with inward gradient and then to the whole of $Y$. Denote the extension by $\Xi_\varphi^Y$.  
If $f$ is the Morse function on $X$, for $x\in \Crit(f)$ and $\alpha \in\cF$ we have by definition of $\ol \varphi_*:C_*(X,\Xi^X;\varphi^*\calf)\ri C_*(Y,\Xi^Y;\calf)$: 
$$\ol\varphi_*(\alpha\otimes x) \, =\, \Psi(\alpha\otimes x),$$
where $\Psi$ is the continuation map between  the data $\Xi^Y_\varphi$ and $\Xi^Y$ (compare to~\eqref{eq:formula-direct-complexes} where the same formula was given for the embedding $\varphi^\chi$).  Let us now express $\varphi_*$.   If  $F:Y\ri \R$ is the Morse function of $\Xi^Y_\varphi$ (which is an extension of $f$) we get  by construction the equality between oriented manifolds $\ol W^u_f(x)= \ol W^u_F(x)$ for any  $x\in \Crit(f)$ (since the gradient points inwards along the boundary of the tubular neighborhood of $X$). Therefore, if $G$ is the Morse function of $\Xi^Y$, we get for  $y'\in \Crit(G)$ that 
$$
\ol \cM^\varphi(x,y')  \cong \overline{W}_f^u(x) \, _\varphi\!\times \overline{W}_G^s(y')\, =\,  \overline{W}_F^u(x) \, _\Id\!\times \overline{W}_G^s(y') \cong \ol\calm^\Id(x,y').$$
This identification preserves orientations and enables us to take the same (corrected) representing chain system on both spaces,  cf.~\eqref{eq:correction-rcs} 

  We finally obtain
$$
\phi_*(\alpha\otimes x)\, =\, \Id_*(\alpha\otimes x),
$$ 
where $\Id_*: C_*(Y,\Xi_\varphi^Y;\calf)\ri C_*(Y,\Xi^Y;\calf)$ is the direct map associated to $\Id$ for the two different sets of data, constructed with a homotopy on the domain which extends $H^X$ and the given  homotopy $H^Y$ on the target : for these choices the evaluation maps  $q^\varphi$ and $q^\Id$ obviously match with the identification above. We apply Proposition~\ref{id=continuation} to finish the proof of our  claim that  $\varphi_*=\ol \varphi_*$ for embeddings. 

\smallskip
\underline{Proof of (b).}   We describe the maps $$\pi_*:C_*(D\times Y, \Xi^D\times \Xi^Y;\pi^*\calf)\ri C_*(Y,\Xi^Y;\calf)$$ and $$\ol\pi_*: C_*(D\times Y, \Xi^D\times \Xi^Y;\pi^*\calf)\ri C_*(Y,\Xi^Y;\calf).$$ 
If $G: Y\ri \R$ is the Morse function of $\Xi^Y$ then  $\ol W^u_{h+G}(0,x)  = \{0\}\times \wu_G(x)$ for any $x\in \Crit(G)$. Thus for any $x,y'\in \Crit(G)$ we have 
\begin{align*}
\ol\calm^\pi((0,x),y') & \cong \wu_{h+G}(0,x) \, _\pi\!\times \overline{W}_G^s(y')\\
& \cong \{0\} \times \ol\calm_G^\Id(x,y') \cong \ol\calm_G^\Id(x,y').
\end{align*}
This identification preserves orientations and enables us to choose the representing chain systems as well as the corrected representing chain systems accordingly. Notice that   the evaluations $q^\pi: \ol\calm((0,x),y')\ri \Omega Y$ and $q^\Id: \ol\calm_G^\Id(x,y')\ri\Omega Y$ clearly also  match with this identification. Therefore $\pi_*(\alpha\otimes(0,x))=\Id_*(\alpha\otimes x)$, where $\Id_*: C_*(Y,\Xi^Y;\calf)\ri C_*(Y,\Xi^Y;\calf)$ is the direct map. Applying again Proposition~\ref{id=continuation} we find that $\Id_*=\Id$ and therefore 
$$\pi_*(\alpha\otimes (0,x))\, =\, \alpha\otimes x$$ for any $x\in \crit(G)$ and $\alpha\in \calf$.
On the other hand we know that 
$$\ol\pi_*(\alpha\otimes (0,x))\, =\, (i_*)^{-1} (\alpha\otimes (0,x))\, =\, \alpha\otimes x,$$ 
which proves the desired equality. 

\smallskip
\underline{Proof of (c).} 
Let us  give a formula for $\varphi^\chi_*$. If $G$ is the Morse function of $\Xi^Y$ then we have $\ol W^s_{h+G}(0,y')\, =\, D\times \ol W^s_G(y')$ for any $y'\in \Crit(G)$. Denoting $f: X\ri\R$  the Morse function of $\Xi^X$ we then get for any $x\in \Crit(f)$ and $y'\in \Crit(G)$ : 
\begin{align*}
\ol\calm^{\varphi^\chi}(x, (0,y'))\,  & \cong \wu_{f}(x) \, _{(\chi,\varphi)}\!\times (D\times \overline{W}_{G}^s(y')) \\
& \cong\, \{(a,b, \chi(a))\, |\, (a,b)\in  \wu_{f}(x) \, _{\varphi}\!\times  \overline{W}_{G}^s(y')\}.
\end{align*}
This space projects onto  $\ol\calm^\varphi(x,y')=\wu_{f}(x) \, _{\varphi}\!\times  \overline{W}_{G}^s(y')$ and the projection $\Pi$  is a homeomorphism. Note that $\Pi$ preserves the orientations. Indeed, the orientation rule \eqref{eq:orientation-direct-second} writes here 
$$
\left(\ori \, \ol\calm^{\varphi^\chi}(x,(0,y')), \ori\, \ol W_{h+G}^u(0,y')\right)\, =\, \ori\, \ol W^u_f(x).
$$
Since $\ol W_{h+G}^u(0,y')= \ol W^u_G(y')$ as oriented manifolds, this coincides with the orientation rule for $\ol\calm^\varphi(x,y')$, i.e.,
$$\left(\ori \, \ol\calm^{\varphi}(x,y'), \ori\, \ol W_{G}^u(y')\right)\, =\, \ori\, \ol W^u_f(x).$$
This enables us to  choose 
$(\Pi^{-1})_*(\sigma_{x,y'})$  as a corrected  representing chain system on $\ol \calm^{\varphi^\chi}((0,x), y')$, where $(\sigma_{x,y'})$ is a  corrected representing chain system for $\ol \calm^\varphi(x,y')$. Using  the homotopies $H^X$ on the domain and $\Id\times H^Y$ on the target we define the evaluation map $q^{\varphi^\chi}:\ol\calm^{\varphi^\chi}(x, (0,y'))\ri \Omega(D\times Y)$  using the formula~\eqref{eq:formula-evaluation-direct}. Evaluating the above representing chain system we get 
$$
\nu^{\varphi^\chi}_{x,(0,y')}\, =\, - q^{\varphi^\chi}_*(\Pi^{-1}_*(\sigma_{x,y'}))
$$ 
and then 
$$\varphi^\chi_*(\alpha\otimes x)\, =\, \alpha\sum_{y\in \crit G}\pi_*(\nu_{x,y'}^{\varphi^\chi})\otimes (0,y')$$
for any $\alpha\in \calf$ and $x\in \crit(f)$, where $\pi: D\times Y\ri Y$ is the projection. Let us explain the reason for the appearance of $\pi_*$ in this formula. For a general  direct map $\psi_*: C_*(X;\psi^*\calf)\ri C_*(Y, \calf)$ defined by $\psi_*(\alpha\otimes x)=\alpha\cdot\sum_y\nu_{x,y'}\otimes y'$ the product $\alpha\cdot \nu_{x,y'}$ is the module product for $\calf$ since $\nu_{x,y'}$ is already a chain on $\Omega Y$ by definition. In the above case $\psi= \varphi^\chi$, the role of $\calf$ is played by $\pi^*\calf$ (and $(\varphi^\chi)^*\pi^*\calf=\varphi^*\calf$), so that we have to use the multiplication rule for the module $\pi^*\calf$. 

We proved in  (b) the equality of direct maps $\pi_*=\ol \pi_*$ in homology. Using that formula for the latter  we get at the homology level : 
$$\pi_*\circ \varphi^\chi_*(\alpha \otimes x) \,=  \ol\pi_*\circ \varphi^\chi_*(\alpha \otimes x) \,= \, \alpha\sum_{y\in \crit G}\pi_*(\nu_{x,(0,y')}^{\varphi^\chi})\otimes y',$$
which we want to compare to 
$$\varphi_*(\alpha\otimes x) \, =\, \alpha\sum_{y\in \crit G}\nu_{x,y'}^{\varphi}\otimes y',$$
where $\nu^\varphi_{x,y'}= - q^{\varphi}_*(\sigma_{x,y'})$. 
At this point it would have been nice to have a commutative diagram 
$$\xymatrix{\ol\calm^{\varphi^\chi}(x,(0,y'))\ar[r]^-{q^{\varphi^{\chi}}}\ar[d]_{\Pi} & \Omega(D\times Y)\ar[d]^-\pi\\
\ol \calm^\varphi(x,y')\ar[r]^-{q^\varphi} &\Omega Y}$$
If this was true we would have gotten 
\begin{align*}
\pi_*(\nu^{\varphi^{\chi}}_{x, (0,y')}) & = -\pi_* q^{\varphi^{\chi}}_*(\Pi^{-1}_*(\sigma_{x,y'})) \\
& =  -q^{\varphi}_* \Pi_*(\Pi^{-1}_*(\sigma_{x,y'})) = -q^\varphi_*(\sigma_{x,y'}) = \nu^\varphi_{x,y'},
\end{align*}
and this would have finished the proof. However, the above diagram is only homotopy commutative: 

\begin{lemma}\label{lem:homotopic-evaluations-bis}  There is a homotopy $(q^s)_{s\in[0,1]}$ between $\pi\circ q^{\varphi^\chi}$ and $q^\varphi\circ \Pi$ whose maps $q^s: \ol\calm(x, (0,y'))\ri \Omega Y$, $s\in [0,1]$ satisfy the following conditions on the boundary $\p \ol\calm(x, (0,y'))$: 
\begin{itemize}
\item for all $(\lambda,\ol\lambda)\in \ol\call_f(x,y)\times \ol\calm(y,(0,y'))$, $x,y\in \crit(f)$,  $y'\in \crit(G)$, we have  
\begin{equation} \label{eq:first-homotopy-condition} 
q^{s}(\lambda, \ol\lambda)\, =\, \Omega\varphi (q^X(\lambda))\#q^s(\ol\lambda).
\end{equation} 
\item for all $ (\ol\lambda,\lambda')\in \ol\calm(x, (0,x'))\times \ol\call_G(x',y'))$, $x\in \crit(f)$, $x', y'\in \crit(G)$, we have 
\begin{equation}\label{eq:second-homotopy-condition} 
q^s(\ol\lambda, \lambda')\, =\, q^s(\ol\lambda)\#q^Y(\lambda').
\end{equation}
\end{itemize}
\end{lemma} 

This lemma is sufficient to prove the proposition. Indeed, similarly to the end of the proof of Proposition~\ref{id=continuation}, we are in the situation of Lemma~\ref{lem:homotopic-evaluations}: the representing chain system $\sigma_{x, (0,y')}=\Pi_*^{-1}(\sigma_{x,y'})$ defined on $\ol\calm^{\varphi^\chi}(x,(0,y'))$ satisfies an equation of the form~\eqref{eq:newinvarianceforchains}, and moreover we have two evaluation maps $\pi\circ q^{\varphi^\chi}$ and $q^\varphi\circ\Pi: \ol\calm^{\varphi^\chi}\ri \Omega Y$ which are homotopic through a homotopy satisfying~\eqref{eq:concatenation1} and~\eqref{eq:concatenation2} (which correspond to~\eqref{eq:first-homotopy-condition} and~\eqref{eq:second-homotopy-condition}  in the statement above). We may therefore apply Lemma~\ref{lem:homotopic-evaluations} and infer that the maps constructed with the two cocycles obtained form these evaluations are identical in homology. In other words, we get the claimed equality $\pi_*\circ\varphi^\chi_*\, =\, \varphi_*$. This completes the proof of Step~(c), and hence the proof of Proposition~\ref{alex=mihai}. 
\end{proof}

\begin{proof}[Proof of Lemma~\ref{lem:homotopic-evaluations-bis} ]

Let us take a closer look at the two maps which are claimed to be homotopic. We denote by $(a,\varphi^\chi(a))$ a point in 
$$
\ol \calm^{\varphi^\chi}(x,(0,y'))\cong \ol W^u_f(x)\, _{\varphi^\chi}\!\!\times  \overline{W}_{h+G}^s(0,y'),
$$  
we denote $\lambda_a$ the unparametrized (possibly broken)  flow line on $X$ which joins $x$ to $a$, and we denote  $\lambda'_{\varphi^\chi(a)}= (\lambda'_\chi(a), \lambda'_\varphi(a))$ the unparametrized (possibly broken)  flow line on $D\times Y$ which runs from $(\chi(a),\varphi(a))$ to $(0,y')$. Note that  $\lambda'_{\varphi(a)}=\pi(\lambda'_{\varphi^\chi(a)})$ is the (unparame\-trized) flow line on $Y$ which descends from $\varphi(a)$ to $y'$. 
By definition we have 
\begin{align*}
\pi\circ & q^{\varphi^\chi}(a, \varphi^\chi(a))\\
& =\pi\left[\varphi^{\chi}(q^X(\lambda_a))\#\varphi^\chi(H_{1-t}^X(a)) \# (\Id\times H_t^Y)(\varphi^\chi(a))
\#q^{D\times Y}(\lambda'_{\varphi^{\chi}(a)})
\right],
\end{align*}
which after distributing $\pi$ becomes 
\begin{align*}
\pi\circ & q^{\varphi^\chi}(a, \varphi^\chi(a))\\
& = \varphi(q^X(\lambda_a))\#\varphi(H_{1-t}(a))\# H^Y_t(\varphi(a))\#\pi\left(q^{D\times Y}(\lambda'_{\varphi^{\chi}(a)})\right).
\end{align*}
On the other hand, again by definition 
\begin{align*}
q^\varphi(\Pi(a,\varphi^\chi(a))) & = q^\varphi(a,\varphi(a)) \\
& = \varphi(q^X(\lambda_a))\#\varphi(H_{1-t}(a))\# H^Y_t(\varphi(a))\#q^Y(\lambda'_{\varphi(a)})).
\end{align*}
These two concatenations coincide except for the fourth term. Let us compare these last terms in the two concatenations. By definition $$q^{D\times Y}\, =\,  \theta^{D\times Y}\circ p^{D\times Y}\circ \Gamma^{D\times Y}\, =\, [\Id\times (\theta^Y\circ p^Y) ]\circ \Gamma^{D\times Y},$$
and therefore $\pi\circ q^{D\times Y} (\lambda'_{\varphi^\chi(a)})= \theta^Y\circ p^Y\circ \pi\circ  \Gamma^{D\times Y}(\lambda'_{\varphi^\chi(a)}))$. The last term of the second concatenation writes $q^Y(\lambda'_\varphi(a))= \theta^Y\circ p^Y\circ \Gamma^Y(\lambda'_{\varphi(a)})$. This means that we have to compare $\pi\circ  \Gamma^{D\times Y}(\lambda'_{\varphi^\chi(a)})$ and $\Gamma^Y(\lambda'_{\varphi(a)})$. We recall that by definition (see the proof of Lemma~\ref{lecture}),  the function  $\Gamma$ applied to some  gradient  flow line  yields its parametrization by the values of the Morse function. Therefore both $\pi\circ  \Gamma^{D\times Y}(\lambda'_{\varphi^\chi(a)})$ and $\Gamma^Y(\lambda'_{\varphi(a)})$ are parametrizations of the same gradient line $\lambda'_{\varphi(a)}$ obtained using different Morse functions : the function $h+G: D\times Y\ri \R$ for the flow line $(\lambda'_{\chi(a)},\lambda'_{\varphi(a)})$  for the former, and the function $G:Y\ri \R$ for the latter. So one is a reparametrization of the other, which implies in particular that they are homotopic. 

Let us now check relations~\eqref{eq:first-homotopy-condition} and~\eqref{eq:second-homotopy-condition}. The first one is immediate: it only concerns the first term of the expressions for $\pi\circ q^{\varphi^{\chi}}$ and $q^\varphi\circ \Pi$ above, whereas the homotopy only affects the last term as we have just seen. Actually this relation is a direct consequence of  the property 
$$q^X(\lambda,\lambda_a)\, =\, q^X(\lambda)\#q^X(\lambda_a)$$
for the broken orbits $(\lambda,\lambda_a)\in \ol\call_f(x,y)\times \ol W_f^u(y)\subset  \ol W^u_f(x)$, which is satisfied by the evaluation maps on $X$.

More care needs to be taken with~\eqref{eq:second-homotopy-condition} since it concerns the fourth and last term of the two expressions. As above we have two parametrizations $\Gamma^{D\times Y}$ and $\Gamma^Y$ of a broken orbit $(\lambda'_{\varphi(a)},\lambda')\in\ol W^s_G(x')\times \ol \call_G(x',y') \subset \ol W^s(y')$: the first one is given by the values of $h+G:D\times Y\ri \R$ for the broken orbit $[(\lambda'_{\chi(a)},\lambda'_{\varphi(a)}), (0,\lambda')]\in \ol W^s_G(0,x')\times \ol \call_{h+G}((0,x'),(0,y'))$, and the other one is given by the values of $G:Y\ri \R$. On the last factor $(0,\lambda')$ the values of $h+G$ are those of $G$, so that we actually get two parametrized paths defined by concatenations $\gamma_0\#\gamma$ and $\gamma_1\#\gamma$  such that $\gamma_1$ is a reparametrization of $\gamma_0$. We have to use a homotopy of the form $\gamma_s\#\gamma$ between these paths in order to ensure~\eqref{eq:second-homotopy-condition}. This is of course possible (see also the end of the proof of Lemma~\ref{lem:homotopic-evaluations-bisbis}), and finishes the proof of Lemma~\ref{lem:homotopic-evaluations-bis}.
\end{proof}

\section{Properties of direct maps revisited} \label{sec:properties-f*-revisited}

As we have already mentioned, in view of the fact that the map $\varphi_*$ defined in this section coincides with the one  from~\S\ref{sec:functoriality-first-definition}, it follows that it satisfies all the desired properties stated in Theorem~\ref{thm:f*!}. We find it nevertheless instructive to sketch alternative proofs of these properties which rely directly on the moduli spaces $\cM^{\varphi}(x,y')$ involved in the definition of the map $\varphi_*$. 

\begin{proof}[Sketch alternative proof of Theorem~\ref{thm:f*!} for direct maps] 
\qquad 

\smallskip 
1. {\it Proof of the {\sc (Identity)} property, i.e., $\mathrm{Id}_*=\mathrm{Id}$.}  We will sometimes denote in this paragraph $\varphi=\mathrm{Id}:X\to X$. We will prove that $\varphi_*$ induces in homology an isomorphism which is also an idempotent, i.e., $\varphi_*\circ\varphi_*=\varphi_*$, hence $\varphi_*$ is the identity.

The start of the proof is the observation that we can use the same auxiliary data (Morse function, pseudo-gradient, embedded collapsing tree, homotopy inverse $\theta$ for the collapsing map $p$) in the source and in the target, because such a choice satisfies the correct transversality conditions. Given $x,y'\in \Crit(f)$, the moduli spaces $\cM(x,y')$ which were previously defined for $\varphi$ become now the moduli spaces of \emph{parametrized} gradient trajectories from $x$ to $y'$. In contrast to the case of classical Morse homology, the map $\mathrm{Id}$ typically does not induce the identity at chain level because the higher dimensional moduli spaces of gradient trajectories are in general non-empty. Nevertheless, $\mathrm{Id}_*$ differs at chain level from the identity by terms which are of strictly lower order in the filtration: if $|x|=|y'|$ then $\cM(x,y')=\varnothing$ unless $x=y'$, in which case $m_{x, x}$ is the 0-chain given by the loop $H(1-\cdot, x)\#H(\cdot,x)$, where $H$ is the chosen homotopy from $\mathrm{Id}$ to $\theta\circ p$. This loop is homotopic to the constant loop at $x$ and therefore $\mathrm{Id}_*$ induces the identity on the 2nd page of the spectral sequence. As a consequence, it induces an isomorphism in homology.

To prove that $\mathrm{Id}_*$ is an idempotent in homology we view the moduli spaces of parametrized gradient trajectories $\cM(x,y')$ from $x$ to $y'$ as moduli spaces of unparametrized gradient trajectories in $\cL(x,y')$ carrying one additional marked point. We then consider the moduli spaces $\cH(x,y')$ consisting of unparametrized gradient trajectories in $\cL(x,y')$ carrying \emph{two} additional marked points. Any two such points are naturally ordered by the gradient direction along the gradient trajectory. The moduli spaces $\cH(x,y')$ have dimension $|x|-|y'|+1$ if $|x|-|y'|\ge 1$, respectively $0$ if $|x|=|y'|$. The compactification $\ol \cH(x,y')$ is a manifold with boundary and corners such that 
\begin{align*}
\p \ol \cH(x,y')  = & \bigcup_y \ol\cL(x,y)\times \ol \cH(y,y') \cup \bigcup_{x'} \ol\cH(x,x')\times \ol\cL(x',y')  \\
& \cup \ol \cM(x,y') \cup \bigcup_{\tilde z} \ol \cM(x,\tilde z)\times \ol \cM(\tilde z,y').
\end{align*}
The first two kinds of terms in the boundary correspond to breaking of gradient trajectories with the marked points staying at finite distance, the third term in the boundary corresponds to the two points colliding, and the fourth term corresponds to the gradient distance between the two points becoming infinite and resulting in an intermediate breaking.
The construction of a representing chain system for the compactified moduli spaces $\ol\cH(x,y')$ gives rise to a collection of chains $h_{x, x} = \star$ and 
$$
h_{x, y'}\in C_{|x|-|y'|+1}(\Omega X) \quad \mbox{ for } |x|-|y'|\ge 1,
$$
such that 
\begin{align*}
\p h_{x, y'}  = & \sum_{\tilde z} \nu_{x, \tilde z}\nu_{\tilde z,  y'} - \nu_{x, y'} \\
& + \sum_{y} (-1)^{|x|-|y|} m_{x, y}h_{y, y'} + \sum_{x'} (-1)^{|x|-|x'|} h_{x, x'}m_{x', \, y'}.
\end{align*}
The first term on the right hand side describes the composition $\varphi_*\circ\varphi_*$, the second term describes $\varphi_*$, and the equation amounts to saying that the collection $(h_{x, y'})$ defines a chain homotopy between $\varphi_*\circ\varphi_*$ and $\varphi_*$ (see~\S\ref{sec:DG-Morse}). Therefore we have $\varphi_*\circ\varphi_*=\varphi_*$ in homology. 
 
\smallskip 
2. {\it Proof of {\sc (Functoriality)}}. 
The identity 
$$
(\psi\varphi)_*=\psi_*\varphi_*
$$
is proved by showing that the corresponding chain maps are homotopic at chain level. 
Denote $\xi$, $\eta$, $\zeta$ the pseudo-gradients on $X$, $Y$, $Z$, and $\phi^t_\eta$ the flow of $\eta$ for $t\in \R$. The composition $\psi_*\circ\varphi_*$ is described using product moduli spaces $\ol \cM^\varphi(x,y')\times \ol \cM^\psi(y',z'')$, where the first factor involves the map $\varphi$ and the second factor involves the map $\psi$. The homotopy is constructed by considering the moduli spaces 
$$
\cH(x,z'') = \bigcup_{t>0} W^u(x) \cap (\psi \phi^t_{\eta} \varphi)^{-1}(W^s(z'')). 
$$
These have to be interpreted as matching configurations consisting of a half-infinite trajectory of $\xi$ flowing out of $x$ to a certain point $a$ in $X$, followed by a trajectory of $\eta$ of finite length $t>0$ starting at $\varphi(a)$ and ending at the point $b=\varphi^t_\eta \varphi(a)$ in $Y$, followed by a half-infinite trajectory of $\zeta$ starting from $\psi(b)$ and flowing into $z''$ in $Z$. The compactification $\ol\cH(x,z'')$ is a manifold with boundary and corners such that 
\begin{align*}
\p \ol\cH(x,z'') = & \bigcup_y \ol\cL(x,y)\times \ol\cH(y,z'') \cup \bigcup_{y''}\ol\cH(x,y'')\times \ol\cL(y'',z'')\\
& \cup \ol\cM^{\psi\varphi}(x,z'') \cup \bigcup_{y'} \ol \cM^\varphi(x,y')\times \ol\cM^\psi(y',z''). 
\end{align*}
Here the first two terms on the right hand side correspond to gradient breaking of the half-infinite trajectories at their asymptotes, the third term corresponds to the limit $t\to 0$ and involves the composition $\psi\varphi$, whereas the fourth term corresponds to gradient breaking as $t\to\infty$. The construction of a representing chain system for the compactified moduli spaces $\ol\cH(x,z'')$ results in a collection of chains 
$$
h_{x, z''}\in C_{|x|-|z''|+1}(\Omega Z)
$$
such that 
\begin{align*}
\p h_{x, z''} = & \sum_{y'} (\Omega\psi)_*\nu_{x, y'}\nu_{y',\, z''} - \nu_{x, z''} \\ 
& + \sum_{y} (-1)^{|x|-|y|} (\Omega \psi\varphi)_*m_{x, y}h_{y, z''} + \sum_{y''} (-1)^{|x|-|y''|} h_{x, y''}m_{y'',\, z''}.
\end{align*}
This provides the desired homotopy, in the manner of~\S\ref{sec:DG-Morse}. 

\smallskip 
{\it Proof of {\sc (Homotopy)}}. Denote $C_*(g;\cF)$ the Morse complex on $Y$ with coefficients in $\cF$, and $C_*(f;\varphi_0^*\cF)$, $C_*(f;\varphi_1^*\cF)$ the Morse complexes on $X$ with coefficients in $\varphi_0^*\cF$, respectively $\varphi_1^*\cF$, where $\varphi_0,\varphi_1:X\to Y$ are homotopic via $\varphi = (\varphi_t)_{t\in [0,1]}$. We claim that there is a canonical up to homotopy chain homotopy equivalence 
$$
\Psi^\varphi:C_*(f;\varphi_0^*\cF)\stackrel\simeq\longrightarrow C_*(f;\varphi_1^*\cF),
$$
and a homotopy commutative diagram 
$$
\xymatrix{
C_*(f;\varphi_0^*\cF)\ar[rr]^{\varphi_{0*}} \ar[dr]_{\Psi^\varphi} & & C_*(g;\cF) \\
& C_*(f;\varphi_1^*\cF)\ar[ur]_{\varphi_{1*}} &  
}
$$ 

\smallskip
{\it Step~1: definition of $\Psi^\varphi:C_*(f;\varphi_0^*\cF)\stackrel\simeq\longrightarrow C_*(f;\varphi_1^*\cF)$}. 
The chain map $\Psi^\varphi$ is defined from a continuation cocycle $\{h^\varphi_{x,y}\in C_{|x|-|y|}(\Omega Y)\, : \, x,y\in\Crit(f)\}$, and the idea is to build the latter by evaluating the moduli spaces of parametrized trajectories for $f$ into $\Omega Y$ using the homotopy $\varphi$. Denote $\cM(x,y)$ the moduli space of parametrized Morse trajectories running between $x,y\in\Crit(f)$, viewed as the moduli space of pairs $(\gamma,a)$ consisting of an unparametrized trajectory $\gamma\in\cL(x,y)$ and a marked point $a$ on $\gamma$. Note that $\cM(x,y)$ is a smooth manifold of dimension $\dim \cM(x,y)=|x|-|y|$, and in the case $|x|=|y|$ this moduli space is nonempty only for $x=y$, when it consists of the single constant trajectory at $x$. This moduli space compactifies to a manifold with boundary and corners such that 
$$
\p \ol \cM(x,y)=\bigcup_z \ol \cL(x,z)\times \ol \cM(z,y) \cup \bigcup_z \ol \cM(x,z)\times \ol \cL(z,y).
$$
To define the evaluation map we introduce the following notation: given $\gamma\in\cL(x,y)$ and $s\in [0,f(x)-f(y)]$, we denote $\gamma(s)$ the unique point on $\overline{\im\, \gamma}$ such that $f(\gamma(s))=f(x)-s$. Thus $s\mapsto \gamma(s)$ is the reparametrization of $\gamma$ by the levels of $f$, which was previously denoted $\Gamma^X(\gamma)$ (see Lemma~\ref{lecture}). We have in particular $x=\gamma(0)$, $a=\gamma(f(x)-f(a))$, and $y=\gamma(f(x)-f(y))$. Denote $\gamma_{x,a}=\gamma|_{[0,f(x)-f(a)]}$ and $\gamma_{a,y}=\gamma|_{[f(x)-f(a),f(x)-f(y)]}$. As in the previous sections, denote $p^X$ the collapsing map on $X$ and $\theta^X$ its chosen homotopy inverse. We define an evaluation map $\ev:\cM(x,y)\to \Omega Y$ as follows: the image of $(\gamma,a)$ is the Moore loop $[0,f(x)-f(y)+1]\to Y$ given by 
\begin{align*}
\ev(\gamma,a)=\varphi_0 (\theta^Xp^X\gamma_{x,a}) 
\, \# \, \{\varphi_t(\theta^Xp^X(a))\}_{t\in [0,1]} \, \# \, \varphi_1(\theta^Xp^X\gamma_{a,y}).
\end{align*}
\begin{figure}
  \centering
  \includegraphics[scale=.75]{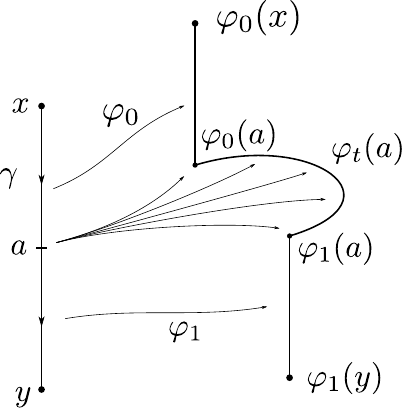}
  \caption{The map $\ev:\cM(x,y)\to \Omega Y$}
  \label{fig:Homotopy1}
\end{figure}
Intuitively, this evaluation map sends $\gamma$ to the concatenation of the following paths: first $\varphi_0\circ\gamma$, where $\gamma$ is traversed between $x$ and $a$, then $\varphi_t(a)$, $t\in[0,1]$, and finally $\varphi_1\circ\gamma$, where $\gamma$ is traversed between $a$ and $y$. As discussed in the previous sections, the role of the composition $\theta^Xp^X$ is to transform paths in loops. The evaluation map extends to the compactification and is compatible with breaking. Applying it to a representing chain system for the moduli spaces $\ol\cM(x,y)$, $x,y\in\Crit(f)$ we obtain the desired continuation cocycle $\{h^\varphi_{x,y}\in C_{|x|-|y|}(\Omega Y)\, : \, x,y\in\Crit(f)\}$, subordinated to the twisting cocycles $(\Omega\varphi_{0*}m_{x,y})$ and $(\Omega\varphi_{1*}m_{x,y})$, which defines the chain map $\Psi^\varphi$. The continuation cocycle is such that $h^\varphi_{x,y}=0$ for $|x|=|y|$ and $x\neq y$, hence the chain map is lower triangular with respect to the Morse filtration. If the homotopy preserves the basepoint, the map $\Psi^\varphi$ is the identity on the diagonal and therefore a chain isomorphism. In general, it is only a chain homotopy equivalence, the homotopy inverse being given by the analogous map constructed from the reverse homotopy.

\smallskip 
{\it Step~2: proof of the homotopy identity $\varphi_{1*}\circ \Psi^\varphi = \varphi_{0*}$}. 
We consider moduli spaces $\cM^{[0,1]}(x,y')$ consisting of pairs $(t,a)$ with $t\in [0,1]$ and $a\in  W^u(x) \cap \varphi_t^{-1}(W^s(y'))$, 
with unstable and stable manifolds considered with respect to $t$-dependent negative pseudo-gradients $\xi_t$ and $\eta_t$. The moduli space $\cM^{[0,1]}(x,y')$ has dimension $|x|-|y'|+1$ and its compactification is a manifold with boundary and corners such that 
\begin{align*}
\p \ol \cM^{[0,1]}(x,y')  = & \bigcup_y \ol\cL(x,y)\times \ol \cM^{[0,1]}(y,y') \cup \bigcup_{x'} \ol\cM^{[0,1]}(x,x')\times \ol\cL(x',y')  \\
& \cup \ol \cM^{\varphi_1}(x,y') \cup \ol \cM^{\varphi_0}(x,y').
\end{align*}
This produces a system of chain representatives $c^{[0,1]}_{x,y'}$ such that 
\begin{align*}
\p c^{[0,1]}_{x,y'} = c^1_{x,y'} & -c^0_{x,y'} \\
& + \sum_y(-1)^{|x|-|y|} s_{x,y}\times c^{[0,1]}_{y,y'} + \sum_{x'}(-1)^{|x|-|x'|} c^{[0,1]}_{x,x'}\times s_{x',y'}.
\end{align*}
At this point we have to deal with the following issue: while each element $(t,a)\in\cM^{[0,1]}(x,y')$ evaluates naturally into $\Omega Y$ via $\varphi_t$ and this evaluation extends continuously to the compactification, it does not give rise to a correct homotopy cocycle between the continuation cocycles defining $\varphi_{0*}$ and $\varphi_{1*}$. The problem is that the chains $s_{x,y}$ are evaluated into $Y$ in an uncontrolled way via the maps $\varphi_t$, $t\in[0,1]$. 
We solve this issue using adapted evaluation maps.

(i) We define an evaluation map $\ev_0:\cM^{[0,1]}(x,y')\to \Omega Y$ as follows. Given a point $(t,a)\in\cM^{[0,1]}(x,y')$, denote $\gamma'_{X,t}:(-\infty,0]\to X$ the half-infinite trajectory of $\xi_t$ running from $x$ to $a$, and $\gamma'_{Y,t}:[0,\infty)\to Y$ the half-infinite trajectory of $\eta_t$ running from $\varphi_t(a)$ to $y'$. As in Step~1, let $\gamma^X_t:[0,f(x)-f(a)]\to X$ be the reparametrization of $\gamma'_{X,t}$ by the levels of $f$, and $\gamma^Y_t:[0,g(\varphi_t(a))-g(y')]\to Y$ the reparametrization of $\gamma'_{Y,t}$ by the levels of $g$. 
Denote $f_{x,a}=f(x)-f(a)$ and $g_{t,a,y'}=g(\varphi_t(a))-g(y')$. 
We define $\ev_0(t,a)$ to be the Moore loop $\ev_0(t,a):[0,f_{x,a}+t+2+g_{t,a,y'}]\to Y$ given by 
\begin{align*}
\ev_0(t,a)=\varphi_0(\theta^Xp^X\gamma^X_t)\, \# \, \{\varphi_\tau & (\theta^Xp^X(a))\}_{\tau\in [0,t]} \\ 
& \# \, \varphi_t(H^X(1-\cdot,a)) \, \# \, H^Y(\cdot,\varphi_t(a)) \, \# \, \gamma^Y_t.
\end{align*}
\begin{figure}
  \centering
  \includegraphics[scale=.75]{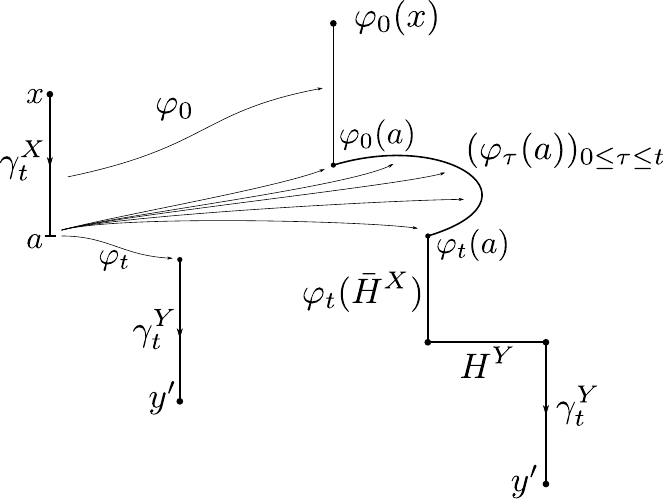}
  \caption{The map $\ev_0:\cM^{[0,1]}(x,y')\to \Omega Y$.}
  \label{fig:Homotopy2}
\end{figure}
This evaluation map extends continuously to the compactification. By applying it to the representing chain system $c^{[0,1]}_{x,y'}$ 
we obtain a homotopy cocycle $h^{[0,1]}_{x,y'}$ between the continuation cocycle $h^0_{x,y'}$ which defines $\varphi_{0*}:C_*(f;\varphi_0^*\cF)\to C_*(g;\cF)$ and the continuation cocycle $\tilde h^1_{x,y'}$ obtained by evaluating via $\ev_0$ the representing chains for the moduli spaces $\ol\cM^{\varphi_1}(x,y')$, which also defines a chain map $C_*(f;\varphi_0^*\cF)\to C_*(g;\cF)$. 

(ii) We consider now the $1$-parameter family of evaluation maps $\ev_L$, $L\ge 0$ defined on elements of $\cM^{\varphi_1}(x,y')$ by varying from $a$ to $x$ the point of insertion of the homotopy $\varphi_\tau$, $\tau\in[0,1]$ along the half-infinite gradient trajectory $\gamma'_X$. Denote $f_{x,b}=f(x)-f(b)$ and $g_{a,y'}=g_{1,a,y'}=g(\varphi_1(a))-g(y')$. Denote also $\gamma'_X=\gamma'_{X,1}$, $\gamma^X=\gamma^X_1$ and $\gamma'_Y=\gamma'_{Y,1}$, $\gamma^Y=\gamma^Y_1$. We define $\ev_L(a)$ to be the Moore loop $\ev_L(a):[0,f_{x,a}+3+g_{a,y'}]\to Y$ given by 
\begin{align*}
\ev&_L (a) \, = \, \varphi_0  (\theta^Xp^X\gamma^X|_{[0,f_{x,\gamma'_X(-L)}]}) \, \# \, 
\{\varphi_\tau(\theta^Xp^X\gamma'_X(-L))\}_{\tau\in[0,1]} \\
& \#\, \varphi_1(\theta^Xp^X\gamma^X|_{[f_{x,\gamma'_X(-L)},f_{x,a}]}) \, \# \, \varphi_1(H^X(1-\cdot,a)) \, \# \, 
H^Y(\cdot,\varphi_1(a)) \, \# \, \gamma^Y. 
\end{align*}
\begin{figure}
  \centering
  \includegraphics[scale=.75]{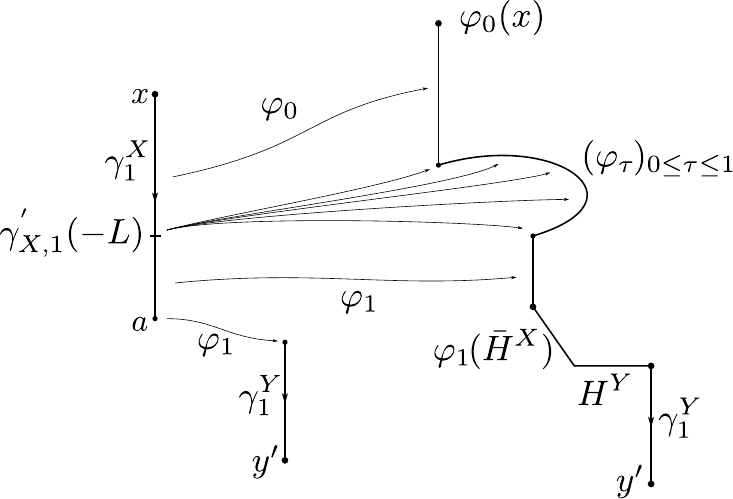}
  \caption{The maps $\ev_{L}:\cM^{\varphi_1}(x,y')$ for $L\geq0$.}
  \label{fig:Homotopy3}
\end{figure}
For $L=0$ we read the continuation cocycle $\tilde h^1_{x,y'}$ from (i). For $L=\infty$ we read the composition between the cocycle $h^1_{y,y'}$ which defines $\varphi_{1*}:C_*(f;\varphi_1^*\cF)\to C_*(g;\cF)$, and the continuation cocycle $h^\varphi_{x,y}$ from Step~1. Upon evaluating $\cM^{\varphi_1}(x,y')$ using the family $\ev_L$, $L\ge 0$ we obtain a homotopy cocycle between these two continuation cocycles.  

\smallskip 
The outcome of (i) and (ii) is that the identity $\varphi_{1*}\circ \Psi^\varphi = \varphi_{0*}$ holds up to homotopy, as claimed. 

\smallskip
4. {\it Proof of the {\sc (Spectral sequence)}}. This follows much in the manner of the proof given using the first definition of the direct maps, and we omit the details. 
\end{proof}

\section{Shriek map induced in homology}\label{sec:second-definition-shriek} 

In this section we give our alternative definition for the shriek map $\varphi_!$ associated to a smooth map $\varphi:X\to Y$ between \emph{orientable} manifolds. 

The proof of the corresponding statements from Theorem~\ref{thm:f*!} will follow from Proposition~\ref{alex=mihai-shriek}, which is proved in the next section and which asserts that this definition of the shriek map is the same as the one from~\S\ref{sec:functoriality-first-definition}.

We proceed as for the definition of direct maps in \S\ref{sec:funct-direct-alex} but with the roles of the stable and unstable manifolds interchanged. We denote $\star_X\in X$, $\star_Y\in Y$ the basepoints, and we choose two sets of data $\Xi^X$ and $\Xi^Y$ on $X$, respectively $Y$. We work under the standing assumption $\varphi(\star_X)=\star_Y$, we assume w.l.o.g. that $\varphi$ is smooth and 
for all $x'\in \Crit(g)$ and $y\in \Crit(f)$ we have 
$$
\varphi|_{W^s(y)}\pitchfork W^u(x'),
$$
where $f$ and $g$ are the Morse functions on $X$, respectively $Y$. We define 
$$
\cM^{\varphi_!}(x',y) = W^s(y)\cap \varphi^{-1}(W^u(x')).
$$
When no confusion can arise we also use the abridged notation 
$$
\cM(x',y)=\cM^{\varphi_!}(x',y).
$$ 
Because the intersection is transverse this is a smooth manifold of dimension 
$$
\dim\, \cM(x',y)=|x'|-|y|+m-n. 
$$
It is convenient to use the notation 
\begin{equation}\label{eq: shifted-degrees} 
[x']\, =\, |x'|+m \gol \mathrm{and}\gol [y]\, =\, |y|+n,
\end{equation} 
so that the dimension of $\cM(x',y)$ equals $[x']-[y]$. By shifting  the indices of the critical points in this way, we get new  gradings  on the twisted complexes for which the shriek map $H_*(Y;\calf)\ri H_*(X; \varphi^*\calf)$ that we will construct below becomes degree preserving. The compactification 
$$
\ol \cM(x',y)\cong \overline{W}^s(y) \, _\varphi\!\times \overline{W}^u(x')
$$ 
is a manifold with boundary and corners of dimension $[x']-[y]$  such that 
$$
\p \ol \cM(x',y) = \bigcup_{y'} \ol \cL(x',y')\times \ol\cM(y',y) \cup \bigcup_{x} \ol\cM(x',x)\times\ol \cL(x,y).
$$

In order to construct a representing chain system on these spaces we need to orient them. We orient first the stable manifolds by setting 
\begin{equation}\label{eq:orientation-rule-stable} 
\left(\ori\,  \ol W^s(y),\ori\, \ol W^u(y)\right)\, =\, \ori\, X,
\end{equation} 
and similarly for the stable manifolds on $Y$. Then we orient $\ol\calm(x',y)$ by 
\begin{equation}\label{eq:shriek-orientation-first}
\left(\coori\,  \ol W^u_g(x'), \ori\, \ol\calm(x',y)\right)\, =\, \ori\, \ol  W_f^s(y),
\end{equation} 
which we may also write
\begin{equation}\label{eq:shriek-orientation-second}
\left( \ori\,  \ol W^s_g(x'), \ori\, \ol\calm(x',y), \right)\, =\, \ori\, \ol  W_f^s(y).
\end{equation} 
  This convention is different  from our usual orientation rule for  transverse intersections \eqref{eq:transverse-orientation}. We chose it in the particular case of shriek maps since it yields  simpler orientation computations  for the  different spaces involved in the construction. For instance we remark the following equality :
  \begin{lemma}\label{lem:direct=shriek-for-Id}  Consider the identity map $\Id:X\ri X$, where $X$ is endowed with two possibly different Morse-Smale couples, $(f,\xi_f)$ on the domain and $(g,\xi_g)$ on  the target. Then for any $x'\in \crit(g)$ and $y\in \crit(f)$ we have 
  $$ \ol\calm^{\Id_!}(x',y)\, =\,  \ol\calm^{\Id}(x',y)$$
  {\bf as oriented manifolds}, where on the right hand side we have  the moduli space associated to the identity viewed as $\Id^{-1}: (X,g,\xi_g)\ri (X,f,\xi_f)$.
  \end{lemma}  
  
  \begin{proof} The two manifolds above are defined by the same (transverse) intersection $\ol W^u_{g}(x') \pitchfork \ol W^s_f(y)$ (more precisely, it is the fibered product of $\ol W^u_{g}(x')$ and $ W^s_f(y)$ defined by the canonical applications which send these Latour cells into $X$). So the only thing to prove is that their orientations are identical. 
  Using the conventions \eqref{eq:orientation-rule-stable} and \eqref{eq:shriek-orientation-second} we write  
 $$ (\ori \, \ol W_g^s(x'), \ori\,  \ol\calm^{\Id_!}(x',y),  \ori\,  \ol W_f^u(y))  \!=\! (\ori\,  \ol W_f^s(y), \ori\,  \ol W^u_f(y)) \!=\!\ori \, X.$$
 On the other hand our convention \eqref{eq:orientation-direct-second} for  direct maps yields 
 $$ (\ori\, \ol W_g^s(x'), \ori\, \ol \calm^{\Id}(x',y), \ori\, \ol W_f^u(y)) \! = \!(\ori \, \ol W_g^s(x'), \ori\, \ol W^u_g(x'))\! =\! \ori\, X,$$
 proving thus our claim
 $$ \ori\, \ol \calm^{\Id_!}(x',y)\, =\, \ori\, \ol \calm^{\Id}(x',y).$$
  \end{proof}

Prior to the construction of a representing chain system  we have to compare the boundary orientation for the strata $\ol \cL(x',y')\times \ol\cM(y',y)$ and $\ol\cM(x',x)\times\ol \cL(x,y)$ of $\partial \ol\calm(x',y)$ with their product orientation. We find:
   
\begin{lemma}\label{lem:orientation-difference-shriek} We have the following equalities of orientations: 
$$
\ori \, \, \partial\ol\calm^{\varphi_!}(x',y) \, =\, (-1)^{[x']-[y']} \left(\ori\, \ol\call_g(x',y'), \ori\, \ol \calm^{\varphi_!}(y',y)\right)
$$
and 
$$
\ori\, \, \partial\ol\calm^{\varphi_!}(x',y)\, =\, (-1)^{[x']-[y]-1}\left(\ori\, \ol\calm^{\varphi_!}(x',x)), \ori\, \ol\call_f(x,y')\right).
$$
\end{lemma} 
\begin{proof} Recall that the orientation rule \eqref{eq:orientation-rule}  for $\ol\call_g(x',y')$ is
$$\left(\ori\, \ol \call_g(x',y'), -\xi , \ori\, \ol  W_g^u(y')\right) \, =\, \ori \, \ol W^u_g(x'),$$
where $\xi$ is a gradient vector field. We use the relation \eqref{eq:orientation-rule-stable} to express this orientation  in terms of stable manifolds instead of unstable manifolds. We start by adding $\ori\, \ol W_g^s(x')$ at the left on both sides and infer
\begin{align*}
(\ori\, \ol W^s(x'), \ori\, \ol \call_g(x',y'), & -\xi , \ori\, \ol  W_g^u(y')) \\
& = (\ori\, \ol W^s_g(x'), \ori \, \ol W^u_g(x')) = \ori\, X. 
\end{align*}

Writing $\ori\,  X =\left(\ori\, \ol W_g^s(y'), \ori \, \ol W^u_g(y')\right)$ and removing $\ori \, \ol W^u_g(y')$ from both sides yields 
\begin{equation}\label{eq:changed-orientation-rule} 
\left(\ori\, \ol W^s(x'), \ori\, \ol \call_g(x',y'), -\xi\right)\, =\, \ori\, \ol W^s(y').
\end{equation}

The rest of the proof proceeds as in Lemma~\ref{lem:orientation-difference-direct}. Let us start with the  first relation. Denote as usual by $n$ the outward normal of $ \ol\calm^{\varphi_!}(x',y)$ at some boundary point $(\lambda',\ol \lambda)\in \ol\call_g(x',y')\times \ol \calm^{\varphi_!}(y',y)$.


We view an element of $\cM(x',y)$ as consisting of a half-infinite trajectory $\lambda$  in $X$ starting at a point $a$ and flowing into $y$, of a half-infinite trajectory $\lambda'$  in $Y$ flowing from the critical point $x'$ to a certain point $b$, the two trajectories being coupled by the condition $\varphi(a)=b$. Denoting $\ol \lambda = (\lambda_a, \lambda'_{\varphi(a)})$ and arguing as in the proof of Lemma~\ref{lem:orientation-difference-direct} we may view $n$ as the outward normal to the boundary of $\ol W^u(x')$ at the point $(\lambda,  \lambda'_a)\in \ol\call_g(x',y')\times \ol W^u(y')$. Then, considering the boundary point $(\lambda, y')\in \partial \ol W^u(x')$, we may take $n=\xi$, the gradient vector field. 

With that in mind, we first write  the usual rule for the boundary orientation 
$$\left(n, \ori\,  \partial \ol \calm^{\varphi_!}(x',y)\right) \, =\, \ol \calm^{\varphi_!}(x',y)$$
and therefore, by \eqref{eq:shriek-orientation-second} 
\begin{align*}
(\ori\, \ol W_g^s(x'),  n, & \ori\,  \partial \ol \calm^{\varphi_!}(x',y))  \\
& = (\ori\, \ol W_g^s(x'),  \ori\, \calm^{\varphi_!}(x',y)) = \ori\, \ol W_f^s(y).
\end{align*}
Now, in order to compare it with the product orientation, we write using \eqref{eq:changed-orientation-rule} in the fourth equality below and \eqref{eq:shriek-orientation-second} in the fifth one :
\begin{align*}\big(\ori\, & \ol W_g^s(x'),  n, \ori\, \ol \call_g(x',y'), \ori\,  \ol \calm^{\varphi_!}(y',y)\big) \\
& =\left(\ori\, \ol W_g^s(x'),  \xi, \ori\, \ol \call_g(x',y'), \ori\,  \ol \calm^{\varphi_!}(y',y)\right)\\
& =-\left(\ori\, \ol W_g^s(x'),  -\xi, \ori\, \ol \call_g(x',y'), \ori\,  \ol \calm^{\varphi_!}(y',y)\right)\\
& =(-1)^{|x'|-|y'|}\left(\ori\, \ol W_g^s(x'),  \ori\, \ol \call_g(x',y'), -\xi, \ori\,  \ol \calm^{\varphi_!}(y',y)\right)\\
& =(-1)^{|x'|-|y'|}\left(\ori\, \ol W_g^s(y'), \ori\,  \ol \calm^{\varphi_!}(y',y)\right)\\
&=(-1)^{|x'|-|y'|}\ori\, \ol W^s_f(y)\\
&=(-1)^{[x']-[y']}\ori\, \ol W^s_f(y).
\end{align*} 
This proves our first claim. 

Let us now prove the second relation in an analogous way. The outward normal $n$ of $\ol\calm^{\varphi_!}(x',y)$ at some boundary  point $(\ol\lambda, \lambda)\in \ol\calm^{\varphi_!}(x',x)\times \ol\call_f(x,y) $ with $\ol \lambda= (\lambda_a, \lambda'_{\varphi(a)})$  may be viewed as the outward normal of $\ol W_f^s(y)$ at $(\lambda_a, \lambda)\in \ol W_f^s(x)\times \ol\call_f(x,y)\in \partial \ol W_f^s(y)$. Moving to the point $(x,\lambda)$ of this boundary, we may consider $n=-\xi$, the opposite of the gradient vector field which points outwards at this point (after being normalised and extended by continuity as in the proof of Proposition~\ref{orientationsMorse}). 
As above, we have on the one hand 
\begin{align*}
(\ori\, \ol W_g^s(x'),  & n, \ori\,  \partial \ol \calm^{\varphi_!}(x',y))  \\
& = (\ori\, \ol W_g^s(x'),  \ori\, \calm^{\varphi_!}(x',y)) = \ori\, \ol W_f^s(y),
\end{align*}
and on the other hand (using \eqref{eq:shriek-orientation-second} in the third equality as well as  \eqref{eq:changed-orientation-rule} for $\ol\call_f(x,y)$  in the last equality)  
\begin{align*}\big(\ori\, \ol W_g^s(x'), &  n, \ori\,  \ol \calm^{\varphi_!}(x',x), \ori\, \ol\call_f(x,y)\big) \\
& =\left(\ori\, \ol W_g^s(x'),  -\xi,  \ori\,  \ol \calm^{\varphi_!}(x',x), \ori\, \ol \call_f(x,y)\right)\\
& =(-1)^{[x']-[x]}\left(\ori\, \ol W_g^s(x'),  \ori\,  \ol \calm^{\varphi_!}(x',x), -\xi, \ori\, \ol \call_{f}(x,y)\right)\\
& =(-1)^{[x']-[x]}\left(\ori\, \ol W_f^s(x), -\xi,  \ori\, \ol \call_f(x,y)\right)\\
& =(-1)^{[x']-[x]+[x]-[y]-1}\left(\ori\, \ol W_f^s(x), \ori\,  \ol \call_f(x,y), -\xi\right)\\
&=(-1)^{[x']-[y]-1}\ori\, \ol W^s_f(y).
\end{align*} 
The comparison between the two finishes our proof. 
\end{proof}

Now following once again the recipe  of  Proposition~\ref{representingchain} we construct  a representing chain system $(\tau_{x',y})$ for the moduli spaces   $\ol\calm^{\varphi_!}(x',y)$ which according to the previous lemma satisfies the following equation : 
$$
\p\tau_{x',y}\, =\, (-1)^{[x']-[y']}\sum_{y'\in \crit(g)}s_{x',y'}\times \tau_{y',y}-(-1)^{[x']-[y]}\sum_{x\in \crit(f)} \tau_{x',x}\times s_{x,y}.
$$ 
As in the case of direct maps we need to correct this representing chain system with a sign in order to get an equation of the form \eqref{eq:newinvarianceforchains}, namely by setting  as in \eqref{eq:correction-rcs}: \begin{equation}\label{eq:correction-rcs-shriek}
\sigma_{x',y}\, =\, (-1)^{[x']-[y]-1}\tau_{x',y}.
\end{equation} 
which indeed satisfies an equation of the desired form 
$$\p\sigma_{x',y}\, =\, \sum_{y'\in \crit(g)}s_{x',y'}\times \sigma_{y',y}-(-1)^{[x']-[x]}\sum_{x\in \crit(f)} \sigma_{x',x}\times s_{x,y}.$$
As previously the next step is to evaluate this corrected representing chain system into $\Omega Y$. At this purpose we use the homotopies $H^Y$ and $H^X$ on $Y$ resp. $X$ that we have already defined to construct the evaluation maps for direct maps; they respectively join $\Id_Y$ to $\theta^Y\circ p^Y$ and $\Id_X$ to $\theta^X\circ p^X$. To define the evaluation map we start with the half-infinite trajectory $\lambda'_{\varphi(a)}$  in $Y$ which we evaluate with $q^Y$ getting thus a path from $\star_Y$ to $\theta^Y\circ p^Y(\varphi(a))$. Then we use the reversed homotopy $H^Y(1-t,\varphi(a))$ to extend this path until $\varphi(a)$. Afterwards we do the same to the half-infinite trajectory $\lambda_a$ in $X$  whose evaluation $q^X$ is a path in $X$ from $\theta^X\circ p^X(a)$
to $\star_X$. In order to close the loop in $Y$ we have to precede it by the homotopy $H^X(t,a)$ and then transport the concatenation  to $Y$ via $\varphi$.  This gives a continuous application $$q^{\varphi_!}_{x',y}:\ol\calm^{\varphi_!}(x',y)\ri \Omega Y$$ whose formula is 
\begin{equation}\label{eq:evaluation-shriek}
q^{\varphi_!}_{x',y}(\lambda_a, \lambda'_{\varphi(a)})=q^Y(\lambda'_{\varphi(a)})\#H^Y(1-t, \varphi(a))\#\varphi(H^X(t,a))\#\varphi(q^X(\lambda_a))
\end{equation} 

By definition, the restriction of $q^{\varphi_!}_{x',y}$ on the boundary $\partial \ol\calm^{\varphi_!}(x',y)$ satisfies 
$$q^{\varphi_!}_{x',y}(\lambda', \ol \lambda) \, =\, q_{x',y'}^Y(\lambda')\# q^{\varphi_!}_{y',y}(\ol\lambda)$$
for all $(\lambda',\ol \lambda)\in \ol\call_g(x',y')\times \ol\calm^{\varphi_!}(y',y)$  and
$$q^{\varphi_!}_{x',y}( \ol \lambda, \lambda) \, =\, q^{\varphi_!}_{x',x}(\ol\lambda)\#\varphi(q_{x,y}^X(\lambda))$$
for all $(\ol \lambda,\lambda)\in  \ol\calm^{\varphi_!}(x',x)\times \ol\call_f(x,y)$.

 
 Defining  $$\nu_{x,y'}\, =\, - q^{\varphi_!}_{x',y,*}(\sigma_{x',y})$$  leads to a collection of chains 
$$
\nu_{x',\, y}\in C_{[x']-[y]}(\Omega Y)
$$
such that 
$$
\p \nu_{x',\, y}= \sum_{y'} (m_{x', \, y'})\nu_{y', \, y} - \sum_{x}(-1)^{[x']-[x]}\nu_{x',\, x}(\Omega \varphi)_*m_{x, y}.
$$

We now define 
\begin{equation}\label{eq:shriek-second-definition} 
\varphi_! : \cF\otimes \Crit(g)\to \cF\otimes \Crit(f), \qquad \varphi_!(\alpha\otimes x') = \sum_{y} \alpha\cdot \nu_{x',\, y}\otimes y
\end{equation} 
where the critical points are graded by the modified index $[\, \, \, ]$. 
A direct verification shows that this is a chain map, see also~\S\ref{sec:invariance-homology}. 
\rmk\label{rem:shriek-for-boundary} As in Remark \ref{rmk:direct-for-boundary} we point out that the definition of $\varphi_!$ adapts to the case when $X$ and $Y$ have boundary.  In order to get the same compactification of $\calm^{\varphi_!}(x',y)$ and define $\varphi_!$ in a similar way, we need to take a (negative) gradient which points outwards $\p X$ and we need to assume that $\varphi(\ol X)\subset \mathrm{int}\, Y$. The target of $\varphi_!$ is therefore the relative homology $H_*(X, \p X; \varphi^*\calf)$.  If $Y$ has boundary, the shriek map $\varphi_!$ factors as $H_*(Y; \calf)\ri H_*(Y, \p Y; \calf) \ri H_*(X,\p X; \varphi^*\calf)$.
\kmr
\section{The two definitions of shriek maps coincide} 

The goal of this section is to prove Proposition~\ref{alex=mihai-shriek} which asserts that  the shriek  map from~\S\ref{sec:functoriality-first-definition} and the one defined in the previous section coincide in homology. This will imply 
the independence of the second definition of the shriek map   with respect to the various choices of auxiliary data. Furthermore we infer all the properties from Theorem~\ref{thm:f*!} for this alternative definition of the shriek map. 
 

\begin{proposition}\label{alex=mihai-shriek} Let $\varphi:(X, \star_X)\ri (Y, \star_Y)$ be a smooth map between oriented manifolds,  $\Xi^X$ and $\Xi^Y$ two sets of data defined respectively on $(X,\star_X)$ and $(Y, \star_Y)$, and $\calf$ a DG-module over $C_*(\Omega_{\star_Y}Y)$. 

Then the  map $\varphi_! : C_*(Y,\Xi^Y;\calf)\ri C_{*+\mathrm{dim}(X)-\mathrm{dim}(Y)}(X,\Xi^X;\varphi^\ast\calf)$ defined by~\eqref{eq:shriek-second-definition}  coincides with the one defined in the last part of \S\ref{sec:funct-direct-general}.

\end{proposition} 

\begin{proof}
As for its counterpart for direct maps (Proposition~\ref{alex=mihai}), the proof relies on the following statement on the shriek map of the identity. 

\begin{proposition} \label{id=continuation-shriek}  (compare to Proposition~\ref{id=continuation} )
For any two sets of data  $\Xi_0$ and $\Xi_1$ defined on $(X,\star)$ and for any DG-module $\calf$  over $C_*(\Omega X)$, the map $\Id_!: C_*(X,\Xi_1;\calf)\ri C_*(X,\Xi_0;\calf)$ defined as in \S\ref{sec:second-definition-shriek} coincides with the continuation map $\Psi_{10}$ in homology. In particular, it  does not depend on the choices made for its definition (representing chain system, homotopies etc). 
\end{proposition} 
\begin{proof}[Proof of Proposition~\ref{id=continuation-shriek}] We claim  that $\Id_!$ actually coincides with  the direct map $\Id^{-1}_*: C_*(X,\Xi_1;\calf)\ri C_*(X,\Xi_0;\calf)$. The conclusion then follows by applying Proposition~\ref{id=continuation} to the inverse of the direct map of the identity  $\Id^{-1}_*$. 

The proof that $\Id_!=\Id^{-1}_*$ is straightforward. We have already pointed out in Lemma~\ref{lem:direct=shriek-for-Id} that the moduli spaces $\ol\calm^{\Id^{-1}}(x',y)$ and $\ol\calm^{\Id_!}(x',y)$ are the same as oriented manifolds, which enables us to pick the same representing chain system $(\tau_{x',y})$ for both of them. Then, the signs in \eqref{eq:correction-rcs} and \eqref{eq:correction-rcs-shriek} also coincide (note that $[x']-[y]=|x'|-|y|$ here), which means that we have the same corrected representing chain system $(\sigma_{x',y})$. Finally, as one may easily check, formula \eqref{eq:formula-evaluation-direct} for $\varphi=\Id^{-1}$ which defines the evaluation map $q_{x',y}^{\Id^{-1}}$ is identical to formula \eqref{eq:evaluation-shriek} for $q_{x',y}^{\Id_!}$. Therefore the two morphisms of complexes are equal. 
\end{proof}

The remaining part of the proof of Proposition~\ref{alex=mihai-shriek} is pretty similar to that of Proposition~\ref{alex=mihai}. Namely, denoting by $\ol\varphi_!$ the shriek map defined in \S\ref{sec:funct-direct-general} and by $\varphi_!$ the one defined in \S\ref{sec:second-definition-shriek} we successively prove the following three claims. 

 (a) The two maps coincide in homology for embeddings $\varphi:X\ri Y$. 

(b)   The two maps coincide in homology for the projection $\pi:D\times Y\ri Y$, where $D= D^m$ is a disc with $0$ as  basepoint. We denote by $\ori \, D$ the usual  orientation of the disk $D$.  In this statement $Y$ is endowed with a set of data $\Xi^Y$ and a homotopy $H^Y$, and these induce on $D\times Y$ a set of data $\Xi^D\times\Xi^Y$ with $\Xi^D= (h, -\nabla h, o_{\{0\}}= \ori\, D, \caly=0, \theta =\Id)$ the standard set of data associated to the Morse function $h(x)=-\|x\|^2$, and a homotopy given by $\Id\times H^Y$. 

(c) The shriek map $\varphi_!$ satisfies the composition property  in the following particular case:  Consider an arbitrary smooth map $\varphi$ and an embedding $\varphi^\chi=(\chi, \varphi) : (X,\star_X)\ri \big(D\times Y, (0,\star_Y)\big)$.  Then 
$$\varphi^\chi_!\circ\pi_!\, =\, (\pi\circ\varphi^\chi)_!\, =\, \varphi_!.$$
Here the map $\varphi_!$ is defined using arbitrary choices for the representing chain system on $\ol\calm^{\varphi_!}(x',y)$ and  for the homotopies, whereas the two other maps are defined using some adapted choices for the representing chain systems and using the homotopies $H^X$, $H^Y$ on $X$, resp. $Y$, and $\Id\times H^Y$ on $D\times Y$.  

Denote $i:Y\ri D\times  Y$ the inclusion $x\mapsto(0,x)$. 
Given that $\ol\varphi_!= \ol\varphi^\chi_!\circ \ol i^{-1}_!$ by definition, and  $\ol i_!\circ \ol \pi_!=\Id$ by the composition property (which is satisfied by the maps defined in \S\ref{sec:functoriality-first-definition}), we get the desired conclusion by applying these three steps as follows:
$$
\ol \varphi_!= \ol  \varphi^\chi_!\circ \ol i^{-1}_!= \ol  \varphi^\chi_!\circ \ol \pi_!= \varphi^\chi_! \circ\pi_!= (\pi\circ\varphi^\chi)_!=\varphi_!.
$$

\smallskip 

\underline{Proof of (a).} We first assume that $\varphi :X\hookrightarrow Y$ is an inclusion. Recall the definition of $\ol\varphi_!$ from  \S\ref{sec:funct-closed-submanifolds}. Denote respectively by $f:X\ri\R$ and $G:Y\ri \R$ the Morse functions of $\Xi^X$ and $\Xi^Y$.   Extend the set of data $\Xi^X$ on $X$  first to a tubular neighborhood $U$  of $X$ in $Y$ with  gradient pointing outwards and then to the whole of $Y$. Denote the extension by $\Xi_\varphi^Y$ and $F: Y\ri \R$ its Morse function: $F$ has no critical points in $U$ except for those of $f$. 
The definition of $\ol \varphi_!$  in \S\ref{sec:funct-closed-submanifolds} was given for the data $\Xi^Y_\varphi$ on $Y$, therefore we have to insert a continuation map in order to define $\ol \varphi_!:C_*(Y,\Xi^Y;\calf)\ri C_{*+\mathrm{dim}(X)-\mathrm{dim}(Y)}(X,\Xi^X;\varphi^*\calf)$.  We get   for $x\in \Crit(G)$ and $\alpha \in\cF$ the formula
\begin{equation}\label{eq:formula-phi-bar}
\ol\varphi_!(\alpha\otimes x)  =\mathrm{pr}(\Psi(\alpha\otimes x)),
\end{equation} 
where $\Psi$ is the continuation map between  the data $\Xi^Y$ and $\Xi^Y_\varphi$ and $\mathrm{pr}$ is the projection 
 $$\mathrm{pr}(\alpha\otimes y)\, =\, \begin{cases}
  \alpha \otimes y& \text{ if } y\in \crit(f),\\
  0&\text{ if }y\not\in \crit(f),
\end{cases}$$
defined for $\alpha\in \cF$ and $y\in \crit(F)$. 
 Let us now show that the same formula is valid for $\varphi_!$.   Since the negative gradient of $F$ points outwards along the boundary of the tubular neighbourhood of $X$ in $Y$, we obviously have  $\ol W^s_f(x)= \ol W^s_F(x)$ for any  $x\in \Crit(f)$. Thus for any $x'\in \crit(G)$ and $y\in \Crit(F)$ we have 
$$
\ol \cM^{\varphi_!} (x',y)  \cong \overline{W}_G^u(x') \, _\varphi\!\times \overline{W}_f^s(y)\, =\,  \overline{W}_G^u(x') \, _\Id\!\times \overline{W}_F^s(y) \cong \ol\calm^{\Id_!}(x',y).$$
This will enable us to relate $\varphi_!$ and $\Id_!$. 

We start by showing that  the identification above  preserves orientations.  The orientation rule \eqref{eq:shriek-orientation-second} for  the two spaces yields 
$$\left(\ori\, \ol W_G^s(x'), \ori\, \ol\calm^{\varphi_!}(x',y)\right) \, =\, \ori\, \ol W^s_f(y)$$
and 
$$\left(\ori\, \ol W_G^s(x'), \ori\, \ol\calm^{\Id_!}(x',y)\right) \, =\, \ori\, \ol W^s_F(y).$$
We therefore have to compare the orientations of $\ol W^s_f(y)$ and $\ol W^s_F(y)$. The orientation convention \eqref{eq:orientation-rule-stable} writes here
$$\left(\ori\, \ol W^s_f(y), \ori\, \ol W^u_f(y)\right)\, =\, \ori\, X$$ and 
$$\left(\ori\, \ol W^s_F(y), \ori\, \ol W^u_F(y)\right)\, =\, \ori\, Y.$$ 
Denoting by $N$ the normal bundle of the submanifold $X\subset Y$, we fixed in \S\ref{sec:funct-closed-submanifolds} the  orientation rule \eqref{eq:orientation-N} 
$$\left(\ori\, \ol W^u_f(y), \ori\, N\right) \, =\, \ori\, \ol W^u_F(y).$$
Combined with the above this gives 
$$\left(\ori\,  \ol W_F^s(y), \ori\,  \ol W^u_f(y), \ori\, N\right)\ \, =\, \ori \, Y.$$
On the other hand,  the orientation rule \eqref{eq:both-oriented} 
$$\left(\ori\, X, \ori\, N\right)\, =\, \ori\,  Y$$ 
implies
$$\left(\ori\, \ol W^s_f(y), \ori\, \ol W^u_f(y), \ori \, N\right)\, =\, \left(\ori\, X, \ori\, N\right)\, =\, \ori \, Y. $$
We thus find \begin{equation}\label{eq:stable=stable} \ori\, \ol W^s_f(y)=\ori\, \ol W_F^s(y),\end{equation}  and hence we have $\ori \, \ol \calm^{\varphi_!}(x',y)=\ori\, \ol\calm^{\Id_!}(x',y)$ as claimed. 
Therefore we may (and will)  take the same  representing chain system on both spaces. This has to be corrected by the sign $(-1)^{[x'] -[y]-1}$  (cf.~\eqref{eq:correction-rcs-shriek}),  but this sign is the same for both spaces since equals $\mathrm{dim}(\ol \calm^{\varphi_!}(x',y))-1 =\mathrm{dim}(\ol\calm^{\Id_!}(x',y))-1$. Therefore we get identical corrected representing chain systems.
  
 Now using a homotopy which extends $H^X$ on the domain of $\Id:Y\ri Y$, and using the homotopy $H^Y$ on  the target, it is easy to check that the evaluation maps $q_{x',y}^{\varphi_!}$ and $q^{\Id_!}_{x',y}$ are also identical for any $x'\in \crit(G)$ and $y\in \crit(f)$. Putting this together we infer that $\varphi_!\, =\, \mathrm{pr}\circ \Id_!$.
 We finish the proof of the argument for inclusions by applying Proposition~\ref{id=continuation-shriek}, which implies
 $\varphi_! \, =\,  \mathrm{pr}\circ \Psi\, \, =\, \ol \varphi_!$
 using \eqref{eq:formula-phi-bar}.
 
 More generally, if $\varphi:X\ri Y$ is an embedding between oriented manifolds, we see it as the composition $i\circ \varphi_X$ where $\varphi_X:X\ri \varphi(X)$ is a diffeomorphism which preserves orientations and $i:\varphi(X) \hookrightarrow Y$ is the inclusion. We use $\varphi$ to  transport the set of data $\Xi^X$  on $\varphi(X)$,  where we denote it by  $\varphi_*(\Xi^X)$. We consider  $\ol{i}_!, i_!: C_*(Y, \Xi^Y;\calf)\ri C_*(\varphi(X), \varphi_*(\Xi^X); i^*\calf)$, the two shriek maps defined by the inclusion;  by the above we have $\ol{i}_!= i_!$. 
 
 By definition we have $\ol\varphi_!= (\varphi_X)_!\circ \ol i_!= (\varphi_X^{-1})_*\circ \ol i_!$. On the other hand there is a natural identification  $\ol\calm^{\varphi_!}(x',y)\cong \ol \calm^{i_!}(x',\varphi(y))$ which moreover preserves orientations.  It is easy to see that this identification matches with the evaluations $q_{x',y}^{\varphi_!}:\ol\calm^{\varphi_!}(x',y)\ri \Omega Y$ and $q_{x',\varphi(y)}: \ol\calm^{i_!}(x',\varphi(y))\ri \Omega Y$. Since $(\phi_X)_!=(\varphi_X^{-1})_*$ maps $\alpha\otimes \varphi(y)$ to $\alpha\otimes y$ we infer that $\varphi_!= (\phi_X)_!\circ i_!$  and therefore, as $i_!=\ol i_!$, we get $\varphi_!=\ol\varphi_!$, as claimed. 
 
 \smallskip
\underline{Proof of (b).}  The map  $\ol\pi_!$ is induced in  homology by the map of complexes  $C_*(Y,\Xi^Y;\calf)\ri C_{*+m}(D^m\times Y, \Xi^D\times \Xi^Y;\pi^*\calf)$  defined by 
$$\ol\pi_!(\alpha\otimes x) \, =\, \alpha\otimes(0,x).$$
Indeed, the shriek map of the inclusion $i: Y\hookrightarrow D\times Y$ was defined in~\S\ref{sec:funct-closed-submanifolds} by $\ol i_!(\alpha\otimes(0,x))=\alpha \otimes x$, and since $\pi\circ i=\Id$ and $i_!$ was proved to be an isomorphism,  we get  $\ol\pi_!=(\ol i_!)^{-1}$ in homology using  the composition rule. Note that this particular formula for $\ol i_!$, and therefore also the one for $\ol\pi_!$, are valid under the orientation convention~\eqref{eq:orientation-N} for unstable manifolds from~\S\ref{sec:funct-closed-submanifolds}, as was discussed above. We will use in the sequel the fact that the orientations conventions 
yield the shriek map of the inclusion $i$.

 Let us now describe   the   map $\pi_!: C_*(Y,\Xi^Y;\calf)\ri C_{*+m}(D\times Y, \Xi^D\times \Xi^Y;\pi^*\calf)$ and prove that it coincides with $\ol\pi_!$ in homology.

  If $G: Y\ri \R$ is the Morse function of $\Xi^Y$ then we  have  $\ol W^s_{h+G}(0,y)  = \{0\}\times \ol W^s_G(y)$ for any $y\in \Crit(G)$ since $0$ is a maximum for  the function $h$ . Thus for any $x',y\in \Crit(G)$ we have 
\begin{align*}
\ol\calm^{\pi_!}(x',(0,y)) & \cong \ol W^s_{h+G}(0,y) \, _\pi\!\times \overline{W}_G^u(x') \\
& \cong \{0\} \times \ol\calm_G^{\Id_!}(x',y)\cong \ol\calm_G^{\Id_!}(x',y).
\end{align*}
Moreover this identification preserves the orientations. To see this, recall that the orientations are given by the rule \eqref{eq:shriek-orientation-second}  and  it suffices to justify that $\ori\, \ol W_{h+G}^s(0,y)= \ori \, \ol W^s_G(y)$. We proved this property \eqref{eq:stable=stable} at point (a) above for any embedding $\varphi:X\ri Y$, so we may apply it to our particular situation $i: Y\hookrightarrow D\times Y$. 

We are therefore allowed to choose the same representing chain system for the two sets of moduli spaces $ \ol\calm^{\pi_!}(x',(0,y))$ and  $\ol\calm_G^{\Id_!}(x',y)$, and moreover the same corrected representing chain system (see \eqref{eq:correction-rcs-shriek}), since the sign of the correction is $(-1)^{\mathrm{dimension}-1}$. 
 
  It is also easy to check that, due to the particular  choice of our data  on $D$,   the evaluations $q^{\pi_!}: \ol\calm^{\pi_!}(x',(0,y))\ri \Omega Y$ and $q^{\Id_!}: \ol\calm_G^{\Id_!}(x',y)\ri\Omega Y$ coincide (via the identification above). Recalling  the formula for $\ol\pi_!$ given at the beginning of the proof we  infer that 
  $$\pi_! = \ol\pi_!\circ \Id_!,$$
   where $\Id_!: C_*(Y,\Xi^Y;\calf)\ri C_*(Y,\Xi^Y;\calf)$ is the shriek map of the identity.  Applying Proposition~\ref{id=continuation-shriek} we find that $\Id_!=\Id$ and the conclusion follows. 
   
   \smallskip \underline{Proof of (c).} The proof goes along the same lines as the proof of Proposition~\ref{alex=mihai}(c).
We establish a formula at the level of complexes for the shriek map of the embedding  $\varphi^\chi: X\ri D\times Y$ and try to relate it to the formula for $\varphi_!$.  We begin with the description of the relevant moduli spaces.   Denote as above by  $G$  the Morse function for $\Xi^Y$ and by $h: D\ri \R$ the function with a unique maximum at $0$ for the set of data $\Xi^D$.  The negative gradient of $h$ points outwards along $\partial D$, so we have $\ol W^u_{h+G}(0,x')\, =\, D\times \ol W^s_G(x')$ for any $x'\in \Crit(G)$. Denoting by  $f: X\ri\R$  the Morse function of $\Xi^X$ we then get for any   $x'\in \Crit(G)$ and $y\in \crit(f)$: 
\begin{align*}
\ol\calm^{\varphi^\chi_!}((0,x'),y)\,  & \cong \ol W^s_{f}(y) \, _{(\chi,\varphi)}\!\times (D\times \overline{W}_{G}^u(x')) \\
& \cong\, \{(a,b, \chi(a))\, |\, (a,b)\in  \ol W^s_{f}(y) \, _{\varphi}\!\times  \overline{W}_{G}^u(x')\}.
\end{align*}
This space projects onto  $\ol\calm^{\varphi_!}(x',y)=\ol W^s_{f}(y) \, _{\varphi}\!\times  \overline{W}_{G}^u(x')$ and the projection $\Pi$  is a homeomorphism. Let us show that $\Pi$ preserves the orientations. The orientation rule \eqref{eq:shriek-orientation-second} writes here 
$$
\left(\ori\, \ol W_{h+G}^s(0,x'), \ori \, \ol\calm^{\varphi^\chi_!}((0,x'),y))\right)\, =\, \ori\, \ol W^s_f(y).
$$
According to \eqref{eq:stable=stable} we have  $\ol W_{h+G}^s(0,x')= \ol W^s_G(x')$ {\it as oriented manifolds} for any embedding.  Therefore we find the same  orientation as for  $\ol\calm^\varphi(x,y')$, i.e.,
$$\left(\ori\, \ol W_{G}^s(x'), \ori \, \ol\calm^{\varphi_!}(x',y)\right)\, =\, \ori\, \ol W^s_f(y),$$
which proves that $\Pi$ preserves orientations. 
As a consequence, we can choose 
$(\Pi^{-1})_*(\sigma_{x',y})$  as a corrected  representing chain system on $\ol \calm^{\varphi^\chi}((0,x'), y)$, where $(\sigma_{x',y})$ is a  corrected representing chain system for $\ol \calm^\varphi(x',y)$ (the correcting sign is the same for both moduli spaces since they have the same dimension). 

Using  the homotopies $H^X$ on  the source and $\Id\times H^Y$ on the target we then  define the evaluation map $q^{\varphi^\chi_!}:\ol\calm^{\varphi^\chi_!}((0,x'),y)\ri \Omega(D\times Y)$  using the formula~\eqref{eq:evaluation-shriek}. Evaluating the above representing chain system we get 
$$
\nu^{\varphi^\chi_!}_{(0,x'),y}\, =\, - q^{\varphi^\chi_!}_*(\Pi^{-1}_*(\sigma_{x',y})),
$$ 
and then 
$$\varphi^\chi_!(\alpha\otimes (0,x'))\, =\, \alpha\sum_{y\in \crit f}\pi_*(\nu_{(0,x') ,y}^{\varphi^\chi_!})\otimes y$$
for any $\alpha\in \calf$ and $x'\in \crit(G)$, where $\pi: D\times Y\ri Y$ is the projection. As in the proof of Proposition~\ref{alex=mihai} the presence of  $\pi_*$ in the above formula is due to the fact that the DG-module on $D\times Y$ is $\pi^*\calf$. 

Now using the equality (in homology)  $\pi_!= \ol\pi_!$ proved in (b) and the formula for the latter we obtain for any $x'\in \crit(G)$ and $\alpha\in \calf$ 
\begin{align*}
\varphi^\chi_!\circ \pi_!(\alpha\otimes x') & = \varphi^\chi_!\circ \ol\pi_!(\alpha\otimes x')\\
& = \varphi^\chi_!(\alpha\otimes (0,x'))\, =\, \alpha\sum_{y\in \crit f}\pi_*(\nu_{(0,x') ,y}^{\varphi^\chi_!})\otimes y.
\end{align*}
This has to be compared with 
$$\varphi_!(\alpha\otimes x') \, =\, \alpha\sum_{y\in \crit f}\nu_{x',y}^{\varphi_!}\otimes y,$$
where $\nu^{\varphi_!}_{x,y'}= - q^{\varphi_!}_*(\sigma_{x',y})$. 
There would be equality between the two if the diagram 
$$\xymatrix{\ol\calm^{\varphi^\chi_!}((0,x'),y)\ar[r]^-{q^{\varphi^{\chi}_!}}\ar[d]_{\Pi} & \Omega(D\times Y)\ar[d]^-\pi\\
\ol \calm^{\varphi_!}(x,y')\ar[r]^-{q^{\varphi_!}} &\Omega Y}$$
was commutative. However, just like in the proof Proposition~\ref{alex=mihai}, we are only able to show that the diagram is homotopy commutative. The proof of this result is completely analogous to the one of Lemma~\ref{lem:homotopic-evaluations-bis} and we leave the details to the reader. 

Having a homotopy commutative diagram however suffices to conclude that $\varphi^\chi_!\circ \pi_!=\varphi_!$ at the level of homology, in a similar manner as in the proof of Proposition~\ref{alex=mihai}. 
This concludes the proof of Proposition~\ref{alex=mihai-shriek}.
\end{proof} 

\section[Maps between manifolds of equal dimension]{Maps between closed oriented manifolds of equal dimension} \label{sec:degree}

We prove in this section in the context of DG-coefficients two results about shriek maps between closed orientable manifolds of the same dimension. Their counterparts with local coefficients are classical (see for example~\cite[Chapter~VIII, Proposition~10.10]{Dold}).

\begin{proposition} \label{prop:phi*phi!}
Let $\varphi:X\to Y$ be a map of degree $d\in\Z$ between closed oriented manifolds of equal dimensions. Let $\cF$ be a DG local system on $Y$. Then 
$$
\varphi_* \varphi_!=d\cdot \Id.
$$
In particular, if $d=\pm 1$ the map $\varphi_!$ is injective and the map $\varphi_*$ is surjective. 
\end{proposition}

\begin{proof} The identity $\varphi_* \varphi_!=d\cdot \Id$ is proved by showing that the corresponding chain maps are homotopic. We  
use the definition of direct maps and shriek maps from~\S\ref{sec:funct-direct-alex} and~\S\ref{sec:second-definition-shriek}. Denote  by $\Xi^X=(f,\eta,\ldots)$ and by $\Xi^Y=(g,\xi, \ldots)$   two sets of data on $X$, resp.  $Y$, and denote by $\phi^t_\eta$ the flow of the pseudo-gradient  $\eta$ for $t\in\R$. The composition $\varphi_*\circ\varphi_!$ is described using product moduli spaces $\ol{\cM}^{\varphi_!}(x',y)\times \ol{\cM}^\varphi(y,z')$  for $x',z'\in \crit(g)$ and $y\in\crit(f)$, where the first factor involves the pseudo-gradients $\xi$ and $\eta$, and the second factor involves the pseudo-gradients $\eta$ and $\xi$. 
 The homotopy is constructed by considering the moduli spaces
$$
\cH(x',z')=\bigcup_{t>0} \phi^t_\eta\varphi^{-1}(W^u(x')) \cap \varphi^{-1}(W^s(z')).
$$

\begin{figure}
  \centering
  \includegraphics[width=\textwidth]{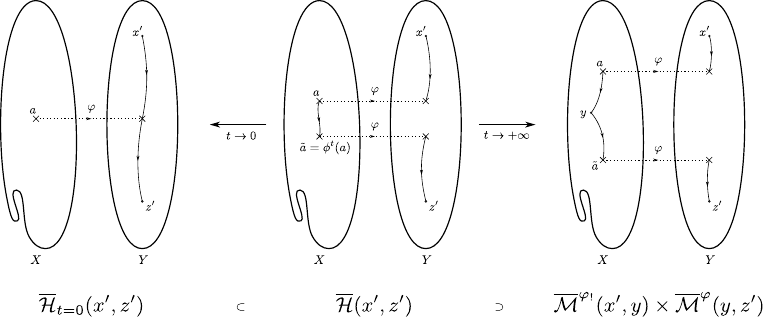}
  \caption{The space $\cH(x',z')$, and its boundary parts associated to $t=0$
    and $t=+\infty$.}
  \label{fig:ShriekStar}
\end{figure}

These have to be interpreted as configurations consisting of a
half-infinite trajectory of $\xi$ flowing out of $x'$ to a point $b$ in
$Y$, a trajectory of $\eta$ of length $t>0$ in $X$ beginning at a point
$a$ such that $\varphi(a)=b$, and a half-infinite trajectory of $\xi$ in
$Y$ beginning at $\varphi(\phi^t_\eta(a))$ and flowing into $z'$. The
compactification $\ol{\cH}(x',z')$ is a manifold with boundary and
corners of dimension $|x'|-|z'|+1$. Its boundary can be described as 
\begin{align*}
\p\ol{\cH}(x',z') = & \bigcup_{x''\in\crit(g)} \!\!\!\ol{\cL}(x',x'')\times \ol{\cH}(x'',z') \cup \bigcup_{z''\in\crit(g)} \!\!\!\ol{\cH}(x',z'')\times \ol{\cL}(z'',z') \\
& \cup \ol{\cH}_{t=0}(x',z') \cup \bigcup_{y \in\crit(f)}\ol{\cM}^{\varphi_!}(x',y)\times \ol{\cM}^\varphi(y,z'),
\end{align*}
where the first two terms correspond to Morse breaking at $\pm\infty$, the fourth term corresponds to Morse breaking in the limit $t\to\infty$, and the third term $\ol{\cH}_{t=0}(x',z')$ corresponds to the limit $t\to 0$ and is described as 
$$
\ol{\cH}_{t=0}(x',z') = \varphi^{-1}(\ol{W}^u(x')\cap \ol{W}^s(z')). 
$$
 See Figure \ref{fig:ShriekStar}. The choice of a representing chain system for the compactified moduli spaces $\ol{\cH}(x',z')$ provides a homotopy between $\varphi_*\varphi_!$ and the morphism defined by the moduli spaces $\ol{\cH}_{t=0}(x',z')$. 
We will show that  the latter induces  $d\cdot \mathrm{Id}$ in homology. 

The moduli space $\ol{\cH}_{t=0}(x',z')$ can be interpreted as the compactification of the space of pairs $(a,\gamma)$, $a\in X$, $\gamma\in \cM(x',z')$ a parametrized trajectory from $x'$ to $z'$, and $\varphi(a)=\gamma(0)$. 
If $x'=z'$ the moduli space $\ol{\cH}_{t=0}(x',x')$ is identified with $\varphi^{-1}(x')$. 
If $|x'|>|z'|$ the moduli space  $\ol{\cH}_{t=0}(x',z')$ is a compact  manifold with boundary and corners of dimension $|x'|-|z'|$; its boundary is 
\begin{align*} \p & \ol{\cH}_{t=0}(x',z')\\
&=\bigcup_{x''\in \crit( g)} \!\!\!\ol\call(x',x'')\times  \ol{\cH}_{t=0}(x'',z')\, \cup\bigcup_{z''\in\crit(g)} \!\!\!\ol{\cH}_{t=0}(x',z'')\times\ol\call(z'',z').\end{align*}

 In order  to prove that $\varphi_*\varphi_!$ and $d\cdot\Id$ are homotopic using the moduli spaces $\ol\cH(x',z')$ and $\ol\cH_{t=0}(x',z')$, we first need to orient these spaces and then to construct  a representing chain system on each one  of them.

 \underline{Orientations.}  To orient $\ol\cH(x',z')$ we proceed as follows: We choose orientations on $X$ and on $Y$ and on the unstable manifolds of all critical points of $f$ and  $g$. Consider the submersion $\Gamma: X\times [0,+\infty)\times X\ri X\times X$ defined by $$\Gamma(a_1, t, a_2)\, =\, (a_1, \phi_\eta^t(a_2))$$
and denote by $Z$ the pre-image of the diagonal $Z=\Gamma^{-1}(\Delta_X)$ which may be seen as the set of  segments of pseudo-gradient trajectories in $X$. This manifold is diffeomorphic to $X\times [0, +\infty)$ and it inherits the orientation of the latter: 
\begin{equation}\label{eq:rule1}
\ori \, Z\, =\, \left(\ori\, X, \p_t\right)
\end{equation}
We compactify it to the manifold with boundary and corners $$\ol Z=Z\cup \bigcup_{y\in \Crit(f)}\ol W^s(y)\times \{+\infty\}\times \ol W^u(y).$$ Note that the boundary $\p\ol Z$ also contains the diagonal $\{(a,0,a)\, |\, a\in X\}$. Now if we define $\chi :\ol Z\ri Y\times Y$ by $$\chi(a_1,t,a_2)= (\varphi(a_1), \varphi(a_2)),$$
 we remark that $\ol\cH(x',z')= \chi^{-1}(\ol W^u(x')\times \ol W^s(z'))$ and we will orient our moduli space using this description: the computations will be simpler if instead of the usual rule $$\left(\ori\, \ol\cH(x',z'), \coori \, \ol W^u(x'), \coori\, \ol W^s(z')\right)\, =\, \ori \, \ol Z$$
we take $$\left( \coori\,  \ol W^u(x'), \ori\, \ol\cH(x',z'), \coori\, \ol W^s(z')\right)\, =\, \ori \, \ol Z$$
which we write using \eqref{eq:rule1}: 
\begin{equation}\label{eq:rule2}
\left( \ori \, \ol W^s(x'), \ori\, \ol\cH(x',z'), \ori\, \ol W^u(z')\right)\, =\, \left( \ori \, X, \p_t\right).\end{equation} 

To orient the stable manifolds we recall the  convention \eqref{eq:orientation-rule-stable}  from \S\ref{sec:second-definition-shriek}:
$$
\left(\ori\,  \ol W^s,\ori\, \ol W^u\right)\, =\, \ori\, X,
$$
and similarly for $Y$. 

We choose to orient  $\ol\cH_{t=0}(x',z')$ as the boundary of $\ol\cH(x',z')$, using the outward normal vector field which is $-\p_t$, so $$\left(-\p_t, \ori \, \ol\cH_{t=0}(x',z')\right)= \ori \, \ol\cH(x',z').$$  Using  \eqref{eq:rule2} this yields the following orientation rule for $\ol\cH_{t=0}(x',z')$: 
$$\left( \ori \, \ol W^s(x'), -\p_t, \ori\, \ol\cH_{t=0}(x',z'), \ori\, \ol W^u(z')\right)\, =\, \left( \ori \, X, \p_t\right). $$
Therefore, moving $\p_t$ to the end of the left hand side 
\begin{equation}\label{eq:rule3}
\left( \ori \, \ol W^s(x'), \ori\, \ol\cH_{t=0}(x',z'), \ori\, \ol W^u(z')\right)\, =\, (-1)^{|x'|+1}  \ori \, X\end{equation} 

The case $x'=z'$ is of particular interest. The manifold $\cH(x',x')$ being $0$-dimensional equation (\ref{eq:rule3}) writes 
$$\left( \ori\, \cH_{t=0}(x',x'), \ori \, \ol W^s(x'), \ori\, \ol W^u(x')\right)\, =\, (-1)^{|x'|+1}  \ori \, X$$
which using \eqref{eq:orientation-rule-stable} becomes: 
\begin{equation}\label{eq:rule4}
\left( \ori\, \cH_{t=0}(x',x'), \ori\, Y\right)\, =\, (-1)^{|x'|+1}\ori\, X.
\end{equation}
This means that the orientation of a point $a\in \varphi^{-1}(x')= \cH_{t=0}(x',x')$ differs from the usual one (which is $\pm 1$ depending on whether $T_a\varphi$ preserves orientations or not) by the sign $(-1)^{|x'|+1}$. 


 The next step is to  determine the difference between the  boundary orientation of   $\p \cH(x',z')$ and the orientations already assigned to the different parts of this boundary. Using the orientation rule \eqref{eq:rule2} we get: 
 
 \begin{lemma}\label{lem:last-orientation-differences} We have: 
 $$\ori\,  \, \p\ol \cH(x',z')\, = \, (-1)^{|x|'-|x''|}\ori\, \left(\ol\call(x',x'')\times \ol\cH(x'',z')\right)$$
 $$\ori\,  \, \p\ol \cH(x',z')\, = \, (-1)^{|x|'-|z'|}\ori\, \left( \ol\cH(x',z'')\times\ol\call(z'',z')\right)$$
 $$\ori\, \,  \p\ol \cH(x',z')\, = \, (-1)^{|x|'}\ori\, \left(\ol\calm^ {\varphi_!}(x',y)\times \ol\calm^\varphi(y,z')\right)$$
 $$\ori\, \,  \p\ol \cH(x',z')\, = \, \ori\, \, \ol\cH_{t=0}(x',z')$$
 \end{lemma}
 \begin{proof} Let $n$ be the outward normal to the space $\p\ol\cH(x',z')$ at a point of $\call(x',x'')\times \cH(x'',z')$. According to \eqref{eq:rule2} we have 
 \begin{equation}\label{compare1}\left( \ori \, \ol W^s(x'), n, \ori\, \p\ol\cH(x',z'), \ori\, \ol W^u(z')\right)\, =\, \left( \ori \, X, \p_t\right).\end{equation}
 
 Reasoning like in the proof of Lemma \ref{lem:orientation-difference-shriek} we may suppose that $n=\xi$,  the outward normal of $\p \ol W^u(x')$ along $\call(x',x'')\times W^u(x'')$, which at $(\lambda,x'')$ is the gradient vector on $Y$. Applying  \eqref{eq:changed-orientation-rule} in the third equality below and \eqref{eq:rule2} in the fourth we get: 
 \begin{align*}\big(& \ori \, \ol W^s(x'),  n, \ori\, \ol\call(x',x''), \ori\, \ol\cH(x'',z'), \ori\, \ol W^u(z')\big)\\
 &=\, \left( \ori \, \ol W^s(x'), \xi,  \ori\,  \ol\call(x',x''), \ori\, \ol\cH(x'',z'), \ori\, \ol W^u(z')\right)\\
  &=\, (-1)^{|x'|-|x''|}\left( \ori \, \ol W^s(x'),   \ori\,  \ol\call(x',x''), -\xi, \ori\, \ol\cH(x'',z'), \ori\, \ol W^u(z')\right)\\
  &=\, (-1)^{|x'|-|x''|}\left( \ori \, \ol W^s(x''),   \ori\, \ol\cH(x'',z'), \ori\, \ol W^u(z')\right)\\
  &=\, (-1)^{|x'|-|x''|}\left( \ori \, X, \p_t\right),
 \end{align*}
 according to \eqref{eq:rule1}. Comparing the above to \eqref{compare1} we get our first relation. 
 
 The proof of the second one is very much similar. Again like in the proof of Lemma \ref{lem:orientation-difference-shriek} we may take $n=-\xi$, the outward normal of $\p\ol W^s(z')$ along $W^s(z'')\times \call(z'',z')$, and we infer applying the orientation rule  \eqref{eq:orientation-rule} in the third equality and \eqref{eq:rule2} in the fourth: 
 \begin{align*}\big( &\ori \, \ol W^s(x'),  n,  \ori\, \ol\cH(x',z''), \ori\, \ol\call(z'',z'),  \ori\, \ol W^u(z')\big)\\
 &=\, \left( \ori \, \ol W^s(x'), -\xi,    \ori\, \ol\cH(x',z''), \ori\, \ol\call(z'',z'), \ori\, \ol W^u(z')\right)\\
  &=\, (-1)^{|x'|-|z'|}\left( \ori \, \ol W^s(x'),     \ori\, \ol\cH(x',z''), \ori\, \ol\call(z'',z'), -\xi, \ori\, \ol W^u(z')\right)\\
  &=\, (-1)^{|x'|-|z'|}\left( \ori \, \ol W^s(x'),     \ori\, \ol\cH(x',z''), \ori\,  \ol W^u(z'')\right)\\
  &=\, (-1)^{|x'|-|z'|}\left( \ori \, X, \p_t\right),
 \end{align*}
which compared to \eqref{compare1} implies the claimed sign difference.

Let us now analyse the third relation. By definition of $\ol\cH(x',z')$ the  outward normal vector  at a point of $\calm^{\varphi_!}(x',y)\times \calm^\varphi(y,z')\subset \p\ol\cH(x',z')$ is also an outward normal vector for $\p\ol Z$ at a point of $W^s(y)\times\{+\infty\}\times W^u(y)$ and via the (orientation preserving)  diffeomorphism between $Z$ and $X\times [0,+\infty)$ we may consider that $n=\p_t$. We get by applying the orientation rules  \eqref{eq:shriek-orientation-second} and  \eqref{eq:orientation-direct-second} for $\ol\calm^{\varphi_!}(x',y)$ resp. $\ol\calm^\varphi(y,z')$: 
\begin{align*} 
\big(& \ori \, \ol W^s(x'),  n,  \ori\, \ol\calm^{\varphi_!}(x',y), \ori\, \ol\calm^\varphi(y,z'),  \ori\, \ol W^u(z')\big)\\
&=\, \left( \ori \, \ol W^s(x'),  \p_t,  \ori\, \ol\calm^{\varphi_!}(x',y), \ori\, \ol\calm^\varphi(y,z'),  \ori\, \ol W^u(z')\right)\\
&=\, (-1)^{|x'|}\left( \ori \, \ol W^s(x'),  \ori\, \ol\calm^{\varphi_!}(x',y), \ori\, \ol\calm^\varphi(y,z'),  \ori\, \ol W^u(z'),\p_t\right)\\
&=\, (-1)^{|x'|}\left( \ori \, \ol W^s(y), \ori\, \ol W^u(y),\p_t\right)\\
&=\, (-1)^{|x'|}\left( \ori X,\p_t\right)
\end{align*}
which implies the third relation after comparison with \eqref{compare1}.

Finally the last relation is valid by definition of the orientation of $\ol\cH_{t=0}(x',z')$. The lemma is proved.  
 \end{proof} 
 
 Now we do the same for the manifold $\ol\cH_{t=0}(x',z')$. Based on the orientation rule $\eqref{eq:rule3}$ we get the following sign differences: 
\begin{lemma}\label{lem:last-last-orientation-differences} We have: 
 $$\ori\,  \, \p\ol \cH_{t=0}(x',z')\, = \, \ori\, \left(\ol\call(x',x'')\times \ol\cH_{t=0}(x'',z')\right)$$
 and
 $$\ori\,  \, \p\ol \cH_{t=0}(x',z')\, = \, (-1)^{|x|'-|z'|-1}\ori\, \left( \ol\cH_{t=0}(x',z'')\times\ol\call(z'',z')\right)$$
 \end{lemma}
\begin{proof}

 As above, we may consider the outward normal $n$ at some point belonging to $\ol\call(x',x'')\times\ol\cH_{t=0}(x'',z')\subset \p\cH_{t=0}(x',z')$ as being the outward normal of $\p\ol W^u(x')$ at a point of $\ol\call(x',x'')\times W^u(x'')$ and choosing this point of the form $(\lambda, x'')$ we may suppose $n=\xi$, the (negative) gradient on $Y$. As in the proof of the preceding lemma we compare the boundary orientation given by 
$$\left(n, \ori \, \p\ol\cH_{t=0}(x',z')\right)\, =\, \ori\, \ol\cH_{t=0}(x',z')$$ with the product orientation 
$$\left(n, \ori\, \ol\call(x',x''), \ori \, \p\ol\cH_{t=0}(x'',z')\right)$$
For this purpose we use the relation \eqref{eq:rule3} which transforms the first relation in 
\begin{equation}\label{eq:compare2} 
\left( \ori \, \ol W^s(x'), n, \ori\, \p\ol\cH_{t=0}(x',z'), \ori\, \ol W^u(z')\right)\, =\, (-1)^{|x'|+1}  \ori \, X,
\end{equation} 
while the second writes using \eqref{eq:changed-orientation-rule} in the third equality below and then \eqref{eq:rule3} in the fourth:
\begin{align*}
&\big( \ori \, \ol W^s(x'), n, \ori\, \ol\call(x',x''), \ori\, \ol\cH_{t=0}(x'',z'), \ori\, \ol W^u(z')\big)\\
&= ( \ori \, \ol W^s(x'), \xi, \ori\, \ol\call(x',x''), \ori\, \ol\cH_{t=0}(x'',z'), \ori\, \ol W^u(z'))\\
&= (-1)^{|x'|-|x''|}( \ori \, \ol W^s(x'), \ori\, \ol\call(x',x''), -\xi, \ori\, \ol\cH_{t=0}(x'',z'), \ori\, \ol W^u(z'))\\
&= (-1)^{|x'|-|x''|}( \ori \, \ol W^s(x''), \ori\, \ol\cH_{t=0}(x'',z'), \ori\, \ol W^u(z'))\\
&= (-1)^{|x'|+1}\ori\, X.
\end{align*} 
The comparison with \eqref{eq:compare2} shows the desired relation. 

To prove the second one we analogously may suppose that the outward normal vector $n$ at some point of $\ol\cH_{t=0}(x',z'')\times\ol\call(z'',z')\subset\p\ol\cH_{t=0}(x',z')$ is also pointing outwards $\ol W^s(z')$ along its  boundary at a point of $\ol W^s(z'')\times \ol\call(z'',z')$ and therefore we may take  $n=-\xi$. We determine the sign difference by putting the product orientation in the relation: 
\begin{align*}
&\big( \ori \, \ol W^s(x'), n,  \ori\, \ol\cH_{t=0}(x',z''), \ori\, \ol\call(z'',z'), \ori\, \ol W^u(z')\big)\\
&= (   \ori \, \ol W^s(x'), -\xi,   \ori\, \ol\cH_{t=0}(x',z''), \ori\, \ol\call(z'',z'), \ori\, \ol W^u(z')  )\\
&=\!\!(-1)^{|x'|-|z'|-1} (   \ori \, \ol W^s(x'),    \ori\, \ol\cH_{t=0}(x',z''), \ori\, \ol\call(z'',z'), -\xi, \ori\, \ol W^u(z')  )\\
&=\!\!(-1)^{|x'|-|z'|-1} (   \ori \, \ol W^s(x'),    \ori\, \ol\cH_{t=0}(x',z''), \ori\,   \ol W^u(z'')  )\\
&=\!\!(-1)^{|z'|} \ori \, X,
\end{align*}
where we applied the orientation rule \eqref{eq:orientation-rule} in the third equality and then \eqref{eq:rule3} in the fourth.  By comparing to \eqref{eq:compare2} we observe a sign difference of $(-1)^{|x'|-|z'|-1}$,  as claimed. 
\end{proof} 

\underline{Representing chain systems.} Taking into account all these sign differences we use the same inductive procedure as in Proposition~\ref{representingchain} to construct  representing chain systems on $\ol\cH_{t=0}(x',z')$ and on $\ol \cH(x',z')$. According to Lemma \ref{lem:last-last-orientation-differences}, the first one, denoted by $\sigma^0_{x',z'}$, satisfies 
$$\p\sigma^0_{x',z'}\, =\, \sum_{x''\in \crit(g)}s_{x',x''}\times \sigma^0_{x'',z'}-\sum_{z''\in\crit(g)}(-1)^{|x'|-|z'|}\sigma^0_{x',z''}\times s_{z'',z'}.$$
We want this representing chain system to define a chain map between DG-complexes, so we transform the above relation  in an equation which looks like the one of the continuation cocycle \eqref{eq:DGcont} by setting $\tilde{\sigma}^0_{x',z'} =(-1)^{|z'|}\sigma^0_{x',z'}$. We indeed get 
\begin{equation}\label{eq-rcs0}
\p\tilde{\sigma}^0_{x',z'}\, =\, \sum_{x''\in \crit(g)}s_{x',x''}\times \tilde{\sigma}^0_{x'',z'}-\sum_{z''\in\crit(g)}(-1)^{|x'|-|z''|}\tilde{\sigma}^0_{x',z''}\times s_{z'',z'}.
\end{equation} 

We may construct the chain system  $\tilde{\sigma}^0_{x',z'}$ such that it satisfies an extra property:  Note that there is a canonical projection $\pi:\ol{\cH}_{t=0}(x',z')\to \ol{\cL}(x',z')$ to the moduli space of unparametrized gradient trajectories which forgets the point $a$ and the parametrization of $\gamma$, and the fiber of $\pi$ has dimension $1$ whenever $\ol{\cH}_{t=0}(x',z')$ is nonempty. Reasoning as in the proof of Lemma~\ref{special-rcs} one may construct the representing chain system $(\sigma^0_{x',z'})$ such that $\pi_*\sigma^0_{x',z'}=0$ if $|x'|>|z'|$  and therefore 
\begin{equation}\label{eq:extra-property}
\pi_*(\tilde{\sigma}^0_{x',z'})\, =\, 0\gol \gol \forall \, |x'|>|z'|
\end{equation} 
For $|x'|=|z'|$ we have by definition   $\tilde{\sigma}^0_{x',z'}=0$ except for the case $x'=z'$ when according to \eqref{eq:rule4} we have: 
\begin{equation}\label{eq:extra2} 
\tilde{\sigma}^0_{x',x'}\, =\, -\sum_{a\in \varphi^{-1}(x')}\mathrm{sgn}(T_a\varphi)\cdot a.
\end{equation}

We now study the representing chain system on $\ol\cH(x',z')$. By lemma \ref{lem:last-orientation-differences} the inductive construction yields 
 $S_{x',z'}\in  C_{|x'|-|y'|+1}(\ol \cH(x',z'))$ such that : 
\begin{align*}\label{eq:last-rcs}
\p & S_{x',z'} \\
& = \sigma^0_{x',z'}+ (-1)^{|x'|}\sum_{y\in\crit(f)}\tau_{x',y}\times\tau_{y,z'}\\
&\ + \sum_{x''\in \crit(g)} \!\!\!(-1)^{|x'|-|x''|}s_{x',x''}\times S_{x'',z'}+ \sum_{z''\in\crit(g)}\!\!\!(-1)^{|x'|-|z'|}S_{x',z''}\times s_{z'',z'},
\end{align*} 
In the relation above $\sigma_{x',z'}^0$ is the representing chain system previously constructed on $\ol\cH_{t=0}(x',z')$ whereas $\tau_{x',y}$ and $\tau_{y,z'}$ are respectively the representing chain systems constructed  in \S\ref{sec:second-definition-shriek} and  \S\ref{sec:funct-direct-alex} on $\ol\calm^{\varphi_!}(x',y)$ and $\ol\calm^\varphi(y,z')$. Now remember that in order to define the maps $\varphi_!$ and $\varphi_*$ we had to apply the sign corrections \eqref{eq:correction-rcs-shriek} and \eqref{eq:correction-rcs}, namely: 
$$\sigma_{x',y}\, =\, (-1)^{[x']-[y]-1}\tau_{x',y}\, = \, (-1)^{|x'|-|y|-1}\tau_{x',y}$$
and $$\sigma_{y,z'}\, =\, (-1)^{|y|-|z'|-1}\tau_{y,z'}.$$
Replacing, we  obtain
\begin{align*}
\p &  S_{x',z'} \\
&= \sigma_{x',z'}^0+ ( -1)^{|z'|}\sum_{y\in\crit(f)}\sigma_{x',y}\times\sigma_{y,z'}\\
&\ + \sum_{x''\in \crit(g)}\!\!\! (-1)^{|x'|-|x''|}s_{x',x''}\times S_{x'',z'}+ \sum_{z''\in\crit(g)}\!\!\!(-1)^{|x'|-|z'|}S_{x',z''}\times s_{z'',z'},
\end{align*} 
Finally, set $\widetilde{S}_{x',z'}= (-1)^{|z'|}S_{x',z'}$ and recall that $\tilde{\sigma}^0_{x',z'}= (-1)^{|z'|}\sigma^0_{x',z'}$ by definition.  The relation becomes: 
\begin{align*}
\p & \widetilde{S}_{x',z'} \\
& = \tilde{\sigma}^0_{x',z'}+ \sum_{y\in\crit(f)}\sigma_{x',y}\times\sigma_{y,z'}\\
& \ + \sum_{x''\in \crit(g)} \!\!\!(-1)^{|x'|-|x''|}s_{x',x''}\times \widetilde{S}_{x'',z'}+ \!\!\!\sum_{z''\in\crit(g)}\!\!\!(-1)^{|x'|-|z''|}\widetilde{S}_{x',z''}\times s_{z'',z'}.
\end{align*} 
This looks quite similar to the homotopy relation \eqref{eq:alg-homotopy-eqn}, but we still have to convert it into an equation on $\Omega Y$. 

\underline{Evaluation maps.} We start with the definition of the evaluation on $\ol\cH(x',z')$: Denote by $p^Y:Y\ri Y/\caly^Y$ the projection and $\theta^Y:Y/\caly^Y\ri Y$ its homotopy inverse which is part of   the datum $\Xi^Y$.  Use similar notation for $X$. As in  \S\ref{sec:funct-direct-alex} and \S\ref{sec:second-definition-shriek}
denote by $H^Y: [0,1]\times Y\ri Y$ a homotopy between $\Id$ and $\theta^Y\circ p^Y$ and by $H^X$ its analogue on $X$. We define evaluation maps $q_{x',z'}: \ol\cH(x',z') \ri \Omega Y$ in the following way. Take an element $(\lambda_Y',\lambda_X,\lambda''_Y)\in \ol\cH(x',z')$, where $\lambda_Y'$ is a gradient line on $Y$  from $x'$ to some point $b=\varphi(a)$,  $\lambda_X$ is a gradient line on $X$ from $a$ to some point $\tilde{a}$ and finally $\lambda_Y''$ is a gradient line on $Y$ from $\varphi(\tilde{a})$ to $z'$. If we apply the evaluation maps $q^X$ and $q^Y$ defined by the data $\Xi^X$ and $\Xi^Y$ respectively to these three paths we get $\gamma_1=q^Y(\lambda'_Y)\in \calp_{\star\ri\theta^Y\circ p^Y(\varphi(a))}Y$, then $\gamma_2=q^X(\lambda_X)\in \calp_{\theta^X\circ p^X(p)\ri\theta^X\circ p^X(\tilde{a})}X$ and  finally  $\gamma_3=q^Y(\lambda'_Y)\in \calp_{\theta^Y\circ p^Y(\varphi(\tilde{a}))\ri\star}Y$. To get a loop on $Y$ from these paths we use the homotopies $H^X$ and $H^Y$ and define: 
\begin{align*} 
q_{x',z'}(\lambda',\lambda,\lambda'') =\gamma_1\,\#\,H^Y& (1\!-\!t,\varphi(a))\,\#\,\varphi (H^X(t,a)  \# \gamma_2\#H^X(1\!-\!t,\tilde{a})) \\
& \#\,H^Y(t, \varphi(\tilde{a}))\,\#\,\gamma_3.
\end{align*}

\begin{figure}
  \centering
  \includegraphics[width=\textwidth]{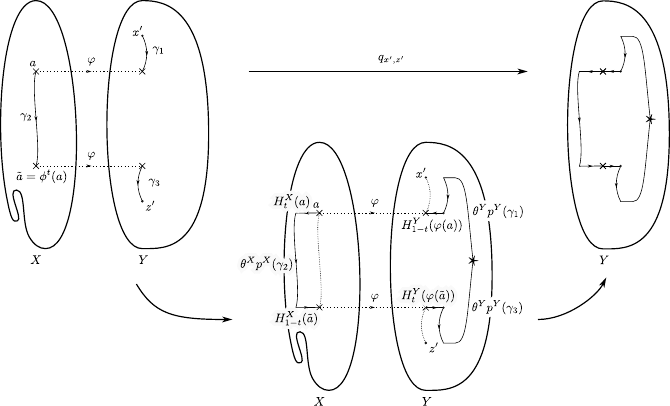}
  \caption{The evaluation map $q_{x',z'}:\ol\cH(x',z') \to \Omega Y$}
  \label{fig:ShriekStarEval}
\end{figure}

Let us see how these evaluations act on  the boundary $\p\ol\cH(x',z')$. It is easy to verify by looking at the formulas \eqref{eq:formula-evaluation-direct} and \eqref{eq:evaluation-shriek} which define the evaluations $q^\varphi_{y,z'}:\ol\calm^\varphi(y,z')\ri\Omega Y $ resp.  $q^{\varphi_!}_{x'y}:\ol\calm^{\varphi_!}(x',y)\ri \Omega Y$ that $$q_{x',z'}\, =\, q_{x',y}^{\varphi_!}\#q_{y,z'}^\varphi$$
on $\ol\calm^{\varphi_!}(x',y)\times\ol\calm^\varphi(y,z')$. In particular, 
\begin{equation}\label{eq:cocycle1}q_{x',y',*}(\sigma_{x',y}\times\sigma_{y,z'}) =(-q^{\varphi_!}_{x',y,*}(\sigma_{x',y}))\cdot (-q^\varphi_{y,z',*}(\sigma_{y,z'}))\, =\, \nu^{\varphi_!}_{x',y}\cdot\nu^{\varphi}_{y,z'},\end{equation}
the cocycle which defines the composition $\varphi_*\varphi_!$.

It is also immediate that $$q_{x',z'}\, =\, q^Y_{x',x''}\#q_{x'',z'}$$
on $\ol\call(x',x'')\times \ol\cH(x'',z')$ and 
$$q_{x',z'}\, =\, q_{x',z''}\#q^Y_{z'',z'}$$
on $\ol\cH(x',z'')\times\ol\call(z'',z')$ where $q^Y$ denotes the evaluation on the trajectory spaces of $Y$. 

On the last part of $\p\ol\cH$ which is $\ol\cH_{t=0}(x',z')$ the evaluation $q_{x',z'}$  equals by definition 
$$\gamma_1\,\#\,H^Y(1\!-\!t,\varphi(a))\,\#\,\varphi\left (H^X(t,a)\#H^X(1\!-\!t,a)\right)\,\#\,H^Y(t, \varphi(a))\,\#\,\gamma_3$$ 


But  we have another natural evaluation on this space. Recall that there is an obvious projection $\pi: \ol\cH_{t=0}(x',z')\ri \ol\call(x',z')$ and that we had an  evaluation   $q^Y:\ol\call(x',z')\ri \Omega Y$ defined by the datum $\Xi^Y$ on $Y$. It is natural to consider as evaluation the composition $q^0_{x',z'} =q^Y\circ\pi$. This formula is also valid  for   $x'=z'$ since $\pi$ factors through $\ol\calm(x',z')$ and we may define $q^Y$ on this space. Notice that with the above notation  $q^0_{x',z'}$ equals $\gamma_1\#\gamma_3$.
There is an obvious homotopy between  $q_{x',z'}$ and $q^0_{x'z'}$ which moreover satisfies the hypothesis of Lemma \ref{lem:homotopic-evaluations}. The representing chain system $(\tilde{\sigma}^0_{x',z'})$ verifies Equation \eqref{eq-rcs0}, so it also fulfils  the conditions required by this lemma.  We may therefore apply it and infer that the cocycles  $\nu_{x',z'}:= -q_{x'z',*}(\tilde{\sigma}^0_{x',z'})$ and  $\nu^0_{x',z'}:= -q^0_{x'z',*}(\tilde{\sigma}^0_{x',z'})$ define homotopic chain maps $\Psi,\Psi^0: C_*(Y,\Xi^Y;\calf)\ri C_*(Y, \Xi^Y; \calf)$ (their formula is $\Psi(\alpha\otimes x')=\sum_{z'\in\crit(g)}\alpha\nu_{x',z'}\otimes z'$ and analogously for $\Psi^0$.)  

On the other hand, the cocycle $h_{x',z'}\in C_{|x'|+|z'|+1} (\Omega Y)$ defined by $h_{x',z'}=q_{x',z',*}(\widetilde{S}_{x',z'})$ satisfies an equation of the form \eqref{eq:alg-homotopy-eqn} and therefore yields a chain homotopy between $\Psi$ and $\varphi_*\varphi_!$ (we also used \eqref{eq:cocycle1} here). Therefore $\varphi_*\varphi_!$ is chain homotopic to $\Psi^0$ and it suffices to show that the latter equals $\mathrm{deg}(\varphi)\cdot\Id$ to finish our proof. This is straightforward: by the property \eqref{eq:extra-property} combined with the definition of $q^0$ we get $\nu^0_{x',z'}=0$ for $x'\neq z'$ and finally, since $q^0_{x',x'}$ is constant equal to the basepoint $\star$, the property \eqref{eq:extra2} implies $$\nu^0_{x',x'}=-q_{x',x', *}(\tilde{\sigma}^0_{x',x'})= \mathrm{deg}(\varphi)\cdot \star,$$
as claimed. 

The proof of Proposition \ref{prop:phi*phi!} is now complete. 
\end{proof}

The following corollary was previously announced as Corollary~\ref{cor:homotopy_equivalence_shriek_announcement}.  

\begin{corollary} \label{cor:homotopy_equivalence_shriek}
Let $\varphi:X\to Y$ be an orientation preserving homotopy equivalence between closed oriented manifolds and let $\cF$ be a DG local system on $Y$. The canonical maps $\varphi_!:H_*(Y;\cF)\to H_*(X;\varphi^*\cF)$ and $\varphi_*:H_*(X;\varphi^*\cF)\to H_*(Y;\cF)$ are isomorphisms inverse to each other.  
\end{corollary}

\begin{proof} An orientation preserving homotopy equivalence has degree $1$, hence $\varphi_*\varphi_!=\Id$ by Proposition~\ref{prop:phi*phi!}. On the other hand, by functoriality (Corollary~\ref{cor:homotopy_equivalence}) we know that $\varphi_*$ is an isomorphism. Thus $\varphi_!$ is an isomorphism and $\varphi_!=\varphi_*^{-1}$.
\end{proof}

We close this section with a slight generalization of 
Proposition~\ref{prop:phi*phi!} regarding transverse maps.

Consider two smooth manifolds $X$ and $Y$ and two smooth and transverse
maps $X\xrightarrow{\varphi}Z$ and $Y\xrightarrow{\psi}Z$ into a third
manifold $Z$. The fibre product of $X$ and $Y$ above $\varphi$ and $\psi$ then
fits into the following diagram~:
$$
\begin{tikzcd}
   X \fibreprod{\varphi}{\psi}Y
  \arrow[d,"\pi_{1}"']\arrow[r,"\pi_{2}"]&Y\arrow[d,"\psi"]\\
  X\arrow[r,"\varphi"]& Z\\
\end{tikzcd}
$$

We pick some DG coefficients $\cF$ on $Z$, and pull them back on $X$,
$Y$, and $X\fibreprod{\varphi}{\psi} Y$ via $\varphi$, $\psi$, and
$\varphi\circ \pi_{1}=\psi\circ \pi_{2}$.

\begin{proposition}\label{prop:TransverseSbmfds}
  In the above situation, we have
  $$
  \psi_{!}\varphi_{*} = \pi_{2*}\pi_{1!}
  $$
  and
  $$
  \varphi_{!}\psi_{*} = \pi_{1*}\pi_{2!}.
  $$
\end{proposition}
\begin{proof}
  The proof is very similar to that of Proposition~\ref{prop:phi*phi!}.
  We choose Morse functions $X\xrightarrow{f_{X}}\R$,
  $Y\xrightarrow{f_{Y}}\R$, $Z\xrightarrow{f_{Z}}\R$, and
  $X\fibreprod{\varphi}{\psi}Y \xrightarrow{f_{\cap}} \R$ on the four
  manifolds under consideration, and associated pseudo-gradients such
  that we can make use of the definition of the shriek and direct maps
  given in \S\ref{sec:funct-direct-alex}
  and~\S\ref{sec:second-definition-shriek}. Let $\phi_{X}$, $\phi_{Y}$,
  $\phi_{Z}$, $\phi_{\cap}$ denote the associated flows.  The composition
  $\psi_{!}\varphi_{*}$ (resp. $\pi_{2*}\pi_{1!}$) is given by the moduli
  spaces of the form $\cM^{\varphi}(x, z)\times\cM^{\psi_{!}}(z, y)$
    (resp. $\cM^{\pi_{1!}}(x, t) \times \cM^{\pi_{2}}(t, y)$).

  Consider the homotopy induced by the moduli spaces
  \begin{align*}
  \cH_{Z}(x,y)
    &=\bigcup_{\tau>0}\{\phi_{Z}^{\tau}(\varphi (W^{u}(x)))\cap\psi (W^{s}(y))\}
  \end{align*}
  i.e. by the spaces of configurations consisting of 3 components~:
  \begin{itemize}
  \item
    a half infinite flow line in $X$ from a critical point $x$ to some
    point $p$,
  \item
    a finite piece of flow line in $Z$ from a point $p'$ to another point
    $q$,
  \item
    a half infinite flow line in $Y$ from a point $q'$ to a critical
    point $y$,
  \end{itemize}
  such that $\varphi(p)=p'$ and $q=\psi(q')$.
  

  The compactification of these moduli spaces can be described as
  follows~:
  \begin{multline}
  \overline\cH_{Z}= {\cH_Z\, \cup\, }
  \overline\cL_{X}(x,x')\times\overline\cH_{Z}(x',y)
  \cup
  \overline\cH_{Z}(x,y')\times\overline\cL_{Y}(y',y)
  \\
  \cup
  \overline\cM^{\varphi}(x,z)\times\overline\cM^{\psi!}_{Y}(z,y)
  \cup
  \overline\cH_{Z,t=0}(x,y)
  \end{multline}
  
  
  The first three terms of the boundary $\p \cH_Z$ correspond to Morse
  breaking in the infinite and finite parts, and the fourth to the limit
  $t\to 0$, which can also be described as
  \begin{align*}
   \overline\cH_{Z,t=0}(x,y)
   &= \{(p,q)\in W^{u}(x)\times W^{s}(y), \varphi(p)=\psi(q)\}.
 \end{align*}

 The analogous moduli spaces for the composition $\pi_{2*}\pi_{1!}$ are
 \begin{align*}
   \cH_{\cap}(x,y)
   &=\bigcup_{\tau>0}\{\phi_{\cap}^{\tau}\pi_{1}^{-1} W^{u}(x)\cap\pi_{2}^{-1}W^{s}(y)\},
 \end{align*}
 and the part of the compactification coresponding to the limit $t=0$ is
 \begin{align*}
 \overline\cH_{\cap,t=0}(x,y)
 = \{(p,z,q)\in W^{u}(x)\times & (X\fibreprod{\varphi}{\psi}Y)\times W^{s}(y) \, : \\
& \pi_{1}(z)=p \text{ and }\pi_{2}(z)=q\}.
 \end{align*}
 A choice of representing chains of $\cH_{Z}(x,y)$ (resp.
 $\cH_{\cap}(x,y)$) induces a homotopy from $\psi_{!}\varphi_{*}$ (resp.
 $\pi_{1!}\pi_{2*}$) to the morphism associated to the moduli spaces
 $\cH_{Z,t=0}(x,y)$ (resp. $\cH_{\cap,t=0}(x,y)$).       

 On the other hand, we have 
 \begin{align*}
   \overline\cH_{\cap,t=0}(x,y)
   &= \{(p,q)\in W^{u}(x)\times W^{s}(y) \, : \, 
     (p,q)\in X\fibreprod{\varphi}{\psi}Y\}\\
   &= \{(p,q)\in W^{u}(x)\times W^{s}(y) \, : \, 
     \varphi(p)=\psi(q)\}\\
   &=\overline\cH_{Z,t=0}(x,y).
 \end{align*}
  We conclude that at $\psi_{!}\varphi_{*}$ and $\pi_{1!}\pi_{2*}$ are
  related by a homotopy, and this proves the first relation in the
  statement.

  \medskip

  The second relation follows by symmetry.
\end{proof}

\begin{remark}
  Proposition~\ref{prop:TransverseSbmfds} can be seen as a
  generalization of Proposition~\ref{prop:phi*phi!}, which can be recovered as follows. Given a degree $d$ smooth map
  $X\xrightarrow{\varphi}Y$ between manifolds of the same dimension we
  have the following diagram~:
  $$
  \begin{tikzcd}
    (X\times Y) \fibreprod{(\varphi\times\id)}{\Delta}Y
    \arrow[d,"\pi_{1}"']\arrow[r,"\pi_{2}"]&Y\arrow[d,"\Delta"]\\
    X\times Y \arrow[r,"\varphi\times\id"]&Y\times Y,\\
  \end{tikzcd}
  $$
  where $\Delta$ is the inclusion of the diagonal. But $(X\times Y)
  \fibreprod{\varphi\times\id}{\Delta}Y \simeq X$, and with this
  identification, $\pi_{1}=\Id\times\varphi$, $\pi_{2}=\varphi$, and the
  above diagram can be written as   
  $$
  \begin{tikzcd}
    X
    \arrow[d,"\id\times\varphi"']\arrow[r,"\varphi"]&Y\arrow[d,"\Delta"]\\
    X\times Y \arrow[r,"\varphi\times\id"]\arrow[d,"\pi_{2}"]&
    Y\times Y
    \\
    Y
  \end{tikzcd}
  $$
  Proposition~\ref{prop:TransverseSbmfds} then implies that 
  $$
  (\id\times\varphi)_{*}\varphi_{!} =
  (\varphi\times\id)_{!}\Delta_{*}.
  $$
  But $\pi_{2*}(\id\times\varphi)_{*} = \varphi_{*}$ by functoriality,
  and we leave it as an exercise for the interested reader to prove that
  $\pi_{2*}(\varphi\times\id)_{!} = d\cdot  \pi_{2*}$.

  This implies
  $\varphi_{*}\varphi_{!}=d\cdot \pi_{2*}\Delta_*=d\cdot \id$.
\end{remark}

\chapter{Cohomology and Poincar\'e duality} \label{sec:cohomology-PD}

In this chapter we prove a version of Poincar\'e duality with DG coefficients for closed orientable manifolds. This result has been proved before by E.~Malm in his thesis, see~\cite[Theorem~3.1.2]{Malm-thesis} and the references therein.

\section{Cohomology} We first define Morse cohomology groups with DG local coefficients on a closed orientable manifold $M$. For that purpose we consider $C_{-*}(\Omega M)$ as a DGA of cohomological type, i.e., with differential of degree $+1$. Given an element $a\in C_*(\Omega M)$ we denote $\ol a$ the same element viewed in $C_{-*}(\Omega M)$ with opposite degree. 

\begin{definition} A \emph{cohomological DG local system on $M$}, or equivalently a \emph{cohomological right $C_{-*}(\Omega M)$-module}, is a DG module $\cG^*$ whose differential has degree $+1$ and which is endowed with a multiplication 
$$
\cG^i\otimes C_{-j}(\Omega M)\to \cG^{i+j}.
$$
\end{definition} 

\begin{example} \label{example:coh-loc-sys}
Given a right $C_*(\Omega M)$-module $\cF$ (homological DG local system), we denote $\ol\cF$ the right $C_{-*}(\Omega M)$-module defined by grading $\cF$ in opposite degree and leaving the differential unchanged:
$$
\ol \cF^i = \cF_{-i}.
$$
Then $\ol\cF$ is a cohomological DG local system. Given an element $f\in \cF$, we denote $\ol f$ the same element viewed in $\ol\cF$ and graded in opposite degree. 
\end{example} 

\begin{definition} 
	Given a Morse function $f:M\to \R$, a \emph{cohomological cocycle adapted to $f$} is a collection $(m_{y^\vee, \, x^\vee})$ indexed by $x^\vee,y^\vee\in\Crit(f)$
 with $m_{y^\vee, \, x^\vee}\in C_{-*}(\Omega M)$ of degree 
	$$
	|m_{y^\vee, \, x^\vee}| = -(|x^\vee|-|y^\vee|-1)=|y^\vee|-|x^\vee| +1,
	$$
	and satisfying the relation 
	$$
	\p m_{y^\vee, \, x^\vee} = \sum_{z^\vee} (-1)^{|y^\vee|-|z^\vee|} m_{y^\vee, \, z^\vee} m_{z^\vee, \,  x^\vee}. 
	$$
\end{definition} 	

The definition is motivated by the following example. 

\begin{example}  \label{ex:cohomological_cocycle}
	Given a Morse function $f$ and a negative pseudo-gradient vector field $\xi$, the moduli spaces $\cL^\uparrow(y,x)$ of \emph{positive} pseudo-gradient trajectories from $y$ to $x$, i.e., trajectories of $-\xi$ running from $y$ at $-\infty$ to $x$ at $+\infty$, possess compactifications $\ol\cL^\uparrow(y,x)$ which are manifolds with boundary with corners of dimension $|x|-|y|-1$. They give rise through the procedure used in the construction of the homological Barraud-Cornea cocycle to elements $m^\uparrow_{y, x}\in C_{|x|-|y|-1}(\Omega M)$ which satisfy the relation 
	$$
	\p m^\uparrow_{y, x} = \sum_z (-1)^{|y|-|z|} m^\uparrow_{y, z}m^\uparrow_{z, x}.
	$$
	When viewing $m^\uparrow_{y, x}$ as an element in $C_{-*}(\Omega M)$ of opposite degree $|y|-|x|+1$ we denote it by $m_{y^\vee, \, x^\vee}$. With the convention $|x^\vee|=|x|$, we have $|m_{y^\vee, \, x^\vee}| = |y^\vee|-|x^\vee|+1$. Hence $(m_{y^\vee, \, x^\vee})$ is a cohomological cocycle adapted to $f$. 
\end{example}

For the next example denote $I:\Omega M\to \Omega M$ the involution given by reparametrizing loops backwards, and denote $I_*$ the map induced at chain level. 
Note that $I_*$ is an anti-algebra homomorphism, meaning that $I_*(\alpha\beta)=(-1)^{|\alpha||\beta|} I_*\beta I_*\alpha$.

\begin{example} \label{example:coh-cocycle-I*}
Given a homological cocycle $\{m_{x, y}\}$, the formula 
$$
\{m_{y^\vee, \, x^\vee}=(-1)^{|x||y|+|x|+1}I_*m_{x, y}\}
$$ 
defines a cohomological cocycle, with $m_{y^\vee, \, x^\vee}$ understood to live in opposite degree $-(|x|-|y|-1)=|y^\vee|-|x^\vee|+1$. To check the cocycle relation we start with the equation $\p m_{x, y}=\sum_z (-1)^{|x|-|z|} m_{x, z}m_{z, y}$, we apply $I_*$ and use that it is an anti-algebra homomorphism to obtain 
\begin{align*}
\p m_{y^\vee, \, x^\vee} & = (-1)^{|x||y|+|x|+1} I_* \left(\sum_z (-1)^{|x|-|z|} m_{x, z}m_{z, y}\right) \\
& = \sum_z (-1)^{|x||y|+1 -|z|  + (|x|-|z|-1)(|z|-|y|-1)} I_*m_{z, y} I_* m_{x, z}\\
& = \sum_z \eps(|x|,|y|,|z|) m_{y^\vee, \, z^\vee} m_{z^\vee, \,  x^\vee} \\
& = \sum_z (-1)^{|y|+|z|} m_{y^\vee, \, z^\vee} m_{z^\vee, \,  x^\vee} \\
& = \sum_z (-1)^{|y^\vee| - |z^\vee|} m_{y^\vee, \, z^\vee} m_{z^\vee, \,  x^\vee},
\end{align*}
with $\eps(|x|,|y|,|z|)=(-1)^{|x||y|+1 -|z|  + (|x|-|z|-1)(|z|-|y|-1) + |z||y| + |z|  + |x||z|+|x|}$.
\end{example}

\begin{definition}
	Given a cohomological DG local system $\cG$ and a cohomological cocycle $(m_{y^\vee, \, x^\vee})$ adapted to the Morse function $f$, the \emph{cohomological Morse complex of $f$ with coefficients in $\cG$} is defined as 
	$$
	C^*(f;\cG) = \cG\otimes \langle \Crit(f)^\vee \rangle,
	$$
	where $\Crit(f)^\vee$ is the set of critical points of $f$ with elements denoted $x^\vee$ for each $x\in\Crit(f)$, endowed with the differential of degree $+1$ given by 
	$$
	d(\alpha\otimes x^\vee) = \p \alpha\otimes x^\vee + (-1)^{|\alpha|} \sum_{y^\vee} \alpha m_{x^\vee, \,  y^\vee} \otimes y^\vee.
	$$ 
\end{definition}

In the remaining of this section we reinterpret this definition as a derived Hom. Given a right $C_*(\Omega M)$-module $\cF$, denote $\cF^{\mathrm{left}}$ the left $C_*(\Omega M)$-module whose underlying chain complex is $\cF$ and with module structure given by 
$$
\alpha\cdot f = (-1)^{|\alpha||f|} f\cdot I_*\alpha.
$$
Recall also from Example~\ref{example:coh-loc-sys} the cohomological DG local system $\ol\cF$.

The following proposition should be compared with~\cite[Definition~3.1.1]{Malm-thesis}.
 
\begin{proposition} \label{prop:Morse-coh-as-derived-hom}
We have a chain homotopy equivalence 
$$
C^*(f;\ol\cF)\simeq R \Hom_{C_*(\Omega M)} (\Z,\cF^{\mathrm{left}}).
$$
\end{proposition}

\begin{proof}
Using the semi-free resolution of the trivial left $C_*(\Omega M)$-module $\Z$ given by the Barraud-Cornea cocycle we can write the right hand side as 
$\Hom_{C_*(\Omega M)}(C_*(\Omega M)\otimes C_*(M),\cF^{\mathrm{left}})$. To write the differential explicitly we consider an element $\ell\in \Hom_{C_*(\Omega M)}(C_*(\Omega M)\otimes C_*(M),\cF^{\mathrm{left}})$ and we let $f_x=\ell(1\otimes x)$, so that $|f_x|=|\ell|+|x|$. Then 
\begin{align*} 
(\delta \ell)& (\alpha\otimes x) \\
& = \p(\ell(\alpha\otimes x)) - (-1)^{|\ell|}\ell(\p(\alpha\otimes x))\\
& = (-1)^{|\ell||\alpha|} \p(\alpha\cdot f_x) - (-1)^{|\ell|}\ell(\p \alpha\otimes x + (-1)^{|\alpha|}\sum_y\alpha m_{x, y}\otimes y)\\
& = (-1)^{|\ell||\alpha| +|\alpha|(|\ell|+|x|)} \p(f_x\cdot I_*\alpha) - (-1)^{|\ell|+|\ell|(|\alpha|-1)} \p\alpha\cdot f_x\\
& \quad -\sum_y(-1)^{|\ell|+|\alpha| + |\ell| (|\alpha| +|x|-|y|-1)} \alpha m_{x, y}\cdot f_y\\
& = (-1)^{|\alpha||x|} \p f_x \cdot I_*\alpha\\
& \quad -\sum_y (-1)^{|\ell|+|\alpha|+|\ell||\alpha|+|\ell|(|x|-|y|-1)+(|\ell|+|y|)(|x|-|y|-1)} \alpha \cdot f_y \cdot I_* m_{x, y}\\
& = (-1)^{|\ell||\alpha| + |\alpha|}\alpha\left(\p f_x - (-1)^{|\ell|}\sum_y (-1)^{|x||y|} f_y\cdot I_* m_{x, y}\right).
\end{align*}

We now consider the cohomological Morse complex $C^*(f;\ol\cF)$ defined using the cohomological cocycle $\{m_{y^\vee, \, x^\vee}=(-1)^{|x||y|+|x|+1}I_*m_{x, y}\}$ from Example~\ref{example:coh-cocycle-I*} (this is possible since the Morse complex with DG coefficients does not depend up to chain homotopy on the choice of cocycle). We define the $\Z$-linear map 
$$
C^*(f;\ol\cF)=\ol\cF\otimes \langle \Crit(f)^\vee\rangle \stackrel\Phi\longrightarrow \Hom_{C_*(\Omega M)}(C_*(\Omega M)\otimes C_*(M),\cF^{\mathrm{left}}),
$$
$$
\ol f\otimes y^\vee\mapsto (-1)^{|y|} \ell_{f,y},
$$
where $\ell_{f,y}(1\otimes y)=f$ and $\ell_{f,y}(1\otimes x)=0$ for $x\neq y$. Thus $|\ell_{f,y}|=|y|-|f|=|y^\vee|+|\ol f|$ and $\Phi$ has degree $0$. Clearly $\Phi$ is $\Z$-linear and is an isomorphism at chain level. To finish the proof it is enough to prove that it is a chain map, and for this we compute 
\begin{align*}
\delta\Phi(\ol f\otimes y^\vee) & =(-1)^{|y|}\delta(\ell_{f,y})\\
& \mapsto
\left\{\begin{array}{lll} 
1\otimes y& \mapsto & (-1)^{|y|}\p f,\\
1\otimes x& \mapsto & -(-1)^{|y|+ |y|-|f| +|x||y|}f\cdot I_*m_{x, y},
\end{array}\right.
\end{align*}
and 
\begin{align*} 
\Phi\delta(\ol f\otimes y^\vee) & = \Phi(\delta \ol f\otimes y^\vee + (-1)^{|\ol f|}\sum_x \ol f \cdot m_{y^\vee, \, x^\vee}\otimes x^\vee)\\
& = \Phi(\ol{\p f}\otimes y^\vee + (-1)^{|f|}\sum_x (-1)^{|x||y|+|x|+1}\ol{f \cdot I_*m_{x, y}}\otimes x^\vee)\\ 
& \mapsto 
\left\{\begin{array}{lll}
1\otimes y&\mapsto& (-1)^{|y|} \p f,\\
1\otimes x&\mapsto& (-1)^{|x|+|f|+|x||y|+|x|+1} f\cdot I_*m_{x, y}.
\end{array}\right.
\end{align*}
These two expressions are equal.
\end{proof}


\section{Poincar\'e duality}

Let $M$ be a closed manifold of dimension $n$. {\bf We assume in this section that $M$ is orientable and oriented.} Let $f:M\to \R$ be a Morse function and $\cF$ a homological DG local system. Recall that $\cF$ determines a cohomological local system $\ol\cF$ as in Example~\ref{example:coh-loc-sys}. Choose a negative pseudo-gradient $\xi$, an embedded collapsing tree $\cY$, and a homotopy inverse $\theta$ for the projection $M\to M/\cY$. Define the homological Morse complex $C_*(-f;\cF)$ using $\cY,\theta$ and the cocycle $(m^{-f}_{x, y})$ determined by $-\xi$. Define the cohomological Morse complex $C^*(f;\ol \cF)$ using $\cY,\theta$ and the cocycle $(m^f_{x^\vee, \,  y^\vee})$ determined by $\xi$ as in Example~\ref{ex:cohomological_cocycle}. With these choices we have: 

\begin{proposition} \label{prop:PD-chain-level}
	There is a canonical isomorphism of chain complexes 
	$$
	PD : C_*(-f;\cF)\stackrel\simeq\longrightarrow C^{n-*}(f;\ol \cF)
	$$
defined on generators by 
$$
PD(\alpha\otimes x)=\ol\alpha\otimes x^\vee.
$$
\end{proposition}

In the previous formula $x$ is a critical point of $-f$ and hence a generator of the homological Morse complex, but it is also a critical point of $f$ and hence determines a generator $x^\vee$ of the cohomological Morse complex. Their degrees are related by 
$$
|x^\vee|=n-|x|.
$$

\begin{proof} It is clear that $PD$ is a bijection. To prove that $PD$ is a chain map, 
note that the cocycles $m_{x, y}^{-f}$ and $m_{x^\vee, \,  y^\vee}^f$ coincide except for their degrees which are opposite, so that $m_{x^\vee, \,  y^\vee}^f = \ol{m_{x, y}^{-f}}$.
We compute:
\begin{align*}
PD(\p(\alpha\otimes x)) & = PD (\p \alpha\otimes x + (-1)^{|\alpha|} \sum_y \alpha  m_{x, y}^{-f}\otimes y) \\
& = \ol{\p \alpha}\otimes x^\vee + (-1)^{|\alpha|} \sum_y \ol{\alpha m_{x, y}^{-f}}\otimes y^\vee \\
& = \p\ol\alpha \otimes x^\vee + (-1)^{|\alpha|} \sum_y \ol\alpha \cdot \ol{m_{x, y}^{-f}} \otimes y^\vee \\
& = \p\ol\alpha \otimes x^\vee + (-1)^{|\alpha|} \sum_y \ol\alpha \cdot m_{x^\vee, \,  y^\vee}^f\otimes y^\vee \\
& = d(\ol\alpha\otimes x^\vee) = d (PD(\alpha\otimes x)). 
\end{align*}
\end{proof}

In the next statement we denote by $H_*$ the Morse homology groups, and by $H^*$ the Morse cohomology groups with DG-coefficients. 

\begin{theorem}[Poincar\'e duality with DG coefficients, see also~{\cite[Theorem~3.1.2]{Malm-thesis}}] \label{thm:PD-orientable}
Let $M$ be a closed {\bf oriented} manifold of dimension $n$, let $\cF$ be a homological DG local system and denote $\ol\cF$ the cohomological local system obtained from $\cF$ by reversing the sign of the grading. The following Poincar\'e duality isomorphism holds: 
$$
PD: H_*(M;\cF) \stackrel\simeq\longrightarrow H^{n-*}(M;\ol\cF).
$$
\end{theorem}

\begin{proof}
This follows directly from Proposition~\ref{prop:PD-chain-level} using the invariance of DG Morse homology with respect to the choice of Morse function. 
\end{proof}

\section{Poincar\'e duality and shriek maps}

\begin{proposition}[map induced in cohomology] \label{prop:f_upper_*}
Let $X$, $Y$ be closed smooth manifolds. Let $\cF$ be a cohomological DG local system on $Y$. A continuous map $\varphi:X\to Y$ induces in cohomology a canonical degree $0$ morphism  
$$
\varphi^*: H^*(Y;\cF)\to H^*(X;\varphi^*\cF)
$$
with the following properties: 
\begin{enumerate} 
\item {\sc (Identity)} We have $\mathrm{Id}^*=\mathrm{Id}$. 
\item {\sc (Composition)} Given maps $X\stackrel{\varphi}\longrightarrow Y \stackrel{\psi}\longrightarrow Z$ and a cohomological DG local system $\cF$ on $Z$, we have 
$$
(\psi\varphi)^* = \varphi^*\psi^* : H^*(Z;\cF)\to H^*(X;\varphi^*\psi^*\cF).
$$
\item {\sc (Homotopy)} Homotopic maps induce equal morphisms. 
\item {\sc (Spectral sequence)} The morphism $\varphi^*$ is the limit of a morphism between the spectral sequences associated to the corresponding enriched complexes, given at the second page by 
$$ 
\varphi^{p,*}: H^{p}( Y ; H^q(\calf) )\ri H^p( X; \varphi^*H^q(\calf)).
$$ 
i.e., the map induced by $\varphi$ in cohomology with coefficients in $H^q(\calf)$. 
\end{enumerate}
\end{proposition} 

\begin{proof}
The construction of $\varphi^*$ uses the same moduli spaces as the ones used for $\varphi_*$, except that the outputs and inputs are exchanged as in Example~\ref{ex:cohomological_cocycle}. All the statements are proved as in Theorem~\ref{thm:f*!}.
\end{proof}

\begin{proposition}\label{prop:PD-shriek} Let $X^m$, $Y^n$ be closed smooth manifolds of respective dimensions $m$ and $n$. {\bf We assume that both $X$ and $Y$ are orientable and oriented.} Given a DG local system on $Y$, we have a commutative diagram 
$$
\xymatrix{ 
H_{*+m-n}(X;\varphi^*\cF) & \ar[l]_-{\varphi_!} H_*(Y;\cF) \ar[d]^{PD}_\simeq \\
H^{n-*}(X;\varphi^*\ol\cF) \ar[u]^{PD^{-1}}_\simeq & \ar[l]_-{\varphi^*} H^{n-*}(Y;\ol\cF). 
}
$$
\end{proposition}

\begin{proof}
This follows directly from the constructions of the maps $\varphi^*$, $\varphi_!$, and from the proof of Poincar\'e duality. The equality $\ol{\varphi^*\cF}=\varphi^*\ol\cF$ is obvious.
\end{proof}

One consequence of the above relationship between shriek maps and maps induced in cohomology is that the statements in Proposition~\ref{prop:f_upper_*} are equivalent to the corresponding statements for shriek maps in Theorem~\ref{thm:f*!}.

\chapter[Duality for non-orientable manifolds]{Shriek maps and Poincaré duality for non-orientable manifolds} \label{sec:non-orientable}

In this chapter we prove a Poincar\'e duality theorem on non-orientable manifolds, and we also construct shriek maps in that setting. Along the way we recast the definition of the Morse complex with DG-coefficients using the notion of orientation line.  
To the best of our knowledge, this notion was originally used by Latour~\cite{Latour} in the context of constant coefficients. It has been extensively used in recent years in Floer theory, e.g. in~\cite{Abouzaid-cotangent}. 

\section{Algebraic preliminaries} 

\subsection{Tensor product of DG local systems}

\begin{definition}[Tensor product of DG local systems] Given two right $C_*(\Omega M)$-modules $\cF_1,\cF_2$ with common ground ring $\K$, their tensor product $\cF_1\otimes_{\mathbb{K}} \cF_2$ is naturally a right $C_*(\Omega M)$-module with multiplication given by the composition
\begin{align*} 
\cF_1\otimes \cF_2\otimes & C_*(\Omega M)\stackrel{1\otimes 1 \otimes \Delta_*}\longrightarrow \cF_1\otimes \cF_2\otimes C_*(\Omega M) \otimes C_*(\Omega M) \\ 
& \stackrel{1\otimes \tau_{23}\otimes 1}\longrightarrow \cF_1\otimes C_*(\Omega M)\otimes \cF_2 \otimes C_*(\Omega M) \to \cF_1\otimes \cF_2.
\end{align*}
Here $\Delta_*$ is the diagonal map on cubical chains, $1$ denotes the identity map and $\tau_{23}$ denotes the twist on the 2nd and 3rd factor.
\end{definition} 

\begin{example}
Given a fibration $F\hookrightarrow E\to M$, we explain in~\S\ref{sec:fibrations} that it determines a unique-up-to-homotopy right $C_*(\Omega M)$-module structure on $\cF=C_*(F)$. Consider now two fibrations $E_1\longrightarrow M \longleftarrow E_2$ over the same base with fibers $F_1$, $F_2$ and corresponding right $C_*(\Omega M)$-module structures on $\cF_1=C_*(F_1)$ and $\cF_2=C_*(F_2)$. The fiber product $E_1 \times_M E_2$ is a fibration with fiber $F_1\times F_2$, and the corresponding right $C_*(\Omega M)$-module is chain homotopy equivalent to the tensor product $\cF_1\otimes \cF_2$. 
\end{example}

\subsection{Orientation lines}

We describe in this subsection a formalism for orientations that is due to Abouzaid~\cite{Abouzaid-cotangent} and is well-adapted for working with twisting local systems. We rephrase in this language our previous orientation conventions from~\S\ref{subsec: orientation-conventions}, see~\eqref{eq:transverse-orientation-new}. We use these same conventions in~\cite{BDHO-cotangent}. 

{\bf Orientation lines.} We call a $\Z$-graded free abelian group of rank 1 an \emph{orientation line}. An isomorphism of orientation lines is a graded isomorphism of the underlying free rank 1 abelian groups. In this setting we have the following canonical isomorphisms.
\begin{itemize}
\item Given two orientation lines $\ell_1,\ell_2$, their tensor product $\ell_1\otimes \ell_2$ is by definition supported in degree $\deg \ell_1 +\deg \ell_2$. The abelian group $\Z$ 
supported in degree $0$ is a neutral element for the tensor product.  
We have a canonical \emph{twist isomorphism} 
\begin{equation} \label{eq:ori_lines_twist}
\ell_1\otimes \ell_2\stackrel \simeq \longrightarrow \ell_2\otimes \ell_1
\end{equation}
given by $v_1\otimes v_2\mapsto (-1)^{\deg \ell_1 \cdot \deg \ell_2} v_2\otimes v_1$. 
\item  Given an orientation line $\ell$, we denote by $\ell^{-1}=\Hom_\Z(\ell,\Z)$ its dual, supported in degree $\deg \ell^{-1}=-\deg \ell$. There is a canonical \emph{evaluation isomorphism} 
\begin{equation} \label{eq:ori_lines_evaluation}
\ell^{-1}\otimes \ell\stackrel \simeq \longrightarrow \Z,\qquad \alpha\otimes v\mapsto \alpha(v). 
\end{equation}
Thus $\ell^{-1}$ plays the role of an inverse for $\ell$ with respect to the tensor product. 
There is a canonical isomorphism 
\begin{equation} \label{eq:ori_lines_inverse_of_product}
\ell_2^{-1}\otimes\ell_1^{-1}\simeq (\ell_1\otimes\ell_2)^{-1}, 
\end{equation}  
where a tensor product $\alpha_2\otimes\alpha_1\in\ell_2^{-1}\otimes\ell_1^{-1}$ is seen as an element of $(\ell_1\otimes \ell_2)^{-1}$ via $\langle \alpha_2\otimes \alpha_1,v_1\otimes v_2\rangle = \alpha_2(v_2) \alpha_1(v_1)$. Also, there is a canonical isomorphism 
\begin{equation} \label{eq:ori_lines_inverse_of_inverse}
\ell\simeq \left(\ell^{-1}\right)^{-1}
\end{equation}
given by $v\mapsto \left(\alpha\mapsto \alpha(v)\right)$.
\end{itemize}

{\bf Shift.} Given an orientation line $\ell$ and an integer $k\in\Z$, let $\ell[k]$ denote the same orientation line with degree shifted \emph{down} by $k$, i.e., $\ell[k]_*=\ell_{*+k}$. 

Given an orientation line $\ell$ supported in degree $\deg \ell$, we use the notation 
$$
\ul\ell=\ell[\deg \ell]
$$ 
for the orientation line seen as supported in degree $0$. 
 
{\bf Oriented orientation lines.} To orient an orientation line means to choose one of its two generators, called ``positive". An isomorphism of orientation lines  has a well-defined sign $\pm 1$ if its source and target are oriented.

Given two oriented orientation lines $\ell_1$ and $\ell_2$, we induce canonically an orientation on their tensor product $\ell_1\otimes \ell_2$ as follows: the positive generator is $v_1\otimes v_2$, where $v_1\in\ell_1$ and $v_2\in\ell_2$ are the positive generators. With this convention, the following \emph{Koszul sign rule} holds:
\begin{center}
{\it The twist isomorphism~\eqref{eq:ori_lines_twist} has sign $(-1)^{\deg\ell_1\cdot \deg\ell_2}$.}
\end{center}

The neutral element $\Z$ for the tensor product is canonically oriented by its generator $1$. As a consequence, given an orientation on $\ell$ we induce canonically an orientation on $\ell^{-1}$ by requiring that the evaluation isomorphism~\eqref{eq:ori_lines_evaluation} be orientation preserving. In other words, a generator $\alpha:\ell\to\Z$ is positive if and only if $\alpha(v)=1$ for the positive generator $v\in\ell$. (Note that the twist isomorphism $\ell^{-1}\otimes \ell\simeq \ell\otimes \ell^{-1}$ has sign $(-1)^{\deg \ell}$.)

These conventions for orienting the tensor product and the inverse imply that, given oriented lines $\ell_1$, $\ell_2$, and $\ell$, the isomorphisms~\eqref{eq:ori_lines_inverse_of_product} and~\eqref{eq:ori_lines_inverse_of_inverse} are orientation preserving.

{\bf Orientation lines of real vector spaces.} To a graded $1$-dimensional real vector space $L$ one associates canonically an orientation line $|L|$, defined as the free rank $1$ abelian group generated by the two orientations of $L$ modulo the relation that their sum vanishes, supported in $\deg L$. 

Given a finite dimensional real vector space $V$, its \emph{determinant line} is the graded $1$-dimensional real vector space $\det V=\Lambda^{\max}V$ supported in degree $\dim V$. Let $|V|=|\det V|$ be the corresponding orientation line. 

Since an orientation of $\det V$ is canonically equivalent to an orientation of $V$, we can see $|V|$ as being the free rank 1 abelian group generated by the two orientations of $V$ modulo the relation that their sum vanishes. There is a canonical isomorphism $|0|\simeq \Z$, which sends the positive orientation $+1$ of the $0$-dimensional vector space to $1$.  

{\bf Linear orientation conventions.} 

{\it Direct sum.} Given two finite dimensional real vector spaces $V$ and $W$, we induce an orientation on $V\oplus W$ from orientations of $V$ and $W$ as follows: given positive bases $(v_1,\dots,v_m)$ of $V$ and $(w_1,\dots,w_n)$ of $W$, the basis $(v_1,\dots,v_m,w_1,\dots,w_n)$ is positive. At the level of orientation lines, this is phrased as a canonical isomorphism 
\begin{equation} \label{eq:ori_lines_direct_sum}
|V|\otimes |W| \simeq |V\oplus W|. 
\end{equation}

{\it Short exact sequences.} Given a short exact sequence of finite dimensional real vector spaces 
\begin{equation} 
0\to A\to B\to C\to 0, 
\end{equation}
we induce an orientation on $B$ from orientations of $A$ and $C$ as follows: given positive bases $(a_1,\dots,a_m)$ of $A$ and $(c_1,\dots,c_n)$ of $C$, we choose lifts $\tilde c_1,\dots,\tilde c_n$ for $c_1,\dots,c_n$ and declare the basis $(a_1,\dots,a_m,\tilde c_1,\dots,\tilde c_n)$ of $B$ to be positive. This yields a canonical isomorphism 
\begin{equation} \label{eq:ori_lines_short_exact_sequence}
|A|\otimes |C|\simeq |B|. 
\end{equation}
This isomorphism can be used to induce an orientation on any of the factors $A$, $B$, $C$ from orientations of the two other factors. Note also that this orientation rule is equivalent to the one for the direct sum under the convention that a choice of splitting $C\to B$ gives rise to an isomorphism $B\simeq A\oplus C$.

{\it Transverse intersection and co-orientation.} Let $V$ be a real vector space (with no specified orientation). A \emph{co-orientation} of a subspace $F\subset V$ is an orientation of $V/F$. Let $E\subset V$ be an oriented subspace, let $F\subset V$ be a co-oriented subspace, and assume that $E$ and $F$ are \emph{transverse}, i.e., $E+F=V$. Then $E\cap F$ inherits a canonical orientation from the short exact sequence 
$$
0\to E\cap F \to E \to \frac{E}{E\cap F} \simeq \frac{E+F}{F} = \frac{V}{F}\to 0,
$$
where the isomorphism $E/E\cap F\simeq (E+F)/F$ is the canonical one. We write 
$|E\cap F|\otimes |V/F|\simeq |E|$ or, in the language of~\S\ref{subsec: orientation-conventions}, as 
\begin{equation} \label{eq:transverse-orientation-new} 
(\ori(E\cap F),\coori(F))=\ori(E).
\end{equation}
This formula is the linear analogue of~\eqref{eq:transverse-orientation} and shows that the orientation convention that we describe in this section coincides with the one from~\S\ref{subsec: orientation-conventions}.

{\bf Orientation local system of a manifold.} Given a manifold $M$, let $|M|$ be the local system of graded free rank $1$ abelian groups whose fiber at a point $p$ is the orientation line $|T_pM|$, supported in degree $n=\dim M$. We call it the \emph{orientation local system of $M$} (in this definition, we implicitly view a local system as a bundle of groups). 

{\bf Vector bundles.} Given a real vector bundle $E\to M$, let $|E^{\text{fiber}}|$ be the local system on $M$ whose fiber at $p$ is the orientation line $|E_p|$ of the fiber $E_p$. Yet another local system of interest is $|E||_M$, the restriction of $|E|$ to $M$, whose fiber at $p$ is the orientation line of the total tangent space $|T_pE|$. 

The canonical short exact sequence $0\to E_p\to T_pE \stackrel {d\pi} \longrightarrow T_pM\to 0$ gives rise to a canonical isomorphism 
$$
 |E_p|\otimes |T_pM|\simeq |T_p E|. 
$$ 
By requiring that this isomorphism preserves orientations we induce an orientation on $T_pE$ from orientations of $E_p$ and $T_pM$ (``fiber first, base second"). Pasting these canonical isomorphisms together we get an isomorphism of graded free rank 1 local systems on $M$ 
$$
 |E^{\text{fiber}}|\otimes |M|\simeq |E||_M.
$$

{\bf Manifolds with boundary.} Given a manifold $M$ with boundary, consider the normal bundle $\nu\to \p M$ along the boundary. The previous recipe provides an isomorphism $|\nu^{\text{fiber}}|\otimes |\p M|\simeq |M||_{\p M} $ of local systems on $\p M$. 

In this situation the normal bundle is trivial. {\it We trivialize the normal bundle along $\p M$ using an outward pointing vector field along the boundary $\nu^{\text{out}}$.} This determines a canonical isomorphism $|\R|\otimes |\p M|\simeq |M||_{\p M}$. 

Explicitly, we split $T_pM$ as $T_pM\simeq \R\nu_p^{\text{out}}\oplus T_p\p M$, or equivalently we project $T_pM$ onto $T_p\p M$ with kernel $\R\nu_p^{\text{out}}$, leading to the exact sequence $0\to \R\nu_p^{\text{out}}\to T_pM\to T_p\p M\to 0$. This induces a canonical isomorphism 
$$
|\R\nu_p^{\text{out}}|\otimes |T_p\p M|\simeq |T_pM|. 
$$
The canonical isomorphism $\R\nu_p^{\text{out}}\simeq \R$ induces $|\R|\otimes |\p M|\simeq |M||_{\p M}$.

\subsection{DG orientation local system on a manifold}

Let $M$ be a manifold of dimension $n$ with basepoint $\star$. The orientation local system $|M|$, described previously as a bundle of groups, can be viewed as a DG local system as follows. The underlying complex is $|T_\star M|$, supported in degree $n$, and the right $C_*(\Omega M)$-action is given as follows: 
\begin{itemize} 
\item $C_i(\Omega M)$ acts trivially for $i>0$, 
\item an element $\sum n_i \gamma_i \in C_0(\Omega M)$ acts by $(\sum n_i \mathrm{sign}(\gamma_i))\mathrm{Id}$ with $\mathrm{sign}(\gamma_i)=\pm 1$ according to whether $\gamma_i$ reverses or preserves the orientation. In particular the action of $C_0(\Omega M)$ factors through that of the group ring $\Z[\pi_1(M)]$. 
\end{itemize}

\begin{definition} We denote $\scro^M$ the above DG local system with fiber $|T_\star M|$, and call it \emph{(DG) orientation local system of $M$}. 
\end{definition} 

The local system $\scro^M$ is supported in degree $n=\dim M$, and 
$$
\ul\scro^M=\scro^M[n]
$$ 
is supported in degree $0$. In the sequel statement of Poincar\'e duality we will need to consider tensor products $\cF\otimes \ul\scro^M$, where $\cF$ is a DG local system.

\section{Morse complex and orientation lines} 

For the sequel arguments we need to rephrase in a more intrinsic way the definition of Morse homology and cohomology groups using orientation lines. 

\bigskip
{\it (i) Constant coefficients.} 
To the best of our knowledge, the point of view adopted here is originally due to Latour~\cite{Latour}, 
see also Abouzaid \cite{Abouzaid-cotangent}. 
Let $f:M\to\R$ be a Morse function, and let $\xi, \cY, \theta$ be choices of Morse-Smale pseudo-gradient vector field, collapsing tree and homotopy inverse for the projection $M\to M/\cY$. Given $x\in\Crit(f)$ we denote $\scro_x=|T_x W^u(x)|$ the orientation line of the unstable manifold at $x$. 
Given $y\in\Crit(f)$ such that $|y|=|x|-1$, any isolated gradient line $\gamma$ from $x$ to $y$ induces an isomorphism $\tau^\gamma_{x, y}:\scro_x\stackrel\sim\longrightarrow \scro_y$ determined by the isomorphism $|\R \p_s| \otimes |T_yW^u(y)|\simeq |T_xW^u(x)|$, where $\p_s$ is the $\xi$-direction along the gradient line. The Morse complex $C_*(f;\Z)$ with constant coefficients is defined to be $\bigoplus_x \scro_x$ with differential 
$$
\p|\scro_x=\sum_{|y|=|x|-1} \sum_{\gamma\in\cL(x,y)} \tau^\gamma_{x, y}.
$$
This definition recovers the one from~\S\ref{sec:construction-enriched-morse} as follows. Fix orientations of $\scro_x$, $x\in \Crit(f)$ given by generators $o_x\in\scro_x$. Then $\tau^\gamma_{x, y}(o_x)=\eps_{x,y} o_y$, and the sign is specified by requiring that the isomorphism $|\R \p_s| \otimes |T_yW^u(y)|\simeq |T_xW^u(x)|$ is orientation preserving. On the other hand, with the conventions from~\S\ref{sec:construction-enriched-morse}, and for the set of orientations $\{o_x\in \scro_x\}$, the sign $\eps'_{x,y}$ of a trajectory $\gamma$ is specified by requiring that the isomorphism $|T_yW^u(y)|\otimes |\R\langle -\xi\rangle|\simeq |T_xW^u(x)|$ is orientation preserving (Remark~\ref{orientation-rule}), and therefore differs from $\eps_{x,y}$ by $(-1)^{|x|-|y|}$. By replacing in the computation of $\eps'_{x,y}$ the orientation $o_x$ by $(-1)^{|x|}o_x$, we obtain $\eps'_{x,y}=\eps_{x,y}$. This phenomenon is similar to the one described in Appendix~\ref{app:ori-geo-an}, see Proposition~\ref{prop:C*-an-geom}.

\bigskip 
{\it (ii) DG coefficients.} To define the homological Morse complex with DG coefficients we adapt the previous construction as follows. Recall the space of parametrized Morse trajectories $\widehat{\cL}(x,y)=W^u(x)\cap W^s(y)$, oriented at a point $p$ by the exact sequence $0\to T_p W^u(x)\cap T_p W^s(y)\to T_p W^u(x)\oplus T_p W^s(y)\to T_pM\to 0$, which gives rise to a canonical isomorphism of orientation lines 
$$
|T_p \widehat{\cL}(x,y)|\otimes |T_pM|\stackrel\sim\longrightarrow |T_pW^u(x)|\otimes |T_pW^s(y)|.
$$
Combining this with the isomorphism $|T_yM|\simeq |T_yW^u(y)|\otimes |T_yW^s(y)|$, recalling the moduli space of Morse trajectories $\cL(x,y)\simeq \widehat{\cL}(x,y)/\R \p_s$, and transporting orientations from $p$ to $y$, we obtain a canonical isomorphism  
$$
|\R\p_s|\otimes |\cL(x,y)|\otimes \scro_y\stackrel\sim\longrightarrow \scro_x,
$$
which induces a canonical isomorphism  
\begin{equation} \label{eq:ori_x, y_Lx, y}
\tau_{x, y}:|\cL(x,y)|^{-1}\otimes \scro_x[1]\stackrel\sim\longrightarrow \scro_y. 
\end{equation}
This splits along connected components $\gamma$ of $\cL(x,y)$ as
$\tau_{x, y}=\oplus_\gamma \tau^\gamma_{x, y}$.
\footnote{Note that $\cL(x,y)$ is always orientable, and the
  definition of the Barraud-Cornea cocycle is the same for both
  orientable and nonorientable manifolds. Given orientations of $\scro_x$,
  $x\in \Crit(f)$, the moduli space $\cL(x,y)$ can be oriented by
  requiring that the canonical isomorphism $\tau_{x, y}$ be orientation
  preserving.}

%

Given a Barraud-Cornea cocycle $(m_{x, y})$ we write $m_{x, y}=\sum_\gamma m^\gamma_{x, y}$, where the sum runs over the connected components $\gamma$ of $\cL(x,y)$. Given a $C_*(\Omega M)$-right module $\cF$ we define 
$$
C_*(f;\cF) = \oplus_x \cF\otimes \scro_x 
$$
with differential 
$$
\p| \cF\otimes \scro_x: \, \alpha\otimes o_x\mapsto \p\alpha\otimes o_x + (-1)^{|\alpha|} \sum_y \sum_\gamma \alpha m^\gamma_{x, y} \otimes \tau^\gamma_{x, y}([m^\gamma_{x, y}]\otimes o_x).
$$ 
Here $[m^\gamma_{x, y}]$ is the orientation of $\cL^\gamma(x,y)$ determined by $m^\gamma_{x, y}$. 

In case the cocycle $(m_{x, y})$ is chosen such that $\tau_{x, y}$ is orientation preserving with respect to fixed orientations $o_x$ of $\scro_x$, $x\in \Crit(f)$, the differential becomes 
$$
\p| \cF\otimes \scro_x: \, \alpha\otimes o_x\mapsto \p\alpha\otimes o_x + (-1)^{|\alpha|} \sum_y \alpha m_{x, y} \otimes o_y
$$ 
and the above definition coincides with the one that we have used previously. 

\rmk Here is an alternative point of view for this construction. Recall that the orientation line $\scro_x$ is the free rank $1$ abelian group defined by the quotient $$\scro_x=\Z\langle o_x, \ol{o}_x\rangle/\{o_x+ \ol{o}_x=0\},$$
where $o_x$ and $\ol{o}_x$ are the two possible orientations of $W^u(x)$. Adapting the  construction of the Barraud-Cornea cocycle in \S\ref{sec:construction-enriched-morse}, we naturally get a  family of morphisms of abelian groups $\mu_{x,y}: \scro_x\otimes\scro_y \ri C_{|x|-|y|-1}(\Omega M)$ indexed by $x,y\in \Crit(f)$. Indeed, once we fix a collection of generators $o_x$ respectively for each  $\scro_x$, i.e. a collection of orientations of the unstable manifolds $W^u(x)$, the aforementioned construction yields a cocycle $(m_{x,y})$ and we set $\mu_{x,y}(o_x\otimes o_y)= m_{x,y}$.  In order to have $\mu_{x,y}$ well defined, we use the trivial observation that we may perform the inductive construction of the representing chain system $(s_{x,y})$ in the proof of Proposition~\ref{representingchain}  such that, when we change $o_x$ into $\ol{o}_x$ (or $o_y$ into $\ol{o}_y$), the corresponding chain $s_{x,y}$ changes its sign and therefore so does $m_{x,y}$.  

Within this framework, the differential  of the complex $C_*(f;\calf) =\oplus_x \cF\otimes \scro_x$ defined above writes: 
$$\partial(\alpha\otimes a) \, =\, \partial \alpha \otimes a + (-1)^{|\alpha|}\sum_y\alpha \mu_{x,y}(a\otimes o_y)\otimes o_y. 
$$
\kmr
\begin{remark} \label{rmk:ori-lines-coh}
One can rephrase in a similar way using orientation lines the cohomological Morse complex. 
\end{remark}

\section{Poincar\'e duality in the non-orientable case}

\begin{theorem}[Poincar\'e duality] \label{thm:PD-non-orientable}
Let $M$ be a closed manifold of dimension $n$, denote $\ul\scro^M$ its orientation local system supported in degree $0$, let $\cF$ be a homological DG local system and denote $\ol\cF$ the cohomological local system obtained from $\cF$ by reversing the sign of the grading. The following Poincar\'e duality isomorphism holds: 
$$
PD: HM_*(M;\cF) \stackrel\simeq\longrightarrow HM^{n-*}(M;\ol\cF\otimes \ul\scro^M).
$$
\end{theorem}

\begin{proof} We repeat the proof of Theorem~\ref{thm:PD-orientable}. The key step is to establish an analogue of Proposition~\ref{prop:PD-chain-level}, which is done in Proposition~\ref{prop:PD-chain-level-oM} below. 
\end{proof} 

We choose a Morse function $f:M\to \R$, a negative pseudo-gradient $\xi$, an embedded collapsing tree $\cY$ and a homotopy inverse $\theta$ for the projection $M\to M/\cY$. 
Denote $\scro^u_x=|T_x W^u(x)|$, $\scro^s_x=|T_x W^s(x)|$, and $\scro^M_x=|T_x M|$.  Since $T_x W^u(x)\oplus T_x W^s(x)=T_x M$, we obtain a canonical isomorphism $\scro^u_x\otimes \scro^s_x\simeq \scro^M_x$, and further a canonical isomorphism 
$$
\iota_x: \scro^s_x \stackrel\simeq\longrightarrow (\scro^u_x)^{-1}\otimes \scro^M_x.
$$

Let $\cF$ be a DG local system. We define the homological Morse complex $C_*(-f;\cF)$ using $\cY,\theta$ and the cocycle $(m^{-f}_{x, y})$ determined by $-\xi$ and the moduli spaces $\cL^{-f}(x,y)$. We define the cohomological Morse complex $C^*(f;\ol\cF\otimes \ul\scro^M)$ using $\cY,\theta$ and the cocycle $(m^f_{x^\vee, \,  y^\vee})$ determined by $\xi$ and the moduli spaces $\cL^{f\uparrow}(x,y)$ as in Example~\ref{ex:cohomological_cocycle}. We denote $\tau^M_{x, y}:\scro^M_x\stackrel\simeq\longrightarrow \scro^M_y$ the canonical isomorphism given by parallel transport of the orientation along a Morse trajectory from $x$ to $y$ (we do not include the trajectory in the notation for readability). 

We have a commutative diagram of canonical isomorphisms 
\begin{equation} \label{eq:comm-diag-taufx, y}
\xymatrix
@C=50pt
{
|\cL^{-f}(x,y)|^{-1}\otimes \scro^u_x(-f)[1] \ar[r]^-{\tau^{-f}_{x, y}} \ar@{=}[d] & \scro^u_y(-f) \ar@{=}[d] \\
|\ol\cL^{f\uparrow}(x,y)|\otimes \scro^s_x(f)[1] \ar[d]_{\iota_x} & \scro^s_y(f) \ar[d]^{\iota_y} \\
|\ol\cL^{f\uparrow}(x,y)|\otimes \scro^u_x(f)[-1]^{-1} \otimes \scro^M_x \ar[r]_-{\tau^{f\uparrow}_{x, y}\otimes \tau^M_{x, y}} & \scro^u_y(f)^{-1} \otimes \scro^M_y.
}
\end{equation}

For the next statement and proof we choose a generator $o^s_x$ of $\scro^u_x(-f)=\scro^s_x(f)$, and we denote $\iota_x(o^s_x)=(o^u_x)^{-1}\otimes o^M_x$, with $o^u_x$, $o^M_x$ generators of $\scro^u_x(f)$, $\scro^M_x$. We also denote $\tau\iota_x:\scro^u_x(-f)=\scro^s_x(f)\stackrel\simeq\longrightarrow \scro^M_x\otimes \scro^u_x(f)^{-1}$ the map $\iota_x$ postcomposed with the twist on the tensor product.

\begin{proposition} \label{prop:PD-chain-level-oM}
There is a canonical isomorphism of chain complexes 
$$
PD:C_*(-f;\cF)\stackrel\simeq\longrightarrow C^{n-*}(f;\ol\cF\otimes o^M)
$$
defined by
$$
\oplus_x \cF\otimes \scro^u_x(-f) \longrightarrow \oplus_{x^\vee} \ol\cF\otimes \scro^M\otimes \scro^u_x(f)^{-1},\qquad \alpha\otimes o^s_x\mapsto \ol\alpha\otimes \tau\iota_x(o^s_x).
$$ 
\end{proposition} 

\begin{proof}
As in the proof of Proposition~\ref{prop:PD-chain-level}, it is clear
that $PD$ is a bijection and we prove that $PD$ is a chain map. The
cocycles $m_{x, y}^{-f}$ and $m_{x^\vee, \,  y^\vee}^f$ coincide except for
their degrees which are opposite, so that $m_{x^\vee, \,  y^\vee}^f =
\ol{m_{x, y}^{-f}}$. We compute:
\begin{align*}
PD(\p& (\alpha\otimes o^s_x)) \\
& = PD \left(\p \alpha\otimes o^s_x + (-1)^{|\alpha|} \sum_y \alpha  m_{x, y}^{-f}\otimes \tau^{-f}_{x, y}(o^s_x)\right) \\
& = \ol{\p \alpha}\otimes \tau \iota_x(o^s_x) + (-1)^{|\alpha|} \sum_y \ol{\alpha m_{x, y}^{-f}}\otimes \tau\iota_y\tau^{-f}_{x, y}(o^s_x) \\
& = \p\ol\alpha \otimes  \tau \iota_x(o^s_x) + (-1)^{|\alpha|} \sum_y \ol\alpha \cdot \ol{m_{x, y}^{-f}} \otimes \tau (\tau^{f\uparrow}_{x, y}\otimes \tau^M_{x, y})\iota_x(o^s_x) \\
& = \p\ol\alpha \otimes \tau \iota_x(o^s_x) + (-1)^{|\alpha|} \sum_y \ol\alpha \cdot m_{x^\vee, \,  y^\vee}^f\otimes \tau^M_{x, y} o^M_x \otimes \tau^{f\uparrow}_{x, y} (o^u_x)^{-1} \\
& = d(\ol\alpha\otimes \iota_x(o^s_x)) \\
& = d (PD(\alpha\otimes x)). 
\end{align*}
For the third equality we used the commutativity of the diagram~\eqref{eq:comm-diag-taufx, y}. For the fifth equality we used the definition of the $C_{-*}(\Omega M)$-module structure on the tensor product $\ol\cF\otimes \ul\scro^M$, and the definition of the cohomological Morse complex for $f$ (see Remark~\ref{rmk:ori-lines-coh}). 
\end{proof}

\begin{remark}
The previous isomorphism is canonical, and this reflects the fact that the group $H_n(M;\ul\scro^M)$ has a canonical generator. This is akin to the fact that the square of an orientation line is canonically oriented: reasoning in terms of Morse theory, one can choose a Morse function on $M$ having a unique maximum $x$, and the canonical generator of $H_n(M;\ul\scro^M)$ is the positive generator of $\ul\scro_x\otimes\scro_x=\ul\scro_x\otimes \scro_x^M$, with $\scro_x$ the orientation line of $T_xW^u(x)=T_xM$. 
\end{remark}

\section{Shriek map in the non-orientable case}

\begin{proposition}[shriek map] \label{prop:f!-non-ori}
Let $X^m$, $Y^n$ be closed smooth manifolds of respective dimensions $m$ and $n$. Denote their respective orientation local systems $\ul\scro_X$ and $\ul\scro_Y$, viewed as supported in degree $0$. Let $\cF$ be a DG local system on $Y$. A continuous map $\varphi:X\to Y$ induces in homology a canonical \emph{shriek map}  
$$
\varphi_!: H_*(Y;\cF\otimes \ul\scro_Y)\to H_{*+m-n}(X;\varphi^*\cF\otimes \ul\scro_X)
$$
which has the same functoriality properties as in the orientable case. In particular,
given maps $X^m\stackrel{\varphi}\longrightarrow Y^n \stackrel{\psi}\longrightarrow Z^p$ and a DG local system $\cF$ on $Z$, we have 
$$
(\psi\varphi)_! = \varphi_!\psi_! : H_*(Z;\cF\otimes \ul\scro_Z)\to H_{*+m-p}(X;\varphi^*\psi^*\cF\otimes \ul\scro_X).
$$
\qed
\end{proposition}

The construction of the shriek map can be done either directly as in~\S\S~\ref{sec:functoriality-first-definition}-\ref{sec:second-definition}, or through Poincar\'e duality via the commutative diagram (see Proposition~\ref{prop:PD-shriek})
$$
\xymatrix{ 
H_{*+m-n}(X;\varphi^*\cF\otimes \ul\scro_X) & \ar[l]_-{\varphi_!} H_*(Y;\cF\otimes \ul\scro_Y) \ar[d]^{PD}_\simeq \\
H^{n-*}(X;\varphi^*\ol\cF) \ar[u]^{PD^{-1}}_\simeq & \ar[l]_-{\varphi^*} H^{n-*}(Y;\ol\cF). 
}
$$

The following corollary is a straightforward consequence of functoriality of shriek maps.

\begin{corollary} \label{cor:homotopy_equivalence_non-ori}
Let $\varphi:X\to Y$ be a homotopy equivalence and let $\cF$ be a DG local system on $Y$. Denote the orientation local systems on $X$, $Y$ by $\scro_X$, respectively $\scro_Y$, viewed as being supported in degree $0$. The canonical map 
$$
\varphi_!:H_*(Y;\cF\otimes \scro_Y)\to H_*(X;\varphi^*\cF\otimes \scro_X)
$$
is an isomorphism. \qed
\end{corollary}

\chapter[Beyond manifolds of finite dimension]{Beyond the case of manifolds of finite dimension} 

\label{sec:beyond-manifolds}
The example we have in mind while writing this chapter is the space $\call Q$ of free loops on a compact manifold (possibly with boundary). There are several  Morse-type theories which compute the singular  homology of $\call Q$. In the forthcoming paper~\cite{BDHO-cotangent} we will address the case of the symplectic homology of the cotangent bundle. 
Another example is the Morse theory of  the energy functional $E(\gamma)=\int ||\gamma'(t)||^2dt$  associated to a metric on~$Q$. But its generalisation to  DG-coefficients leads to some serious technical issues. First issue, the energy functional (defined on the Sobolev space $H^{1,2}(\bfs^1,Q)$, in order to work with a Banach manifold) is not ${\mathcal{C}}^\infty$ (or at least regular enough), as required by Sard's theorem in order to get manifold structures on the trajectory spaces between arbitrary critical points.  Second issue, we need to prove that the Latour cells $\ol W^u(x)$ defined by the critical points -- which have all finite index -- are homeomorphic to disks. While the first issue may be  overcome  using   the  work~\cite{AS4} of Abbondandolo  and Schwarz, for the second issue the result of Qin~\cite{Qin-Lizhen} for Morse functions  on Hilbert manifolds is proved under the hypothesis that the gradient is standard near the critical points, i.e. there exists a Morse local model.  We are not aware of a more general statement  in infinite dimensions. 

 Another approach to the Morse theory of the free loop spaces was described by Abouzaid in~\cite{Abouzaid-cotangent}. It is the one which we will generalise in this chapter. The idea is to see $\call Q$ as the direct limit (union) of the spaces $\call^r Q$ of free loops of length bounded above by $r$. For an appropriate choice of  a sequence $(r_n)$ these spaces turn out to be homotopy equivalent to compact manifolds with boundary and corners, on which the classical Morse theory is valid. Our generalisation will first treat  the case of topological spaces which are homotopy equivalent to manifolds. We then discuss in~\S\ref{sec:direct-limits} the Morse theory with DG-coefficients on direct limits, and finally in~\S\ref{sec:DG-for-loops}  the particular  case of $\call Q$.  
 
 Throughout  this chapter all  the spaces have a fixed   basepoint $\star$ and unless otherwise mentioned the applications between them preserve the basepoints.
 Before starting, let us list the expected properties of the homology $H_*(X;\calf)$ when $X$ is  a connected topological space which is not necessarily a finite dimensional manifold and $\calf$ is a DG-module over $C_*(\Omega X)$:
 \begin{enumerate} 
 \item {\it Spectral sequence}. There is a spectral sequence whose second page $E_{pq}^2$ is isomorphic to $H_p(X;H_q(\calf))$ and which converges to $H_*(X;\calf)$. 
 \item {\it Fibration}. If $\cale:F\ri E\ri X$ is a Hurewicz fibration and $\calf=C_*(F)$ then $H_*(X;\calf)\simeq H_*(E)$. 
 \item {\it Direct maps}. If $\varphi:X\ri Y$ is continuous then there exists a map $\varphi_*:H_*(X;\varphi^*\calf)\ri H_*(Y;\calf)$; these maps satisfy the properties 1-4 of~\S\ref{sec:functoriality-properties}.
 \item{\it Pullback.} Let $\cale : F\ri E_Y\ri Y$ be a Hurewicz fibration, $\calf = C_*(F)$ and $E_X$ the total space of the pullback fibration $\varphi^*\cale$, where $\varphi: X\ri Y$ is a continuous map. Let $\widetilde{\varphi}:E_X\ri E_Y$ the application induced by~$\varphi$. Then the direct map $\varphi_*$ coincides with the direct map $\widetilde{\varphi}_*$ via the isomorphisms of (ii). 
 \end{enumerate}

We will show that all these properties are satisfied by our generalisation below.

\section[Manifolds up to homotopy]{Topological spaces which are homotopy equivalent to manifolds}\label{sec:top-spaces}

 Let $u: X \ri K$ be a homotopy equivalence between a compact connected manifold $X$ (possibly with non-empty boundary) and a connected topological space $K$  endowed with a DG-module $\calf$ over $C_*(\Omega K)$. \begin{definition}\label{def:DG-for-top} The DG-homology of $K$ with coefficients in $\calf$ associated to $u$ is 
\begin{equation}\label{eq:defDG-for-top}
H_*^u(K; \calf)\, :=\, H_*(X; u^*\calf).
\end{equation} 
\end{definition}
A first remark is that, if $K$ is itself a manifold, then $u_*: H_*^u(K;\calf )\ri H_*(K;\calf)$ is an isomorphism by Theorem D, since $u$ is a homotopy equivalence. Another obvious observation is given in the next remark.

\rmk\label{rem:spectral-for-top} The homology $H_*^u (K;\calf)$ is the limit of a spectral sequence whose second page $E_{pq}^2$ is canonically isomorphic to the homology with local coefficients  $H_p(K, H_q(\calf))$. This spectral sequence is the one which computes $H_*(X; u^*\calf)$. Its second page is $E_{pq}^2=H_{p}(X, u^*H_q(\calf))$, and since $u$ is a homotopy equivalence, it yields an isomorphism $u_*: E_{pq}^2\ri H_p(K, H_q(\calf))$.
\kmr
We also have the following analogue of Theorem~A/Theorem~\ref{thm:fibration}.
\begin{proposition} \label{prop:fibration-for-top} 
 Let  $u:X\ri K$ as above, $F\ri E\ri K$ a Hurewicz fibration and $\calf=C_*(F)$ the DG-module over $C_*(\Omega K)$ defined by the choice of  a lifting function. Then $H_*^u(K;\calf)$ is isomorphic to $H_*(E)$. 
\end{proposition} 

\begin{proof} 
    If $E'$ is the total space of the pullback fibration $u^*E$ then Theorem~\ref{thm:fibration} asserts that $H^u_*(K;\calf)=H_*(X;u^*\calf)$ is isomorphic to $H_*(E')$. Since $u$ is a homotopy equivalence, it defines an isomorphism between $H_*(E')$ and $H_*(E)$. 
\end{proof} 

Now let $u, v: X\ri K$ two {\it homotopic} continuous maps and $\calf$ a DG-module over $C_*(\Omega K)$. Recall that in  Proposition~\ref{prop:homotopy-identification} from~\S\ref{sec:comments-homotopy} we defined an isomorphism $$\Psi: H_*(X;u^*\calf)\ri H_*(X; v^*\calf)$$
which we called \emph{identification isomorphism}. In the aforementioned section, $K$ was supposed to be a manifold for simplicity, but this was actually not needed in the definition of $\Psi$. We also proved in Proposition~\ref{prop-indep-homotopy-identifications} that $\Psi$ does not depend on the homotopy between $u$ and $v$ (see also Remark \ref{indep-homotopy-identifications-for-continuous}), but we do not know how to adapt this proof to the case when $K$ is not a manifold. So in the sequel we will only allow the use of properties (i)-(iii) from Proposition~\ref{prop:homotopy-identification} in the general case of topological spaces, whereas Proposition~\ref{prop-indep-homotopy-identifications} may only be used for manifolds. In the latter case we will denote the identification morphism by $\Psi^{u,v}$ to emphasize that it only depends on $u$ and $v$. As explained in \S\ref{sec:comments-homotopy} the homotopy property for direct maps $u,v:X\ri Y$ between manifolds is expressed as the commutation of the following diagram : 

$$\xymatrix{ 
H_*(X ;u^*\calf) \ar[r]^-{u_*}\ar[d]_-{\Psi^{u,v}} &H_*(Y;\calf)\\
H_*(X;v^*\calf)  \ar[ur]_-{v_*}& 
}$$


\paragraph{Direct maps.} 
We now define the direct map induced by a continuous map $\beta : K\ri L$ between topological spaces, both of which are homotopy equivalent to manifolds.  Let $u: X\ri K$ and $v: Y\ri L$ be homotopy equivalences. Choose a continuous map between manifolds $\varphi:X\ri Y$ such that the following diagram is commutative {\it up to homotopy}: 

\begin{equation}
\xymatrix{ 
X  \ar[r]^-{u}\ar[d]^-{\varphi} &K \ar[d]^-{\beta}\\
Y  \ar[r]^-{v}& L. }\label{eq:diagram-direct-map}
\end{equation}
For instance one may take $\varphi=v'\circ\beta\circ u$ where $v'$ is a homotopy inverse for $v$. We therefore have an identification  isomorphism 
$$\Psi : H_*(X,u^*\beta^*\calf)\ri H_*(X; \varphi^*v^*\calf)$$
which a priori depends on the homotopy between $\beta\circ u$ and $v\circ \varphi$. 

\begin{definition}\label{def:direct-map-for-top} We define the direct map $\beta^{u,v}_*: H_*^u(K, \beta^*\calf)\ri H^v_*(L; \calf)$ by the composition 
\begin{equation}\label{eq:def-direct-top}
\varphi_*\circ\Psi : H_*(X; u^*\beta^*\calf)\ri H_*(Y; v^*\calf),
\end{equation}
where $\varphi_*:H_*(X; \varphi^*v^*\calf)\ri H_*(Y;v^*\calf)$ is the direct map and $\Psi$ is the identification isomorphism above. 
\end{definition} 

In order to have a well-defined direct map we have to prove the following:
\begin{proposition}\label{prop:independence-direct-for-top} The map $\beta^{u,v}_*$ does not depend on $\varphi:X\ri Y$ from the homotopy diagram (\ref{eq:diagram-direct-map}), nor on the identification isomorphism $\Psi$.
\end{proposition}

Building towards the proof we start with a simple lemma: 
\begin{lemma}\label{lem:identification=id}
Let $X$ be a manifold, $\calf$ a DG-module over $C_*(\Omega X)$  and  $u: X\ri X$ a continuous map which is homotopic to the identity. Then the identification  morphism $\Psi^{u,\Id}$ equals the direct map $u_*:H_*(X;\calf)\ri H_*(X;\calf)$.
\end{lemma} 
\begin{proof} This is an immediate application of the identity property $\Id_*=\Id$ and the homotopy property (Theorem D or \ref{thm:f*!}), which in this case writes 
$$\xymatrix{ 
H_*(X ;u^*\calf) \ar[r]^-{u_*}\ar[d]_-{\Psi^{u,\Id}} &H_*(Y;\calf)\\
H_*(X;\calf)  \ar[ur]_-{\Id_*}& \, 
}$$
\end{proof}

The next lemma is a generalization of the homotopy property.
\begin{lemma}\label{lem:homotopy-generalization}
Let $v: Y\ri L$ be a homotopy equivalence between  a manifold $Y$  and a topological space $L$. Let $X$ be a manifold and $\varphi, \chi:X\ri Y$ two (continuous) maps such that $v\circ\varphi$ and $v\circ\chi$ are homotopic. Consider  a DG-module $\calf$ over $C_*(\Omega L)$. Then for any homotopy $\Upsilon$ between $v\circ\varphi$ and $v\circ\chi$  the following diagram is commutative:
$$\xymatrix{ 
H_*(X ;\varphi^*v^*\calf) \ar[r]^-{\varphi_{*}}\ar[d]_-{\Psi^{\Upsilon}} &H_*(Y;v^*\calf)\\
H_*(X;\chi^*v^*\calf)  \ar[ur]_-{\chi_{*}}& \, 
}$$
\end{lemma} 
\begin{proof} Note that if $\varphi$ and $\chi$ are homotopic  then the statement is just the usual homotopy property for these maps and the DG-module $v^*\calf$ (via  Remark \ref{identical-identifications}). Our hypothesis being weaker, we use a different argument.  Let $w:L\ri Y$ be a homotopy inverse for $v$. Consider the following diagram which only contains maps between manifolds: 
$${\scriptsize \xymatrix@C+1.2pc{ H_*(Y;(w\circ v)^*v^*\calf)\ar[r]^-{\Psi^{v\circ(w\circ v),v}}&H_*(Y; v^*\calf)\ar[r]^-{=}&H_*(Y;v^*\calf) &H_*(Y;(w\circ v)^*v^*\calf)\ar[l]_(.55){\Psi^{v\circ(w\circ v),v}}\\
H_*(X;\varphi^*(w\circ v)^*v^*\calf)\ar[u]^- {\varphi_*} \ar[r]^-{\Psi^{v\circ(w\circ v)\circ\varphi, v\circ\varphi}}\ar@/_2pc/[rrr]^-{\Psi^{v\circ(w\circ v)\circ\varphi,v\circ(w\circ v)\circ\chi}}&H_*(X;\varphi^*v^*\calf)\ar[r]^-{\Psi^\Upsilon}\ar[u]^-{\varphi_*}&H_*(X;\chi^*v^*\calf)\ar[u]^-{\chi_*}&H_*(X;\chi^*(w\circ v)^*v^*\calf)\ar[u]^- {\chi_*} \ar[l]_-{\Psi^{v\circ(w\circ v)\circ\chi, v\circ\chi}}}}
$$
Our goal is to prove that the middle square of this diagram is commutative. The left and the right squares are commutative by the naturality result Lemma~\ref{lem:naturality-for-identification}. The lower part is also commutative by Proposition~\ref{prop:homotopy-identification}, item (iii). There is another commutation relation, the one concerning  the lower curved arrow,  the leftmost and the rightmost arrows and the upper horizontal  part of the diagram. Indeed, using Remark \ref{identical-identifications} this commutation writes $$\Psi^{w\circ v, \Id}\circ\chi_*\circ\Psi^{(w\circ v)\circ\varphi, (v\circ w)\circ\chi}=\Psi^{w\circ v, \Id}\circ\varphi_*,$$
and applying Lemma~\ref{lem:identification=id} we infer that this is equivalent to 
$$(w\circ v)_*\circ \chi_*\circ \Psi^{(w\circ v)\circ\varphi, (v\circ w)\circ\chi} \, =\, (w\circ v)_*\circ\varphi_*,$$
which is exactly the homotopy property for the two homotopic applications between manifolds  $(w\circ v)\circ \varphi$ and $(w\circ v)\circ\chi$. 

Using all these commutations we easily deduce that $$\chi_*\circ\Psi^\Upsilon\circ \Psi^{v\circ(w\circ v)\circ\varphi, v\circ\varphi}\, =\, \varphi_*\circ \Psi^{v\circ(w\circ v)\circ\varphi, v\circ\varphi}.
$$
Since the identification map $\Psi^{v\circ(w\circ v)\circ\varphi, v\circ\varphi}$ is an isomorphism, we get the desired relation $\chi_*\circ\Psi^\Upsilon=\varphi_*$. This finishes the proof of the lemma. 
\end{proof}

\begin{proof}[Proof of Proposition~\ref{prop:independence-direct-for-top}] Let $\varphi : X\ri Y$ and $\chi:X\ri Y$ be two continuous maps which fit into the homotopy commutative diagram (\ref{eq:diagram-direct-map}), i.e. $\beta\circ u$ is homotopic to both $v\circ\varphi$ and $v\circ\chi$.   Choose  two homotopies between $\beta\circ u$ and these maps and denote by   $\Psi: H_*(X; u^*\beta^*\calf)\ri H_*(X; \varphi^*v^*\calf)$ and  $\Psi': H_*(X; u^*\beta^*\calf)\ri H_*(X; \chi^*v^*\calf)$ the corresponding identification morphisms; denote also by $\beta_*^{u,v}$ respectively $\beta_*^{'u,v}$ the two direct maps thus defined by \eqref{eq:def-direct-top}; our goal is to prove $\beta_*^{u,v}=\beta_*^{'u,v}$. 

Consider the homotopy between $v\circ\varphi$ and $v\circ\chi$ obtained by the concatenation of the two above and denote by $\Psi^\Upsilon: H_*(X;\varphi^*v^*\calf)\ri H_*(X;\chi^*v^*\calf)$ the corresponding identification isomorphism. Our proof is implied by the following  diagram: 
$$\xymatrix@R+1.5pc@C+1.3pc{H_*(X;\chi^*v^*\calf)\ar[dr]^-{\chi_*}&\, \\
H_*(X; u^*\beta^*\calf)\ar[u]^-{\Psi'}\ar[d]_-{\Psi}\ar@<.5ex>[r]^-{\beta_*^{'u,v}}\ar@<-.5ex>[r]_-{\beta^{u,v}_*}&H_*(Y;v^ *\calf)\\
H_*(X;\varphi^*v^*\calf)\ar[ur]_-{\varphi_*}\ar@/^4pc/[uu]^-{\Psi^\Upsilon}&\, }$$
The upper and the lower right triangles are commutative by definition of the maps $\beta^{'u,v}_*$ and $\beta^{u,v}_*$. The leftmost part is also commutative by Proposition~\ref{prop:homotopy-identification}, item (iii). Finally the previous lemma yields the commutation of the exterior part of the diagram: $\Psi^\Upsilon\circ\chi_*=\varphi_*$. We easily infer the desired equality from all these commutations: 
$$
\beta^{'u,v}_*\, =\, \chi_*\circ\Psi'\, =\, \chi_*\circ\Psi^\Upsilon\circ\Psi\, =\,  \varphi_*\circ\Psi\, =\, \beta^{u,v}_*.
$$
The proof of Proposition~\ref{prop:independence-direct-for-top} is now complete. 
\end{proof}

The direct maps $\beta^{u,v}_*$ satisfy properties 1-4 from~\S\ref{sec:functoriality-properties}, which in this case are stated as follows:
\begin{proposition}\label{prop:funct-properties-for-top} The maps $\beta^{u,v}_*$ defined above satisfy the following properties:
  \begin{itemize}
  \item[1.] Identity: $\Id^{u,u}_*=\Id_*$. 
  \item[2.] Composition: $\delta^{v,w}_*\circ\beta^{u,v}_*=(\delta\circ\beta)^{u,w}_*$. 
  \item[3.] Homotopy: If $\beta_0, \beta_1:K\ri L$ are homotopic through a homotopy $\beta$ then $\beta^{u,v}_{0,*}=\beta_{1;*}^{u,v}\circ\Psi^{\beta\circ u}$, where $$\Psi^{\beta\circ u}: H_*(X; u^*\beta_0^* \calf)\ri H_*(X; u^*\beta_1^*\calf)$$ is the identification isomorphism induced by the homotopy $\beta\circ u $. 
  \item[4.] Spectral sequence:  The direct map  $\beta^{u,v}_*$ is the limit of a morphism between the  spectral sequences defined in Remark~\ref{rem:spectral-for-top}, which at the second page $E_{pq}^2$ 
     identifies with the direct map between homologies with local coefficients 
$$\beta_p: H_*(K, \beta^*H_q(\calf))\ri H_p(L, H_q(\calf)).$$
\end{itemize}

\end{proposition} 

\begin{proof} The proof is straightforward. We will only sketch the arguments and leave the details to the reader.
  
  {\it 1. Identity. } Obvious.
  
{\it 2. Composition.} For the following diagram which is commutative in homotopy 
$$\xymatrix{X\ar[r]^-{u}\ar[d]_-{\varphi}\ar@/_2pc/[dd]_-{\zeta\circ\varphi}& K\ar[d]^-{\beta}\ar@/^2pc/[dd]^-{\delta\circ\beta}\\
Y\ar[r]^-{v}\ar[d]_-{\zeta}&L\ar[d]^-{\delta}\\
Z\ar[r]^-{w}&M}$$
we define the direct maps $\beta^{u,v}_*$, $\delta^{v,w}_*$ and $(\delta\circ\beta)^{u,w}_*$ and get the diagram 
$$\xymatrix@R+1.5pc{ H_*(X;u^*\beta^*\delta^*\calf)\ar[d]_-{\Psi} \ar[dr]^-{\beta^{u,v}_*}\ar@/^3pc/ [ddrr]^-{(\delta\circ\beta)^{u,w}_*}\ar@/_5pc/ [dd]_-{\Psi^{''}}&\, &\, \\
H_*(X; \varphi^*v^*\delta^*\calf)\ar[d]^-{\Psi'}\ar[r]^-{\varphi_*}&H_*(Y;v^*\delta^*\calf)\ar[dr]^-{\delta^{v,w}_*} \ar[d]^-{\Psi^{'''}}&\, \\
H_*(X;\varphi^*\zeta^*w^*\calf)\ar[r]^-{\varphi_*}\ar@/_2pc/ [rr]_-{(\zeta\circ\varphi)_*}&H_*(Y;\zeta^*w^*\calf)\ar[r]^-{\zeta_*} &H_*(Z;w^*\calf)}$$

The goal is to prove the commutativity  of the upper left part. All the other parts  of the diagram are commutative by definition (the upper, the lower right and the curved exterior  triangles), by Proposition~\ref{prop:homotopy-identification}.iii (the leftmost part) and by the naturality Lemma~\ref{lem:naturality-for-identification} (the square). The conclusion follows.

{\it 3. Homotopy.} For the homotopy commutative diagram 
$$\xymatrix{X\ar[r]^-{u}\ar[d]_-{\varphi}& K\ar[d]^-{\beta_0,\beta_1}\\
Y\ar[r]^-{v}&L}$$
we get 
$$\xymatrix@R+1.5pc@C+1.3pc{H_*(X;u^*\beta_0^*\calf)\ar[d]_-{\Psi}\ar[dr]^-{\beta^{u,v}_{0,*}}\ar@/_4pc/[dd]_-{\Psi^{\beta\circ u}}&\, \\
H_*(X; \varphi^*v^*\calf)\ar[r]^-{\varphi_*}&H_*(Y;v^ *\calf)\\
H_*(X;u^*\beta_1^*\calf)\ar[u]^-{\Psi'}\ar[ur]_-{\beta^{u,v}_{1,*}}&\, }$$

The goal is to prove the commutativity of the curved exterior triangle; the upper and lower right triangles are commutative (by definition) and the leftmost part is also commutative (Proposition~\ref{prop:homotopy-identification}.iii). This implies the homotopy property.

{\it 4. Spectral sequence.} Straightforward. 
\end{proof}

\rmk\label{rem:independence-direct-for-top}A consequence of this proposition is that $H_*^u(K;\calf)$ does not depend on the homotopy equivalence $u:X\ri K$. Indeed, if $v:Y\ri K$ is another one, the direct map 
$$\Id^{u,v}_*: H_*^u(K;\calf)\ri H_*^v(K; \calf)$$ associated to $\Id_K$ is an isomorphism, with inverse $\Id^{v,u}_*$.  Using these identifications  we may therefore define  $H_*(K;\calf)$. Similarly we may define the direct map $\beta_*:H_*(K;\beta^*\calf)\ri H_*(L;\calf)$ induced by $\beta:K\ri L$ without specifying the homotopy equivalences since the maps $\beta^{u,v}_*$ defined above are compatible with the identifications. Indeed, the diagram 
$$\xymatrix{H_*^{u}(K;\beta^*\calf)\ar[r]^-{\Id^{u,u'}_*}\ar[d]_-{\beta^{u,v}_*}&H_*^{u'}(K;\beta^*\calf)\ar[d]^-{\beta^{u',v'}_*}\\
H_*^{u}(L;\calf)\ar[r]^-{\Id^{v,v'}_*}&H_*^{v'}(L; \calf)}$$
is commutative by the composition property.
\kmr

We now prove the property (iv) (Pullback) above: 
\begin{proposition}\label{prop:direct-map-fibrations-for-top}Let $\beta: K\ri L$ be a continuous map between topological spaces which are homotopy equivalent to compact manifolds and let $\cale: F\ri E_L\ri L$ be a Hurewicz fibration over $L$. Denote by $\calf$ the DG-module~$ C_*(F)$ over $C_*(\Omega L)$
and by  $E_K$ the total space of the pullback fibration $\beta^*\cale$. Then  the direct map $\beta_*: H_*(K;\beta^*\calf)\ri H_*(L; \calf)$ coincides via the isomorphisms of Proposition~\ref{prop:fibration-for-top} with $\widetilde{\beta}_*: H_*(E_K)\ri H_*(E_L)$, the map induced by $\beta$ between the singular homologies of the total spaces. 
\end{proposition}
\begin{proof} We have to check the commutativity  of the diagram 
$$\xymatrix{H_*(K;\beta^*\calf)\ar[d]_-{\beta_*}\ar[r]^-{\Psi_K}_-{\simeq}&H_*(E_K)\ar[d]_-{\widetilde{\beta}_*}\\
H_*(L;\calf)\ar[r]^-{\Psi_L}_-{\simeq}& H_*(E_L)}
$$
where $\Psi_K$ and $\Psi_L$ are the isomorphisms given by Proposition~\ref{prop:fibration-for-top}. Choose two homotopy equivalences $u:X\ri K$ and $v:Y\ri L$, where $X$ and $Y$ are compact manifolds,   and a continuous map $\varphi:X \ri Y$ such that $v\circ \varphi$ and $\beta\circ u$ are homotopic. By definition $H_*(K;\beta^*\calf)= H_*(X; u^*\beta^*\calf)$, $H_*(L;\calf)= H_*(Y;v^*\calf)$, and describing the isomorphisms $\Psi_K$ and $\Psi_L$ as in the proof of Proposition~\ref{prop:fibration-for-top} the diagram above becomes: 
$$\xymatrix{H_*(X;u^*\beta^*\calf)\ar[r]^-{\Psi_X^{\beta\circ u}}_-{\simeq} \ar[d]^- {\beta^{u,v}_*}\ar@{.>}@/^2pc/[rr]^{\Psi_K}&H_*(E_X^{\beta\circ u})\ar[r]^-{\widetilde{u}_*}_-{\simeq} &H_*(E_K)\ar[d]^{\widetilde{\beta}_*}\\
H_*(Y;v^*\calf)\ar[r]^{\Psi_Y}_-{\simeq}\ar@{.>}@/_2pc/[rr]_-{\Psi_L}&H_*(E_Y)\ar[r]^{\widetilde{v}_*}_-{\simeq}&H_*(E_L)}$$
Here $E_X^{\beta\circ u}$ and $E_Y$ are the total spaces of the pullback fibrations $u^*\beta^*\cale$ resp. $v^*\cale$, the isomorphisms $\Psi_X^{\beta\circ u}$ and $\Psi_Y$ are given by Theorem \ref{thm:fibration}, and $\widetilde{u}, \widetilde{v}$ are induced by $u$ resp. $v$ between the corresponding total spaces.

In order to prove that it is commutative we complete the diagram into 

$$\xymatrix@R+1.5pc@C+1.5pc{H_*(X;u^*\beta^*\calf)\ar[r]^-{\Psi_X^{\beta\circ u}}_-{\simeq} \ar[d]^- {\beta^{u,v}_*}\ar@/_3.5pc/[dd]_-{\Psi}&H_*(E_X^{\beta\circ u})\ar[r]^-{\widetilde{u}_*}_-{\simeq} \ar[dr]^-{(\widetilde{\beta\circ u})_*}&H_*(E_K)\ar[d]^{\widetilde{\beta}_*}\\
H_*(Y;v^*\calf)\ar[r]^{\Psi_Y}_-{\simeq}&H_*(E_Y)\ar[r]^{\widetilde{v}_*}_-{\simeq}&H_*(E_L)\\
H_*(X;\varphi^*v^*\calf)\ar[r]^-{\Psi_X^{v\circ\varphi}}_-{\simeq}\ar[u]_-{\varphi_*}& H_*(E_X^{v\circ\varphi})\ar[u]_-{\widetilde{\varphi}_*}\ar[ur]_-{(\widetilde{v\circ\varphi})_*}&\, 
}$$

Note that the parts of the new addition to the diagram are commutative: the two triangles on the right side are commutative by naturality of the maps induced 	on the total spaces. On the left side we added the identification isomorphism defined by the homotopy between $\beta\circ u$ and $v\circ\varphi$; the commutativity $\beta^{u,v}_*=\varphi_*\circ\Psi$ is the   definition \eqref{eq:def-direct-top} of $\beta^{u,v}_*$. The lower square is commutative by Proposition~\ref{prop:direct-for-fibrations}, which is the version for manifolds of the current proposition. We have another commutation on the diagram, by applying Lemma \ref{lem:homotopy-invariance-fibrations} to the homotopic maps $\beta\circ u$ and $v\circ\varphi:X\ri L$ (this lemma was stated for a manifold $L$ but the proof in the case when $L$ is topological space is the same): 
$$(\widetilde{\beta\circ u})_*\circ \Psi^{\beta\circ u}_X\, =\, (\widetilde{v\circ \varphi})_*\circ \Psi^{v\circ \varphi}\circ \Psi.$$
With all these relations in hand we are able to complete the proof of our proposition. We have: 
\begin{align*}\widetilde{\beta}_*\circ\widetilde{u}_*\circ \Psi_X^{\beta\circ u}&\, = \, (\widetilde{\beta\circ u})_*\circ \Psi_X^{\beta\circ u}\, =\, 
 (\widetilde{v\circ \varphi})_*\circ \Psi^{v\circ \varphi}\circ \Psi \\
 &= \widetilde{v}_*\circ\widetilde{\varphi}_*\circ \Psi^{v\circ \varphi}\circ \Psi \, =\, 
\widetilde{v}_*\circ\Psi_Y\circ \varphi_*\circ\Psi\\
&= \widetilde{v}_*\circ\Psi_Y\circ\beta^{u,v}_*,
\end{align*} 
which is the commutativity relation claimed by our statement.
\end{proof} 

\section{Direct limits}\label{sec:direct-limits} 

Let $(X_n)_{n\in{\bf N}}$ be a collection of compact connected manifolds, or more generally of connected topological spaces which are homotopy equivalent to manifolds. We suppose that they are  endowed with basepoints and with  continuous maps $i_n:X_n\ri X_{n+1}$ which preserve these basepoints. Consider the direct system defined by $(X_n,i_n)$ and its direct limit in the category of pointed topological spaces, which we denote by $X$. We denote by $j_n:X_n\ri X$ the continuous maps defined by the direct limit; we have $j_{n+1}\circ i_n=j_n$. Let $\calf$ be a DG-module over $C_*(\Omega X)$, the chains on the space of loops based at the basepoint of $X$. 
We define the homology of $X$ with coefficients in $\calf$ by 
\begin{equation}\label{eq:homology-direct-limit} 
H_*(X;\calf)\, =\, \varinjlim H_*(X_n; j_n^*\calf),
\end{equation}
where the direct system in the right hand side is defined by the direct maps $i_{n,*}: H_{*}(X_n;i_n^*j^*_{n+1}\calf)\ri H_{*}(X_n;j^*_{n+1}\calf)$

\xmpl\label{ex:-direct-limit-manifolds} Suppose that $X_n$ is a submanifold  of $X_{n+1}$ for any   $n$ and  that  $i_n: X_n\hookrightarrow X_{n+1}$ is the inclusion. In this case $X=\bigcup_n X_n$ and we are able to describe an enriched complex whose homology is $H_*(X;\calf)$. We take a set of data $\Xi^0$ on $X_0$ and extend it to a set of data $\Xi^1$ on $X_1$ as in the definition of the direct maps for submanifolds in \S\ref{sec:funct-closed-submanifolds}; for this choice the enriched complex $C_*(X_0,\Xi^0; j_0^*\calf)$ is a subcomplex of $C_*(X_1, \Xi^1;j_1^*\calf)$. Then we extend $\Xi^1$ to $\Xi^2$ on $X_2$, then to $\Xi^3$ on $X_3$ and so on. We thus get a set of data $\Xi$ on $X$ yielding  a complex $C_*(X,\Xi;\calf)$ which can be identified with the direct limit of $C_*(X_n, \Xi^n ; j_n^*\calf)$. Its homology is $H_*(X; \calf)$, using the commutation between  direct limit and homology. 
\lpmx

\begin{proposition}\label{prop:spectral-for-limits} If $X$ and $\calf$ are as above then $H_*(X;\calf)$ is the limit of a spectral sequence whose second page is isomorphic to the homology with local coefficients $H_p(X, H_q\calf)$. 
\end{proposition}

\begin{proof} 
This is an easy consequence of the fact that the direct maps $i_{n*}: H_*(X_n;j_n^*\calf)\ri H_*(X_{n+1};j_{n+1}^*\calf)$ satisfy the spectral sequence property, i.e. they are the (spectral sequence) limits of direct maps  $i_{n,p,q}^r: E_{n,p,q}^r\ri E_{n+1,p,q}^r$ between the spectral sequences corresponding to the enriched homologies of $X_n$, resp. $X_{n+1}$, such that at the second page the map $i_{n, p, q}^2: H_p(X_n, H_q(\calf))\ri H_p(X_{n+1}, H_q(\calf))$ is the usual direct map induced by $i_n$  in homology with local coefficients. This enables us to define $$E_{p,q}^r\, =\, \varinjlim_n E_{n,p,q}^r.$$
Since direct limits commute with homology we infer that $E_{p,q}^r$ is a spectral sequence: 
$$E_{p,q}^{r+1}\, =\,  \varinjlim_n E_{n,p,q}^{r+1}\, =\,  \varinjlim_n H_*(E_{n,p,q}^r)\, =\,  H_*(\varinjlim_n E_{n,p,q}^r)\, =\, H_*(E_{p,q}^r),$$
and using the same argument 
\begin{align*}
H_*(X;\calf) & =\varinjlim_n H_*(X_n;j_n^*\calf)  \\
& =  \varinjlim_n H_* (E_{n,p,q}^\infty)  =  H_*( \varinjlim_n E_{n,p,q}^\infty)=H_* (E_{pq}^\infty).
\end{align*}
This means that the spectral sequence $E_{p,q}^r$ converges to $H_*(X;\calf)$.  Moreover $$
E_{p,q}^2\, =\,  \varinjlim_n E_{n,p,q}^2\, =\,  \varinjlim_n H_p(X_n, H_q\calf) \, =\, H_p(X, H_q(\calf)),$$
as claimed. 
\end{proof}

The next statement is the analogue of Theorem~\ref{thm:fibration} and of Proposition~\ref{prop:fibration-for-top} for direct limits: 
\begin{proposition}\label{prop:fibration-direct-limits} Let $X= \varinjlim X_n$ and $\cale: F\ri E\ri X$ a Hurewicz fibration. Denote by $\calf$ the DG-module $C_*(F)$ over $C_*(\Omega X)$. Then $H_*(X;\calf)$ is isomorphic to $H_*(E)$.
\end{proposition} 

\begin{proof} Let $i_n: X_n\ri X_{n+1}$, $j_n: X_n\ri X$ be the continuous maps associated to the direct limit, $E_n$ the total space of the pullback fibration $j_n^*\cale$ and $\widetilde{i}_n: E_n\ri E_{n+1}$ the maps induced by $i_n$. Writing $E_n=X_n \, _{j_n}\!\!\times_{\pi}E$ (where $\pi :E\ri X$ is the projection) we see that $\varinjlim E_n=E$. We infer 
$$H_*(X; \calf)\, =\, \varinjlim_{i_{n*}} H_*(X_n; j_n^*\calf)\, \simeq\, \varinjlim_{\widetilde{i}_{n*}} H_*(E_n)\, = \, H_*(\varinjlim E_n)\, =\, H_*(E).$$
Note that we used  here Proposition~\ref{prop:direct-map-fibrations-for-top} which implies  that the direct map  $i_{n*}:H_*(X_n;j_n^*\calf)\ri H_*(X ; j_{n+1}^*\calf)$ coincides with $\widetilde{i}_n: H_*(E_n)\ri H_*(E_{n+1})$ via the isomorphisms between homology with DG coefficients and homology of the total spaces. 
\end{proof}

\paragraph{Direct maps.}

Let $(X_n, i_n)$ and $(Y_n, k_n)$ be two direct systems of connected topological spaces which are homotopy equivalent to compact manifolds and denote by $X$, resp. $Y$, their direct limits. Denote by $j_n:X_n\ri X$,  $l _n:Y_n\ri Y$ the maps associated to the direct limits. Consider $(\varphi_n)$,  a family of continuous maps $\varphi_n:X_n\ri Y_n$ which define a morphism of direct systems, i.e. satisfy $\varphi_{n+1}\circ i_n=k_n\circ \varphi_n$. This defines a continuous map $\varphi=\varinjlim\varphi_n : X\ri Y$. If $\calf$ is a DG-module over  $C_*(\Omega Y)$, the direct maps 
$$\varphi_{n*}: H_*(X_n; j_n^*\varphi^*\calf)=H_*(X_n; \varphi_n^* l_n^*\calf)\ri H_*(Y_n;l_n^*\calf)$$ define a morphism of direct systems. We define the direct map as
\begin{equation}\label{eq:def-direct-for-limits}
\varphi_*: H_*(X; \varphi^*\calf)\ri H_*(Y;\calf),\qquad \varphi_*\, =\, \varinjlim \varphi_{n*}.
\end{equation}
\rmk\label{rmk:properties-map-direct-limits} The map $\varphi_*$ between  limits satisfies the four properties from~\S\ref{sec:functoriality-properties}: identity, composition, homotopy and spectral sequence. This is a straightforward consequence  of the fact that the direct maps $\varphi_{n*}$ satisfy these properties. We point out that the hypothesis we need for the homotopy property to be satisfied  is that the homotopy between two maps $\varphi, \chi: \varinjlim X_n \ri\varinjlim Y_n$ is the direct limit of homotopies $[0,1]\times X_n\ri Y_n$ between $\varphi_n$ and $\chi_n$. The identification isomorphism $\Psi: H_*(X; \varphi^*\calf)\ri H_*(X; \chi^*\calf)$ will be the direct limit of the identification isomorphisms defined by these homotopies. 
\kmr
  
  We finish this section with the following analogue of Proposition \ref{prop:direct-map-fibrations-for-top}.
  \begin{proposition}\label{rmk:map-fibration-direct-limits} Let $\varphi_n:X_n\ri Y_n$ be a family of continuous maps between direct systems $(X_n,i_n)$ and $(Y_n,k_n)$ and denote by $\varphi: X\ri Y$ the direct limit $\varinjlim\varphi_n$ between $X=\varinjlim X_n$ and $Y=\varinjlim Y_n$. Consider  a Hurewicz fibration $\cale:F\ri E_Y\ri Y$ and its pullback $\varphi^*\cale$ whose total space is denoted by $E_X$.  Then, if  $\calf$ is the DG-module $C_*(F)$,  the direct morphism $\varphi_*: H_*(X;\varphi^*\calf)\ri H_*(Y;\calf)$ coincides -- via the isomorphisms of Proposition~\ref{prop:fibration-direct-limits}, -- with the map $\widetilde{\varphi}_*:H_*(E_X)\ri H_*(E_Y)$ induced by $\varphi$ between the singular homologies of the total spaces.  \end{proposition}
  
  \begin{proof} Let $j_n:X_n\ri X$ and $l_n:Y_n\ri Y$ be the maps defined by the direct limits and denote $E_n^X$, $E_n^Y$ the total spaces of the pullback fibrations $j_n^*\varphi^*\cale$ respectively $l_n^*\cale$.   The proof of Proposition~\ref{prop:fibration-direct-limits} shows that the isomorphism $$\Psi^X:H_*(X;\varphi^*\calf) \stackrel{\simeq}{\longrightarrow} H_*(E_X)$$ is the direct limit of the isomorphisms 
  $$
  \Psi^X_n: H_*(X_n; j_n^*\varphi^*\calf) \stackrel{\simeq}{\longrightarrow} H_*(E_n^X),
  $$
  and similarly 
  $$\Psi^Y:H_*(Y;\calf) \stackrel{\simeq}{\longrightarrow} H_*(E_Y)$$ is the direct limit of  $$\Psi^Y_n: H_*(Y_n; l_n^*\calf) \stackrel{\simeq}{\longrightarrow} H_*(E_n^Y).$$
Our result is proved by passing  to the direct limit in the following  diagram,  which is commutative by Proposition~\ref{prop:direct-map-fibrations-for-top}: 
$$\xymatrix@C+1.5pc{H_*(E_n^X)\ar[r]^-{\widetilde{\varphi}_{n*}} & H_*(E^Y_n)\\
H_*(X_n; j_n^* \varphi^*\calf)\ar[r]^-{\varphi_{n*}}\ar[u]^-{\Psi_X^n}_-{\simeq}&H_*(Y_n;l_n^*\calf)\ar[u]^-{\Psi_Y^n}_-{\simeq}}$$
  \end{proof}

\section[DG Morse homology on the free loop space]{Homology with DG-coefficients on the free loop space}\label{sec:DG-for-loops}

We adapt to DG-coefficients the presentation of Abouzaid (\cite{Abouzaid-cotangent}, Chapter 11) of the twisted homology of the free loop space.  We start with the contractible loops and then explain how the definition adapts to the other connected components of the free loop space.  Let $Q$ be a compact manifold and $X=\call Q$ the space of {\it contractible} free loops on $Q$. We will suppose w.l.o.g. that they are piecewise smooth; the space $\call_{\mathrm{ps}}Q$ of such loops has the same homotopy type as $\call Q$ and, if $i:\call_{\mathrm{ps}}Q\hookrightarrow \call Q$ is the inclusion, we will define $H_*(\call Q;\calf)= H_*(\call_{\mathrm{ps}}Q; i^*\calf)$.    Fix a Riemannian metric $g$ on $Q$ and denote for each $r\geq 0$ by  $\call^r Q$ the connected  component of $\star$  in  $\call_{\mathrm{ps}}Q$  formed by loops of length $L\leq r$ . If $(r_n)$ is an increasing sequence which tends to $+\infty$ and $i_n :\call^{r_n}Q\hookrightarrow \call^{r_{n+1}}Q$ are the inclusions, then clearly $\call Q=\varinjlim \call^{r_n} Q$. 

\begin{definition}\label{def:admissible-sequence} 
We call a sequence $(r_n)$ as above \emph{admissible} if, for each $n$, the space $\call^{r_n}Q$ is homotopy equivalent to a compact manifold (possibly with boundary).
\end{definition} 

M. Abouzaid proves in \cite[Proposition~2.4]{Abouzaid-cotangent} the following result: 
\begin{proposition}\label{prop:existence-admissible} Given a metric $g$ on $Q$, there exist  sequences$\, \, (r_n)$ which are admissible and tend to infinity.

\end{proposition}
 
 The idea of the proof is that a  piecewise smooth free loop is homotopic to a  concatenation of (short)  geodesics and  the endpoints of these geodesics define  a subset of $Q^N$ for some integer $N$. We will present a sketch of the proof at the end of this section. 
 
 The two preceding sections then allow us to define the homology of the space $\call Q$ with DG-coefficients: Choose a basepoint $\star$ on $Q$ and take the constant loop $\star$ as basepoint on $\call Q$. Let $\calf$ be a $DG$ module over $C_*(\Omega\call Q)$, $g$ a metric on $Q$ and $(r_n)$ an admissible sequence. Define 
 \begin{equation}\label{eq:homology-free-loops}
 H_*(\call Q;\calf)\, =\, \varinjlim H_*(\call^{r_n}Q; j_n^*\calf) 
 \end{equation} 
 
 where $j_n:\call^{r_n}Q\ri \call Q$ are the inclusions. It is easy to see that: 
 \begin{proposition}\label{prop:independence-of-metric}
 The homology $H_*(\call Q;\calf)$ does not depend on the metric $g$, nor on the choice of an admissible sequence. 
 \end{proposition} 
 
 \begin{proof} Let $g$,  $g'$ be two metrics on $Q$ and $(r_n)$, $(r_n')$ admissible sequences corresponding to $g$, resp. $g'$. Since $Q$ is compact there is some $C>0$ such that, for any $q\in Q$ and $v\in T_qQ$, 
 $$\frac{1}{C}||v||_g\, \leq\, ||v||_{g'}\, \leq\, C||v||_g.$$
 We infer that there exist subsequences $(r'_{k_n})$ of $(r'_n)$ and $(r_{l_n})$ of $(r_n)$ such that 
 $$\call_g^{r_n}Q\, \subset\, \call^{r'_{k_n}}_{g'}Q\, \subset\, \call^{r_{l_n}}_gQ.$$
 We denote by $\varphi_n$ and $\varphi'_n$ these two inclusions. They induce morphisms of direct systems, whose limits are both the identity map $\Id:\call Q\ri \call Q$. However, they do not necessarily define the identity morphism at the level of homology with DG-coefficients (for the same kind of reason why the direct map $\Id_*: C_*(X; \Xi;\calf) \ri C_*(X, \Xi';\calf)$ is not  the identity in homology, but rather the continuation isomorphism for the data  $\Xi$ and $\Xi'$, see Proposition~\ref{id=continuation}).  We therefore use two different notations:   $$\Id_*=\varinjlim \varphi_{n*}: H_*(\call Q, g, (r_n); \calf)\ri H_*(\call Q, g', (r'_{k_n});\calf)$$ and $$\Id_{*}'=\varinjlim\varphi'_{n*}:H_*(\call Q, g', (r'_{k_n}); \calf)\ri H_*(\call Q,  g, (r_{l_n});\calf)$$
 for some chosen DG-module $\calf$ over $C_*(\Omega\call Q)$. Now direct maps satisfy the composition property (Remark \ref{rmk:properties-map-direct-limits}) and therefore $$\Id'_*\circ \Id_*=(\Id'\circ \Id)_*=\varinjlim(\varphi'_n\circ\varphi_n)_*.$$
 But $\varphi'_n\circ\varphi_n$ is the inclusion $\call^{r_n}Q\subset \call^{r_{k_n}} Q$  and by definition of the direct limit $\varinjlim(\varphi'_n\circ\varphi_n)_*$ is an isomorphism. By an analogous argument, taking a subsequence of $(k_n)$ we find some direct map $\Id''_*$ such that $\Id''_*\circ \Id'_*$ is an isomorphism. As a consequence, $\Id'_*$ is an isomorphism and therefore $\Id_*$ is also an isomorphism.  Therefore our claimed isomorphism is given by the composition \begin{equation}\label{eq:identification-for-loops}\xymatrix@C=20pt{H_*(\call Q, g, (r_n); \calf)\ar[r]_-{\Id_*}^-{\simeq}& H_*(\call Q, g', (r'_{k_n});\calf) & H_*(\call Q, g', (r'_n);\calf)\ar[l]_-{\simeq}}.\end{equation}
 It is an easy exercise to prove that the isomorphism above does not depend on the choice of the subsequence $(k'_n)$. Therefore we are allowed to use the notation $H_*(\call Q;\calf)$ and use these identification isomorphisms for different sets of data $[g, (r_n)]$ used to define this homology. 
 
 \end{proof}

\rmk\label{rmk:properties-homology-loop-space} By Proposition~\ref{prop:spectral-for-limits} we get that $H_*(\call Q; \calf)$ is the limit of a spectral sequence with second page $E_{pq}^2= H_p(\call Q; H_q(\calf))$. Also, by Proposition~\ref{prop:fibration-direct-limits} we obtain that, for a Hurewicz fibration $\cale: F\ri E\ri \call Q$, the homology $H_*(\call Q; C_*(F))$ is isomorphic to the singular homology $H_*(E)$. 
\kmr

\paragraph{Direct maps.} According to~\S\ref{sec:direct-limits} a family $\varphi=(\varphi_n)$ of continuous maps which defines a morphism between two direct systems yields a direct map in homology with DG-coefficients between their limits. We will define such a map for free loop spaces in two cases: 

\emph{Case 1}. Consider a smooth  map $\varphi:Q\ri V$ between two compact manifolds. Denote by $\call\varphi: \call Q\ri \call V$ the map induced by $\varphi$ on the free loop space. Let $\calf$ be a DG-module over $C_*(\Omega \call V)$ . Our purpose is to define a direct map $\call\varphi_*: H_*(\call Q; \call\varphi^*\calf)\ri H_*(\call V;\calf)$. 

Choose $[g^Q,(r_n^Q)]$ and $[g^V, (r_n^V)]$  sets of data (metric + admissible sequence) on these manifolds. Notice that  for any $v \in T_q Q$ we have $||d\varphi(v) ||_{g^V} \leq C\cdot ||v||_{g^Q}$, for $C=\mathrm{sup}_{q\in Q}||d_q\varphi||$. In particular, by taking a subsequence we may suppose that  the admissible sequence $(r_n^V)$ satisfies  $\call\varphi (\call^{r_n^Q}Q) \subset \call^{r_n^V}V$. Let $\call\varphi_n$ be the restriction of $\call\varphi$ to $\call\varphi (\call^{r_n^Q}Q)$;  the maps $\call\varphi_n$  define a morphism of direct systems and we obviously  have $\varinjlim\call\varphi_n= \call\varphi$. Define the direct map  $$\call\varphi_*: H_*(\call Q, g^Q, (r_n^Q);\call\varphi^*\calf)\ri H_*(\call V, g^V, (r_{k_n}^V); \calf)$$ as the direct limit of   $\call\varphi_{n*}$, as in the previous section. 

One may check that $\call\varphi_*$ is well defined, i.e. compatible with the identifications~\eqref{eq:identification-for-loops}, by considering the following commutative diagram: 
$$\xymatrix{H_*(\call Q, g^Q, (r^Q_n); \call\varphi^*\calf)\ar[r]_-{\call\varphi_*}\ar[d]_-{\Id^Q_*}^-{\simeq}& H_*(\call V, g^V, (r^V_{n});\calf) \ar[d]_-{\Id^V_*}^-{\simeq}\\
H_*(\call Q, g^{'Q}, (r^{'Q}_{j_n}); \call\varphi^*\calf)\ar[r]_-{\call\varphi_*}& H_*(\call V, g^{'V}, (r^{'V}_{p_n});\calf) \\
H_*(\call Q, g^{'Q}, (r^{'Q}_n); \call\varphi^*\calf)\ar[r]_-{\call\varphi_*}\ar[u]^-{i_*}_-{\simeq}& H_*(\call V, g^{'V}, (r^{'V}_{n});\calf) \ar[u]^-{i_*}_-{\simeq}}
$$

where the leftmost and rightmost arrows define the identifications and the subsequence $(r^{'V}_{p_n})$  satisfies $r^{'V}_{p_n}\geq \mathrm{max}(r^{'V}_{j_n}, C r^{V}_{n})$.

\rmk\label{prop:properties-direct-for-loops}  The map $\call\varphi_*$ has the following properties: 
\begin{itemize}
\item It satisfies the properties from~\S\ref{sec:functoriality-properties}:  identity, composition, homotopy and spectral sequence. These follow immediately from the corresponding properties of direct maps between direct limits (Remark ~\ref{rmk:properties-map-direct-limits}). 
\item If $\cale: F\ri E_{\call V}\ri \call V$ is a Hurewicz fibration and $E_{\call Q}$ is the total space of the pullback $\call\varphi^*\cale$ then, denoting $\calf=C_*(F)$,  we have that  $\call\varphi_*:H_*(\call Q; \call\varphi^*\calf)\ri H_*(\call V;\calf)$ coincides with $\call\varphi_*: H_*(E_Q)\ri H_*(E_V)$ via the isomorphisms of Remark \ref{rmk:properties-homology-loop-space}. This is an easy consequence of Proposition~\ref{rmk:map-fibration-direct-limits}. 
\end{itemize}
\kmr

Let us also define the direct map $\call\varphi_*$ for $\varphi:Q\ri V$ {\it continuous}. Recall the notation $\call_{\mathrm{ps}}Q$ for the piecewise smooth loops, and denote by  $i_Q: \call_{\mathrm{ps}}Q\hookrightarrow \call Q$ the inclusion. We recall that, by definition, $H_*(\call Q; \calf)= H_*(\call_{\mathrm{ps}}Q;i_Q^*\calf)$.  Choose $\ol\varphi: Q\ri V$ a smooth approximation of $\varphi$. We get a diagram
$$\xymatrix {\call_{\mathrm{ps}} Q\ar[d]_-{\call\ol\varphi}\ar[r]^-{i_Q}&\call Q\ar[d]^-{\call\varphi}\\
\call_{\mathrm{ps}}V\ar[r]^-{i_V}&\call V}$$
which is commutative in homotopy. Let $\calf$ be a DG-module over $C_*(\Omega \call V)$ and denote by  $\Psi: H_*(\call_{\mathrm{ps}} Q; i_Q^*\call\varphi^*\calf)\ri H_*(\call_{\mathrm{ps}} Q; \call\ol\varphi^*i_V^*\calf)$ the identification isomorphism defined by the above homotopy. We define the direct map $\call\varphi_* : H_*(\call_{\mathrm{ps}}Q; i_Q^*\call\varphi^*\calf) \ri H_*(\call_{\mathrm{ps}}V;i_V^*\calf)$ as the composition
$$
\call\varphi_* \, =\, \call\ol\varphi_*\circ \Psi.
$$
(Compare to \eqref{eq:def-direct-top} in Definition \ref{def:direct-map-for-top}.) The fact that this definition does not depend on the choice of the approximation $\ol\varphi$, and the fact that the properties from Remark~\ref{prop:properties-direct-for-loops} continue to hold in this case, are proved in the same way as in~\S\ref{sec:top-spaces}. We omit the details.

\medskip
\emph{Case 2}. Let $Q$ be a smooth manifold and $i:Q \hookrightarrow \call Q$  the inclusion. Consider a DG-module $\calf$ over $C_*(\Omega\call Q)$. We define the direct map $i_*:H_*(Q;i^*\calf)\ri H_*(\call Q;\calf)$.  As above, we replace $\call Q$ by the space $\call_{\mathrm{ps}}Q$ of piecewise smooth loops. We take a metric $g^Q$ on $Q$ and an admissible sequence $(r_n)$. Then, writing $Q=\varinjlim Q_n$ with $Q_n=Q$, we notice that $i: Q\ri \call Q$ may be seen as the direct limit of $i_n=i:Q_n \ri \call^{r_n}Q$. Following \eqref{eq:def-direct-for-limits} in \S \ref{sec:direct-limits}, we  define $i_*=\varinjlim i_{n*}$. As in the aforementioned section, if $E$ is the total space of a Hurewicz fibration $\cale:F\ri E\ri \call Q$ and $\calf =C_*(F)$, then $i_*$ is the direct map in singular homology $H_*(E|_Q)\ri  H_*(E)$, where $E|_Q$ is the total space of the pullback fibration $i^*\cale$. 





For further purposes it is useful to  present a

\begin{proof}[Sketch of the proof of Proposition~\ref{prop:existence-admissible}] Let $g$ be a metric on $Q$. We may suppose w.l.o.g. that its injectivity radius is larger than $4$. For any integer $n\geq 1$ consider $n$ numbers $\delta_i^n\in (0,2)$ such that $r_n=\sum_{i=1}^n\delta_i^n$ is an increasing sequence which tends to infinity. Denote by $d$ the distance on $Q$ and define 
\begin{align*}
X_n^*= \{(q_0, \ldots q_{n-1})\subset Q^n\, : \, d(q_{i-1},q_i)\leq \delta_i^n, \, \forall & i=1,\ldots, n,\\
& \mathrm{where}\,  \,  q_{n}\equiv q_0\}.
\end{align*}

For a generic choice of $\delta_i^n$ the space $X_n^*$ is a manifold with boundary and corners. Note that the results of the preceding two sections are valid for manifolds with corners: we take the (negative) gradient pointing inwards along the boundary and thus the trajectories between critical points  stay away from $\p X_n^*$. We may think of an element in $X_n^*$ as encoding the vertices of an $n$-gon in $Q$ whose edges are geodesics. Now denote by $\call_*^rQ$ the space of all free loops (not necessarily contractible)  of length less than $r$. Define $u_n: X^*_n\ri \call^{r_n}_*Q$ by mapping an element of $X_n^*$ into its corresponding polygon. Abouzaid proves in \cite{Abouzaid-cotangent} that $u_n$ is a homotopy equivalence; if $L(\gamma)$ is the length of a loop $\gamma$ (supposed piecewise smooth, as above, and parametrized at unit speed), then
$$
v_n (\gamma) = (\gamma(0), \gamma(\delta_1^n\tfrac{L(\gamma)}{r_n}),  \gamma((\delta_1^n+ \delta_2^n)\tfrac{L(\gamma)}{r_n}), \ldots, 
 \gamma((\delta_1^n+\delta_2^n+\cdots\delta_{n-1}^n)\tfrac{L(\gamma)}{r_n}))
$$
is a homotopy inverse of $u_n$. 

Considering  $X_n$, the connected component of $X_n^*$ which contains the basepoint $(\star, \star, \cdots, \star)$, we infer that $u_n: X_n\ri \call^{r_n}Q$ is a homotopy equivalence which preserves the basepoints,  as claimed.
\end{proof} 

\rmk We may also choose the numbers $\delta_i^n$ to satisfy the property $\delta_i^n~\leq~ \delta_{i+1}^{n+1}$ for all $i=1,\ldots, n$. For this choice we have a natural inclusion $\varphi_n:X_n\ri X_{n+1}$ defined by $\varphi_n(q_0,\ldots, q_{n-1})=(q_0,q_0,\ldots q_{n-1})$ and the diagram 
$$\xymatrix{X_n\ar[r]^-{u_n}\ar[d]_-{\varphi_n}&\call^{r_n}Q\ar[d]^-{i_n}\\
X_{n+1}\ar[r]^-{u_{n+1}} &\call^{r_{n+1}}Q}$$
is commutative. This gives a description of the direct system of manifolds in the limit of which we may define a Morse complex with DG-coefficients, whose homology is by definition the one of $\call Q$. 
\kmr

\rmk\label{rmk:fixed-homotopy-class} {\bf Definition for non-contractible free loops.}   Let $a$ be a free homotopy class of loops and $\call_{a} Q$ the connected component of the free loop space  defined by this class. We defined above the homology with DG-coefficients for $a=0$; the definition easily adapts to the general case. We first fix a basepoint $\gamma_a\in \call_{a}Q$. After deforming it into a piecewise geodesic loop we may suppose that it belongs to the image of $u_{n_0}(X_{n_0}^*)$ for some sufficiently large $n_0$. Let $\star_{n_0}\in X_{n_0}^*$ such that $u_{n_0}(\star_{n_0})=\gamma_a$ and take as basepoint of $X_n^*$ the point $\star_n=\varphi_{n-1} \circ\varphi_{n-2}\circ\cdots\circ\varphi_{n_0}(\star_{n_0})$ for $ n\geq n_0 +1$, where $\varphi_k$ are the maps of the previous remark. For $n\geq n_0$ we denote  by $\call_{a}^{r_n}Q$ the connected component containing $\gamma_a$ of the free loops belonging to  the class $a$ of length $\leq r$ and by $X_n^{a}$ the component of $X_n^*$ which contains  $\star_n$. We remark that the proof of Proposition~\ref{prop:existence-admissible} immediately implies that $u_n:X_n^{a}\ri \call_{a}^{r_n}$ is a homotopy equivalence which preserves the basepoints.  On the other hand we also have $\call_{a}Q=\varinjlim\call_{a}^{r_n}Q$. We may therefore define $H_*(\call_{a}Q; \calf)$  analogously to~\eqref{eq:homology-free-loops}, the direct system $(\call^{r_n}_{a}Q)$ being considered  for $n\geq n_0$. The assertions of Proposition~\ref{prop:independence-of-metric},  remain valid, as well as Remark \ref{rmk:properties-homology-loop-space}. Direct maps are also defined in a similar manner: for any continuous map $\varphi : Q\ri V$  we have $\varphi_*: H_*(\call_{a}Q;\varphi^*\calf) \ri H_*(\call_{\varphi(a)}V;\calf)$ and these maps satisfy the properties of Remark \ref{prop:properties-direct-for-loops}. \kmr

\appendix

\chapter[Geometric and analytic orientations]{Comparison of geometric and analytic orientations in Morse theory} \label{app:ori-geo-an}

Our purpose in this appendix is to compare two sets of orientations for moduli spaces of Morse trajectories. The first one, which we refer to as \emph{geometric orientation}, is the one used in the rest of this book and is based on presenting the moduli spaces of Morse trajectories as transverse intersections of submanifolds of the ambient manifold, see~\S\ref{sec:construction-enriched-morse}. The second one, which we refer to as \emph{analytic orientation}, was used in~\cite{Schwarz_book} and is based on the notion of coherent orientation for suitable spaces of Fredholm operators defined on the compactified real line. Such a comparison is needed in order to relate Floer theory on aspherical manifolds to Morse theory~\cite{BDHO-cotangent}.  

This appendix is structured as follows: we recall in~\S\ref{sec:geometric_ori} the definition of the geometric orientation, we recall in~\S\ref{sec:analytic_ori} the definition of the analytic orientation, and we compare the two in~\S\ref{sec:geometric-analytic-comparison}. To make the appendix self-contained, we recall some relevant definitions from the body of the book. 

We consider a closed manifold $X$ of dimension $d$, a Morse function $f:X\to\R$, and a Morse-Smale negative pseudo-gradient vector field for $f$, denoted by $\xi$. Given $x\in\Crit(f)$ we denote by $W^u(x)$ its \emph{unstable manifold} with respect to $\xi$ and $W^s(x)$ its \emph{stable manifold} with respect to $\xi$. That $\xi$ is Morse-Smale means that the intersection $W^u(x)\cap W^s(y)$ is transverse for all $x,y\in\Crit(f)$. 

\section{Geometric orientation}  \label{sec:geometric_ori}

Following~\S\ref{sec:construction-enriched-morse} we let 
$$
\cM(x,y)=W^u(x)\cap W^s(y)
$$ 
and call it \emph{the space of parametrized $\xi$-trajectories}. The terminology is motivated by the fact that any point $p\in\cM(x,y)$ determines a unique trajectory $\gamma:\R\to X$ such that $\dot\gamma=\xi\circ \gamma$ and $\gamma(0)=p$, which satisfies in addition $\lim_{s\to-\infty}\gamma(s)=x$, $\lim_{s\to+\infty}\gamma(s)=y$. Conversely, any trajectory $\gamma:\R\to X$ such that $\dot\gamma=\xi\circ \gamma$ and $\lim_{s\to-\infty}\gamma(s)=x$, $\lim_{s\to+\infty}\gamma(s)=y$ can be identified with the point $p=\gamma(0)$. 

Let $S^s(y)=W^s(y)\cap f^{-1}(f(y)+\eps)$ be the \emph{stable sphere of $y$}, where $\eps>0$ is chosen such that the interval $(f(y),f(y)+\eps]$ does not contain critical values of $f$. We let 
$$
\cL(x,y)=W^u(x)\cap S^s(y)
$$
and call it \emph{the moduli space of unparametrized $\xi$-trajectories}. The terminology is motivated by the fact that, given any trajectory $\gamma:\R\to X$ such that $\dot\gamma=\xi\circ \gamma$ and $\lim_{s\to-\infty}\gamma(s)=x$, $\lim_{s\to+\infty}\gamma(s)=y$, the intersection $\im\,\gamma\cap S^s(y)$ consists of a single point, which belongs to $\cL(x,y)$. Thus $\cL(x,y)$ is naturally identified with the quotient of $\cM(x,y)$ by the $\R$-action that translates a point along the unique $\xi$-trajectory passing through it. 

In~\S\ref{sec:construction-enriched-morse} we oriented the moduli spaces $\cL(x,y)$ by first choosing orientations of all the unstable manifolds $W^u(x)$, $x\in\Crit(f)$ and then requiring that the canonical isomorphism  
\begin{equation} \label{eq:iso-ori}
|\cL(x,y)|\otimes |\R\langle -\xi\rangle|\otimes |W^u(y)|\simeq |W^u(x)|
\end{equation}
be orientation preserving (see Remark~\ref{orientation-rule}). 

\begin{definition}
Given a choice of orientation of the unstable manifolds $W^u(x)$, $x\in\Crit(f)$, the resulting family of orientations of the moduli spaces $\cL(x,y)$ is called \emph{geometric orientation}. 
\end{definition}

For the purpose of this section we now recall the definition of the isomorphism~\eqref{eq:iso-ori}. The inclusion $\cL(x,y)\hookrightarrow W^u(x)$ induces at any point $p\in\cL(x,y)$ the short exact sequence 
$$
0\to T_p\cL(x,y)\to T_pW^u(x)\to T_pW^u(x)/T_p\cL(x,y) \equiv T_pX/T_pS^s(y)\to 0
$$ 
and thus determines a canonical isomorphism 
$$
|\cL(x,y)|\otimes |T_pX/T_pS^s(y)|\simeq |W^u(x)|.
$$ 
Here the identification $T_pW^u(x)/T_p\cL(x,y) \equiv T_pX/T_pS^s(y)$ is canonical and holds because the intersection $W^u(x)\cap S^s(y)$ is transverse, i.e. $T_pX=T_pW^u(x)+T_pS^s(y)$ (see also~\eqref{eq:transverse-orientation-new}). On the other hand, the inclusions $T_pS^s(y)\subset T_pW^s(y)\subset T_pX$ determine the short exact sequence 
$$
0\to T_pW^s(y)/T_pS^s(y)\to \break T_pX/T_pS^s(y)\to T_pX/T_pW^s(y)\to 0
$$ 
which gives rise to a canonical isomorphism
$$
|T_pW^s(y)/T_pS^s(y)|\otimes |T_pX/T_pW^s(y)|\simeq |T_pX/T_pS^s(y)|.
$$ 
Now $S^s(y)$ is canonically cooriented in $W^s(y)$ as the boundary of the disc $W^s(y)\cap \{f\le f(y)+\eps\}$ by its exterior normal $-\xi$, and $W^s(y)$ is canonically cooriented in $X$ by $W^u(y)$ since $T_yW^u(y)\equiv T_yX/T_y W^s(y)$. In other words we have canonical isomorphisms 
$$
|T_pW^s(y)/T_pS^s(y)|\simeq |\R\langle -\xi(p)\rangle|\quad \mbox{and} \quad |T_pX/T_pW^s(y)|\simeq |W^u(y)|. 
$$
The above isomorphisms combine into~\eqref{eq:iso-ori}.

\section{Analytic orientation} \label{sec:analytic_ori}

Given $x,y\in\Crit(f)$, the space $\cM(x,y)$ of $\xi$-trajectories connecting $x$ and $y$ can be interpreted as the zero set of the Fredholm section $\gamma\mapsto \dot\gamma-\xi(\gamma)$ of a suitable Hilbert bundle over a space of paths connecting $x$ to $y$, see~\cite[Appendix~A]{Schwarz_book}.\footnote{More precisely, one considers the space $\cP$ of paths $\R\to X$ of Sobolev class $W^{1,2}$ and asymptotic to $x$, resp. $y$, at $\pm\infty$, and the Hilbert bundle over $\cP$ with fiber  $L^2(\gamma^*TX)$, the space of $L^2$-vector fields along $\gamma$. Here the $L^2$-scalar product is understood with respect to some reference Riemannian metric on $X$.} The linearization $D_\gamma$ of this section at $\gamma\in\cM(x,y)$ is a Fredholm operator that acts on vector fields along $\gamma$. Once a reference Riemannian metric on $X$ has been chosen, the linearization can be expressed as $D_\gamma X = \nabla_s X - \nabla_X\xi(\gamma)$. The Fredholm index of $D_\gamma$ is equal to the difference of Morse indices $|x|-|y|$, and $D_\gamma$ is surjective if $\xi$ is Morse-Smale~\cite[Theorem~3.3]{Salamon_BLMS90} (surjectivity of $D_\gamma$ at all $\gamma\in\cM(x,y)$ and the Morse-Smale condition are actually equivalent). As a consequence, under the assumption that $\xi$ is Morse-Smale we obtain a canonical isomorphism of vector spaces 
$$
T_\gamma\cM(x,y)\simeq \ker D_\gamma
$$
and therefore a canonical isomorphism of orientation lines 
\begin{equation} \label{eq:TMxyDgamma}
|T_\gamma\cM(x,y)|\simeq |D_\gamma|.
\end{equation}

Choose once and for all orthogonal isomorphisms $T_xX\simeq \R^d$ for $x\in\Crit(f)$ and denote by $A_x\in Sym_d(\R)$ the $d\times d$-symmetric matrix that represents the Hessian of $f$ at $x$ in this trivialization. 
There is a unique up to homotopy orthogonal trivialization $\gamma^*TM\simeq \R\times \R^d$ that is asymptotic at $\pm\infty$ to the chosen trivializations at $x$ and $y$, and in this trivialization the operator $D_\gamma$ takes the form 
\begin{equation} \label{eq:Dgamma}
D_\gamma:W^{1,2}(\R,\R^d)\to L^2(\R,\R^d),\qquad X\mapsto \p_s X +A(s)X(s),
\end{equation}
where $A:\R\to Sym_d(\R)$ is a path of symmetric matrices such that\break $\lim_{s\to-\infty}A(s)=A_x$, the Hessian of $f$ at $x$, and $\lim_{s\to+\infty}A(s)=A_y$, the Hessian of $f$ at $y$. Since $A_x$ and $A_y$ are nondegenerate, any operator of the form~\eqref{eq:Dgamma} with fixed asymptotic behavior $A_x$ and $A_y$ is Fredholm, see~\cite[Theorem~3.3]{Salamon_BLMS90} or~\cite[Proposition~2.12]{Schwarz_book}, and the space $\cD_{x,y}$ of such operators is convex. The determinant bundle over $\cD_{x,y}$, denoted by ${\det}_{x,y}$, is therefore trivializable. As a consequence, the orientation bundle $|\cM(x,y)|$ is also trivializable because, by~\eqref{eq:TMxyDgamma}, it is isomorphic to the pullback of $|{\det}_{x,y}|$ under the map $\cM(x,y)\to \cD_{x,y}$, $\gamma\mapsto D_\gamma$. This implies that the spaces of Morse trajectories $\cM(x,y)$ are orientable. 

Given $x,z,y\in\Crit(f)$ and $R\gg 0$ there is a gluing map  
$$
\#_R:\cK_R\subset \cM(x,z)\times\cM(z,y)\to \cM(x,y),
$$
where $\cK_R\subset \cM(x,z)\times\cM(z,y)$ is a family of compact sets defined for $R$ large enough, which can be chosen to exhaust $\cM(x,z)\times\cM(z,y)$ as $R\to\infty$. This family of gluing maps induces a canonical isomorphism 
\begin{equation} \label{eq:oriMxzy}
|\cM(x,z)|\otimes |\cM(z,y)|\simeq |\cM(x,y)|. 
\end{equation}
By definition, a \emph{coherent orientation} of the moduli spaces $\cM(x,y)$ is a trivialization of $|\cM(x,y)|$ for all $x,y\in\Crit(f)$ such that the above canonical isomorphisms are orientation preserving. 

The key fact is that coherent orientations exist. This is proved in the context of Morse theory in~\cite[\S3.2]{Schwarz_book} by a method similar to the classical one in Floer theory~\cite{FH-coherent}. The starting point for the proof is that the isomorphisms~\eqref{eq:oriMxzy} are induced by canonical isomorphisms~\cite{FH-coherent,Schwarz_book} 
\begin{equation} \label{eq:detxzy}
|{\det}_{x,z}|\otimes |{\det}_{z,y}|\simeq |{\det}_{x,y}|
\end{equation}
determined by the \emph{linear gluing map}\footnote{To define the linear gluing map one first chooses a smooth nondecreasing cutoff function $\rho:\R\to[0,1]$ such that $\rho\equiv 0$ on $(-\infty,0]$ and $\rho\equiv 1$ on $[1,+\infty)$. Then, given $R>0$ and operators $D_1=\p_s + A_1(s)$ and $D_2=\p_s + A_2(s)$, one defines $D_1\#_R D_2= \p_s + A^R(s)$ with $A^R(s)=A_1(s+R)$ for $s\le -1$, $A^R(s)=(1-\rho(-s)-\rho(s))A_z + \rho(-s)A_1(s+R) + \rho(s)A_2(s-R)$ for $s\in[-1,1]$, and $A^R(s)=A_2(s-R)$ for $s\ge 1$.}
$$
\#_R:\cD_{x,z}\times\cD_{z,y}\to \cD_{x,y}, \qquad (D_1,D_2)\mapsto D_1\#_R D_2.
$$
Thus, in order to prove the existence of coherent orientations for the spaces of connecting trajectories $\cM(x,y)$, it is enough to prove the existence of orientations of the determinant bundles ${\det}_{x,y}$, $x,y\in\Crit(f)$ such that the isomorphisms~\eqref{eq:detxzy} are orientation preserving (we call such orientations \emph{coherent} as well). A key property of the isomorphisms~\eqref{eq:detxzy} is \emph{associativity}, i.e. commutativity of the diagram
$$
\xymatrix
@C=50pt
{
|{\det}_{x,z}|\otimes |{\det}_{z,y}|\otimes |{\det}_{y,w}| \ar[r]^-\simeq \ar[d]^-\simeq & |{\det}_{x,y}|\otimes |{\det}_{y,w}| \ar[d]^-\simeq\\
|{\det}_{x,z}|\otimes |{\det}_{z,w}|\ar[r]^-\simeq &  |{\det}_{x,w}|.
}
$$
(See~\cite[Theorem~6]{Schwarz_book} or~\cite[Theorem~10]{FH-coherent}.) 

Coherent orientations on the determinant bundles ${\det}_{x,y}$ can be constructed in various ways, but the overall scheme is predetermined: an initial choice of orientations for \emph{some} determinant lines induces via~\eqref{eq:detxzy} orientations for \emph{all} determinant lines, and the resulting set of orientations is coherent by the associativity property. The recipe from~\cite{Schwarz_book}, which is analogous to the one for Cauchy-Riemann operators from~\cite{FH-coherent}, stays in the realm of operators defined on the real line: it starts with the choice of orientations of ${\det}_{x_0,y}$ with $x_0\in\Crit(f)$ fixed and $y\in\Crit(f)$ arbitrary, such that the orientation of ${\det}_{x_0,x_0}$ is the canonical one. These induce orientations for ${\det}_{y,x_0}$, which further induce coherent orientations of ${\det}_{x,y}$. 

We will proceed in a slightly different manner, inspired by the recipe for Cauchy-Riemann operators from~\cite{Bourgeois-Mohnke,BOauto}, and deduce coherent orientations from a choice of orientations of determinant bundles of Fredholm operators defined over half-lines. More precisely, we consider operators of the form~\eqref{eq:Dgamma} with domain of definition $(-\infty,0]$ and asymptotic condition imposed only at $-\infty$, i.e., 
$$
D^u_x:W^{1,2}((-\infty,0],\R^d)\to L^2((-\infty,0],\R^d),\qquad X\mapsto \p_s X +A(s)X(s),
$$
where $A:(-\infty,0]\to Sym_d(\R)$ is a path of symmetric matrices such that $\lim_{s\to-\infty}A(s)=A_x$. The arguments of~\cite[Theorem~3.3]{Salamon_BLMS90} or~\cite[Proposition~2.12]{Schwarz_book} adapt in a straightforward way to show that the operator $D^u_x$ is Fredholm of index $\ind D^u_x=|x|$, the Morse index of $x$. We denote by $\cD^u_x$ the space of such operators and note as before that it is convex, so that its determinant bundle, denoted by ${\det}_x$, is orientable. 

\begin{remark}
Although we will not use them in the sequel, we introduce for completeness two other contractible spaces consisting of Fredholm operators of the form~\eqref{eq:Dgamma}. The space $\cD^s_x$ consists of operators 
$$
D^s_x:W^{1,2}([0,+\infty),\R^d)\to L^2([0,+\infty),\R^d),\qquad X\mapsto \p_s X +A(s)X(s),
$$
where $A:[0,+\infty)\to Sym_d(\R)$ is a path of symmetric matrices such that $\lim_{s\to+\infty}A(s)=A_x$. Any operator $D^s_x$ is Fredholm of index $\ind D^s_x=d-|x|$. 
The space $\cD$ consists of operators 
$$
D:W^{1,2}(I,\R^d)\to L^2(I,\R^d),\qquad X\mapsto \p_s X +A(s)X(s),
$$
where $I\subset \R$ is a closed interval of nonzero length and $A:I\to Sym_d(\R)$ is a path of symmetric matrices. Any operator $D$ is Fredholm of index $\ind D=d$.
\end{remark}

We now choose an orientation of ${\det}_x$ for each $x\in\Crit(f)$ and induce from it a coherent orientation of the determinant lines ${\det}_{x,y}$ for all $x,y\in\Crit(f)$ as follows. There is again a family of linear gluing maps 
$$
\#_R:\cK_R\subset \cD_{x,y}\times \cD^u_y\to \cD^u_x
$$
defined for $R\gg 0$, which induces a canonical isomorphism 
\begin{equation} \label{eq:detxydetxdety}
|{\det}_{x,y}|\otimes |{\det}_y|\simeq |{\det}_x|.
\end{equation}
Having chosen orientations of ${\det}_x$ and ${\det}_y$, there is a unique orientation of ${\det}_{x,y}$ such that this isomorphism is orientation preserving. Associativity of gluing now takes the form of the commutative diagram  
$$
\xymatrix
@C=50pt
{
|{\det}_{x,z}|\otimes |{\det}_{z,y}|\otimes |{\det}_y| \ar[d] \ar[r]^-\simeq & |{\det}_{x,z}|\otimes |{\det}_z| \ar[d]^\simeq \\
|{\det}_{x,y}|\otimes |{\det}_y| \ar[r]^-\simeq &  |{\det}_x|
}
$$
By definition the horizontal arrows and the right vertical arrow in the above diagram are orientation preserving, and therefore the left vertical arrow is also orientation preserving, i.e., the canonical isomorphism~\eqref{eq:oriMxzy} is orientation preserving. This proves that the resulting system of orientations on ${\det}_{x,y}$ is coherent. 

We infer a coherent system of orientations for the spaces of trajectories $\cM(x,y)$, and these induce a system of orientations for the moduli spaces $\cL(x,y)=\cM(x,y)/\R$, with $\R$-action determined by the infinitesimal generator $\xi$, by requiring that the canonical isomorphism 
\begin{equation} \label{eq:RLM}
|\R\langle\xi\rangle|\otimes |\cL(x,y)|\simeq |\cM(x,y)|
\end{equation}
be orientation preserving. 

\begin{definition}
Given a choice of orientations of the determinant bundles ${\det}_x$, $x\in\Crit(f)$, the resulting family of orientations of the moduli spaces $\cL(x,y)$ is called \emph{analytic orientation}. 
\end{definition}

\begin{remark}
While our recipe for constructing coherent orientations starting from choices of orientations of the determinant lines ${\det}_x$, $x\in\Crit(f)$ can be rephrased in the terms of~\cite[\S3.2]{Schwarz_book}, its advantage is that it is directly suited for a comparison with the geometric orientation of the moduli spaces.   
\end{remark}

\section{Comparison between the two orientations} \label{sec:geometric-analytic-comparison}

By construction, the geometric orientation of the moduli spaces $\cL(x,y)$ depends on a choice of orientations of the unstable manifolds $W^u(x)$, $x\in\Crit(f)$, and the analytic orientation depends on a choice of orientations of the determinant bundles ${\det}_x$, $x\in\Crit(f)$. A first step in comparing the geometric and analytic orientations is to show that these choices are canonically equivalent. 

\begin{lemma}
A choice of orientation of the unstable manifold $W^u(x)$ for some $x\in\Crit(f)$ is canonically equivalent to a choice of orientation of the determinant bundle ${\det}_x$.  
\end{lemma}

\begin{proof} A choice of orientation of ${\det}_x$ is equivalent to a choice of orientation of any of its fibers, so that it is enough to study the operator $D^u_x$ given by a constant path $A(s)\equiv A_x$. Denote by $\lambda_1,\dots,\lambda_d$ the eigenvalues of $A_x$, ordered such that $\lambda_1,\dots,\lambda_k<0$ and $\lambda_{k+1},\dots,\lambda_d>0$ with $k=|x|$, and denote by $E_1,\dots,E_d$ an orthonormal basis of eigenvectors. The operator $D^u_x$ is surjective (this can be deduced for example by a direct truncation argument from the bijectivity of the operator $D_{x,x}=\p_s+A_x \in\cD_{x,x}$~\cite[Theorem~3.3]{Salamon_BLMS90}), so that $\det D^u_x=\Lambda^{\max} \ker D^u_x$. A general solution to the equation $D_{x,x}X=0$ is $X(s)=\sum_{i=1}^dc_ie^{-\lambda_i s}E_i$ with $c_i\in\R$. This map belongs to $W^{1,2}((-\infty,0],\R^d)$ if and only if $c_{k+1}=\dots=c_d=0$, so that $\ker D^u_x=\mathrm{Vect}\langle e^{-\lambda_i s}E_i\, : \, i=1,\dots,k\rangle$. An orientation of $\ker D^u_x$ is therefore equivalent to an orientation of $\Vect\langle E_1,\dots,E_k\rangle$, the negative eigenspace of $A_x$, and this is in turn equivalent to an orientation of $W^u(x)$.
\end{proof}

\begin{proposition} \label{prop:comparison-or-geom-analytic}
Given equivalent choices of orientations of $W^u(x)$ and ${\det}_x$ for $x\in\Crit(f)$, the induced geometric and analytic orientations of the moduli spaces $\cL(x,y)$ differ by a sign equal to $(-1)^{|x|-|y|}$. 
\end{proposition}

\begin{proof}
The geometric orientation is determined by requiring that the canonical isomorphism
\begin{equation} \label{eq:geometric-ori-app}
|\cL(x,y)|\otimes |\R\langle -\xi\rangle|\otimes |W^u(y)|\simeq |W^u(x)| 
\end{equation}
is orientation preserving. 

The analytic orientation is determined by requiring that the canonical isomorphisms~\eqref{eq:detxydetxdety} and~\eqref{eq:RLM} are orientation preserving, hence requiring that the canonical isomorphism
$
|\R\langle\xi\rangle|\otimes |\cL(x,y)|\otimes |{\det}_y|\simeq |{\det}_x|
$
is orientation preserving. Under the equivalence of orientations between $W^u(x)$ and ${\det}_x$ for $x\in \Crit(f)$, this is equivalent to requiring that the canonical isomorphism 
\begin{equation} \label{eq:analytic-ori-app}
|\R\langle\xi\rangle|\otimes |\cL(x,y)|\otimes |W^u(y)|\simeq |W^u(x)|
\end{equation}
is orientation preserving. The definitions of the canonical isomorphisms~\eqref{eq:analytic-ori-app} and~\eqref{eq:geometric-ori-app} imply that they only differ by rearranging the factors, and therefore the analytic orientation of $\cL(x,y)$ differs from the geometric orientation by the sign $(-1)\cdot (-1)^{\dim\cL(x,y)}=(-1)^{|x|-|y|}$.  
\end{proof}

\section{Comparing the homologies}

\begin{corollary} \label{cor:o-obar}
Let $\{o_x\, : \, x\in\Crit(f)\}$ be a choice of orientations of $W^u(x)$ or, equivalently, of ${\det}_x$. The analytic orientation of the moduli spaces $\cL(x,y)$ determined by $\{o_x\}$ is equal to the geometric orientation determined by $\{(-1)^{|x|}o_x\}$. 
\end{corollary}

\begin{proof}
This follows directly from Proposition~\ref{prop:comparison-or-geom-analytic} and the geometric orientation rule~\eqref{eq:geometric-ori-app}. 
\end{proof}

We now explain how the corollary implies that Morse homology with DG local coefficients does not depend on whether we choose geometric- or analytic orientations. 
\begin{center}
Given $o=\{o_x\}$, we set $\bar o=\{(-1)^{|x|}o_x\}$. 
\end{center}
Let $\Xi=(f,\xi,o, s^{\text{geom}}_{x,y}(o),\cY,\theta)$ be Morse data, where $f$ is a Morse function, $\xi$ is a Morse-Smale negative pseudo-gradient vector field, $o=\{o_x\}$ is a choice of orientations for $W^u(x)$, $\{s^{\text{geom}}_{x,y}(o)\}$ is a representing chain system for the compactified moduli spaces $\ol\cL(x,y)$ endowed with the \emph{geometric} orientation induced by $o$, $\cY$ is a collapsing tree, and $\theta$ is a homotopy inverse for the collapsing map $X\to X/\cY$.

We denote $s^{\text{an}}_{x,y}(o)=(-1)^{|x|-|y|}s^{\text{geom}}_{x,y}(o)$. By Proposition~\ref{prop:comparison-or-geom-analytic}, this is a representing chain system for the compactified moduli spaces $\ol\cL(x,y)$ endowed with the \emph{analytic} orientation induced by $o$. On the other hand, by Corollary~\ref{cor:o-obar} we have $s^{\text{an}}_{x,y}(o)=s^{\text{geom}}_{x,y}(\bar o)$. 

Twisting cocycles $m_{x,y}\in C_{|x|-|y|-1}(\Omega X)$ are obtained from representing chain systems $s_{x,y}$ using the collapsing tree $\cY$ and the homotopy inverse $\theta$. Accordingly, we obtain twisting cocycles $m^{\text{geom}}_{x,y}(o)$, $m^{\text{an}}_{x,y}(o)$ and $m^{\text{geom}}_{x,y}(\bar o)$ and we have 
$$
(-1)^{|x|-|y|}m^{\text{geom}}_{x,y}=m^{\text{an}}_{x,y}(o)=m^{\text{geom}}_{x,y}(\bar o).
$$
This implies the following.
\begin{proposition} \label{prop:C*-an-geom} Given a DG-local system $\cF$, we have an equality of chain complexes
$$
FC_*(X,f,\xi,o,s^{\text{an}}_{x,y}(o),\cY,\theta;\cF)= FC_*(X,f,\xi,\bar o,s^{\text{geom}}_{x,y}(\bar o),\cY,\theta;\cF).
$$
\qed
\end{proposition}

\bibliographystyle{alpha}
\bibliography{000_dgfloer}

%
%
\bigskip

{\small

\medskip
\noindent Jean-François Barraud \\
\noindent Institut de mathématiques de Toulouse, Université Paul Sabatier – Toulouse
III, 118 route de Narbonne, F-31062 Toulouse Cedex 9, France.\\
{\it e-mail:} barraud@math.univ-toulouse.fr
\medskip

\medskip
\noindent Mihai Damian \\
\noindent Université de Strasbourg, Institut de recherche mathématique avancée, IRMA, Strasbourg, France.\\
{\it e-mail:} damian@math.unistra.fr 
\medskip

\medskip
\noindent Vincent Humili\`ere \\
\noindent Sorbonne Université and Université de Paris, CNRS, IMJ-PRG, F-75006 Paris, France\\
\noindent \& Institut Universitaire de France.\\
{\it e-mail:} vincent.humiliere@imj-prg.fr
\medskip

\medskip
\noindent Alexandru Oancea \\
\noindent Université de Strasbourg, Institut de recherche mathématique avancée, IRMA, Strasbourg, France.\\
{\it e-mail:} oancea@unistra.fr 
\medskip

}

\end{document}